\newif\ifjournalstyle
\journalstyletrue
\journalstylefalse

\ifdefined\articleversion
  \journalstylefalse
\fi
\ifdefined\journalversion
  \journalstyletrue
\fi
\ifjournalstyle
  \documentclass[envcountsect,envcountsame,orivec,oribibl]{svproc}

\else
  \documentclass{article}
\fi
\pdfoutput=1

\usepackage{graphicx}       
\usepackage{subfig} 

\usepackage[T1]{fontenc}
\usepackage{microtype} 

\usepackage{cite}
\usepackage{array,colortbl}
\usepackage{amsmath,amsthm,mathtools}
\usepackage{amsfonts,bm,amssymb}
\DeclareFontFamily{U}{mathx}{\hyphenchar\font45}
\DeclareFontShape{U}{mathx}{m}{n}{
      <5> <6> <7> <8> <9> <10>
      <10.95> <12> <14.4> <17.28> <20.74> <24.88>
      mathx10
      }{}
\DeclareSymbolFont{mathx}{U}{mathx}{m}{n}
\DeclareFontSubstitution{U}{mathx}{m}{n}
\DeclareMathAccent{\widecheck}{0}{mathx}{"71}

\def\citep#1#2{\cite[{#1}]{#2}}

\usepackage{algorithm}
\usepackage{algorithmic}

\newcommand{\bbC}{{\mathbb{C}}}

\newcommand{\bbE}{{\mathbb{E}}}

\newcommand{\bbT}{{\mathbb{T}}}

\newcommand{\N}{{\mathbb{N}}} 
\newcommand{\Q}{{\mathbb{Q}}} 
\newcommand{\R}{{\mathbb{R}}} 
\newcommand{\Z}{{\mathbb{Z}}} 

\DeclareSymbolFont{bbold}{U}{bbold}{m}{n}
\DeclareSymbolFontAlphabet{\mathbbold}{bbold}

\newcommand{\bbone}{{\mathbbold{1}}}

\newcommand{\bsa}{{\boldsymbol{a}}}
\newcommand{\bsb}{{\boldsymbol{b}}}
\newcommand{\bsc}{{\boldsymbol{c}}}

\newcommand{\bse}{{\boldsymbol{e}}}

\newcommand{\bsh}{{\boldsymbol{h}}}

\newcommand{\bsk}{{\boldsymbol{k}}}

\newcommand{\bsv}{{\boldsymbol{v}}}
\newcommand{\bsw}{{\boldsymbol{w}}}
\newcommand{\bsx}{{\boldsymbol{x}}}
\newcommand{\bsy}{{\boldsymbol{y}}}
\newcommand{\bsz}{{\boldsymbol{z}}}

\newcommand{\bsK}{{\boldsymbol{K}}}

\newcommand{\bszero}{{\boldsymbol{0}}} 
\newcommand{\bsone}{{\boldsymbol{1}}}  
\newcommand{\bsalpha}{{\boldsymbol{\alpha}}}
\newcommand{\bsbeta}{{\boldsymbol{\beta}}}
\newcommand{\bsgamma}{{\boldsymbol{\gamma}}}

\newcommand{\bstau}{{\boldsymbol{\tau}}}

\newcommand{\bsDelta}{{\boldsymbol{\Delta}}}

\newcommand{\calH}{{\mathcal{H}}}

\newcommand{\calQ}{{\mathcal{Q}}}

\newcommand{\frakQ}{{\mathfrak{Q}}}

\newcommand{\frakU}{{\mathfrak{U}}}

\newcommand{\setu}{{\mathfrak{u}}}
\newcommand{\setv}{{\mathfrak{v}}}
\newcommand{\setw}{{\mathfrak{w}}}

\newcommand{\rmi}{{\mathrm{i}}}

\newcommand{\rmu}{{\mathrm{u}}}

\newcommand{\rd}{\,\mathrm{d}} 

\ifjournalstyle
  \usepackage[hidelinks,hypertexnames=false,hyperindex=true,pdfpagelabels,linktoc=all]{hyperref}
  \usepackage{bookmark}
  \makeatletter
    \providecommand*{\toclevel@author}{999}
    \providecommand*{\toclevel@title}{0}
  \makeatother
  \usepackage{enumitem}
  \numberwithin{equation}{section}
  \spnewtheorem{alg}[theorem]{Algorithm}{\bfseries}{\itshape}
  \spnewtheorem{assumption}[theorem]{Assumption}{\bfseries}{\itshape}
\else
\usepackage[hidelinks,hypertexnames=false,hyperindex=true,pdfpagelabels,linktoc=all]{hyperref}
\usepackage{bookmark}
\makeatletter 
  \providecommand*{\toclevel@author}{999}
  \providecommand*{\toclevel@title}{0}
\makeatother

\usepackage{enumitem}

\numberwithin{equation}{section}

\theoremstyle{plain}
  \newtheorem{theorem}{Theorem}[section]
  
  \newtheorem{lemma}[theorem]{Lemma}
  
\theoremstyle{definition}

\theoremstyle{remark}

\fi

\newcommand{\fmarkone}{+}
\newcommand{\fmarktwo}{\circ}

\numberwithin{figure}{section}
\numberwithin{table}{section}
\graphicspath{{figs/}}

\ifjournalstyle
  \usepackage[paperwidth=155mm,paperheight=235mm,textwidth=122mm,textheight=193mm,centering]{geometry}
  \usepackage{adjustbox}
\else
  \usepackage[a4paper]{geometry}
\fi

\newcommand{\mix}{\mathrm{mix}}
\DeclareMathOperator{\supp}{supp}
\DeclareMathOperator{\Li}{Li}
\newcommand{\divisors}{\operatorname{div}}
\newcommand{\twopii}{2\pi\rmi\,}
\newcommand{\wce}{\operatorname{wce}}
\newcommand{\tpmod}[1]{\,(\operatorname{mod} #1)}
\newcommand{\hookrightleftarrows}{%
  \mathrel{%
    \substack{%
      \hookrightarrow\\[-0.6ex]
      \rotatebox[origin=c]{180}{$\scriptstyle\hookrightarrow$}%
    }%
  }%
}

\usepackage{etoolbox} 
\usepackage{needspace}

\usepackage{nicefrac}

\newcommand{\numberthis}[1]{%
  \addtocounter{equation}{1}%
  \tag{\theequation}%
  \ifstrempty{#1}{}{%
    \label{#1}%
  }%
}

\allowdisplaybreaks

\usepackage[utf8]{inputenc}  
\usepackage{listings}
\usepackage{xcolor}

\lstdefinelanguage{Julia}{
  morekeywords={function, end, if, else, for, while, return, struct, mutable, let, in, begin, hcat, mod, reduce, @test, println, using, collect},
  sensitive=true,
  morecomment=[l]\#,
  morestring=[b]",
}

\def\figuredocument{function_spaces}
\def\figurepretex{\string\def\string\articleversion{1}}
\usepackage{pgfplots,pgfplotstable}
\pgfplotsset{compat=1.18}
\usetikzlibrary{arrows.meta,calc,patterns,positioning}

\definecolor{deepgreen}{RGB}{0,110,80}
\definecolor{softgreen}{RGB}{205,235,224}
\definecolor{deepblue}{RGB}{35,90,155}
\definecolor{softblue}{RGB}{214,228,246}
\definecolor{deeporange}{RGB}{205,100,25}
\definecolor{softorange}{RGB}{249,222,196}
\definecolor{softgray}{RGB}{240,242,244}

\tikzset{
  sample/.style={circle,fill=black,inner sep=1.6pt},
  infoarrow/.style={-{Stealth[length=2.2mm]},thick},
  construction samples/.pic={
    \foreach \samplex/\sampley in {
      0.0625/0.0625,0.1875/0.3125,0.3125/0.5625,0.4375/0.8125,
      0.5625/0.1875,0.6875/0.4375,0.8125/0.6875,0.9375/0.9375
    } {
      \node[sample] at (\samplex,\sampley) {};
    }
  },
}

\usepgfplotslibrary{groupplots}
\usepgfplotslibrary{external}
\providecommand{\figuredocument}{\tikzexternalrealjob}
\providecommand{\figurepretex}{}
\tikzset{external/system call={pdflatex \tikzexternalcheckshellescape -halt-on-error -interaction=batchmode -synctex=1 -jobname "\image" "\figurepretex\string\def\string\tikzexternalrealjob{\tikzexternalrealjob}\string\input{\figuredocument}"}}
\ifdefined\rebuildfigures
  \tikzset{external/force remake}
\fi
\newcommand{\articledata}{data/20261014_e16d21c70ede}

\pgfdeclareplotmark{filled cross}{
  \pgfpathmoveto{\pgfqpoint{-.3\pgfplotmarksize}{\pgfplotmarksize}}
  \pgfpathlineto{\pgfqpoint{.3\pgfplotmarksize}{\pgfplotmarksize}}
  \pgfpathlineto{\pgfqpoint{.3\pgfplotmarksize}{.3\pgfplotmarksize}}
  \pgfpathlineto{\pgfqpoint{\pgfplotmarksize}{.3\pgfplotmarksize}}
  \pgfpathlineto{\pgfqpoint{\pgfplotmarksize}{-.3\pgfplotmarksize}}
  \pgfpathlineto{\pgfqpoint{.3\pgfplotmarksize}{-.3\pgfplotmarksize}}
  \pgfpathlineto{\pgfqpoint{.3\pgfplotmarksize}{-\pgfplotmarksize}}
  \pgfpathlineto{\pgfqpoint{-.3\pgfplotmarksize}{-\pgfplotmarksize}}
  \pgfpathlineto{\pgfqpoint{-.3\pgfplotmarksize}{-.3\pgfplotmarksize}}
  \pgfpathlineto{\pgfqpoint{-\pgfplotmarksize}{-.3\pgfplotmarksize}}
  \pgfpathlineto{\pgfqpoint{-\pgfplotmarksize}{.3\pgfplotmarksize}}
  \pgfpathlineto{\pgfqpoint{-.3\pgfplotmarksize}{.3\pgfplotmarksize}}
  \pgfpathclose
  \pgfusepathqfill
}

\pgfplotsset{
  scatterpoints/.style={
    axis equal image,
    width=4cm,
    height=4cm,
    xmin=0, xmax=1,
    ymin=0, ymax=1,
    axis lines=box,
    enlargelimits=false,
    clip=false,
    grid=none,
    xlabel={},
    ylabel={},
    xtick=\empty,
    ytick=\empty,
    xticklabel=\empty,
    yticklabel=\empty,
    tick style={black, line width=0.5pt},
    scale only axis,
  }
}
\pgfplotsset{
  bluepoints/.style={
    only marks,
    mark=*,
    mark size=1.5pt,
    mark options={fill=blue},
  }
}
\pgfplotsset{
  orangepoints/.style={
    only marks,
    mark=*,
    mark size=1.5pt,
    mark options={fill=orange},
  }
}

\ifjournalstyle
  \pgfmathsetmacro{\journalfigurescale}{122/147}
  \pgfkeys{/pgf/images/include external/.code={\includegraphics[scale=\journalfigurescale]{#1}}}
\fi

\newcommand{\titlebreak}{\ifjournalstyle\space\else\\\fi}
\ifjournalstyle
  \newcommand{\journalonly}[1]{#1}
  \newcommand{\articleonly}[1]{}
  \newcommand{\journalbreak}[1][]{\\* #1 &}
\else
  \newcommand{\journalonly}[1]{}
  \newcommand{\articleonly}[1]{#1}
  \newcommand{\journalbreak}[1][]{}
\fi
\newcommand{\chapterauthor}{Dirk Nuyens}
\newcommand{\chapteraffiliation}{Department of Computer Science \\ KU~Leuven, Belgium}

\begin{document}

\title{A gentle introduction to approximating\titlebreak high-dimensional integrals with lattice rules\titlebreak and a maybe-not-so-gentle introduction\titlebreak to the relevant function spaces\titlebreak for QMC integration}
\ifjournalstyle
  \mainmatter
  \author{\chapterauthor}
  \institute{\chapteraffiliation}
  \titlerunning{Lattice rules and function spaces for QMC integration}
  \authorrunning{\chapterauthor}
\else
  \author{\chapterauthor \\ \normalsize \chapteraffiliation}
\fi
\date{\today}
\maketitle

\begin{abstract}
These lecture notes give both a practical and theoretical introduction to high-dimensional integration with lattice rules.
We introduce lattice points and rules in combination with accompanying Julia code.
Next we introduce the typical function spaces used in the literature.
We compare classical Sobolev spaces with Sobolev spaces of dominating mixed smoothness.
We derive lower bounds and show that the curse of dimensionality is unavoidable for these unweighted function spaces.
Next we switch to weighted periodic spaces and derive a constructive upper bound using the component-by-component construction for the generating vector of a lattice rule.
We pick periodic spaces because they are a perfect match with lattice rules.
The upper bound shows that under suitable conditions on the weights we obtain dimension-independent worst-case error bounds.
We then use embedding constants to transfer these bounds to some common norms for these periodic spaces.
\end{abstract}

\ifjournalstyle
  \begingroup
    \makeatletter
    \section*{\contentsname}
    \let\l@title\@gobbletwo
    \let\l@author\@gobbletwo
    \let\authcount\@gobble
    \@starttoc{toc}
    \makeatother
  \endgroup
\else
  \tableofcontents
\fi

\section{Introduction}\label{sec:introduction}

This text aims to provide an ``easy'' introduction to using lattice rules for approximating high-dimensional integrals and to provide an entry point into worst-case error bounds in the literature.
If one enjoys reading mathematics and code at the same time, then ``easy'' might refer to Sections~\ref{sec:introduction}--\ref{sec:quadrature} which start the explanation from a practical point of view and include Julia code for numerical experiments.
Here, and maybe more importantly throughout the later sections, ``easy'' is meant as well as: you can skip freely parts that you do not need.
To be able to reference the literature we do need to supply enough of the technical details.
One aim is to make it possible for the reader to ``interpret'' error bounds from the literature which are formulated for a particular function space and to then understand if this result could be applied to their problem, or their function space, at hand.
We define membership in a function space as having the associated norm finite, possibly with additional constraints.
The main part of these lecture notes is hence Section~\ref{sec:function-spaces} on ``function spaces'' where we define a whole lot of different spaces/norms.
Since we tirelessly ramble through a lot of different, somewhat equivalent (to be specified later), function spaces/norms, one might need a strong stomach to read this front to back.
So maybe ``easy'' should be interpreted as a hint to the reader to take it easy.
Nevertheless, we hope that the stated results in the text will be useful to cite should the need arise, since no comprehensive overview like this seems to exist, although a lot of individual results can be found scattered throughout the literature.
An advantage is that we provide all the proofs in a consistent way.
To give an idea of the crazy notation: we will define the multivariate Sobolev spaces $H^\alpha([0,1]^d)$, $H^{\alpha,\bsgamma}_\mix([0,1]^d)$, $H^{\alpha,\bsgamma,\rmu}_\mix([0,1]^d)$, $H^{\alpha,\bsgamma,\bsa}_\mix([0,1]^d)$ and their periodic relatives $\widetilde{H}^\alpha([0,1]^d)$, $\widetilde{H}^{\alpha,\bsgamma}_\mix([0,1]^d)$, $\widetilde{H}^{\alpha,\bsgamma,\rmu}_\mix([0,1]^d)$, $\widetilde{H}'^{\alpha,\bsgamma}_\mix([0,1]^d)$, $\widetilde{H}''^{\alpha,\bsgamma}_\mix([0,1]^d)$, $\widetilde{H}^{\alpha,\bsgamma,\rmu}_{\mix,1}([0,1]^d)$, $\widetilde{H}'^{\alpha,\bsgamma}_{\mix,1}([0,1]^d)$, $\widetilde{H}''^{\alpha,\bsgamma}_{\mix,1}([0,1]^d)$ and $\widetilde{H}^{\alpha,\bsgamma,\bsa}_\mix([0,1]^d)$.
The norms of the periodic spaces are all expressed in terms of Fourier coefficients and a specific \emph{Fourier multiplier} dependent on the actual space/norm, except for the anchored periodic space $\widetilde{H}^{\alpha,\bsgamma,\bsa}_\mix([0,1]^d)$.

We first provide some general background, along with remarks related to computer implementations.
The mathematical notation and conventions used throughout these lecture notes are explained at the end of the current section.
In Section~\ref{sec:lattice-points} we provide some Julia code to play with lattice points and lattice sequences, and demonstrate some useful insight on lattice point sets.
Section~\ref{sec:quadrature} reviews some concepts of univariate quadrature formulae in such a way that they relate to the usage of lattice rules, which are quadrature/cubature rules that use lattice points, to approximate multivariate integrals.
We provide Julia code that allows you to play with lattice rules to approximate multivariate integrals and to estimate errors.
After this we are ready for the main part.
Section~\ref{sec:function-spaces} first introduces essential information-based complexity concepts and then motivates different univariate function spaces.
For the multivariate spaces we next explain the difference between ``classical Sobolev space norms'' and ``dominating mixed smoothness Sobolev space norms'' and demonstrate that the latter avoid the dimension-dependent deterioration of the rate of convergence.
We show this by two respective lower bounds, which hold for any algorithm using function values (or derivative values), see Theorems~\ref{thm:Halpha-lower-bound} and~\ref{thm:Hmixalpha-lower-bound}.
We define what it means to suffer the \emph{curse of dimensionality} and we then show that these unweighted function spaces suffer the curse of dimensionality in Theorem~\ref{thm:curse}.
Next we want to show that one possible way to achieve tractability and avoid the curse of dimensionality is to use \emph{weighted function spaces}.
At this point we switch in Section~\ref{sec:periodic-spaces} to weighted periodic Sobolev spaces (which are vector subspaces of the non-periodic ones since we ask additional periodicity properties of the functions), but weighted non-periodic function spaces can also break the curse of dimensionality.
Our focus is on periodic spaces because Fourier expansions interact especially well with lattice rules.
In particular, in Section~\ref{sec:upper-bound-lattice-rules}, we use the \emph{lattice character property}~\eqref{eq:lattice-character-prop} to obtain an exact expression for the integration error for functions with absolutely converging Fourier series, see Lemma~\ref{lem:fourier-sum-error-expression}.
We then use the \emph{component-by-component construction} to prove a constructive upper bound on the worst-case error, see Theorem~\ref{thm:lattice-rule-upper-bounds}.
Under certain conditions on the weights this worst-case error bound is dimension-independent.
Finally, in Section~\ref{sec:rkhs} we discuss \emph{reproducing kernel Hilbert spaces} which are traditionally used in the literature but which we have somehow avoided up to this point.
We give some standard results and use them to show that, for the function spaces discussed here, deterministic algorithms that are nonlinear and/or adaptive cannot improve upon an optimal nonadaptive linear algorithm.

As mentioned above, the literature often uses rather specific function spaces with ``ugly'' looking norms.
These norms are often chosen as the most convenient ones to prove the results one is after, or because one is familiar with these norms, or because they are most often used in other specialised literature, or just for historical reasons.
Luckily it is not always necessary for the application to also use these ugly norms.
We can use embedding constants to transfer worst-case error results from one space to another, see~\eqref{eq:transplant-upper-bound}.
One particular popular function space in the literature is the \emph{Korobov space}.
Its definition is however not completely standardised.
It agrees most of the time with the unanchored periodic Sobolev spaces of dominating mixed smoothness $\widetilde{H}^{\alpha,\bsgamma,\rmu}_\mix([0,1]^d)$ and $\widetilde{H}^{\alpha,\bsgamma,\rmu}_{\mix,1}([0,1]^d)$.
E.g., in Theorem~\ref{thm:periodic-upper-bound-transfer} we transfer worst-case error bounds for $\widetilde{H}^{\alpha,\bsgamma,\rmu}_{\mix,1}([0,1]^d)$ to a whole zoo of periodic spaces, cf.\ Table~\ref{tbl:periodic-equivalent-norms}.
All our spaces are Hilbert spaces, but in the original publications of Korobov the spaces were based on $\ell_\infty$ based Fourier bounds~\cite{Kor59}, see also the discussion before Theorem~\ref{thm:wce-F-d-r}.
Of course, we cannot skip the definitions of all of these norms and so the whole of Section~\ref{sec:function-spaces} will contain plenty of ugliness.
We just hope the reader is happy to surf along.

Having defined these function spaces, one can then analyse the usage of lattice rules to approximate high-dimensional integrals.
We standardise the domain to be the unit cube $[0, 1]^d$ and are interested in approximating
\begin{align}\label{eq:Id}
  I_d(f)
  :=
  \int_{[0,1]^d} f(\bsx) \rd\bsx
  ,
\end{align}
for $f$ in a certain function space---these function spaces will characterise the ``smoothness'' or ``regularity'' of our functions since we expect the problem of integrating ``smooth'' functions to be easier than the problem of integrating not so smooth functions\ldots\

There are many different types of quasi-Monte Carlo (QMC) point sets, but we focus on what is arguably the easiest point set of them all: \emph{lattice points}.
A formula to approximate an integral is often called a \emph{quadrature rule} and, in the multivariate case, a \emph{cubature rule}.
Hence we use \emph{lattice rule} as the name for a quadrature/cubature rule that uses lattice points.
For a given number of points $n$, in $d$ dimensions, a lattice rule is fully specified by the \emph{generating vector} $\bsz \in \Z^d$, and the lattice rule takes the form
\begin{align}\label{eq:R1L-rule}
  Q_{d,n,\bsz}(f)
  &:=
  \frac1n \sum_{k=0}^{n-1} f((\bsz k \bmod n) / n)
  .
\end{align}
As is typical for quasi-Monte Carlo rules the weights in the quadrature/cubature rule are all equal to $\nicefrac1n$, and this then warrants the naming ``quasi-Monte Carlo rule'' as a variant of a plain Monte Carlo rule, with the difference being that we took deterministic quasi-Monte Carlo points instead of i.i.d.\ uniform $[0,1]^d$ sample points.
(However, we do not insist on the weights all being $\nicefrac1n$ and there are many examples in the literature in which one has modified weights.)

The standard, crudely formulated, selling point of quasi-Monte Carlo rules is that they converge like $O(n^{-1})$ for sufficiently smooth functions, i.e., functions of bounded Hardy--Krause variation, while the standard Monte Carlo method only converges like $O(n^{-1/2})$, but under the much more relaxed condition that $f$ needs to have finite variance.
We must warn the reader that such a statement must be interpreted with care: it hides dimension-dependent constants and compares a \emph{deterministic worst-case error} bound for quasi-Monte Carlo with a \emph{root-mean-square error} bound for Monte Carlo.
A precise comparison requires specifying the function space and the error criterion.
Indeed, in the randomised root-mean-square setting the minimal error on the unit ball of $L_2([0,1]^d)$ is also of order $n^{-1/2}$, and hence the standard Monte Carlo method has optimal rate for this class and no randomised quasi-Monte Carlo method can improve on this rate.
We will however not consider randomised methods in these lecture notes, except in Section~\ref{sec:quadrature} to estimate errors.

While lattice rules can get optimal convergence rates in certain function spaces, not every lattice rule will do so.
The quality is determined by the generating vector.
In Figure~\ref{fig:r1l-points} we give some ``good'' and ``bad'' examples of lattice rules.
It should be intuitively obvious that taking $\bsz = (1,\ldots,1)$, which results in all the lattice points being on the diagonal of the unit cube, is not likely to be a good quadrature/cubature rule, as for sure there must be functions $f$ which, say, are zero on the diagonal but have interesting behaviour off-diagonal.
But be warned!
This depends on what type of functions you want to integrate, as even points on the diagonal could have their very specific use case (although in that case the problem should probably have been formulated differently).
Obviously taking $\bsz = (0, \ldots, 0)$ is even worse as we get just a single point whatever $n$ is!

Lattice points are actually a multiset and each lattice point appears with multiplicity $\gcd(z_1, \ldots, z_d, n)$.
Note that $\gcd(0, n) = n$, and $\gcd(a_1,\ldots,a_n) = \allowbreak \gcd(\gcd(a_1,a_2),\ldots,a_n)$.
Hence we have $n$ unique points if and only if $\gcd(z_1, \ldots, z_d, n) = 1$.
Quite often one demands that $\gcd(z_j, n) = 1$ for all~$j$, which is a stronger condition than $\gcd(z_1, \ldots, z_d, n) = 1$: it does not just guarantee that there are $n$ unique $d$-dimensional points in total, but also that every one-dimensional projection has exactly $n$ points.
See also Figures~\ref{fig:r1l-points-b} and~\ref{fig:r1l-points-c} which demonstrate the effect of having $\gcd(z_j, n) > 1$ for specific~$j$.
Code and more comments on these lattices are given in Section~\ref{sec:lattice-points}.

What makes a good lattice rule, i.e., the combination $(d, n, \bsz)$, depends on the class of functions you want to integrate, which we will encode in a function space.
The aim, for those constructing ``good'' lattice rules, is then to produce a good combination $(d, n, \bsz)$ for the particular function space which can then be used for all functions in that space.
This is done by minimising the worst-case error in some sense.
Such ``good'' lattice rules are traditionally constructed in advance for a specific function space.

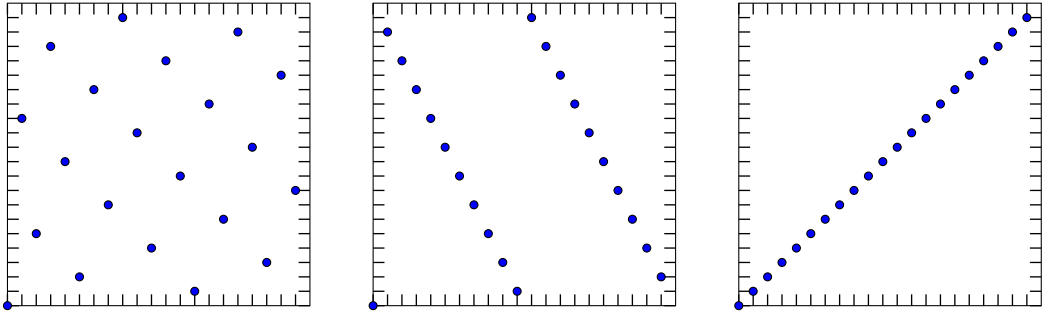
\begin{figure}
  \centering
  \tikzsetnextfilename{figure_lattice_points_1_13_21}
\begin{tikzpicture}
\begin{axis}[
    scatterpoints,
    xtick=data,
    ytick=data,
]
    \addplot[bluepoints] table[x index=0, y index=1] {\articledata/lattice_points_1_13_21.txt};
\end{axis}
\end{tikzpicture}\quad%
  \tikzsetnextfilename{figure_lattice_points_1_19_21}
\begin{tikzpicture}
\begin{axis}[
    scatterpoints,
    xtick=data,
    ytick=data,
]
    \addplot[bluepoints] table[x index=0, y index=1] {\articledata/lattice_points_1_19_21.txt};
\end{axis}
\end{tikzpicture}\quad%
  \tikzsetnextfilename{figure_lattice_points_1_1_21}
\begin{tikzpicture}
\begin{axis}[
    scatterpoints,
    xtick=data,
    ytick=data,
]
    \addplot[bluepoints] table[x index=0, y index=1] {\articledata/lattice_points_1_1_21.txt};
\end{axis}
\end{tikzpicture}
  \caption{Some ``good'' and ``bad'' example lattice points with $n=21$ points. The generating vector $\bsz$ determines the quality of the point set. Left: a ``good'' set with $\bsz = (1, 13)$. Middle: a not so good set with $\bsz = (1, 19)$. Right: a clearly not so good set with $\bsz = (1, 1)$. Intuitively we want the point set to be filling the unit cube as good as possible. If we leave large gaps then we have no idea what the function might do there\ldots}\label{fig:r1l-points}
\end{figure}

\paragraph{Comments on the modulus operation.}
We comment on some possible issues concerning the modulus operation, especially when one wants to implement lattice rules in a programming language.
In the form~\eqref{eq:R1L-rule} we note that instead of writing $\bsz \in \Z^d$ we could as well have written $\bsz \in \Z_n^d$ where we define $\Z_n := \{ 0, 1, \ldots, n-1 \}$ and formally identify $\Z_n$ with $\Z/n\Z$.
The modulus operation in~\eqref{eq:R1L-rule} is meant as a \emph{binary operator}, i.e., acting on two integer arguments and written in infix notation as ``$a \bmod b$'', and, for $a, b \in \Z$ with $b \ne 0$, is defined by $a \bmod b := a - b \lfloor a/b \rfloor$.
On different occasions we also need the equivalence relation modulo~$n$ and we write this in the traditional notation as \mbox{``$a \equiv b \pmod{n}$''}.
These two different notations have different, but related, meanings.
The first argument may be negative, e.g., when taking differences or additive inverses of lattice points.
For our purpose we will always have the second argument to be positive.
Unfortunately, different programming languages use different sign conventions for operations called ``modulus'' or ``remainder''.
With the mathematical definition we have $a \bmod n \in \Z_n$ for $a \in \Z$ and $n \in \N = \{1, 2, \ldots\}$, i.e., $a \bmod n$ is always a non-negative integer number in $[0, n)$ independent of the sign of~$a$, e.g., $-3 \bmod 5 = 5 - 3 = 2$.
In C and C++, where the remainder takes the sign of the first argument, one must write something more elaborate like \mbox{\texttt{(a \% n + n) \% n}} to express $a \bmod n$ for integer arguments.
We also use the modulus operation for real arguments and define $x \bmod y := x - y \lfloor x/y \rfloor$ for $x, y \in \R$, with $y \ne 0$.
See Figure~\ref{fig:mods} for an illustration of how to correct the C-style modulus operation.
Note however that in C/C++ the \texttt{\%} operator only works on integers and one has to use \texttt{fmod} for floating-point arguments.

\begin{figure}
\centering
\tikzsetnextfilename{figure_modulo}
\begin{tikzpicture}
\begin{groupplot}[
  group style={
    group size=3 by 1,
    horizontal sep=0.5cm
  },
  width=6cm,
  height=4.5cm,
  xmin=-3.5, xmax=3.5,
  ymin=-1.2, ymax=1.2,
  samples=400,
  domain=-3.5:3.5,
  axis lines=middle,
  xtick={-3,-2,-1,0,1,2,3},
  ytick={-1, -0.5, 0, 0.5, 1},
  tick label style={font=\scriptsize},
  xlabel={$x$}, ylabel style={rotate=-90},
  xlabel style={xshift=5pt},
  grid=both
]
\nextgroupplot[title={$\{ x \} = x \bmod 1 = x - \lfloor x \rfloor$}]
\pgfplotsinvokeforeach{-4,...,3}{
  \addplot[green!60!black, thick, domain=#1:{#1+1}, samples=2] {x - (#1)};
}
\addplot[only marks, mark=*, mark size=1.5pt, mark options={draw=green!60!black, fill=white}] coordinates {(-3,1) (-2,1) (-1,1) (0,1) (1,1) (2,1) (3,1)};
\addplot[only marks, mark=*, mark size=1.5pt, mark options={draw=green!60!black, fill=green!60!black}] coordinates {(-3,0) (-2,0) (-1,0) (0,0) (1,0) (2,0) (3,0)};
\nextgroupplot[title={$x \,\%\, 1$ (C-style)}]
\pgfplotsinvokeforeach{-4,...,-1}{
  \addplot[red, thick, domain=#1:{#1+1}, samples=2] {x - (#1+1)};
}
\pgfplotsinvokeforeach{0,...,3}{
  \addplot[red, thick, domain=#1:{#1+1}, samples=2] {x - (#1)};
}
\addplot[only marks, mark=*, mark size=1.5pt, mark options={draw=red, fill=white}] coordinates {(-3,-1) (-2,-1) (-1,-1) (1,1) (2,1) (3,1)};
\addplot[only marks, mark=*, mark size=1.5pt, mark options={draw=red, fill=red}] coordinates {(-3,0) (-2,0) (-1,0) (0,0) (1,0) (2,0) (3,0)};
\nextgroupplot[title={$((x \,\%\, 1) + 1) \,\%\, 1$ (C-style)}]
\pgfplotsinvokeforeach{-4,...,3}{
  \addplot[blue, thick, domain=#1:{#1+1}, samples=2] {x - (#1)};
}
\addplot[only marks, mark=*, mark size=1.5pt, mark options={draw=blue, fill=white}] coordinates {(-3,1) (-2,1) (-1,1) (0,1) (1,1) (2,1) (3,1)};
\addplot[only marks, mark=*, mark size=1.5pt, mark options={draw=blue, fill=blue}] coordinates {(-3,0) (-2,0) (-1,0) (0,0) (1,0) (2,0) (3,0)};
\end{groupplot}
\end{tikzpicture}
\caption{Differences in modulo operations: mathematical modulo (what we want) versus C-style modulo and how to correct it. In C/C++ the \texttt{\%} operator only accepts integers and one would use \texttt{fmod} instead for floating-point numbers, which will result in the ``C-style'' graphs.}\label{fig:mods}
\end{figure}
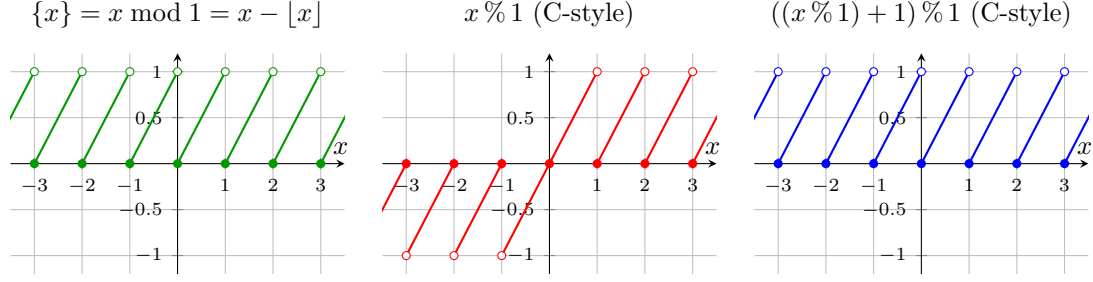

Quite often this ``reduction to the unit cube'' is written using the \emph{fractional braces} notation from number theory, i.e., for $a \in \Z$ and $n \in \N = \{1, 2, \ldots\}$, the common notation in the literature is
\begin{align*}
  \frac{a \bmod n}n
  &=
  \left\{ \frac{a}n \right\}
  =
  \frac{a}n \bmod 1
  .
\end{align*}
Using the proper mathematical definition of the modulus operator, we define the fractional braces to be $\{ x \} := x \bmod 1$ for any $x \in \R$.
In your programming language this then calls for a float version of the modulus operation, i.e., we are doing ``modulus 1.0''.
Since fractional braces are sometimes confusing as they clash with set notation, and since they hide the actual algebraic structure of the lattice points, we will avoid that notation in these notes.

To wrap up.
For our usage, in C/C++ you need to write \texttt{((a \% b) + b) \% b} for integer \texttt{a} and $\texttt{b} > 0$, and \texttt{fmod(fmod(x, y) + y, y)} for floating-point \texttt{x} and $\texttt{y} > 0$.
In Python the \texttt{\%} operator for modulus is fine for both integer and floating-point arguments.
In Julia the explicit \texttt{mod} function is fine for our purpose for both integer and floating-point arguments, but mind that \texttt{\%} also exists in Julia and uses the ``wrong'' C/C++ convention.
Matlab's and Octave's \texttt{mod} functions work for both integer and floating-point arguments and behave correctly for our purpose.
Due to the nature of floating-point arithmetic it is however possible that your implementation of $x \bmod y$ returns a result equal to $y$ instead of strictly less than~$y$.
Applying an extra modulus operation on top of this result would map the result back to~$0$ which is at least in the valid range.
The C/C++ trick does this automatically.
However, this does not stop other artefacts: adding $y$ before the second modulus operation can also introduce extra rounding errors and then it might be safer to first check if we need to add $y$ (i.e., if the first result was negative) and to then explicitly check if the result wrongly has $\ge y$ and if so set it to~$0$.
Finally, for just doing ``modulo $1.0$'' in C/C++ one can also directly do \texttt{x - floor(x)} or \texttt{x - std::floor(x)} respectively.
But also here there are floating-point caveats.
Also for integer modulo you might have to be careful since in almost all cases you will be working with finite width representations and so you have to be aware of the possibility that your result might wrap-around and even get a negative sign because of this.

Throughout the next two sections we will illustrate the text with Julia code.
There is no specific reason to pick Julia, except maybe for personal preference of the author at this occasion.
The particular features used are to keep the code snippets concise, and also due to personal preference.
W.r.t.\ the long discussion on the modulus operator: we are going for simplicity and convenience and will simply make use of the \texttt{mod} function in Julia.

\paragraph{Notation.}
Let us first settle some further notation.
We use $\N := \{1, 2, \ldots\}$ and $\N_0 := \{0, 1, 2, \ldots\}$.
We set $\Z_n := \{0, \ldots, n-1\}$ and define $\Z_n^\times := \{ a \in \Z_n : \gcd(a, n) = 1\}$ as the multiplicative group modulo~$n$.
The elements of $\Z_n^\times$ are those elements of $\Z_n$ that have a multiplicative inverse and are called the \emph{units} of the ring~$\Z_n$ equipped with addition and multiplication modulo~$n$.
We use $A \subset B$ to denote that $A$ is a strict subset of $B$ and $A \subseteq B$ to denote that $A$ is a subset of $B$ and may equal~$B$.
A sum over an empty index set is~$0$, and an empty product is~$1$.

For $\setu \subseteq \{1, \ldots, d\}$ and $\bsx \in \R^d$ we denote $\bsx_\setu = (x_j)_{j\in\setu}$, similarly for $\bsh \in \Z^d$ we denote $\bsh_\setu = (h_j)_{j \in \setu}$.
We define $\supp(\bsh) := \{ j : h_j \ne 0 \}$.
We write $\bsh_\setu \in \Z^{|\setu|}$ to mean $h_j \in \Z$ for all $j \in \setu$.
With $-\setu$ we mean the complement w.r.t.\ $\{1, \ldots, d\}$, i.e., $-\setu := \{1,\ldots,d\} \setminus \setu$.
It should always be clear from the context what $d$ is in this case.
To join up two vectors, we write $(\bsx_\setu; \bsa_{-\setu})$ to denote the $d$-dimensional vector which takes the corresponding components of $\bsx$ for $j \in \setu$ and the corresponding components of ``the anchor'' $\bsa$ for $j \in -\setu$.
As a shorthand when we anchor with all the same value $a$ we write $(\bsx_\setu; a)$, e.g., when $a = 0$ we write $(\bsx_\setu; 0)$.
We always interpret the result as a $d$-dimensional vector.
We will meet functions $r_{\alpha,\bsgamma}$ that are defined in terms of the support of $\bsh$, e.g., the following form in particular
\begin{align*}
  r_{\alpha,\bsgamma}(\bsh)
  &=
  \gamma_{\supp(\bsh)}^{-1/2} \, \prod_{j \in \supp(\bsh)} r_\alpha(h_j)
  .
\end{align*}
Special about this particular form is that it does not care about the input dimensionality (it even works when the input dimensionality is~$0$) and it ``ignores'' extra trailing $0$-values.
For such a function $r_{\alpha,\bsgamma}$ the following shorthand is often used in the literature
\begin{align}\label{eq:fake-0-dimensional-sum}
  \sum_{\bsh_\setu \in (\Z \setminus \{0\})^{|\setu|}}
  r_{\alpha,\bsgamma}(\bsh_\setu)
  &:=
  \sum_{\substack{\bsh \in \Z^d \\ \bsh_\setu \in (\Z \setminus \{0\})^{|\setu|} \\ \bsh_{-\setu} = \bszero}}
  r_{\alpha,\bsgamma}(\bsh)
  =
  \sum_{\substack{\bsh \in \Z^d \\ \bsh_\setu \in (\Z \setminus \{0\})^{|\setu|} \\ \bsh_{-\setu} = \bszero}}
  r_{\alpha,\bsgamma}((\bsh_\setu; 0))
  ,
\end{align}
where the second and third lines under the sum are restrictions to $\bsh \in \Z^d$ each time.
We will also use this notation.
Note that for $\setu = \emptyset$ such a sum has a single term, namely $r_{\alpha,\bsgamma}(\bszero) = \gamma_\emptyset^{-1/2}$.

We write $\|\cdot\|_{\ell_p}$ for the $\ell_p$ norm, which could be on any countable index set, finite or countably infinite.
The particular index set should be clear from the context.
In a few locations we write $\|\bstau\|_1$ or $\|\bstau\|_\infty$ for $\bstau \in \N_0^d$, they just mean $\|\bstau\|_1 = \|\bstau\|_{\ell_1(\{1,\ldots,d\})} = |\tau_1| + \cdots + |\tau_d| = \tau_1 + \cdots + \tau_d$ and $\|\bstau\|_\infty = \|\bstau\|_{\ell_\infty(\{1,\ldots,d\})} = \max_{j \in \{1,\ldots,d\}} \tau_j$.

We will use both big-$O$ and little-$o$ asymptotic notations.
Given a function $f$, the notation $O(f)$ denotes the class of functions that can asymptotically be bounded by a constant times~$f$ and we write memberships as $g \in O(f)$ and $g \in o(f)$.
For functions $f > 0$ and $g \ge 0$ of $n$ we have, for $n \to \infty$,
\begin{align*}
  g \in O(f)
  &\;\Leftrightarrow\;
  \limsup_{n \to \infty} \frac{g(n)}{f(n)} < \infty
  \;\Leftrightarrow\;
  \exists C > 0 : \exists n_0 : \forall n > n_0 : g(n) \le C \, f(n)
  .
\end{align*}
Given a function $f$, the notation $o(f)$ denotes the class of functions that are asymptotically negligible compared to~$f$.
For functions $f > 0$ and $g \ge 0$ of $n$ we have, for $n \to \infty$,
\begin{align*}
  g \in o(f)
  &\;\Leftrightarrow\;
  \lim_{n \to \infty} \frac{g(n)}{f(n)} = 0
  \;\Leftrightarrow\;
  \forall C > 0 : \exists n_0 : \forall n > n_0 : g(n) < C \, f(n)
  .
\end{align*}
Obviously if $g \in o(f)$ then also $g \in O(f)$, in fact we have $o(f) \subset O(f)$.

When defining norms we often write things like
\begin{align*}
  \|f\|_{H^\alpha([0,1])}^2
  &:=
  \sum_{\tau = 0}^\alpha \|D^\tau f\|_{L_2([0, 1])}^2
  ,
\end{align*}
where we have defined the \emph{squared} norm.
Obviously this is just to make the equations slightly more compact and the actual definition should be interpreted as
\begin{align*}
  \|f\|_{H^\alpha([0,1])}
  &:=
  \left( \sum_{\tau = 0}^\alpha \|D^\tau f\|_{L_2([0, 1])}^2 \right)^{1/2}
  .
\end{align*}

All the notations that we use are quite common in the literature, but there are unfortunately no standard symbols for the different function spaces.
It is always necessary to check the particular definitions of the used function spaces in a reference as changes can sometimes be subtle.
Quite often the smoothness parameter $\alpha$ for a Korobov space might be doubled or not.
Also the weights in the weighted norms might be squared or not.

In the literature one might often see $s$ or $d$ as the number of dimensions.
Sometimes there is both $d$ and $s$, e.g., with applications to PDEs with random fields, one often uses $d$ as the dimension of the physical domain on which the PDE is solved, and $s$ as the number of (truncated) stochastic dimensions from the random field, see, e.g., \cite{KN2016,NN2021uniform,NN2021lognormal}.
In this case, if one uses lattice points, or other quasi-Monte Carlo points, for the random field, the dimensionality of this point set is then~$s$.
Also older literature often uses $s$ as the number of dimensions.
In these lecture notes we use $d$ throughout.
As a final remark, we like to think of vectors as column vectors, but we will just write $\bsz = (1, 5)$ for a column vector of length two.

\paragraph{References.}
We give plenty of references throughout the lecture notes, but here are already some general references.
Classical references on quasi-Monte Carlo methods and lattice rules are \cite{Nie92,SJ94}.
Recent introductions and surveys are \cite{DP2010,DKS2013,Nuy2014,DKP2022,KN2016}.
Standard references on information-based complexity are \cite{TWW88,NW2008,NW2010,NW2012}.
For background on the function spaces used in these lecture notes, see, e.g., \cite{NW2008,NW2010,GHHR2017,GHHRW2017,KSU2015,NUU2017}.

\section{A first look at lattice points}\label{sec:lattice-points}

Before we dive into the theory in Section~\ref{sec:function-spaces} we will first try to gain some insights into lattice points by giving some basic facts and illustrating them with Julia code and their corresponding figures.
Necessarily this contains some minimalistic tutorial on Julia.
The Julia code should be read in order: sometimes a snippet uses a certain package or defines a function, e.g., we write \lstinline{using Plots}, and this will not be repeated in subsequent snippets.

\subsection{Lattice points}

Lattice points on the $d$-dimensional unit cube $[0, 1]^d$ are defined by two parameters: the number of points $n \in \N$ and the \emph{generating vector} $\bsz \in \Z^d$.
For the triplet $(d, n, \bsz)$ the points are given by
\begin{align}\label{eq:R1L-points}
  \bsx_k
  :=
  \frac{\bsz k \bmod{n}}{n}
  =
  \frac{\bsz k}{n} \bmod{1}
  ,
  \qquad
  \text{for } k = 0, 1, \ldots, n-1
  .
\end{align}
They can be used directly like this as the nodes of a quadrature/cubature rule with equal weights~$1/n$, as in~\eqref{eq:R1L-rule}.

We briefly need to discuss some aspects of our programming language of choice, Julia.
They are mostly related to syntax and the discussed concepts are also possible in other languages.
In Julia we can make a ``generator'' which ``generates'' elements one by one.
The syntax is similar to that of array comprehension but with the square brackets replaced by round brackets.
E.g.\ array comprehension \lstinline!a = [ cos(k/10) for k in 1:10000 ]! creates an array that we can index like \lstinline!a[1]! or \lstinline!a[10000]!, indexing is $1$-based, and stays resident in memory, but the generator construction \lstinline!b = ( cos(k/10) for k in 1:10000 )! creates a generator which lets us iterate and compute its values one-by-one without storing them all in memory first.
We do not have random access anymore, but we can use constructs like \lstinline!for x in b; ...; end!.
A lot of standard Julia functions also accept generators instead of arrays, e.g., \lstinline!sum(b)! will work.
We will use this concept to create a ``point generator'' for our lattice points.
This has the advantage that we can use it in a similar way as a random number generator, each time asking for the next $d$-dimensional vector.

As in many modern scientific programming languages, Julia supports broadcasting a function to all elements of an array.
In Julia this is explicitly written by using the dot syntax, e.g., \lstinline!mod.([1.2, 4.5, -3], 5)! applies \lstinline!mod! to all three elements of the array and automatically expands scalars and singleton dimensions (dimensions of size~$1$) of other arguments such that all arguments are compatible in size.
In Julia, there is no need for an explicitly written vectorised version of a function to get good performance on arrays.
By using the dot notation, Julia automatically applies the function elementwise via broadcasting, which is compiled and fused for efficiency.
Having said that we can now show how to have a simple one-liner generator of lattice points.
\begin{lstlisting}
using Test # Demonstrate simple calls and unit tests

# Generator for n rank-1 lattice points with
# d-dimensional generating vector z
lattice_points(z, n) = ( mod.(z .* k, n) ./ n for k in 0:n-1 )
# The `.*` is to make it such `z` can be a tuple instead of an array.
# You can use the `StaticArrays` package and do e.g. `z = SA[1, 3]`.
\end{lstlisting}
\Needspace{6\baselineskip}
\begin{lstlisting}
# Example usage: print the points one by one
for x in lattice_points([1, 3], 4)
  println(x)
end
# or `println.(lattice_points([1, 3], 4));` if you like

# We can `collect` all the lattice points in an n-vector of d-vectors
@test collect(lattice_points([1, 3], 4)) ==
  [[0.0, 0.0], [0.25, 0.75], [0.5, 0.5], [0.75, 0.25]]
@test collect(lattice_points((1, 3), 4)) ==
  # z as tuple gives n-vector of d-tuples
  [(0.0, 0.0), (0.25, 0.75), (0.5, 0.5), (0.75, 0.25)]
@test collect.(lattice_points((1, 3), 4)) ==
  # `collect.` gives n-vector of d-vectors
  [[0.0, 0.0], [0.25, 0.75], [0.5, 0.5], [0.75, 0.25]]
# Or we can collect them in a d-by-n array (add
# `collect.` again for tuple)
@test reduce(hcat, lattice_points([1, 3], 4)) ==
  [                      # this is a 2-by-4 matrix
    0.0 0.25 0.5 0.75;
    0.0 0.75 0.5 0.25
  ]
\end{lstlisting}
We make use of the \lstinline!Test! package to demonstrate what the result of a simple example call should be.
This is also a limited check that we did not make any stupid mistakes.

It is useful to be able to make some simple 2D plots.
We will use the \lstinline!Plots! package for this.
To make nice plots we require some fiddling around and hence we define a \lstinline!scatter_points! function to do the work for us.
\begin{lstlisting}
using Plots # To make some basic plots

function unit_square_plot_setup()
  delta = 0.01; lim = (0-delta, 1+delta)
  plot([0, 1, 1, 0, 0], [0, 0, 1, 1, 0], # draw unit square
    lc=:gray, xlim=lim, ylim=lim, ratio=1,
    legend=false, framestyle=:none, plot_margin=0)
end

# Given a point generator `g` (like `lattice_points`), make a
# 2D scatter plot of points
function scatter_points(g; d1=1, d2=2, plt=unit_square_plot_setup())
  P = reduce(hcat, collect.(g))
    # "collect." allows the points, i.e., z, to be a tuple
  scatter!(plt, P[d1,:], P[d2,:])
  plt
end

\end{lstlisting}
\Needspace{6\baselineskip}
\begin{lstlisting}
# Now plot some nice lattice point set.
# The lattice we use is a ``Fibonacci lattice''.
# It has n = F_m and z = (1, F_{m-1}). They are known to be good.
# Such an explicit good form is only known for 2D (and 1D).
scatter_points(lattice_points([1, 55], 89))
\end{lstlisting}
The lattice with $\bsz = (1, 55)$ and $n=89$ can be found in the left subplot of Figure~\ref{fig:r1l-points-b}.

Given $d$ and $n$, the quality of the lattice point set depends on the choice of the generating vector $\bsz$.
Under the intuitive assumption that the lattice points need to be uniformly filling the unit cube it is obvious that e.g.\ $\bsz = (1,\ldots,1)$ is an extremely bad choice, and many more bad choices can be found.
However, one can easily show that there are also enough good choices and most of the proofs show existence of good generating vectors by making use of this fact.
\begin{lstlisting}
# Bad choices with n = 32
scatter_points(lattice_points([5, 5], 32))
  # same as z = (1, 1) since gcd(5, 32) = 1
scatter_points(lattice_points([2, 6], 32))
  # looks good but only 16 unique points since gcd(2, 6, 32) = 2
scatter_points(lattice_points([1, -1], 32))
  # same as z = (1, 31) since mod(-1, 32) = 31
\end{lstlisting}
This produces the ``intuitively bad'' lattice points for $n = 32$ from left to right displayed in Figure~\ref{fig:r1l-points-c} (the middle plot is only bad because we expected to see $32$ points and not $16$).
Multiplying every component of the generating vector by a unit $a \in \Z_n^\times$ does not change the point set, since multiplication by $a$ merely permutes the indices $k \in \Z_n$; this explains, for example, why $\bsz=(5, 5)$ and $\bsz=(1, 1)$ generate the same point set when $n = 32$, just in a different order.

\begin{figure}[p]
  \centering
  \scalebox{0.85}{\hbox{%
    \tikzsetnextfilename{figure_lattice_points_1_55_89}
\begin{tikzpicture}
\begin{axis}[
    scatterpoints,
    xtick=data,
    ytick=data,
]
    \addplot[bluepoints] table[x index=0, y index=1] {\articledata/lattice_points_1_55_89.txt};
\end{axis}
\end{tikzpicture}\quad%
    \tikzsetnextfilename{figure_lattice_points_10_9_90}
\begin{tikzpicture}
\begin{axis}[
    scatterpoints,
    xtick=data,
    ytick=data,
]
    \addplot[bluepoints] table[x index=0, y index=1] {\articledata/lattice_points_10_9_90.txt};
\end{axis}
\end{tikzpicture}\quad%
    \tikzsetnextfilename{figure_lattice_points_10_11_90}
\begin{tikzpicture}
\begin{axis}[
    scatterpoints,
    xtick=data,
    ytick=data,
]
    \addplot[bluepoints] table[x index=0, y index=1] {\articledata/lattice_points_10_11_90.txt};
\end{axis}
\end{tikzpicture}
  }}
  \caption{Rank-$1$ lattices with $89$ and $90$ points. We mark the unique $x_1$ and $x_2$ coordinates on the axes. Left: $\bsz = (1, 55)$ with $n=89$: this is a ``Fibonacci lattice'' which has a nice uniform distribution and $89$ unique coordinates for both $x_1$ and $x_2$. Fibonacci rules are the archetypal examples of ``good'' lattice rules for $d=2$. Middle: $\bsz = (10, 9)$ with $n=90$: this is a $9$ by $10$ grid which can be represented by a rank-$1$ lattice since $\gcd(9, 10) = 1$, there are only $9$ unique $x_1$ coordinates and $10$ unique $x_2$ coordinates. Right: $\bsz=(10, 11)$ with $n=90$: this one has only $9$ distinct coordinates for $x_1$ since $\gcd(10, 90) = 10$ and $90/10 = 9$. The last two are most likely not such ``good'' lattice rules since if functions like $f(x_1, x_2) = g(x_1)$ are to be allowed, and we do not know it is actually a one-dimensional function instead of a two-dimensional function, then these rules will only use $9$ effective quadrature points while the first one will use the full $89$ points.}\label{fig:r1l-points-b}
  \par\medskip
  \scalebox{0.85}{\hbox{%
    \tikzsetnextfilename{figure_lattice_points_1_1_32}
\begin{tikzpicture}
\begin{axis}[
    scatterpoints,
    xtick=data,
    ytick=data,
]
    \addplot[bluepoints] table[x index=0, y index=1] {\articledata/lattice_points_1_1_32.txt};
\end{axis}
\end{tikzpicture}\quad%
    \tikzsetnextfilename{figure_lattice_points_2_6_32}
\begin{tikzpicture}
\begin{axis}[
    scatterpoints,
    xtick=data,
    ytick=data,
]
    \addplot[bluepoints] table[x index=0, y index=1] {\articledata/lattice_points_2_6_32.txt};
\end{axis}
\end{tikzpicture}\quad%
    \tikzsetnextfilename{figure_lattice_points_1_31_32}
\begin{tikzpicture}
\begin{axis}[
    scatterpoints,
    xtick=data,
    ytick=data,
]
    \addplot[bluepoints] table[x index=0, y index=1] {\articledata/lattice_points_1_31_32.txt};
\end{axis}
\end{tikzpicture}
  }}
  \caption{Intuitively bad examples of lattice points for $n=32$. Left: $\bsz = (5, 5)$, which is equivalent to $\bsz = (1, 1)$ since $\gcd(5, 32) = 1$, has all points on the diagonal and hence does not fill the cube nicely. Middle: $\bsz = (2, 6)$ looks good but we only have $16$ distinct points instead of $32$ since $\gcd(2, 6, 32) = 2$, hence this is equivalent to taking $\bsz = (1, 3)$ with $n = 16$. Right: $\bsz = (1, -1)$, or equivalently $\bsz = (1, 31)$ since $-1 \bmod 32 = 31$, has all points except $(0, 0)$ on the anti diagonal and hence also does not fill the cube nicely.}\label{fig:r1l-points-c}
  \par\medskip
  \scalebox{0.85}{\hbox{%
    \tikzsetnextfilename{figure_lattice_sequence_64}
\begin{tikzpicture}
\begin{axis}[
    scatterpoints,
    xtick=data,
    ytick=data,
]
    \addplot[bluepoints] table[x index=0, y index=1] {\articledata/lattice_sequence_64.txt};

\end{axis}
\end{tikzpicture}\quad%
    \tikzsetnextfilename{figure_lattice_sequence_96}
\begin{tikzpicture}
\begin{axis}[
    scatterpoints,
    xtick=data,
    ytick=data,
]
    \addplot[draw=none, mark=none, forget plot] table[x index=0, y index=1] {\articledata/lattice_sequence_96.txt};
    \addplot[bluepoints, restrict expr to domain={\coordindex}{0:63}] table[x index=0, y index=1] {\articledata/lattice_sequence_96.txt};
    \addplot[orangepoints, restrict expr to domain={\coordindex}{64:95}] table[x index=0, y index=1] {\articledata/lattice_sequence_96.txt};

\end{axis}
\end{tikzpicture}\quad%
    \tikzsetnextfilename{figure_lattice_sequence_128}
\begin{tikzpicture}
\begin{axis}[
    scatterpoints,
    xtick=data,
    ytick=data,
]
    \addplot[draw=none, mark=none, forget plot] table[x index=0, y index=1] {\articledata/lattice_sequence_128.txt};
    \addplot[bluepoints, restrict expr to domain={\coordindex}{0:63}] table[x index=0, y index=1] {\articledata/lattice_sequence_128.txt};
    \addplot[orangepoints, restrict expr to domain={\coordindex}{64:127}] table[x index=0, y index=1] {\articledata/lattice_sequence_128.txt};

\end{axis}
\end{tikzpicture}
  }}
  \caption{A lattice sequence in base-$2$ with generating vector $\bsz=(1, 364981)$.
  Left: first $64 = 2^6$ points.
  Middle: first $96$ points.
  Right: first $128 = 2^7$ points.
  In between two powers of~$2$ we can gradually fill up space by making use of the radical inverse ordering.}\label{fig:lattice-sequence-points}
\end{figure}

Technically~\eqref{eq:R1L-points} is a ``rank-$1$ lattice'' point set.
The rank refers to the minimum number of rational generators that are needed to span the full point set.
In case of a rank-$1$ lattice, there is only one rational generator, i.e., $\bsz / n$.
In recent literature we typically only consider rank-$1$ lattices as they are powerful enough to prove useful error bounds.
More information on rank-$r$ lattices, with $1 \le r \le d$, can be found in, e.g., \cite{SJ94,DKP2022,SSN2019}.
One useful fact to remember about general rank-$r$ lattices is that a regular grid in $d$ dimensions can always be represented by $d$ rational generators, and so theory for general rank-$r$ lattices also holds for regular grids, see, e.g., \cite{SSN2019}.
It is even possible to represent some regular grids as rank-$1$ lattices: provided $n_1$ and $n_2$ are relatively prime then the product grid of $n_1 \times n_2$ points can be represented by a rank-$1$ lattice with generating vector $\bsz = (n_2, n_1)$ and $n = n_1 \, n_2$ points.
This can also be done in higher dimensions.
The following produces the regular grid in the middle plot of Figure~\ref{fig:r1l-points-b}.
\Needspace{9\baselineskip}
\begin{lstlisting}
# A regular n_1 by n_2 grid as a rank-1 lattice
# when n_1 and n_2 are relatively prime
# but not relatively prime to n.
# We take z = (n_2, n_1) with n = n_1 * n_2 points.
n_1 = 9; n_2 = 10;
n = n_1 * n_2;
scatter_points(lattice_points([n_2, n_1], n))
\end{lstlisting}

It is particularly useful to stress the group structure of a lattice point set.
Given $(d, n, \bsz)$, if we take two points $\bsx_k$ and $\bsx_{k'}$ as in~\eqref{eq:R1L-points}, then they form a group with addition modulo~$1$ as the operation.
That is, $\bsx_k + \bsx_{k'}$ is another point of the same lattice point set, and so is $\bsx_k - \bsx_{k'}$; where in both cases the addition is interpreted as the group operation, i.e., modulo~$1$.
It can easily be verified that, in our case, we obtain the points $\bsx_{k+k' \bmod{n}}$ and $\bsx_{k-k' \bmod{n}}$ respectively, so under the group operation we have $-\bsx_{k'} = \bsx_{-k' \bmod{n}}$.
For a rank-$1$ lattice point set this structure clearly comes from the underlying operations modulo~$n$ on the ring $\Z_n$ with addition and multiplication modulo~$n$.
This is the reason why we prefer the definitions with modulo~$n$ in~\eqref{eq:R1L-rule} and~\eqref{eq:R1L-points}.
This ring structure is also the basis for the so-called fast component-by-component construction of lattice rules, see, e.g., \cite{NC2006-prime,NC2006,CKN2006,DKS2013,Nuy2014,KN2016,DKP2022}.

From~\eqref{eq:R1L-points} it is clear that it is sufficient to consider the generating vector to be modulo~$n$.
Most of the literature will even consider the components to only come from the multiplicative group modulo~$n$, also called the units of $\Z_n$,
\begin{align*}
  \Z_n^\times
  :=
  \big\{ z \in \Z_n : \gcd(z, n) = 1 \big\}
  .
\end{align*}
A good intuitive interpretation of multiplication with a unit element, i.e., $z \in \Z_n^\times$, is that it actually defines a permutation on $\Z_n$.
I.e., we have $z \, \Z_n = \{ z \, a \bmod{n} : a \in \Z_n \} = \Z_n$ for any $z \in \Z_n^\times$.
A special property of these permutations is that they always map $0$ to~$0$.
In contrast, if $z \in \Z_n \setminus \Z_n^\times$, then, for $d=1$, we have
\begin{align*}
  x_k
  &=
  (k z \bmod{n}) / n
  =
  (k g z' \bmod{g n'})/(g n')
  =
  (k z' \bmod{n'})/n'
  ,
\end{align*}
with $g = \gcd(z, n)$, $z = g \, z'$ and $n = g \, n'$.
This means we only have $n' = n/g$ unique points which are repeated $g$ times.
Clearly, at least for $d = 1$, this should be avoided.
The middle and right plots in Figure~\ref{fig:r1l-points-b} illustrate this loss in uniquely projected points.
In the middle plot this results in a regular grid.
The following produces the right plot in Figure~\ref{fig:r1l-points-b}, where only the first coordinate has fewer unique projected points.
\begin{lstlisting}
# Projected points falling on top of each other for x_1
# since gcd(z_1, n) > 1.
# Note 90 = 2 * 5 * 3^2.
scatter_points(lattice_points([10, 11], 90))
\end{lstlisting}

Often the literature on lattice rules takes $n = p$ to be prime, and in that case we have $\Z_p^\times = \{1, \ldots, p-1\} = \Z_p \setminus \{0\}$.
In number theory it is often convenient to work with prime $p$ since then $\Z_p$ is a field.
This means that every nonzero residue is invertible and some of the counting arguments used in error bounds and construction proofs become simpler.
On the other hand, proofs for general composite $n$ can become quite complicated; but more structure, like $n = 2^m$ can simplify things again.
Proofs for good generating vectors also often restrict the components of $z \in \Z_n^\times$.
At least for prime $n = p$ this is often not necessary and taking $z \in \Z_p$ often results in very similar bounds.
The cardinality of $\Z_n^\times$ is given by the Euler totient function $\varphi(n) = |\Z_n^\times|$.
Bounds from the literature which take components in $\Z_n^\times$ often state error upper bounds in terms of a power of $\varphi(n)^{-1}$ instead of $n^{-1}$.
For prime $n = p$ we have $\varphi(p) = p-1$ and then $\varphi(p)^{-1} = 1/(p-1) \le 2/p$.
For $n = p^m$ a power of a prime, we have $\varphi(p^m) = p^{m-1} (p-1)$ and so here also $\varphi(p^m)^{-1} \le 2/p^m$.
So in these cases it is easy to rewrite the bound in terms of~$n^{-1}$.
Those cases are mostly the more practically interesting.

Let us end this section by the obvious: for $d = 1$ and $z_1 \in \Z_n^\times$, a rank-$1$ lattice point set is precisely the equidistant grid $\{ k/n : k \in \Z_n \}$, possibly generated in a different order when iterating $k = 0, 1, \ldots, n-1$.
Using this as points in an equal-weight quadrature rule then results in the left-rectangle rule.

\subsection{Lattice point sequences}

Often people think about lattice rules and lattice point sets as a fixed collection of $n$ points that we use in a quadrature/cubature rule to approximate an integral.
It is however possible to look at consecutive powers of a base, say $2$ for convenience, and then to have a single good generating vector $\bsz$ that works for all $n$ of the form $2^m$ for some range of~$m$.
Depending on the function space we can also use fast component-by-component algorithms.
E.g., for a Korobov space with product weights the construction cost is $O(d \, N (\log N)^2)$ for a generating vector that works for all $n=2^m$ with $m_1 \le m \le m_2$ and $N = 2^{m_2}$, see, e.g., \cite{CKN2006}.
Moreover, these point sets are embedded into each other.
We hence speak about \emph{embedded lattice point sets} or \emph{lattice sequences} (in base~$2$, or whatever base was used).

The advantage of having a sequence (and an equal weight rule) is that we can just keep on drawing extra samples if we think the approximation is not accurate enough.
Specifically for base~$2$ when you go from the initial $2^m$ points to the set with $2^{m+1}$ points, and considering
\begin{align*}
  \bsx_k
  &=
  \frac{\bsz k \bmod{2^{m+1}}}{2^{m+1}}
\end{align*}
it is easy to see that the even $k$'s reduce to the previous points, since, with $k = 2\ell$,
\begin{align*}
  \bsx_{2\ell}
  &=
  \frac{\bsz 2\ell \bmod{2^{m+1}}}{2^{m+1}}
  =
  \frac{\bsz \ell \bmod{2^m}}{2^m}
  ,
\end{align*}
and hence the odd $k$'s will generate the additional points.

We will make use of this in the following code snippet to generate the sequence in blocks, each block just containing the extra points to make up the full next lattice point set.
The generating vector we use is the one from~\cite{HKKN2011}.
Note that we are using a trick.
We are not writing \lstinline{1:2:2^m - 1} but \lstinline{1:2:2^m} as the range for~\lstinline{k}, this is such that the range also works for $m = 0$, which is the first block for which we need $k = 0$ and we have $1 \bmod 1 = 0$ so this correctly generates the zero vector as the first block.

\Needspace{5\baselineskip}
\begin{lstlisting}
using Statistics

lattice_points_base2_block(z, m) =
  ( mod.(z .* k, 2^m) ./ 2^m for k in 1:2:2^m )
# Note the special case for m = 0 which will give the
# range 1:2:1 = ( 1 ) which
# then correctly gives the point zero as 1 mod 1 = 0.
# All other ranges are fine too since we
# iterate over odd numbers and 2^m
# is always even for integer m >= 1.
\end{lstlisting}
\Needspace{3\baselineskip}
\begin{lstlisting}
# First 3 dimensions of the HKKN2011 vector, with n_max = 2^20
z = (1, 364981, 245389)
@test collect(lattice_points_base2_block(z, 0)) == [(0, 0, 0)]
@test collect(lattice_points_base2_block(z, 1)) == [(0.5, 0.5, 0.5)]
@test collect.(lattice_points_base2_block.(Ref(z), 0:3)) ==
  # vectorise m using Ref(z)
  [
    [(0.0, 0.0, 0.0)],
    [(0.5, 0.5, 0.5)],
    [(0.25, 0.25, 0.25), (0.75, 0.75, 0.75)],
    [(0.125, 0.625, 0.625), (0.375, 0.875, 0.875),
      (0.625, 0.125, 0.125), (0.875, 0.375, 0.375)]
  ]
# Each block is a different-length vector of 3-dimensional tuples.
# `reduce(vcat, ...)` joins them into one vector of tuples.
# `reduce(hcat, collect.(...))` then gives a 3-by-8 matrix.
# We only need the concatenation for the following unit tests:
@test mean(collect.(reduce(vcat,
  collect.(lattice_points_base2_block.(Ref(z), 0:7))))) ==
  [1, 1, 1] * (2^7 - 1)/(2 * 2^7)
@test mean(collect.(reduce(vcat,
  collect.(lattice_points_base2_block.(Ref(z), 1:7))))) ==
  [1, 1, 1] * 1/2

# To generate all the points in the blocks from
# 0:3 you can much easier do:
P_firstpoints_a = collect(lattice_points(z, 2^3))
# which gives the same *collection* of points,
# but in different order, as
P_firstpoints_b = reduce(vcat,
  collect.(lattice_points_base2_block.(Ref(z), 0:3)))
# so they are equal if we sort them lexicographical:
@test sort(P_firstpoints_a) == sort(P_firstpoints_b)
\end{lstlisting}

It is always a good idea to devise a simple unit test in the code to avoid making stupid mistakes.
The two tests using the mean make use of the fact that we have a generating vector that is relatively prime to $n = 2^m$ for each component and any $m$, i.e., the generating vector consists of odd numbers.
Hence for $n = 2^m$ every individual component is a permutation of the equidistant points $\{ k / n : k \in \Z_n \}$.
If we then take the vector mean over all generated points except the first one we have
\begin{align*}
  \frac1{n-1} \sum_{k=1}^{n-1} \frac{k \bsz \bmod{n}}{n}
  &=
  \begin{bmatrix} \displaystyle \frac1{n-1} \sum_{k=1}^{n-1} \frac{k z_j \bmod{n}}{n} \end{bmatrix}_{j=1,\ldots,d}
  \\
  &=
  \begin{bmatrix} \displaystyle \frac1{n-1} \sum_{k=1}^{n-1} \frac{k}{n} \end{bmatrix}_{j=1,\ldots,d}
  =
  \begin{bmatrix} \displaystyle \frac12 \end{bmatrix}_{j=1,\ldots,d}
  .
\end{align*}
This is what the second mean unit test checks for our block generator.
The first mean unit test just also includes the first point which is the zero vector, and hence rescales the result.

In most practical cases it is sufficient to just generate the whole lot of additional points.
Sometimes we actually do want to generate the sequence point by point.
This is possible by using the radical inverse function in the same base as the sequence, see, e.g., \cite{HHLL2001,HN2003,CKN2006,HKKN2011}.
The radical inverse function in base~$2$, denoted by $\phi_2 : \N_0 \to \Q \cap [0, 1)$, just mirrors the bits around the base-$2$ radix point.
E.g., for $k = 3 = (11)_2$ we get $\phi_2(3) = (0.11)_2 = \nicefrac{1}{2} + \nicefrac{1}{4} = \nicefrac{3}{4}$ and for $k = 6 = (110)_2$ we get $\phi_2(6) = (0.011)_2 = \nicefrac{1}{4} + \nicefrac{1}{8} = \nicefrac{3}{8}$.
Using this we can generate the lattice sequence in the correct ordering in each block as follows
\begin{align*}
  \bsx_k
  &=
  \bsz \, \phi_2(k) \bmod{1}
  ,
  &
  \text{for } k = 0, 1, \ldots
  .
\end{align*}
If we split this sequence in blocks of the sizes generated by the block generator and ignore the ordering inside the blocks then we generate exactly the same points per block.
So each time we reach a full power of $2$ we have a full lattice point set and they are embedded into each other just as above.

The following code implements the radical inverse ordering in base~$2$ and produces the three panels in Figure~\ref{fig:lattice-sequence-points}.
\begin{lstlisting}
# Reverse all bits of an unsigned integer and scale by its bit width.
radical_inverse_base2(k) =
  bitreverse(unsigned(k)) / 4 / 2^(8 * sizeof(k) - 2)

lattice_sequence(z) =
  (mod.(z .* radical_inverse_base2(k), 1)
    for k in Iterators.countfrom(0))

z = (1, 364981)
points = collect(Iterators.take(lattice_sequence(z), 128))

@test mean(collect.(points[2:64])) == [1, 1] * 1/2
@test mean(collect.(points[2:128])) == [1, 1] * 1/2

\end{lstlisting}
\Needspace{4\baselineskip}
\begin{lstlisting}
scatter_points(points[1:64])
scatter_points(points[1:96])
scatter_points(points[1:128])
\end{lstlisting}

\section{Setting the scene for numerical integration}\label{sec:quadrature}

We now numerically approximate some integrals.
We deliberately take non-periodic functions as test functions although lattice rules shine when integrating periodic functions.
(For a function in a periodic space like $\widetilde{H}^\alpha_\mix([0,1]^d)$ we can have a convergence of $O(n^{-\alpha+\epsilon})$.)
This is to demonstrate that you can still use lattice rules on non-periodic functions, and there are many ways to do so.
Here we will employ random shifting (when $d \ge 2$) and, without going too much into the function space business yet, we will numerically show a convergence of $O(n^{-1+\epsilon})$ if the integrand is in $H^\alpha_\mix([0,1]^d)$ for any $\alpha \in \N$, i.e., $\alpha \ge 1$, for details, see~\cite{SKJ2002}.
We will also numerically show that the tent transformation can give you $O(n^{-2+\epsilon})$ if the integrand is in $H^\alpha_\mix([0,1]^d)$ with $\alpha \ge 2$, for details, see~\cite{Hic2002}.
Keep in mind that this section is still meant to be very informal, the real function space business begins in Section~\ref{sec:function-spaces}.

\subsection{Just one-dimensional}

Before we even start looking at high-dimensional integrals we should discuss the one-dimensional case.
Say, suppose, you need to get a value for the following integral
\begin{align}\label{eq:example-integral-d=1}
  \int_0^2 \exp(-x^5 + x) \rd{x}
  \qquad\Bigg(\;=\;\quad
  \int_0^1 \exp(-(2x)^5 + (2x)) \, 2 \rd{x} \quad\;\Bigg)
  .
\end{align}
(The reason why we have written the second integral with the substitution will become clear in a minute.)
A picture of the integrand function is given in Figure~\ref{fig:example-integrand-d=1}.
We have lots of options to approximate this integral: Gaussian quadrature, Newton--Cotes rules, Simpson's rule, the trapezoidal rule, \ldots\ and then we have not even considered adaptive integration software such as QUADPACK~\cite{PdDUK83} or Cubature~\cite{Joh2005}.

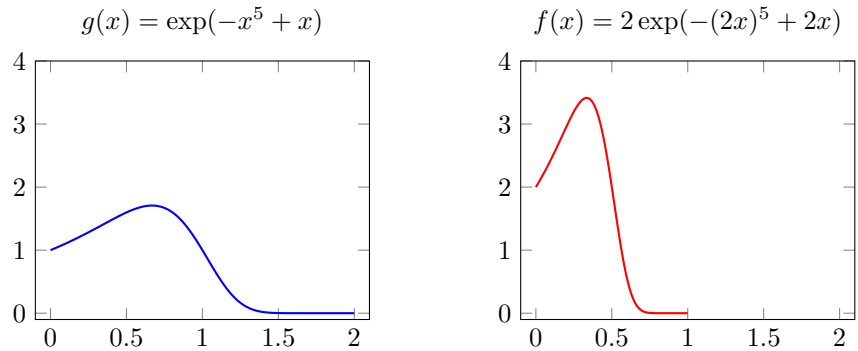
\begin{figure}
  \centering
  \tikzsetnextfilename{figure_example_integrand}
\begin{tikzpicture}
  \begin{groupplot}[
    group style={
      group size=2 by 1,
      horizontal sep=2cm
    },
    width=6cm,
    height=5cm,
    samples=200,
    ymin=-0.1, ymax=4, 
    xmin=-0.1, xmax=2.1
  ]
    \nextgroupplot[
      domain=0:2,
      title={$g(x) = \exp(-x^5 + x)$}
    ]
    \addplot[blue, thick] {exp(-x^5 + x)};

    \nextgroupplot[
      domain=0:1,
      title={$f(x) = 2 \exp(-(2x)^5 + 2x)$}
    ]
    \addplot[red, thick] {2*exp(-(2*x)^5 + 2*x)};
  \end{groupplot}
  \end{tikzpicture}
  \caption{Our one-dimensional example function. Left: original domain. Right: after substitution.}\label{fig:example-integrand-d=1}
\end{figure}

A common trick in array-oriented languages like Python/NumPy and Matlab is to ``vectorise'' your function such that the function can be evaluated in all the sample points at once.
This then results in a speed advantage as array operations are dispatched to optimally compiled code.
One advantage of using Julia is that you do not need to write your function vectorised from the start, you can just use the dot-notation (as explained in the previous section) since Julia will compile your for-loops.
For our application of multivariate integration: our single input is a vector of length~$d$, or a scalar for $d=1$.
So for the function $g(x) = \exp(-x^5 + x)$ over $[0, 2]$ we can just straightforwardly write it as follows.
\begin{lstlisting}
# Define `g' as the name for our function over [0, 2] (defined
# as an anonymous function)
g = x -> exp(-x^5 + x)
\end{lstlisting}
\Needspace{3\baselineskip}
\begin{lstlisting}
# Note that we could also have written
#   g(x) = exp(-x^5 + x)
# which would be a shorthand for
#   function g(x)
#      exp(-x^5 + x)
#   end
# The version `g = x -> ...' might be preferred since it makes
# a difference when you are e.g. in interactive mode (in the REPL)
# since the latter two versions add `g' to the method table and this
# could be confusing if you later add another function which has
# multiple arguments, as the old method will still exist...
# You can check this with `methods(g)'.
\end{lstlisting}

The integral~\eqref{eq:example-integral-d=1} is defined over $[0, 2]$.
Lattice points and rules are defined over the unit cube $[0, 1]^d$.
To apply our lattice rule we will use the simple substitution $x = 2t$ which maps $[0,1]$ to $[0,2]$ and has constant Jacobian~$2$.
We write this change of variable as the integral over $[0, 1]$ for a new function that we will call \lstinline!f!, see also~\eqref{eq:example-integral-d=1} above.
Such substitutions are standard when you want to apply basic quadrature rules, e.g., Gauss--Legendre rules are defined over $[-1,1]$.
\Needspace{3\baselineskip}
\begin{lstlisting}
# Our mapped function over [0, 1]
f = x -> 2 * g(2 * x)
\end{lstlisting}

Since a lattice rule is just the average over the lattice points we can use the \lstinline!mean! function from the \lstinline!Statistics! module.
\begin{lstlisting}
# Note that we pass in z = 1 (a scalar)
# instead of z = [ 1 ] (a vector).
# This is because our function expects scalars.
mean(f, lattice_points(1, 55))
# This gives    1.5777499333497484
# true value is 1.5596783055678136 in double precision

\end{lstlisting}
\Needspace{11\baselineskip}
\begin{lstlisting}
# Let's also do some simple unit tests to see we did
# not make stupid mistakes...
@test mean(x -> 1, lattice_points(1, 55)) == 1
@test isapprox(mean(x -> x - 1/2, lattice_points(1, 55)), -1/2/55)
@test isapprox(mean(x -> x - 1/2, lattice_points(1, 50)), -1/2/50)
@test isapprox(mean(x -> x - 1/2, lattice_points(12, 55)), -1/2/55)
  # gcd(12, 55) = 1
@test isapprox(mean(x -> x - 1/2, lattice_points(12, 50)), -1/2/25)
  # gcd(12, 50) = 2
\end{lstlisting}
Does this perform as you expect when you look at the ``lattice rule approximation'' versus the true value? (See snippet above: we obtain $1.5777499333497484$ versus the reference value $1.5596783055678136$. Not too good.)
First note that a lattice rule for $d=1$ is just a ``left-rectangle rule''.
As long as we take $z_1$ relatively prime to $n$, we get $n$ distinct points $x_k = k / n$, $k = 0, \ldots, n-1$, and $z_1$ just determines a permutation of these points.
Remember your numerical analysis course: we do not expect terribly good performance from the left-rectangle rule.
This is like the simplest approximation of them all.
The typical error bound states that for a function with bounded first derivative we should see convergence $O(n^{-1})$.
We will check this convergence in a minute.

The code snippet above also used another simple unit test.
Integrating the function $x \mapsto x - 1/2$ over $[0, 1]$ gives $0$ since it is antisymmetric around~$1/2$.
The lattice rule acts as a left-rectangle rule which has symmetric points around $1/2$, except for the first point, and possibly a point at $1/2$ when $n$ is even, but our function is zero there.
Because of the antisymmetry of our function all the symmetric evaluations cancel against each other and we have, for $\gcd(z_1, n) = 1$,
\begin{align*}
  \frac1n \sum_{k=0}^{n-1} \left(\frac{k}{n} - \frac12\right)
  &=
  -\frac1{2n}
  .
\end{align*}
While for $\gcd(z_1, n) = g$ we have $n/g$ unique points but with the same properties.

There are other simple quadrature rules that we can build out of our ``one-dimensional lattice points''.
The midpoint rule is particularly easy as well and for functions with bounded second derivative it should converge like $O(n^{-2})$.
\Needspace{5\baselineskip}
\begin{lstlisting}
# Midpoint rule
mean(x -> f(x + 1/2/55), lattice_points(1, 55))
# This gives    1.5597333998528078
# true value is 1.5596783055678136 in double precision
\end{lstlisting}

\begin{figure}
  \centering
\tikzsetnextfilename{figure_tent_transform}
\begin{tikzpicture}
  \begin{groupplot}[
    group style={
      group size=2 by 1,
      horizontal sep=2cm,
    },
    width=5.5cm,
    height=4cm,
    domain=0:1,
    samples=200,
    ymin=0, ymax=1.1,
    xmin=0, xmax=1.1,
    axis lines=middle,
    xlabel=$x$,
    ylabel=$\varphi(x)$,
    extra x ticks={0},
    extra y ticks={0},
  ]
    \nextgroupplot[title={$n=10$}]
    \addplot[blue, thick] {1 - abs(1 - 2*x)};
    \foreach \k in {0,...,9} {
      \addplot[only marks, mark=*, mark options={fill=blue}] coordinates {(\k/10, {1 - abs(1 - 2*(\k/10))})};
    }
    \nextgroupplot[title={$n=11$}]
    \addplot[blue, thick] {1 - abs(1 - 2*x)};
    \foreach \k in {0,...,10} {
      \addplot[only marks, mark=*, mark options={fill=blue}] coordinates {(\k/11, {1 - abs(1 - 2*(\k/11))})};
    }
  \end{groupplot}
\end{tikzpicture}
  \caption{The tent transform maps points on the $x$-axis to the points on the $y$-axis.}\label{fig:tent}
\end{figure}
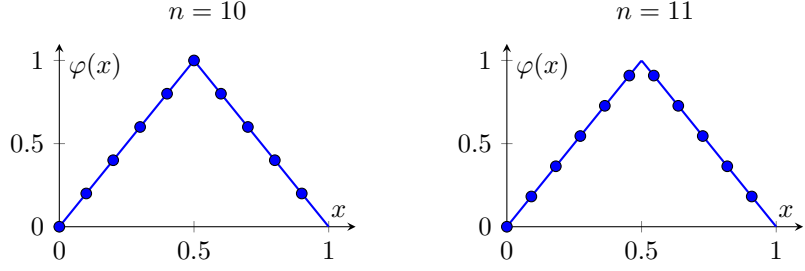

Yet another rule that we can build from ``one-dimensional lattice points'' is the trapezoidal rule.
For this we use the tent transform
\begin{align*}
  \varphi(x)
  &:=
  1 - |1-2x|
  ,
\end{align*}
which was originally proposed in \cite{Hic2002}, see also \cite{DNP2014,GSY2019}.
The tent transform is a Lebesgue preserving transformation.
Splitting the integral over $[0,1/2]$ and $[1/2,1]$ and performing a substitution on each interval gives $\int_0^1 f(\varphi(x)) \rd{x} = \int_0^1 f(x) \rd{x}$.
See Figure~\ref{fig:tent} for an illustration on how this maps one-dimensional lattice points.
If $n = 2m$ is even, we can see this as mapping all the interior points to halve the resolution, each with weight~$2$, and mapping the point at $0$ and at $1/2$ to $0$ and $1$ respectively, with weight~$1$.
That is exactly the trapezoidal rule.
For odd $n$ this is not the case, but you can still obtain $O(n^{-2})$ convergence.
\Needspace{6\baselineskip}
\begin{lstlisting}
# Trapezoidal rule
tent(x) = 1 .- abs.(1 .- 2*x)
mean(x -> f(tent(x)), lattice_points(1, 55))
# This gives    1.5592375736545039
# true value is 1.5596783055678136 in double precision
\end{lstlisting}
The implementation does the straightforward mapping, which is slightly wasteful because of the double function evaluations.
This could be avoided, see \cite{DNP2014} which explains how to avoid this for general~$d$.

The expected $O(n^{-1})$ for the left-rectangle rule (i.e., plain lattice rule) and $O(n^{-2})$ convergence for the midpoint rule (i.e., shifted lattice rule) and trapezoidal rule (i.e., tent transformed lattice rule) can be checked directly on a log-log plot.
The plot produced by the following code is the left panel of Figure~\ref{fig:convergence-plots}.
\Needspace{8\baselineskip}
\begin{lstlisting}
I_f = 1.5596783055678136
ns = 2 .^ (4:20)
lr_errors = [
  abs(mean(f, lattice_points(1, n)) - I_f)
  for n in ns
]
mp_errors = [
  abs(mean(x -> f(x + 1/(2*n)), lattice_points(1, n)) - I_f)
  for n in ns
]
tr_errors = [
  abs(mean(x -> f(tent(x)), lattice_points(1, n)) - I_f)
  for n in ns
]

\end{lstlisting}
\Needspace{6\baselineskip}
\begin{lstlisting}
plt = scatter(ns, [lr_errors tr_errors mp_errors],
  xscale=:log10, yscale=:log10,
  label=["left rectangle" "tent transformed" "midpoint"])
plot!(plt, ns, 1 ./ ns, label="n^-1", linestyle=:dash)
plot!(plt, ns, 1 ./ ns.^2, label="n^-2", linestyle=:dash)
\end{lstlisting}
From the left panel of Figure~\ref{fig:convergence-plots} we clearly see that the tent transformed rule has $O(n^{-2})$ convergence.
Since we used the straightforward implementation, we are doing about twice as many function evaluations as necessary, and this results in the curve being shifted by a factor of~$2$ to the right.
In higher dimensions the tent transformed lattice rule will still be able to show $O(n^{-2})$ convergence for non-periodic functions~\cite{Hic2002,GSY2019} and we demonstrate this in the next section.
On the other hand, our trick to shift a one-dimensional lattice rule to obtain the midpoint rule does not have an immediate multivariate equivalent.

\subsection{Higher dimensional and randomly shifted lattice rules}

Our next step is to do multivariate integrals.
A direct tensor product of a classical univariate rule with $m$ points leads to a multivariate rule with $n = m^d$ points.
Of course we need to take $m \ge 2$ for this to make sense, otherwise, with $m = 1$, we just have $n = 1^d = 1$ point.
Unfortunately, for $m = 2$ and say $d = 100$ dimensions we already have that
\begin{align*}
  n
  =
  2^{100}
  &=
  1\,267\,650\,600\,228\,229\,401\,496\,703\,205\,376
  \approx
  10^{30}
\end{align*}
so even if one function evaluation would just take $1~\textrm{ns}$ then the time of doing these function evaluations would be about $10^{21}~\textrm{s}$ which is about $10^{13}~\textrm{years}$ which is about $3000$ times more than the current estimated age of the universe.
Still, we would like to do $100$-dimensional integrals.
E.g., in financial engineering we meet integrals with hundreds, thousands and even tens of thousands of dimensions, and in uncertainty quantification this can even be much much larger.
That means that simple tensor product rules are out of the picture since their number of sample points increases exponentially in~$d$.

Luckily, lattice rules allow you to freely choose the number of sample points independent of~$d$.
(Of course, that does not mean that a $32$-point lattice rule would be a good idea to approximate a $1000$-dimensional integral.
Unless your function / function space is super super nice.)

Let us experiment on the following function
\begin{align*}
  g_d(\bsx)
  &:=
  \prod_{j=1}^d (x_j + 1/2)
  .
\end{align*}
Clearly for any number of dimensions $d$ we have $I_d(g_d) = 1$, so this is a perfect test integrand function for which we can compute the exact error.

In practice of course we would not be able to compute the exact integral, otherwise what is the point of trying to numerically approximate it?
To get a practical error estimate we will use $M$ independent randomly shifted lattice rules.
Given $d, n \in \N$, $\bsz \in \Z^d$, $\bsDelta \in [0, 1]^d$, a \emph{shifted lattice rule} is
\begin{align}\label{eq:shifted-R1L-rule}
  Q_{d,n,\bsz}^\bsDelta(f)
  &:=
  \frac1n \sum_{k=0}^{n-1} f\bigg(\big((\bsz k / n) + \bsDelta\big) \bmod{1}\bigg)
  .
\end{align}
If we compare this with the standard lattice rule $Q_{d,n,\bsz}$, see~\eqref{eq:R1L-rule}, then we notice that we have shifted all of our points by the given shift $\bsDelta$ modulo~$1$ over the unit cube.
Note that
\begin{align*}
  \big( ((\bsz k \bmod{n}) / n) + \bsDelta \big) \bmod{1}
  &=
  \big( (\bsz k / n) \bmod{1} + \bsDelta \big) \bmod{1}
  \journalbreak =
  \big( (\bsz k / n) + \bsDelta \big) \bmod{1}
  .
\end{align*}
We obtain a \emph{randomly shifted lattice rule} if we consider the shift $\bsDelta$ to be random.
See, e.g.,~\cite{LL2000}, for an overview on randomly shifted lattice rules.
If the shift is uniform over the unit cube, $\bsDelta \sim U([0,1]^d)$, then we can show that $Q_{d,n,\bsz}^\bsDelta$ is unbiased since, for $f \in L_1([0,1]^d)$,
\begin{align*}
  \bbE_{\bsDelta \sim U([0,1]^d)}[Q_{d,n,\bsz}^\bsDelta(f)]
  &=
  \int_{[0,1]^d} \frac1n \sum_{k=0}^{n-1} f\bigg(\big((\bsz k / n) + \bsDelta\big) \bmod{1}\bigg) \rd\bsDelta
  \\
  &=
  \frac1n \sum_{k=0}^{n-1} \int_{[0,1]^d} f\bigg(\big((\bsz k / n) + \bsDelta\big) \bmod{1}\bigg) \rd\bsDelta
  \\
  &=
  \frac1n \sum_{k=0}^{n-1} \int_{[0,1]^d} f(\bsDelta) \rd\bsDelta
  \\
  &=
  I_d(f)
  ,
\end{align*}
where we used that, for fixed $\bsa$, the map $\bsx \mapsto (\bsa + \bsx) \bmod{1}$ is a measure-preserving map, like the tent transform, and hence
\begin{align*}
  \int_{[0,1]^d} f((\bsx + \bsa) \bmod{1}) \rd\bsx
  &=
  \int_{[0,1]^d} f(\bsx) \rd\bsx
  .
\end{align*}

If we now consider $M$ i.i.d.\ random shifts $\bsDelta_r \sim U([0,1]^d)$, $r=1,\ldots,M$, then we get $M$ independent estimators
\begin{align*}
  Q_{d,n,\bsz}^{\bsDelta_1}(f)
  ,\,
  Q_{d,n,\bsz}^{\bsDelta_2}(f)
  ,\,
  \ldots
  ,\,
  Q_{d,n,\bsz}^{\bsDelta_M}(f)
  .
\end{align*}
Write $\omega = (\bsDelta_1,\ldots,\bsDelta_M)$ and take the average as our final (unbiased) estimator
\begin{align*}
  \overline{Q}_{d,n,\bsz,M}^\omega(f)
  &:=
  \frac1M \sum_{r=1}^M Q_{d,n,\bsz}^{\bsDelta_r}(f)
  .
\end{align*}
For $M \ge 2$ and $f \in L_2([0,1]^d)$ the standard deviation of this average is called the \emph{standard error} and we estimate it by
\begin{align*}
  \hat{S}(\overline{Q}_{d,n,\bsz,M}^\omega(f))
  &:=
  \sqrt{
  \frac1{(M-1) M} \sum_{r=1}^M
  \Big| Q_{d,n,\bsz}^{\bsDelta_r}(f) - \overline{Q}_{d,n,\bsz,M}^\omega(f) \Big|^2
  }
  .
\end{align*}
We use this as an error estimator in the next Julia snippets.
In the code we will take the sample standard deviation of the collection of results using the function \lstinline{std} and then divide by $\sqrt{M}$ to get the estimated standard error on $\overline{Q}_{d,n,\bsz,M}^\omega(f)$.
We approximate the integral for $g_d$ for $d = 2, 10, 100$ using $M = 10$ randomly shifted lattice rules.
The generating vector is from~\cite{CKN2006}.
For each of these we plot the estimated standard error in the middle panel of Figure~\ref{fig:convergence-plots} against the number of lattice points~$n$.
Note that this is different from the total number of function evaluations, which is~$M \, n$.

\begin{lstlisting}
using Random

# A test integrand with integral one in every dimension
gd = x -> prod(1 .+ (x .- 1/2))

shift_mod1(x, shift) = mod.(x .+ shift, 1)

\end{lstlisting}
\Needspace{7\baselineskip}
\begin{lstlisting}
# The first 10 components of the vector from CKN2006 are
# (1, 182667, 469891, 498753, 110745, 446247,
# 250185, 118627, 245333, 283199).
# The full vector can be found online at the Magic Point Shop:
# https://people.cs.kuleuven.be/~dirk.nuyens/qmc-generators/
z = parse.(Int, readlines("exod2_base2_m20_CKN.txt"))

# For simplicity, generate each full rule
# separately instead of reusing the
# blocks or the radical-inverse sequence from the previous section.
lattice_rule_sequence(z, m1, m2) =
  (lattice_points(z, 2^m) for m in m1:m2)

m1 = 4; m2 = 20
ns = 2 .^ (m1:m2)
ds = (2, 10, 100)
M = 10
Random.seed!(2026) # For reproducibility.
shifts = rand(maximum(ds), M)

\end{lstlisting}
\Needspace{7\baselineskip}
\begin{lstlisting}
# A list of results for each d in ds
# where each result is a matrix with rows corresponding to 2^m points
# and columns to independent random shifts.
# This takes some time to compute... 2^20 is about one million.
Qs = [
  reduce(hcat,
    [mean(x -> gd(shift_mod1(x, shift)), points)
      for points in lattice_rule_sequence(z[1:d], m1, m2)]
    for shift in eachcol(shifts[1:d, :]))
  for d in ds
]
stderrs = hcat([vec(std(Q, dims=2)) ./ sqrt(M) for Q in Qs]...)

\end{lstlisting}
\Needspace{5\baselineskip}
\begin{lstlisting}
plt = scatter(ns, stderrs, xscale=:log10, yscale=:log10,
  label=["stderr d=$d" for _ in [1], d in ds])
plot!(plt, ns, stderrs[1, 1] .* ns[1] ./ ns,
  label=false, linestyle=:dash)
\end{lstlisting}

From the middle panel of Figure~\ref{fig:convergence-plots} we see that as $d$ increases the problem becomes harder: the estimated standard error curves shift upward and the observed (pre-asymptotic (!)) convergence rate slows down (the lines become flatter).
This is because our test integrand is not so innocent.
We can write it as
\begin{align*}
  g_d(\bsx)
  &=
  \prod_{j=1}^d (x_j + 1/2)
  =
  \prod_{j=1}^d (1 + (x_j - 1/2))
  =
  \sum_{\setu \subseteq \{1,\ldots,d\}} \prod_{j\in\setu} (x_j - 1/2)
  .
\end{align*}
The variance of this function is $\sigma^2(g_d) = (13/12)^d - 1$ and grows exponentially in~$d$.
Hence even Monte Carlo will suffer for increasing~$d$.
As will become clear later when we introduce weighted function spaces, we could try to moderate the importance of the different dimensions by adding in product weights $\beta_j$ and assuming $\beta_1 \ge \beta_2 \ge \cdots \ge \beta_d > 0$ and defining a nicer version
\begin{align*}
  g_{d,\bsbeta}(\bsx)
  &:=
  \prod_{j=1}^d (1 + \beta_j \, (x_j - 1/2))
  =
  \sum_{\setu \subseteq \{1,\ldots,d\}} \prod_{j\in\setu} (\beta_j \, (x_j - 1/2))
  .
\end{align*}
We encourage the reader to play with this function and compare the convergence for $g_{d,\bsbeta}$ with that of the unweighted $g_d$ function.
A possible choice is to take $\beta_j = j^{-2}$, in which case you quickly see that those terms with large cardinality $|\setu|$ will suddenly not matter much anymore.
In this case we have, for $\setu \ne \emptyset$,
\begin{align*}
  \prod_{j\in\setu} j^{-2}
  &\le
  (\max_{j\in\setu} j^{-2})^{|\setu|}
  .
\end{align*}
Of course, here we are playing with test functions, and we can make up test functions all day long, but in reality it depends on the actual function whether it has such nice or very poor behaviour.
This is why we will analyse weighted function spaces in Section~\ref{sec:mixed-spaces}, to understand which dimension weight structures allow for nice convergence results.

Our function $g_d$ actually belongs to the dominating mixed smoothness space $H^\alpha_\mix([0,1]^d)$ for any $\alpha \in \N$, you can check with Section~\ref{sec:mixed-spaces} later.
If we continue applying randomly shifted good lattice rules with larger and larger~$n$, then the root-mean-square convergence will eventually be $O_{d,\epsilon}(n^{-1+\epsilon})$ for every $\epsilon > 0$.
We just have to spend an exponential budget first since the hidden constant in the big-O notation is exponential in~$d$.
The function $g_d$ also belongs to any of the weighted dominating mixed smoothness spaces, say $H^{\alpha,\bsgamma}_\mix([0,1]^d)$ for any $\alpha \in \N$, with $\bsgamma = \{ \gamma_\setu \}_\setu$ and assuming all $\gamma_\setu > 0$, again, see Section~\ref{sec:mixed-spaces}.
The weights will reduce the influence of $d$ on the worst-case error, but they will blow up the influence of $d$ in the norm of the function.
We will discuss this shortly in Section~\ref{sec:ibc}, in particular see~\eqref{eq:weighted-error-bound}.
The weights of the function space make it possible to analyse what happens when we integrate functions like $g_{d,\bsbeta}$: if the weights in the norm correspond to the structure of the function, then the norm is small, and if this family of weights also gives a worst-case error whose dependence on~$d$ remains uniformly bounded as $d \to \infty$, then we achieve the holy grail of dimension-independent error bounds.
So a nice function and a correctly weighted function space will do the trick.

As a last experiment we will show that indeed the tent transform can still improve the convergence from $O(n^{-1})$ to $O(n^{-2})$ for $d \ge 2$ when using a lattice rule.
We will keep on using the function $g_d$ without weights, but we will keep the dimensionality relatively low at $d = 5$.
We apply the tent transform componentwise.
We will again use $M = 10$ randomly shifted lattice rules and estimate the standard error.
In this case it is important that we first apply the random shift to the point set, then apply the tent transform, and then apply the function, i.e., $f(\varphi((\bsx + \bsDelta) \bmod{1}))$.
Swapping the order of the tent transform and the random shift would not improve the convergence.
For more information on the tent transform see, e.g.,~\cite{Hic2002,DNP2014}.

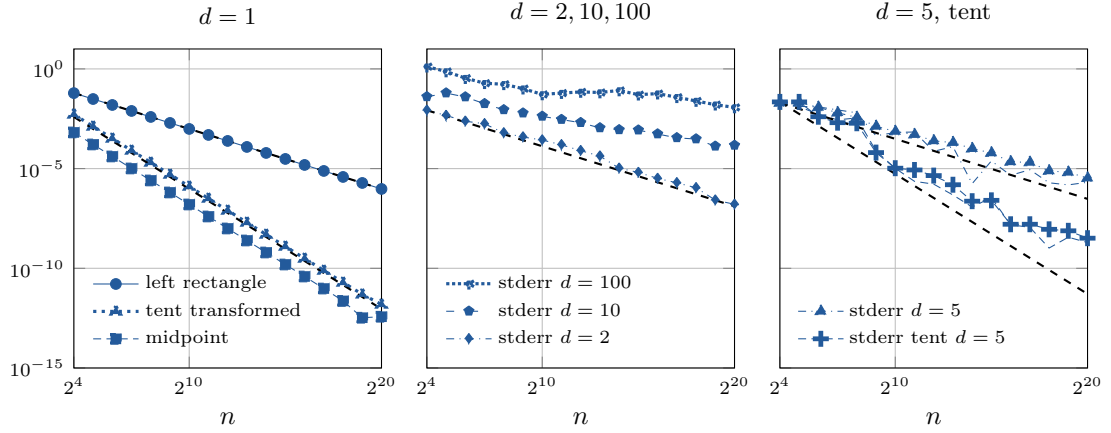
\begin{figure}
  \centering
  \tikzsetnextfilename{figure_convergence}
\def\convergencedata{\articledata/convergence_experiments_data.txt}
\pgfplotstableread[col sep=tab]{\convergencedata}\convergencetable
\pgfplotstablegetelem{0}{n}\of{\convergencetable}
\pgfmathsetmacro{\referencen}{\pgfplotsretval}
\pgfplotstablegetelem{0}{stderr_d2}\of{\convergencetable}
\pgfmathsetmacro{\dimensionreference}{\pgfplotsretval * \referencen}
\pgfplotstablegetelem{0}{shift_stderr_d5}\of{\convergencetable}
\pgfmathsetmacro{\shiftreference}{\pgfplotsretval * \referencen}
\pgfplotstablegetelem{0}{tent_shift_stderr}\of{\convergencetable}
\pgfmathsetmacro{\tentreference}{\pgfplotsretval * \referencen^2}
\begin{tikzpicture}
    \begin{groupplot}[
      group style={group size=3 by 1, horizontal sep=0.6cm, y descriptions at=edge left},
      width=5.65cm,
      height=5.8cm,
      xmode=log,
      ymode=log,
      log basis x=2,
      xmin=16,
      xmax=1048576,
      xtick={16,1024,1048576},
      xticklabels={$2^4$,$2^{10}$,$2^{20}$},
      ymin=1e-15,
      ymax=1e1,
      every mark/.append style={scale=0.6},
      xlabel={$n$},
      grid=major,
      legend style={font=\scriptsize, draw=none, fill=none, cells={anchor=west}},
      tick label style={font=\scriptsize},
      title style={font=\small}
    ]
      \nextgroupplot[title={$d=1$}, legend pos=south west]
      \addplot[deepblue, mark=*] table[x=n, y=left_error, col sep=tab] {\convergencedata};
      \addlegendentry{left rectangle}
      \addplot[deepblue, dotted, very thick, mark=triangle*] table[x=n, y=tent_error, col sep=tab] {\convergencedata};
      \addlegendentry{tent transformed}
      \addplot[deepblue, densely dashed, mark=square*] table[x=n, y=midpoint_error, col sep=tab] {\convergencedata};
      \addlegendentry{midpoint}
      \addplot[black, dashed, thick, domain=16:1048576, samples=2, forget plot] {1/x};
      \addplot[black, dashed, thick, domain=16:1048576, samples=2, forget plot] {1/x^2};

      \nextgroupplot[title={$d=2,10,100$}, legend pos=south west]
      \addplot[deepblue, densely dotted, very thick, mark=asterisk, mark size=2.2pt] table[x=n, y=stderr_d100, col sep=tab] {\convergencedata};
      \addlegendentry{stderr $d=100$}
      \addplot[deepblue, loosely dashed, mark=pentagon*] table[x=n, y=stderr_d10, col sep=tab] {\convergencedata};
      \addlegendentry{stderr $d=10$}
      \addplot[deepblue, dashdotted, mark=diamond*] table[x=n, y=stderr_d2, col sep=tab] {\convergencedata};
      \addlegendentry{stderr $d=2$}
      \addplot[black, dashed, thick, domain=16:1048576, samples=2, forget plot] {\dimensionreference/x};

      \nextgroupplot[title={$d=5$, tent}, legend pos=south west]
      \addplot[deepblue, thin, dashdotted, forget plot] table[x=n, y=shift_error_d5, col sep=tab] {\convergencedata};
      \addplot[deepblue, thin, densely dashed, forget plot] table[x=n, y=tent_shift_error, col sep=tab] {\convergencedata};
      \addplot[deepblue, dashdotted, mark=triangle*, mark size=2.5pt] table[x=n, y=shift_stderr_d5, col sep=tab] {\convergencedata};
      \addlegendentry{stderr $d=5$}
      \addplot[deepblue, densely dashed, mark=filled cross, mark size=3pt] table[x=n, y=tent_shift_stderr, col sep=tab] {\convergencedata};
      \addlegendentry{stderr tent $d=5$}
      \addplot[black, dashed, thick, domain=16:1048576, samples=2, forget plot] {\shiftreference/x};
      \addplot[black, dashed, thick, domain=16:1048576, samples=2, forget plot] {\tentreference/x^2};
    \end{groupplot}
  \end{tikzpicture}
  \caption{Convergence plots for the three experiments using lattice rules.
  Left: absolute errors for three variants of a ``one-dimensional lattice rule'' on the function $f(x) = 2 \exp(-(2x)^5 + 2x)$.
  Middle: estimated standard errors (stderr) for $M = 10$ randomly shifted lattice rules applied to $g_d(\bsx) = \prod_{j=1}^d (x_j+1/2)$ for $d = 2$, $d = 10$ and $d = 100$.
  The dashed line is an $n^{-1}$ reference line.
  Right: estimated standard errors (stderr) with and without the tent transform for $g_5$ using $M = 10$ randomly shifted lattice rules.
  The thin lines correspond to the actual absolute errors (since we know the exact value of the integral) of the $\overline{Q}^\omega$ estimators which take the average over all shifts.
  The dashed lines are $n^{-1}$ and $n^{-2}$ reference lines.}\label{fig:convergence-plots}
\end{figure}

\journalonly{\Needspace{17\baselineskip}}
\begin{lstlisting}
# tent(x) = 1 .- abs.(1 .- 2 .* x) # as above

# First 5 components of exew_base2_m20_a3_HKKN.txt from HKKN2011
z = (1, 364981, 245389, 97823, 488939)
d = 5; m1 = 4; m2 = 20; M = 10
Random.seed!(2026) # For reproducibility.
shifts = rand(d, M)

# Rows correspond to 2^m points and columns to
# independent random shifts.
Qs_lr = reduce(hcat,
  [mean(x -> gd(shift_mod1(x, shift)), points)
    for points in lattice_rule_sequence(z, m1, m2)]
  for shift in eachcol(shifts))
\end{lstlisting}
\Needspace{6\baselineskip}
\begin{lstlisting}
Qs_tr = reduce(hcat,
  [mean(x -> gd(tent(shift_mod1(x, shift))), points)
    for points in lattice_rule_sequence(z, m1, m2)]
  for shift in eachcol(shifts))

Qbar_error_lr = abs.(vec(mean(Qs_lr, dims=2)) .- 1)
Qbar_error_tr = abs.(vec(mean(Qs_tr, dims=2)) .- 1)
stderr_lr = vec(std(Qs_lr, dims=2)) ./ sqrt(M)
stderr_tr = vec(std(Qs_tr, dims=2)) ./ sqrt(M)

\end{lstlisting}
\Needspace{12\baselineskip}
\begin{lstlisting}
ns = 2 .^ (m1:m2)
plt = plot(ns, [Qbar_error_lr Qbar_error_tr],
  xscale=:log10, yscale=:log10, label=false)
scatter!(plt, ns, [stderr_lr stderr_tr],
  label=["stderr d=$d" "stderr tent d=$d"])
plot!(plt, ns, stderr_lr[1] .* (ns[1] ./ ns),
  label=false, linestyle=:dash)
plot!(plt, ns, stderr_tr[1] .* (ns[1] ./ ns).^2,
  label=false, linestyle=:dash)
\end{lstlisting}

The resulting plot is the last panel in Figure~\ref{fig:convergence-plots}.
We see that for the randomly shifted lattice rule we have almost $O(n^{-1})$ convergence.
Consistent with the middle panel we expected the rate to not quite be $-1$ yet for this number of sample points.
But the nice part is the tent transformed randomly shifted lattice rule which gives a convergence akin to $O(n^{-2})$, but not quite yet.

Ready-to-use generating vectors can be found on the \emph{Magic Point Shop} website~\cite{NuyMagicPointShop}, while code for constructing generating vectors is available from the \emph{QMC4PDE} website~\cite{NuyQMC4PDE} which accompanies the overview article~\cite{KN2016}.
Construction of generating vectors is not part of this article.
An overview article on fast CBC methods can be found in~\cite{Nuy2014}.

\section{Function spaces}\label{sec:function-spaces}

After all these experiments from the previous sections it is time to dive into the actual theory.
In Section~\ref{sec:ibc} we first define some necessary information-based complexity concepts.
We start with non-periodic and periodic one-dimensional Sobolev spaces in Section~\ref{sec:spaces-d=1} and move into multivariate non-periodic spaces in Section~\ref{sec:non-periodic-spaces}.
Periodic multivariate spaces are discussed in Section~\ref{sec:periodic-spaces}, these are our main topic concerning lattice rules.
In Section~\ref{sec:upper-bound-lattice-rules} we show an upper bound on the worst-case error of lattice rules in the periodic spaces.

In principle, the reader can take a massive shortcut by just browsing through Section~\ref{sec:ibc} and then skip all intermediate sections to go straight to Section~\ref{sec:upper-bound-lattice-rules}.
Any questions that will then be triggered in Section~\ref{sec:upper-bound-lattice-rules} can then hopefully be resolved by selectively browsing through the intermediate sections.
In fact, the author thinks this is the recommended way.
For example, in Section~\ref{sec:upper-bound-lattice-rules} the reader might wonder why we are using such funny function spaces, or if we really need the mixed space, or the weights\ldots\ %
Well, that is exactly the long and winding story of the intermediate sections.

You will soon notice that our function spaces will be defined as subspaces of $L_2([0,1]^d)$ and you might remember that its elements are equivalence classes of functions that agree almost everywhere.
In these lecture notes we will mainly consider deterministic quadrature/cubature on spaces of continuous functions and we will hence identify each $L_2$ equivalence class with its unique continuous representative whenever one exists.
The reason to do this should be clear from the following example: changing the values at the finitely many cubature nodes, e.g., adding~$100$, does not change the $L_2$ equivalence class nor the integral but definitely changes the cubature result.
By only allowing the continuous representative we remove this ambiguity and make point evaluation well defined.
Admissible inputs to our deterministic cubature rules on these spaces, both in the worst-case error analysis and in an implementation, are the continuous functions themselves, equipped with the norm of the function space under consideration.
The modified function in the example is therefore not an admissible input despite having the same norm and integral.
It is however possible to have discontinuous functions in spaces with bounded (i.e., continuous) point evaluation, e.g., Walsh spaces which are used to analyse digital nets, cf.\ Section~\ref{sec:rkhs}.

\subsection{``Good'' quadrature/cubature rules and the worst-case error}\label{sec:ibc}

If we say that a certain quadrature/cubature rule is ``good'' then what do we mean?
We loosely recall some notions from information-based complexity.
The reader can consult \cite{TWW88,NW2008,NW2010,NW2012,DHP2015,DKS2013} for further details.

\paragraph{Worst-case error.}
Suppose you want to approximate a mapping $L : F \to R$ by an algorithm $A$ where $F$ and $R$ are normed spaces.
Then we can define the \emph{worst-case error} of the algorithm $A$ approximating $L$ by
\begin{align}\label{eq:def-wce}
  \wce(L, A, F, R)
  :=
  \sup_{\substack{f \in F \\ \|f\|_F \le 1}} \|L(f) - A(f)\|_R
  .
\end{align}
Typically we subscript the algorithm by some of the important parameters for which we want to study the worst-case error behaviour.
E.g., if the algorithm is constrained to use at most $n$ function values and we want to study the behaviour of the worst-case error in terms of $n$ then we write $A_n(f)$, if the dimension $d$ is also important to us then we write~$A_{d,n}(f)$.
How we measure $f$ and the approximation error $L(f) - A(f)$ is determined by the choice of the norms.

\paragraph{Deterministic versus randomised algorithms.}
The algorithm $A$ in~\eqref{eq:def-wce} is deterministic, and we call its error the deterministic worst-case error.
Given a probability space $(\Omega, \Sigma, P)$, a randomised algorithm is a family $\{ A^\omega \}_{\omega \in \Omega} = \{ A^\omega \}_\omega$ of deterministic algorithms, indexed by an outcome $\omega \in \Omega$, such that, for every fixed~$f$, the mapping $\omega \mapsto A^\omega(f)$ is an $R$-valued random element.
Here $\omega$ collects all random choices made by the algorithm, and fixing $\omega$ fixes one deterministic realisation~$A^\omega$.
Random elements could be randomly chosen sample points (like the Monte Carlo algorithm), a random shift applied to deterministic sample points, a random choice between possible deterministic sample points, etc.
For general background on randomised algorithms, see, e.g.,~\cite{Novak1988}.
The definition of the (deterministic) worst-case error~\eqref{eq:def-wce} could be applied to $A^\omega$ for a given $\omega$, but then the result is still a function of~$\omega$.
To obtain a single number we have to do something to the randomness such that we can measure the error of a randomised algorithm.
If we want to somehow stick to the worst-case choice, then we have the choice of handling $\omega$ before or after taking the supremum.
There are some common error measures for randomised algorithms, but unfortunately they are often just called the \emph{randomised error}, while in reality their definitions can be very different.
We introduce these common error measures with more specific names.
We define the \emph{worst-case expected error} and the \emph{worst-case root-mean-square error}, respectively, as follows
\begin{align*}
  \wce^{\mathrm{ran}}(L, \{A^\omega\}_\omega, F, R)
  &:=
  \sup_{\substack{f \in F \\ \|f\|_F \le 1}}
  \bbE_\omega\bigl[\|L(f) - A^\omega(f)\|_R\bigr]
  ,\\
  \wce^{\mathrm{rms}}(L, \{A^\omega\}_\omega, F, R)
  &:=
  \sup_{\substack{f \in F \\ \|f\|_F \le 1}}
  \left(\bbE_\omega\bigl[\|L(f) - A^\omega(f)\|_R^2\bigr]\right)^{1/2}
  .
\end{align*}
By Cauchy--Schwarz $\wce^{\mathrm{ran}}(L, \{A^\omega\}_\omega, F, R) \le \wce^{\mathrm{rms}}(L, \{A^\omega\}_\omega, F, R)$.
See, e.g., \cite{KKNU2019,KNW2023} for randomised lattice algorithms.
It is also possible to take the expectation over the worst-case errors.
This is often done when considering randomly shifted lattice rules.
We define the \emph{root-mean-square worst-case error}, or, when using random shifts, more explicitly, the \emph{shift-averaged worst-case error}, by
\begin{align*}
  \wce^{\mathrm{rms-sh}}(L, \{A^\omega\}_\omega, F, R)
  &:=
  \left(\bbE_\omega\left[\wce(L, A^\omega, F, R)^2\right]\right)^{1/2}
  \\
  &\hphantom{:}=
  \Bigg( \bbE_\omega\Big[ \sup_{\substack{f \in F \\ \|f\|_F \le 1}} \|L(f) - A^\omega(f)\|_R^2 \Big] \Bigg)^{1/2}
  ,
\end{align*}
this is, e.g., used in \cite{Hic2002,SKJ2002}.
Putting the expectation on the inside or on the outside of the supremum has different interpretations and effects on bounds that can be proven.
Specifically, if one has a deterministic lower bound on the worst-case error that is uniform over a class of algorithms that is large enough to include every realisation of your randomised algorithm, then taking the expectation on the outside is not going to help you circumvent this lower bound.
So in this case the randomness cannot help you.
However, if you take the expectation on the inside of the supremum, then the class of lattice rules with a random choice of the number of points in between $(n/2, n]$ can achieve a worst-case expected error of $n^{-\alpha - 1/2 + \delta}$, $\delta > 0$, instead of the deterministic $n^{-\alpha + \epsilon}$, $\epsilon > 0$, see~\cite{KKNU2019,KNW2023,NW2024} for the details.
Apart from Monte Carlo integration and using independent random shifts to estimate errors as we did in Section~\ref{sec:quadrature}, we do not study randomised algorithms in these lecture notes.

\paragraph{Standard information / sampling algorithms.}
Although we write $A(f)$, where the function $f$ is just an argument to the algorithm, we will need to specify what kind of information the algorithm can use on~$f$.
We will stick to what is called \emph{standard information} which means that we can evaluate $f$ freely at any point $\bsx$ in its domain and we then obtain the exact value of~$f(\bsx)$.
These are \emph{sampling algorithms}.
This matches neatly to an implementation in a computer program where we can do function calls to $f$ but we cannot look inside the implementation of the function; or more loosely defined, we cannot assume a certain representation of the function.
(The loose formulation covers the case as in, e.g., Julia where sampling algorithms in effect also have access to derivative values in a point $\bsx$ by using automatic differentiation based on the compiler's intermediate representations of user-defined functions.)

It is possible to consider more restricted forms of information, or less restricted forms of information.
E.g., restricted sampling which likely makes the problem harder, or allow a set of linear operators (e.g., access to exact Fourier coefficients), which likely makes the problem easier, but maybe less realistic.
We do not pursue these options here.

\paragraph{Information cost, function evaluation cost, arithmetic cost and total cost.}
The inner workings of the algorithm and the choice of the sample points where to evaluate the function (standard information) can depend on the properties of the space~$F$.
The choice of where to sample the function could be allowed to depend on previous evaluations, in which case we call the algorithm ``adaptive''.
We define the \emph{information cost} as the number of function evaluations (i.e., the information cost is in units of function evaluations) and for an algorithm $A_n$ which is allowed at most $n$ function values the information cost is hence at most~$n$.
The actual cost of these at most $n$ function evaluations depends on the ``cost model''.
For a finite dimensional domain of dimension~$d$ it is reasonable to assume that the cost of evaluating any $f \in F = F_d$ at a point $\bsx$ is bounded by a constant $C = C_d$.
The cost of at most $n$ function evaluations is therefore bounded by $C_d \, n \in O_d(n)$, where the dependence on~$d$ is absorbed into the implied constant.
Different cost models are possible, and become particularly important for infinite-dimensional integration, where the cost can depend on the number of active variables used in a function evaluation.
See, e.g.,~\cite{KSWW2010,PW2011,DG2014,KNPSW2017,GKNW2018}.
We will not discuss that here.
The \emph{arithmetic cost} is the cost of all the logical and arithmetic operations to combine all the function values into the output of the (possibly adaptive) algorithm.
For a non-adaptive algorithm which just takes a linear combination of function evaluations, as for a lattice rule, the arithmetic cost is~$O(n)$ if we assume the weights and nodes are given.
The \emph{total cost} is then the sum of the cost of all the function evaluations (not the information cost) and the arithmetic cost of the algorithm.

\paragraph{Quadrature/cubature and the construction cost of the rule.}
In the context of approximating high-dimensional integrals, as we discuss here, $L$ in~\eqref{eq:def-wce} is the $d$-dimensional integral $I_d$, see~\eqref{eq:Id}, and our algorithm $A = Q_{d,n}$ will be a certain quadrature/cubature rule using (at most) $n$ function evaluations (standard information).
As designers of the algorithm we aim to construct good sampling points that will work for any $f \in F_d$.
So in case of the algorithm being a rank-$1$ lattice rule $Q_{d,n,\bsz}$, see~\eqref{eq:R1L-rule}, this boils down to selecting/constructing a ``good'' generating vector~$\bsz = \bsz(n, F_d)$ with $n$ points for the space~$F_d$.
This generating vector is then constructed with the aim of minimising the worst-case error for the space~$F_d$, cf.~\eqref{eq:def-wce-integration}, \eqref{eq:wce-bound} and~\eqref{eq:error-bound} at the end of this section on how to then make use of this.
We refer to the cost of constructing $\bsz = \bsz(n, F_d)$ as the \emph{construction cost}.
Often it is assumed that this construction can be done off-line in advance, i.e., we have a pre-computed $\bsz$ for given $n$ and $F_d$ and then we do not take the construction cost into account for the total cost of applying the algorithm to an $f \in F_d$.

If the construction cost is acceptable, then we could construct the weights and the nodes on the fly.
For a lattice rule we just need to construct a good generating vector.
For the weighted unanchored periodic Sobolev spaces of dominating mixed smoothness, commonly called Korobov spaces, it is possible under suitable conditions to construct such a generating vector in cost $O(d \, n \log n)$, see, e.g., \cite{NC2006-prime,NC2006,CKN2006,Nuy2014}.
This makes it possible to consider constructing the rule on demand since we just discussed that a reasonable model for the cost of function evaluation is~$O_d(n)$.
When the generating vector is constructed on demand, its construction cost becomes part of the arithmetic cost of applying the algorithm.
Since $\log n \in o(n^\epsilon)$ for arbitrarily small $\epsilon > 0$, we can write this construction part of the arithmetic cost as $O_d(n^{1+\epsilon})$.
We can hence construct the generating vector on the fly and the total cost is then~$O_d(n^{1+\epsilon})$.
Furthermore, for these function spaces it is even possible to construct a generating vector $\bsz$ which works for a range of $n = 2^m$ with similar cost and one obtains a sequence of approximations with embedded sample points, see, e.g.,~\cite{CKN2006}.

So for a lattice rule that uses $n$ function evaluations the total cost is $O_d(n)$ when using a pre-computed rule or $O_d(n^{1+\epsilon})$ when constructing the rule on the fly in a suitable function space.
This means the information cost $n$ and the total cost are comparable.
This is not always the case, it depends on the amount of work the algorithm has to do.
E.g., one could consider an algorithm that first approximates the function based on $n$ function values and then computes an approximation to the integral from the approximant.
In such a case the arithmetic cost of approximating the function using the $n$ function values will likely dominate (e.g., it could be the cost of solving a, possibly structured, linear system of equations).
The algorithm in~\cite{JRH2016} uses lattice rules and estimates the error from the decay of Fourier coefficients.
Its total cost is $O_d(n^{1+\epsilon})$, $\epsilon > 0$, without including any construction cost.

\paragraph{Usage of the worst-case error for error bounds.}
For the particular case of approximating the $d$-dimensional integral $I_d$ by a quadrature/cubature rule $Q_{d,n}$, let us denote the worst-case error by
\begin{align}\label{eq:def-wce-integration}
  \boxed{
  e(Q_{d,n}, F_d)
  :=
  \wce(I_d, Q_{d,n}, F_d, \bbC)
  =
  \sup_{\substack{f \in F_d \\ \|f\|_{F_d} \le 1}} |I_d(f) - Q_{d,n}(f)|
  .
  }
\end{align}
Suppose we now manage to prove the following upper bound for a particular algorithm $Q_{d,n}$ and the space $F_d$:
\begin{align}\label{eq:wce-bound}
  e(Q_{d,n}, F_d)
  \le
  C_{F_d} \, n^{-\nu(F_d)}
  .
\end{align}
Then we call $\nu(F_d)$ the ``rate of convergence'' of the error.
It typically depends on the smoothness of the function space, which hence becomes another parameter in our analysis which we currently have not marked in the subscript yet, and the number of dimensions.
If we assume $Q_{d,n}$ to be linear, or, more generally, $I_d - Q_{d,n}$ to be homogeneous, then, for any $f \in F_d$ it holds
\begin{align}\label{eq:error-bound}
  |I_d(f) - Q_{d,n}(f)|
  &\le
  \|f\|_{F_d} \, e(Q_{d,n}, F_d)
  \le
  \|f\|_{F_d} \, C_{F_d} \, n^{-\nu(F_d)}
  .
\end{align}
Hence the above error bound gives us a split between a particular function $f$, in the form of how we measure our functions in the space by $\|f\|_{F_d}$, and how well the algorithm performs in terms of the information cost $n$ for the worst-case error in the space, by $e(Q_{d,n}, F_d)$.
This prompts us, e.g., for the rank-$1$ lattice rule $Q_{d,n,\bsz}$, see~\eqref{eq:R1L-rule}, to find a $\bsz$ for which we have a favourable bound on $e(Q_{d,n,\bsz}, F_d)$.

\paragraph{Modern QMC view (for integration).}
A given $d$-variate integrand $f$ can often be considered as a member of many different function spaces.
First there is what we will call the smoothness, i.e., how nice the function is, and we will denote this with the smoothness parameter~$\alpha$, which roughly measures how many weak derivatives the function has.
Next we will motivate weighted function spaces by modelling the importance of different subsets of variables by weights~$\gamma_\setu > 0$ for each $\setu \subseteq \{1,\ldots,d\}$.
We collect all $\gamma_\setu$ in a weight family $\bsgamma = \{ \gamma_\setu \}_\setu$.
Even after fixing the smoothness~$\alpha$, varying the weights $\bsgamma$ gives a whole scale of spaces $F_{d,\alpha,\bsgamma}$, and $f$ may belong to all of them.
We will show in Lemmas~\ref{lem:H-alpha-mix-weighted-equivalent-norms} and~\ref{lem:periodic-equivalent-norms-d} that, for the spaces that we are interested in, for any two positive weight families $\bsgamma$ and $\bsgamma'$, the corresponding spaces for each fixed~$d$ are equal as vector spaces and have equivalent norms, but their unit balls need not coincide.
We have~\eqref{eq:error-bound},
\begin{align}\label{eq:weighted-error-bound}
  |I_d(f) - Q_{d,n}(f)|
  &\le
  \|f\|_{F_{d,\alpha,\bsgamma}} \, e(Q_{d,n}, F_{d,\alpha,\bsgamma})
  ,
\end{align}
for each choice of~$\bsgamma$.
If we know additional properties of $f$, for example because $f$ is expressed in terms of the solution of a PDE and estimates for that solution reveal the relative importance of the variables and their interactions, then we can use this information to choose the weights $\bsgamma$ and construct the point set so as to try and minimise the product $\|f\|_{F_{d,\alpha,\bsgamma}} \, e(Q_{d,n}, F_{d,\alpha,\bsgamma})$, see, e.g.,~\cite{DKS2013,KSS2012,KN2016}.
Making the weights smaller decreases the worst-case error but increases the norm of the integrand.
We will also show that the weights will determine if our problem allows for dimension-independent error bounds under certain summability conditions on the weights.

\subsection{Some auxiliary useful bounds derived from Jensen's inequality}\label{sec:Jensen}

We have $\|\cdot\|_{\ell_\infty} \le \|\cdot\|_{\ell_2} \le \|\cdot\|_{\ell_1}$, that is,
\begin{align}\label{eq:ell-p-inequality}
  \|\cdot\|_{\ell_q}
  &\le
  \|\cdot\|_{\ell_p}
  &&
  \text{for } \infty \ge q \ge p > 0
  ,
\end{align}
where the norms (or quasi-norms) range over any index set $I \subseteq \Z^d$.
The direction of the inequality is easy to remember by the trivial case $\|\cdot\|_{\ell_\infty} \le \|\cdot\|_{\ell_1}$, but it holds for $\infty \ge q \ge p > 0$.
This inequality is often used in the following form
\begin{align}\label{eq:fake-jensen}
  \Big(\sum_i |a_i|\Big)^{1/\lambda}
  &\le
  \sum_i |a_i|^{1/\lambda}
  &&
  \text{for } \lambda \ge 1
  .
\end{align}
To show this we use~\eqref{eq:ell-p-inequality} and note that the left-hand side can be written as $\|(|a_i|^{1/\lambda})_i\|_{\ell_\lambda}$ while the right-hand side is $\|(|a_i|^{1/\lambda})_i\|_{\ell_1}$.
The next similar useful inequality follows equally directly
\begin{align}\label{eq:fake-jensen-2}
  \Big(\sum_i |a_i|\Big)^{\lambda}_{\hphantom{1/\lambda}}
  &\ge
  \sum_i |a_i|^{\lambda}_{\hphantom{1/\lambda}}
  &&
  \text{for } \lambda \ge 1
  .
\end{align}
Inequality~\eqref{eq:fake-jensen} is sometimes called Jensen's inequality in the QMC and IBC literature due to a ``misnomer'' in~\cite{HLP34}.
It is informative to also show the true Jensen inequality, which looks similar, but is used in the context of expectations, and goes in the opposite direction.
In our context (a discrete probability distribution and a simple power function), Jensen's inequality states
\begin{align}\label{eq:jensen1}
  \Big(\sum_i w_i \, |a_i|\Big)^{1/\lambda}
  &\ge
  \sum_i w_i \, |a_i|^{1/\lambda}
  &&
  \text{for } \lambda \ge 1,\; w_i \ge 0,\; \sum_i w_i = 1
  ,
\intertext{and similarly}
  \label{eq:jensen2}
  \Big(\sum_i w_i \, |a_i|\Big)^{\lambda}_{\hphantom{1/\lambda}}
  &\le
  \sum_i w_i \, |a_i|^{\lambda}_{\hphantom{1/\lambda}}
  &&
  \text{for } \lambda \ge 1,\; w_i \ge 0,\; \sum_i w_i = 1
  .
\end{align}
Note that Jensen's inequality holds for general concave and convex functions, not just $\cdot^{1/\lambda}$ and~$\cdot^\lambda$, as well as for expectations on arbitrary probability spaces.
Although inequalities~\eqref{eq:fake-jensen} and~\eqref{eq:jensen1} go opposite ways, a straightforward proof of~\eqref{eq:ell-p-inequality}, which we just used to directly prove~\eqref{eq:fake-jensen} above, can be done by normalising the sequence and then actually using Jensen's inequality.
So in a way, \eqref{eq:fake-jensen}~could be called Jensen's inequality as well, hence the title of this section.
A much more direct proof is given in~\cite[Theorem~2.2]{DHP2015}.
As a historical remark, Jensen in~\cite{Jen1906} states in an addendum that~\eqref{eq:fake-jensen-2}, in a finite sum version, was ``proved by L\"uroth according to Pringsheim''.

We note that for $L_p$ norms on $[0,1]^d$ (or on a probability space with measure~$1$) it holds that
\begin{align*}
  \|\cdot\|_{L_1([0,1]^d)}
  \le
  \|\cdot\|_{L_2([0,1]^d)}
  \le
  \|\cdot\|_{L_\infty([0,1]^d)}
  ,
\end{align*}
which is the opposite direction of~\eqref{eq:ell-p-inequality}.
Again, one can use Jensen's inequality and interpret $\int_{[0,1]^d} |f(\bsx)|^p \rd{\bsx} = \bbE[|f|^p]$ and convexity to prove this.

\subsection{Smoothness and one-dimensional Sobolev spaces}\label{sec:spaces-d=1}

We want to give a mathematical meaning to the smoothness of a function.
A function that exhibits large variation over short distances should not be considered very smooth.
Norms involving derivatives make this idea precise: large derivatives increase the norm, while membership of the corresponding function space requires weak derivatives up to a prescribed order to be integrable.
Function spaces defined in this way are called Sobolev spaces.
Since the deterministic algorithms considered in these lecture notes use function values, we will be particularly interested in Sobolev spaces whose functions have continuous representatives, so that point evaluation is well defined.
Sobolev spaces are conventionally defined on open domains in standard PDE texts, which would therefore write $H^\alpha((0,1))$.
After this identification, writing $H^\alpha((0,1))$ or $H^\alpha([0,1])$ does not change the normed space: the $L_2$-spaces and weak derivatives are understood on $(0,1)$, while the values at the endpoints are those of the continuous representative.
Following common notation in the QMC and IBC literature for integration on the unit cube, we hence write $H^\alpha([0,1])$.

\subsubsection{The typical standard Sobolev space}\label{sec:typical-Sobolev-d=1}

Sobolev smoothness of integer order is measured by requiring successively higher weak derivatives to belong to $L_2([0,1])$.
For functions on the interval $[0,1]$ and smoothness order $\alpha \in \N_0 = \{0,1,2,\ldots\}$, consider the norm
\begin{align}\label{eq:H-alpha-d=1}
\boxed{
  \|f\|_{H^\alpha([0,1])}^2
  :=
  \sum_{\tau = 0}^\alpha \|D^\tau f\|_{L_2([0, 1])}^2
  ,
}
\end{align}
where
\begin{align*}
  \|f\|_{L_2([0,1])}^2
  :=
  \int_{[0,1]} |f(x)|^2 \rd{x}
  .
\end{align*}
Equation~\eqref{eq:H-alpha-d=1} is a typical definition of a Sobolev space norm for $d=1$.
The ``order'' is sometimes called the ``regularity'' or ``smoothness'' and often denoted by $r$ or~$s$; we will stick with~$\alpha$.

Here and below, $D^\tau f$ denotes the $\tau$th weak derivative; the norm is understood to be infinite if any of the included weak derivatives does not exist in $L_2([0,1])$.
We write $D^\tau f(x)$ to mean $(D^\tau f)(x)$.
The norm above defines the function space
\begin{align}\label{eq:space-H-alpha-d=1}
  H^\alpha([0,1])
  :=
  \Big\{ f \in L_2([0,1]) : \|f\|_{H^\alpha([0,1])} < \infty \Big\}
  .
\end{align}
Membership in the space $H^\alpha([0,1])$ is to have the associated norm be finite.
In the current definition the smoothness $\alpha$ is integer, but this restriction will be removed when we switch to definitions in terms of Fourier coefficients.
We are mainly interested here in $\alpha \ge 1$; the Fourier spaces considered later will allow real $\alpha > 1/2$.
Sometimes these norms are defined without the squares on the right-hand side, and Sobolev spaces can more generally be defined using $L_p$-norms, but we restrict ourselves to the Hilbert space setting and sums of squared $L_2$-norms.
(With the exception of Theorem~\ref{thm:Walpha-p-lower-bound} which states a lower bound for general $L_p$ Sobolev spaces.)
Additionally we might have some extra conditions on functions to be part of the space.
E.g.\ we could ask that they have a certain property like $f(0) = 0$, or $f(0) = f(1)$ etc.
We note that for $\alpha = 0$ the norm~\eqref{eq:H-alpha-d=1} reduces to just the $L_2$-norm and we have $H^0([0,1]) = L_2([0,1])$.
\paragraph{Continuous embeddings and equivalent norms.}
Let $F$ and $G$ be normed spaces such that $F \subseteq G$.
We say that $F$ is \emph{continuously embedded} in $G$, and write $F \hookrightarrow G$, if there exists a constant $C > 0$ such that
\begin{align*}
  \|f\|_G
  &\le
  C \, \|f\|_F
  \qquad
  \text{for all } f \in F \subseteq G
  .
\end{align*}
Thus a continuous embedding $F \hookrightarrow G$ asserts both that every element of $F$ belongs to $G$, i.e., $F \subseteq G$, and that its $G$-norm is controlled by its $F$-norm.
Equivalently, the inclusion map from $F$ into $G$ is continuous.
If $F$ and $G$ are the same vector space equipped with different norms, then the two norms are equivalent if and only if there exist constants $0 < c \le C$ such that
\begin{align*}
  c \, \|f\|_F
  \le
  \|f\|_G
  \le
  C \, \|f\|_F
  \qquad
  \text{for all } f \in F = G
  .
\end{align*}
In this case the inclusion maps are continuous in both directions and we write
\begin{align*}
  F
  &\hookrightleftarrows
  G
  \qquad
  \Longleftrightarrow
  \qquad
  F \hookrightarrow G
  \quad\text{and}\quad
  G \hookrightarrow F
  .
\end{align*}
Note that $F = G$ just denotes equality of the underlying vector spaces; whether their norms are identical or merely equivalent will be stated separately.
When $\alpha = 0$, we have $H^0([0,1]) = L_2([0,1])$, whose functions need not be continuous and may even have point singularities.
For the integer smoothness considered here, the one-dimensional Sobolev embedding gives
\begin{align*}
  H^\alpha([0,1])
  &\hookrightarrow
  C([0,1])
  \qquad
  \text{for } \alpha \ge 1
  ,
\end{align*}
where $C([0,1]) := \{ f : [0,1] \to \bbC : f \text{ is continuous} \}$ is equipped with the uniform norm.
A direct proof, including an explicit embedding constant, follows from the absolute continuity discussed next; see Lemma~\ref{lem:H-alpha-C-embedding-d=1} below.
When we later define smoothness in terms of Fourier coefficients, the corresponding embedding into $C([0,1])$ will hold for $\alpha > 1/2$.
\paragraph{Absolute continuity.}
If $g \in L_1([0,1])$ has a weak derivative $Dg \in L_2([0,1])$, then $g$ has an absolutely continuous representative satisfying, for all $a,x \in [0,1]$,
\begin{align}\label{eq:absolute-continuity-d=1}
  g(x)
  &=
  g(a)
  +
  \int_a^x Dg(t) \rd{t}
  .
\end{align}
In particular, if $f \in H^\alpha([0,1])$, then $D^\tau f$ is absolutely continuous for all $\tau = 0,\ldots,\alpha-1$, while $D^\alpha f$ exists in the weak sense and belongs to $L_2([0,1])$.

\begin{lemma}\label{lem:H-alpha-C-embedding-d=1}
For $\alpha \in \N$, i.e., $\alpha \ge 1$, the continuous embedding $H^\alpha([0,1]) \hookrightarrow C([0,1])$ satisfies
\begin{align}\label{eq:H-alpha-C-embedding-d=1}
  \|f\|_{C([0,1])}
  &\le
  \sqrt{2} \, \|f\|_{H^\alpha([0,1])}
  .
\end{align}
\end{lemma}
\begin{proof}
Absolute continuity and Cauchy--Schwarz give, for $x, y \in [0,1]$,
\begin{align*}
  |f(x) - f(y)|
  &\le
  \left|\int_y^x Df(t) \rd{t}\right|
  \le
  |x-y|^{1/2} \, \|Df\|_{L_2([0,1])}
  \le
  \|Df\|_{L_2([0,1])}
  .
\end{align*}
Writing $f(x) = \int_0^1 f(y) \rd{y} + \int_0^1 (f(x)-f(y)) \rd{y}$, applying Cauchy--Schwarz to the first integral, and integrating the preceding bound with respect to~$y$, we obtain, for every $x \in [0,1]$,
\begin{align*}
  |f(x)|
  &\le
  \left|\int_0^1 f(y) \rd{y}\right|
  +
  \int_0^1 |f(x)-f(y)| \rd{y}
  \le
  \|f\|_{L_2([0,1])}
  +
  \|Df\|_{L_2([0,1])}
  .
\end{align*}
Using $|a| + |b| \le \sqrt{2} \, (|a|^2 + |b|^2)^{1/2}$ and taking the maximum over $x \in [0,1]$ therefore gives
\begin{align*}
  \|f\|_{C([0,1])}
  =
  \max_{x \in [0,1]} |f(x)|
  &\le
  \|f\|_{L_2([0,1])}
  +
  \|Df\|_{L_2([0,1])}
  \\
  &\le
  \sqrt{2} \, \|f\|_{H^1([0,1])}
  \le
  \sqrt{2} \, \|f\|_{H^\alpha([0,1])}
  .
  \qedhere
\end{align*}
\end{proof}

\subsubsection{Periodic Sobolev spaces and equivalent norms}\label{sec:periodic-spaces-d=1}

We first define $\widetilde{H}^\alpha([0,1])$ for integer $\alpha \in \N_0$ as a subspace of $H^\alpha([0,1])$, and then use Fourier coefficients to define several equivalent versions for arbitrary real~$\alpha \ge 0$.
After proving their equivalence, we use calligraphic
\begin{align*}
  \boxed{
  \widetilde{\calH}^\alpha([0,1])
  }
\end{align*}
to denote their common underlying vector space, without selecting a particular one of the equivalent norms.
These norms and spaces all occur in the literature and they can be used later to define spaces of dominating mixed smoothness by taking tensor products, or more complicated forms using weights, see Table~\ref{tbl:periodic-equivalent-norms} in Section~\ref{sec:periodic-spaces}.
Unfortunately there is no standard notation, and in fact, there are even more norms and spaces which we are not showing here.
There is no need to know all of these norms, but it may be useful to see how we can motivate all these different norms and how they relate to each other as equivalent norms.
Whether one or another of these norms is used in a publication depends on preference and convenience.

\paragraph{Fourier coefficients and convergence.}
For $f \in L_1([0,1])$ we define the Fourier coefficients by
\begin{align}\label{eq:Fourier-coeff-d=1}
  \hat{f}_h
  &:=
  \int_0^1 f(x) \, \exp(-2\pi\rmi \, h x) \rd{x}
  ,
  \qquad
  h \in \Z
  .
\end{align}
The corresponding Fourier series and its partial sums are
\begin{align*}
  \sum_{h \in \Z} \hat{f}_h \, \exp(2\pi\rmi \, h x)
  \qquad\text{and}\qquad
  S_N f(x)
  &:=
  \sum_{|h| \le N} \hat{f}_h \, \exp(2\pi\rmi \, h x)
  .
\end{align*}
The Fourier series of $f$ \emph{converges uniformly} to $f$ if
\begin{align*}
  \sup_{x \in [0,1]} |f(x) - S_N f(x)|
  &\to
  0
  \qquad
  \text{as } N \to \infty
  .
\end{align*}
Further, $f$ has an \emph{absolutely converging Fourier series} if its Fourier coefficients are absolutely summable, i.e.,
\begin{align*}
  \sum_{h \in \Z} |\hat{f}_h|
  &<
  \infty
  .
\end{align*}
The functions with absolutely converging Fourier series form the \emph{Wiener algebra}
\begin{align}\label{eq:def-Wiener-algebra-d=1}
\boxed{
  A([0,1])
  :=
  \Big\{ f \in L_1([0,1]) : \|f\|_{A([0,1])} := \sum_{h \in \Z} |\hat{f}_h| < \infty \Big\}
  .
}
\end{align}
This is often also written as $A(\bbT([0,1]))$, to denote that we identify $[0,1]$ with the torus by identifying the endpoints of~$[0,1]$.
As we show next, each $f \in A([0,1])$ is given pointwise by its absolutely and uniformly convergent Fourier series which automatically satisfies $f(0) = f(1)$.
(To be fair, the Wiener algebra is often just denoted by $A(\bbT)$, but then you do not know if these are $2\pi$-periodic functions or $1$-periodic functions, like here.)

Since $|\hat{f}_h \, \exp(2\pi\rmi \, h x)| = |\hat{f}_h|$ is independent of $x$, the Fourier series converges uniformly to a continuous $1$-periodic function; by uniqueness of Fourier coefficients in $L_1([0,1])$, this function is the continuous representative of~$f$.
Let us denote by $C(\bbT([0,1])) := \{ f \in C([0,1]) : f(0) = f(1) \}$, equipped with the uniform norm, the continuous functions on the torus obtained by identifying the endpoints of~$[0,1]$.
We have
\begin{align*}
  \|f\|_{C(\bbT([0,1]))}
  &=
  \sup_{x \in [0,1]} |f(x)|
  \\
  &=
  \sup_{x \in [0,1]}
  \Big| \sum_{h \in \Z} \hat{f}_h \, \exp(2\pi\rmi \, h x) \Big|
  \le
  \sum_{h \in \Z} |\hat{f}_h|
  =
  \|f\|_{A([0,1])}
  ,
\end{align*}
and hence $A([0,1]) \hookrightarrow C(\bbT([0,1]))$.
Since $C(\bbT([0,1])) \subset L_2([0,1]) \subset L_1([0,1])$ on $[0,1]$, see also the final paragraph of Section~\ref{sec:Jensen}, we could equivalently replace $L_1([0,1])$ by $L_2([0,1])$ in~\eqref{eq:def-Wiener-algebra-d=1}.

\paragraph{Periodic spaces and Fourier-based norms.}
First, for integer smoothness $\alpha \in \N_0$, we define
\begin{align}\label{eq:space-tildeH-alpha-d=1}
  \widetilde{H}^\alpha([0,1])
  &:=
  \Big\{
    f \in H^\alpha([0,1])
    :
    D^\tau f(0) = D^\tau f(1)
    \text{ for } \tau = 0,\ldots,\alpha-1
  \Big\}
  .
\end{align}
The endpoint values are well defined because $D^\tau f$ is absolutely continuous for $\tau = 0,\ldots,\alpha-1$, i.e., we have $D^\tau f \in H^{\alpha-\tau}([0,1]) \subseteq H^1([0,1])$ for all $0 \le \tau \le \alpha - 1$.
Equivalently, for every $\tau = 0,\ldots,\alpha-1$, the continuous representative of $D^\tau f$ on $[0,1]$ is identified with a continuous function on the torus $\bbT([0,1])$ formed by joining the endpoints.
We can thus also specify the space by
\begin{align}\label{eq:space-tildeH-alpha-d=1-b}
  \widetilde{H}^\alpha([0,1])
  &=
  \Big\{
    f \in H^\alpha([0,1])
    :
    D^\tau f \in C(\bbT([0,1]))
    \text{ for } \tau = 0,\ldots,\alpha-1
  \Big\}
  .
\end{align}
This is the sense in which the adjective ``periodic'' is used for the periodic Sobolev spaces.
Note that for $\alpha = 0$ we still have $\widetilde{H}^0([0,1]) = H^0([0,1]) = L_2([0,1])$ and so then there is no periodicity at all.
All the Fourier-based norms below have the form
\begin{align*}
  \|f\|_F^2
  &=
  \sum_{h \in \Z} |\hat{f}_h|^2 \, r_F^2(h)
  ,
\end{align*}
where we call $r_F(h)$ the Fourier multiplier of the norm or space.
Equivalence of such norms can then be proved by pointwise comparison of their Fourier multipliers in~$h \in \Z$.

We equip this subspace $\widetilde{H}^\alpha([0,1]) \subseteq H^\alpha([0,1])$, for $\alpha \in \N_0$, with the norm from $H^\alpha([0,1])$.
The $H^\alpha([0,1])$-norm has an explicit Fourier representation for $f \in \widetilde{H}^\alpha([0,1])$, and we define
\begin{align}\label{eq:norm-tildeH-alpha-d=1}
\boxed{
  \|f\|_{\widetilde{H}^\alpha([0,1])}^2
  :=
  \sum_{h \in \Z}
  |\hat{f}_h|^2
  \left(
    1
    +
    \sum_{\tau=1}^\alpha |2\pi h|^{2\tau}
  \right)
  .
}
\end{align}
We show in Lemma~\ref{lem:periodic-Fourier-characterisation-d=1} below that $\|f\|_{\widetilde{H}^\alpha([0,1])}^2 = \|f\|_{H^\alpha([0,1])}^2$ for $f \in \widetilde{H}^\alpha([0,1])$ and that this norm can be used to define the space~\eqref{eq:space-tildeH-alpha-d=1}, alternatively~\eqref{eq:space-tildeH-alpha-d=1-b}, by just asking the norm~\eqref{eq:norm-tildeH-alpha-d=1} to be finite.
This norm still requires $\alpha$ to be integer, since we have the sum over~$\tau$.
However, if we retain just the highest power, then $\alpha$ is not required to be integer anymore, and, when $\alpha$ is integer, we obtain an equivalent norm.
Hence, for arbitrary real $\alpha > 0$, we define the primed norm
\begin{align}\label{eq:norm-tildeH-prime-alpha-d=1}
\boxed{
  \|f\|_{\widetilde{H}'^\alpha([0,1])}^2
  :=
  \sum_{h \in \Z}
  |\hat{f}_h|^2
  \left(1 + |2\pi h|^{2\alpha}\right)
  ,
}
\end{align}
and the corresponding primed space by
\begin{align}\label{eq:space-tildeH-prime-alpha-d=1}
  \widetilde{H}'^\alpha([0,1])
  &:=
  \Big\{
    f \in L_2([0,1])
    :
    \|f\|_{\widetilde{H}'^\alpha([0,1])} < \infty
  \Big\}
  .
\end{align}
We do not include $\alpha=0$ in this definition, even though it is a valid choice, because then the term $|2\pi h|^{2\alpha}$ still survives, whereas the sum over the derivative norms in~\eqref{eq:norm-tildeH-alpha-d=1} is then empty and there we have $\|\cdot\|_{\widetilde{H}^0([0,1])} = \|\cdot\|_{L_2([0,1])}$ as expected, whilst that would not be the case here.
An easy modification from the previous norm, which works also for $\alpha = 0$, is the \emph{Bessel-potential norm} (on the torus) which we define, for any real $\alpha \ge 0$,
\begin{align}\label{eq:norm-tildeH-double-prime-alpha-d=1}
\boxed{
  \|f\|_{\widetilde{H}''^\alpha([0,1])}^2
  :=
  \sum_{h \in \Z}
  |\hat{f}_h|^2
  \left(1 + |2\pi h|^2\right)^\alpha
  ,
}
\end{align}
and the corresponding space
\begin{align}\label{eq:space-tildeH-double-prime-alpha-d=1}
  \widetilde{H}''^\alpha([0,1])
  &:=
  \Big\{
    f \in L_2([0,1])
    :
    \|f\|_{\widetilde{H}''^\alpha([0,1])} < \infty
  \Big\}
  .
\end{align}
For this last space we also recover $\widetilde{H}''^0([0,1]) = L_2([0,1])$ with identical norms.

Two other equivalent norms which occur frequently in the QMC literature are, for real $\alpha \ge 0$, the \emph{periodic unanchored Sobolev space norm}
\begin{align}\label{eq:norm-tildeH-alpha-u-d=1}
\boxed{
  \|f\|_{\widetilde{H}^{\alpha,\rmu}([0,1])}^2
  :=
  \sum_{h \in \Z}
  |\hat{f}_h|^2 \,
  \max\{1,|2\pi h|\}^{2\alpha}
  ,
}
\end{align}
and the corresponding norm without the factor $2\pi$, which is particularly popular in all kinds of so-called \emph{Korobov spaces} (although that name is used in a much more general sense, also for the norm including the $2\pi$),
\begin{align}\label{eq:norm-tildeH-alpha-u-1-d=1}
\boxed{
  \|f\|_{\widetilde{H}^{\alpha,\rmu}_1([0,1])}^2
  :=
  \sum_{h \in \Z}
  |\hat{f}_h|^2 \,
  \max\{1,|h|\}^{2\alpha}
  .
}
\end{align}
The superscript $\rmu$ denotes that this is an ``unanchored'' norm which we will introduce when discussing equivalent norms for the non-periodic space in~\eqref{eq:H-alpha-u-d=1}, see also Lemma~\ref{lem:H-alpha-u-norm-equivalence-d=1}.
We will use the subscript~$1$ to define alternative versions for any of the Fourier-based norms and it just means to replace the $2\pi$ scaling by~$1$.
A case in point is the following variant of the Bessel potential norm~\eqref{eq:norm-tildeH-double-prime-alpha-d=1}, omitting the factor $2\pi$ in~\eqref{eq:norm-tildeH-double-prime-alpha-d=1}, which occurs often like this in the literature,
\begin{align}\label{eq:norm-tildeH-double-prime-alpha-1-d=1}
\boxed{
  \|f\|_{\widetilde{H}''_1{}^\alpha([0,1])}^2
  :=
  \sum_{h \in \Z}
  |\hat{f}_h|^2
  \left(1 + |h|^2\right)^\alpha
  .
}
\end{align}
Since we are on a roll, we can also define the scale~$1$ variant of~\eqref{eq:norm-tildeH-prime-alpha-d=1}
\begin{align}\label{eq:norm-tildeH-prime-alpha-1-d=1}
\boxed{
  \|f\|_{\widetilde{H}'_1{}^\alpha([0,1])}^2
  :=
  \sum_{h \in \Z}
  |\hat{f}_h|^2
  \left(1 + |h|^{2\alpha}\right)
  .
}
\end{align}
The corresponding spaces are again defined by finiteness of their norms, and, for the relevant ranges of $\alpha$, it will turn out that all of them are equal as vector spaces and they just differ in norm.
We will show embedding constants in Lemma~\ref{lem:periodic-equivalent-norms-d=1}.

\paragraph{Fourier characterisation and equivalent norms.}
For some of the bounds below we will make use of the Riemann zeta function, which, for $x > 1$, is defined by the convergent series
\begin{align*}
  \zeta(x)
  &:=
  \sum_{h \ge 1} h^{-x}
  .
\end{align*}
In particular, $1 < \zeta(x) < \infty$ for $x > 1$, while $\zeta(x) \to \infty$ as $x \to 1^+$.

The following lemma gives the exact Fourier representation of the $\widetilde{H}^\alpha([0,1])$ norm when we start with the $H^\alpha([0,1])$ norm for a periodic function and shows that the Fourier representation automatically encodes the matching conditions in~\eqref{eq:space-tildeH-alpha-d=1}.
Since Lemma~\ref{lem:periodic-equivalent-norms-d=1}, further down below, will show that all of these Fourier-based norms induce the same space for their respective ranges of~$\alpha$, the result below implies that, for $\alpha \in \N_0$, the boundary matching conditions in~\eqref{eq:space-tildeH-alpha-d=1} hold for all these spaces (note however that these conditions are void for $\alpha = 0$).
More generally, for real $\alpha$, the analogous matching conditions hold for every derivative order $\tau \in \N_0$ satisfying $0 \le \tau < \alpha - 1/2$, as we will show later in Lemma~\ref{lem:periodic-derivatives-d=1}.

\begin{lemma}[Encoding ``periodicity'' conditions by the Fourier-based norm]\label{lem:periodic-Fourier-characterisation-d=1}
Let $\alpha \in \N_0$.
The space $\widetilde{H}^\alpha([0,1])$ defined in~\eqref{eq:space-tildeH-alpha-d=1}, or~\eqref{eq:space-tildeH-alpha-d=1-b}, can be equally well defined by asking finiteness of the norm~\eqref{eq:norm-tildeH-alpha-d=1}, i.e.,
\begin{align*}
  \widetilde{H}^\alpha([0,1])
  =
  \Bigg\{
    f \in L_2([0,1])
    :
    \sum_{h \in \Z}
    |\hat{f}_h|^2
    \left(
      1
      +
      \sum_{\tau=1}^\alpha |2\pi h|^{2\tau}
    \right)
    <
    \infty
  \Bigg\}
  .
\end{align*}
For $f \in \widetilde{H}^\alpha([0,1])$ we have $\|f\|_{\widetilde{H}^\alpha([0,1])} = \|f\|_{H^\alpha([0,1])}$ and the Fourier series of $D^\tau f$ converges absolutely and uniformly for $\tau = 0,\ldots,\alpha-1$, and
\begin{align*}
  D^\tau f(x)
  &=
  \sum_{h \in \Z}
  (2\pi\rmi \, h)^\tau \, \hat{f}_h \, \exp(2\pi\rmi \, h x)
  \qquad
  \text{for } x \in [0,1]
  .
\end{align*}
\end{lemma}
\begin{proof}
For $\alpha=0$, the statement is obvious.
For the remainder of the proof, assume $\alpha \ge 1$.
First let $f$ satisfy the conditions in~\eqref{eq:space-tildeH-alpha-d=1}.
Repeated integration by parts, using the matching conditions in~\eqref{eq:space-tildeH-alpha-d=1} to eliminate the boundary terms, gives
\begin{align*}
  \widehat{D^\tau f}_h
  &=
  (2\pi\rmi \, h)^\tau \, \hat{f}_h
  \qquad
  \text{for } \tau = 1,\ldots,\alpha
  .
\end{align*}
Parseval's identity then gives the norm in terms of Fourier coefficients~\eqref{eq:norm-tildeH-alpha-d=1}, and hence one inclusion.
Conversely, suppose that the Fourier-based norm~\eqref{eq:norm-tildeH-alpha-d=1} is finite.
This implies that $\sum_{h\in\Z} |(2\pi\rmi\,h)^\tau \hat{f}_h|^2 < \infty$ for $\tau = 0,\ldots,\alpha$.
Consequently, $f$ has weak derivatives through order $\alpha$ in $L_2([0,1])$ and hence $f \in H^\alpha([0,1])$.
It remains only to verify the matching conditions in~\eqref{eq:space-tildeH-alpha-d=1}.
For $\alpha \ge 1$, define
\begin{align*}
  (\widetilde{C}_{\alpha,1})^2
  &:=
  \sum_{h \in \Z}
  \left(
    1
    +
    \sum_{k=1}^\alpha |2\pi h|^{2k}
  \right)^{-1}
  =
  1 + 2 \sum_{h \ge 1}
  \left(
    1
    +
    \sum_{k=1}^\alpha |2\pi h|^{2k}
  \right)^{-1}
  \\
  &\hphantom{:}\le
  1 + 2 \sum_{h \ge 1} |2\pi h|^{-2\alpha}
  \le
  1
  +
  2 \, (2\pi)^{-2\alpha} \, \zeta(2\alpha)
  <
  \infty
  ,
\end{align*}
which is finite since $2 \alpha > 1$.
For $\alpha \ge 2$, define
\begin{align*}
  (\widetilde{C}_{\alpha,2})^2
  &:=
  \max_{1 \le \tau \le \alpha-1}
  \sum_{h \in \Z}
  |2\pi h|^{2\tau} \,
  \left(
    1
    +
    \sum_{k=1}^\alpha |2\pi h|^{2k}
  \right)^{-1}
  \\
  &\hphantom{:}=
  \max_{1 \le \tau \le \alpha-1}
  2 \sum_{h \ge 1}
  |2\pi h|^{2\tau} \,
  \left(
    1
    +
    \sum_{k=1}^\alpha |2\pi h|^{2k}
  \right)^{-1}
  \\
  &\hphantom{:}\le
  \max_{1 \le \tau \le \alpha-1}
  2 \sum_{h \ge 1}
  |2\pi h|^{2\tau} \,
  |2\pi h|^{-2\alpha}
  \\
  &\hphantom{:}=
  \max_{1 \le \tau \le \alpha-1}
  2 \, (2\pi)^{-2 (\alpha - \tau)} \, \zeta(2 (\alpha - \tau))
  <
  \infty
  ,
\end{align*}
which is finite since $2 (\alpha - \tau) > 1$ for all $1 \le \tau \le \alpha-1$.
We can show that the Fourier series of $f$ is absolutely summable, since Cauchy--Schwarz gives
\begin{align*}
  \sum_{h \in \Z} |\hat{f}_h|
  &\le
  \widetilde{C}_{\alpha,1} \, \|f\|_{\widetilde{H}^\alpha([0,1])}
  <
  \infty
  .
\end{align*}
Hence the Fourier series of $f$ converges absolutely and uniformly and defines a continuous $1$-periodic representative.
When $\alpha \ge 2$, we must additionally consider $\tau = 1,\ldots,\alpha-1$.
For each such $\tau$, Cauchy--Schwarz gives
\begin{align*}
  \sum_{h \in \Z} |(2\pi h)^\tau \, \hat{f}_h|
  &\le
  \widetilde{C}_{\alpha,2} \, \|f\|_{\widetilde{H}^\alpha([0,1])}
  <
  \infty
  .
\end{align*}
Hence the Fourier series of $D^\tau f$ also converges absolutely and uniformly and defines a continuous $1$-periodic representative for $\tau = 1,\ldots,\alpha-1$.
It follows that the conditions in~\eqref{eq:space-tildeH-alpha-d=1} hold, so $f \in \widetilde{H}^\alpha([0,1])$.
\end{proof}

We next show that all of these spaces are continuously embedded into one another for the relevant ranges of~$\alpha$.
(Note that the constants in the lemma below could be improved by taking into account that $h \in \Z$, while the proof considers inequalities for $x \in \R$.)

\begin{lemma}[Reciprocal continuous embeddings]\label{lem:periodic-equivalent-norms-d=1}
For $\alpha > 0$ we have the continuous embeddings $\widetilde{H}'^\alpha([0,1]) \hookrightleftarrows \widetilde{H}''^\alpha([0,1])$ and $\widetilde{H}'_1{}^\alpha([0,1]) \hookrightleftarrows \widetilde{H}''_1{}^\alpha([0,1])$ with
\begin{align*}
  \min\{1, \sqrt{2^{\alpha-1}}\} \,
  \|f\|_{\widetilde{H}'^\alpha([0,1])}
  &\le
  \|f\|_{\widetilde{H}''^\alpha([0,1])}
  \le
  \max\{1, \sqrt{2^{\alpha-1}}\} \,
  \|f\|_{\widetilde{H}'^\alpha([0,1])}
  ,
  \\
  \min\{1, \sqrt{2^{\alpha-1}}\} \,
  \|f\|_{\widetilde{H}'_1{}^\alpha([0,1])}
  &\le
  \|f\|_{\widetilde{H}''_1{}^\alpha([0,1])}
  \le
  \max\{1, \sqrt{2^{\alpha-1}}\} \,
  \|f\|_{\widetilde{H}'_1{}^\alpha([0,1])}
  ,
\end{align*}
and $\widetilde{H}^{\alpha,\rmu}([0,1]) \hookrightleftarrows \widetilde{H}'^\alpha([0,1])$ and $\widetilde{H}^{\alpha,\rmu}_1([0,1]) \hookrightleftarrows \widetilde{H}'_1{}^\alpha([0,1])$ with
\begin{align*}
  \|f\|_{\widetilde{H}^{\alpha,\rmu}([0,1])}
  &\le
  \|f\|_{\widetilde{H}'^\alpha([0,1])}
  \le
  \sqrt{2} \, \|f\|_{\widetilde{H}^{\alpha,\rmu}([0,1])}
  ,
  \\
  \|f\|_{\widetilde{H}^{\alpha,\rmu}_1([0,1])}
  &\le
  \|f\|_{\widetilde{H}'_1{}^\alpha([0,1])}
  \le
  \sqrt{2} \, \|f\|_{\widetilde{H}^{\alpha,\rmu}_1([0,1])}
  ,
\end{align*}
and $\widetilde{H}'_1{}^\alpha([0,1]) \hookrightleftarrows \widetilde{H}'^\alpha([0,1])$ with
\begin{align*}
  \|f\|_{\widetilde{H}'_1{}^\alpha([0,1])}
  &\le
  \|f\|_{\widetilde{H}'^\alpha([0,1])}
  \le
  (2\pi)^\alpha \, \|f\|_{\widetilde{H}'_1{}^\alpha([0,1])}
  .
\end{align*}
For $\alpha \ge 0$ we have the continuous embeddings $\widetilde{H}^{\alpha,\rmu}([0,1]) \hookrightleftarrows \widetilde{H}''^\alpha([0,1])$, $\widetilde{H}^{\alpha,\rmu}_1([0,1]) \hookrightleftarrows \widetilde{H}''_1{}^\alpha([0,1])$, $\widetilde{H}''_1{}^\alpha([0,1]) \hookrightleftarrows \widetilde{H}''^\alpha([0,1])$ and $\widetilde{H}^{\alpha,\rmu}_1([0,1]) \hookrightleftarrows \widetilde{H}^{\alpha,\rmu}([0,1])$ with
\begin{align*}
  \|f\|_{\widetilde{H}^{\alpha,\rmu}([0,1])}
  &\le
  \|f\|_{\widetilde{H}''^\alpha([0,1])}
  \le
  2^{\alpha/2} \, \|f\|_{\widetilde{H}^{\alpha,\rmu}([0,1])}
  ,
  \\
  \|f\|_{\widetilde{H}^{\alpha,\rmu}_1([0,1])}
  &\le
  \|f\|_{\widetilde{H}''_1{}^\alpha([0,1])}
  \le
  2^{\alpha/2} \, \|f\|_{\widetilde{H}^{\alpha,\rmu}_1([0,1])}
  ,
  \\
  \|f\|_{\widetilde{H}''_1{}^\alpha([0,1])}
  &\le
  \|f\|_{\widetilde{H}''^\alpha([0,1])}
  \le
  (2\pi)^\alpha \, \|f\|_{\widetilde{H}''_1{}^\alpha([0,1])}
  ,
  \\
  \|f\|_{\widetilde{H}^{\alpha,\rmu}_1([0,1])}
  &\le
  \|f\|_{\widetilde{H}^{\alpha,\rmu}([0,1])}
  \le
  (2\pi)^\alpha \, \|f\|_{\widetilde{H}^{\alpha,\rmu}_1([0,1])}
  .
\end{align*}
For $\alpha \in \N$ we have the continuous embeddings $\widetilde{H}'^\alpha([0,1]) \hookrightleftarrows \widetilde{H}^\alpha([0,1])$ and $\widetilde{H}^\alpha([0,1]) \hookrightleftarrows \widetilde{H}''^\alpha([0,1])$ with
\begin{align*}
  \|f\|_{\widetilde{H}'^\alpha([0,1])}
  &\le
  \|f\|_{\widetilde{H}^\alpha([0,1])}
  \le
  \sqrt{\frac{\alpha+1}{2}} \, \|f\|_{\widetilde{H}'^\alpha([0,1])}
  ,
  \\
  \|f\|_{\widetilde{H}^\alpha([0,1])}
  &\le
  \|f\|_{\widetilde{H}''^\alpha([0,1])}
  \le
  \sqrt{\frac{2^\alpha}{\alpha+1}} \, \|f\|_{\widetilde{H}^\alpha([0,1])}
  .
\end{align*}
These bounds can be combined to get, e.g., for $\alpha \in \N$,
\begin{align*}
  \|f\|_{\widetilde{H}^\alpha([0,1])}
  &\le
  \sqrt{\frac{\alpha+1}{2}} \, \|f\|_{\widetilde{H}'^\alpha([0,1])}
  \le
  \sqrt{\alpha+1} \, \|f\|_{\widetilde{H}^{\alpha,\rmu}([0,1])}
  .
\end{align*}
A direct comparison gives the better, and $\alpha$-independent, bound
\begin{align*}
  \|f\|_{\widetilde{H}^{\alpha,\rmu}([0,1])}
  &\le
  \|f\|_{\widetilde{H}^\alpha([0,1])}
  \le
  \left(1-(2\pi)^{-2}\right)^{-1/2} \, \|f\|_{\widetilde{H}^{\alpha,\rmu}([0,1])}
  ,
\end{align*}
where the upper bound is sharp for $\alpha \to \infty$.
\end{lemma}
\begin{proof}
For $\alpha \ge 1$ and $a, b \ge 0$ we have, see~\eqref{eq:fake-jensen}--\eqref{eq:jensen2} in Section~\ref{sec:Jensen},
\begin{align*}
  (a^\alpha + b^\alpha)
  \le
  (a + b)^\alpha
  \le
  2^{\alpha-1} \, (a^\alpha + b^\alpha)
  ,
\end{align*}
while for $0 \le \alpha \le 1$ we have
\begin{align*}
  2^{\alpha-1} \, (a^\alpha + b^\alpha)
  \le
  (a + b)^\alpha
  \le
  (a^\alpha + b^\alpha)
  ,
\end{align*}
combining these gives, for $\alpha \ge 0$ and $a, b \ge 0$,
\begin{align*}
  \min\{1, 2^{\alpha-1}\} \, (a^\alpha + b^\alpha)
  \le
  (a + b)^\alpha
  \le
  \max\{1, 2^{\alpha-1}\} \, (a^\alpha + b^\alpha)
  .
\end{align*}
We can now compare the Fourier multipliers and arrive at the following, for $\alpha \ge 0$, $x \in \R$, under the convention that $0^0 = 1$,
\begin{align*}
  \min\{1, 2^{\alpha-1}\} \, (1 + |x|^{2\alpha})
  &\le
  (1 + |x|^2)^\alpha
  \le
  \max\{1, 2^{\alpha-1}\} \, (1 + |x|^{2\alpha})
  ,
  \\
  \max\{1, |x|\}^{2\alpha}
  &\le
  1 + |x|^{2\alpha}
  \le
  2 \, \max\{1, |x|\}^{2\alpha}
  ,
  \\
  \max\{1, |x|\}^{2\alpha}
  =
  (\max\{1, |x|^2\})^\alpha
  &\le
  (1 + |x|^2)^\alpha
  \le
  2^\alpha \, \max\{1, |x|\}^{2\alpha}
  ,
\end{align*}
and
\begin{align*}
  (1 + |x|^{2\alpha})
  &\le
  (1 + |2\pi x|^{2\alpha})
  \le
  (2\pi)^{2\alpha} \, (1 + |x|^{2\alpha})
  ,
  \\
  (1 + |x|^2)^\alpha
  &\le
  (1 + |2\pi x|^2)^\alpha
  \le
  (2\pi)^{2\alpha} \, (1 + |x|^2)^\alpha
  ,
  \\
  \max\{1, |x|\}^{2\alpha}
  &
  \le
  \max\{1, |2\pi x|\}^{2\alpha}
  \le
  (2\pi)^{2\alpha} \, \max\{1, |x|\}^{2\alpha}
  .
\end{align*}

Consider $t \ge 0$.
Note that for $\tau=0,\ldots,\alpha$ we have $(1-t^\tau) \, (1-t^{\alpha-\tau}) = 1 + t^\alpha - t^\tau - t^{\alpha-\tau} \ge 0$ and hence $t^\tau + t^{\alpha-\tau} \le 1 + t^\alpha$.
Summing both sides over $\tau$ gives $\sum_{\tau=0}^\alpha (t^\tau + t^{\alpha-\tau}) = 2 \sum_{\tau=0}^\alpha t^\tau$ for the left-hand side and $\sum_{\tau=0}^\alpha (1 + t^\alpha) = (\alpha+1) \, (1+t^\alpha)$ for the right-hand side, hence
\begin{align*}
  \left( 1 + \sum_{\tau=1}^\alpha t^\tau \right)
  &\le
  \frac{\alpha+1}{2} \, (1 + t^\alpha)
  .
\end{align*}
Next, for $\alpha \ge 1$ consider the function $f(s) = ((1 - s) + s \, t)^\alpha$ for $0 \le s \le 1$.
This function is convex and hence $\int_0^1 f(s) \rd{s} \ge f(1/2) = (1+t)^\alpha / 2^\alpha$ while $\int_0^1 f(s) \rd{s} = (1/(\alpha+1)) \sum_{\tau=0}^\alpha t^\tau$.
Hence
\begin{align*}
  (1 + t)^\alpha
  &\le
  \frac{2^\alpha}{\alpha + 1} (1 + \sum_{\tau=1}^\alpha t^\tau)
  .
\end{align*}

Finally, these let us prove, for $\alpha \in \N$ and $x \in \R$,
\begin{align*}
  (1 + |x|^{2\alpha})
  &\le
  \left( 1 + \sum_{\tau = 1}^\alpha |x|^{2 \tau} \right)
  \le
  \frac{\alpha + 1}{2} \, (1 + |x|^{2\alpha})
  ,
\end{align*}
and for $\alpha \in \N_0$ and $x \in \R$,
\begin{align*}
  \left(1 + \sum_{\tau=1}^\alpha |x|^{2\tau} \right)
  &\le
  (1 + |x|^2)^\alpha
  \le
  \frac{2^\alpha}{\alpha+1} \,
  \left(1 + \sum_{\tau=1}^\alpha |x|^{2\tau} \right)
  .
\end{align*}
From these inequalities all the results follow, except for the last result.
For the last result, Cauchy--Schwarz and Parseval's identity immediately give
\begin{align*}
  \|f\|_{\widetilde{H}^{\alpha,\rmu}([0,1])}^2
  &=
  \left|\int_0^1 f(x) \rd{x}\right|^2
  +
  \|D^\alpha f\|_{L_2([0,1])}^2
  \\
  &\le
  \|f\|_{L_2([0,1])}^2
  +
  \|D^\alpha f\|_{L_2([0,1])}^2
  \le
  \|f\|_{\widetilde{H}^\alpha([0,1])}^2
  .
\end{align*}
For the upper inequality, the Fourier multipliers are equal to one for $h=0$, while for $h \ne 0$ we have
\begin{align*}
  \frac{\sum_{\tau=0}^\alpha |2\pi h|^{2\tau}}{|2\pi h|^{2\alpha}}
  =
  \sum_{j=0}^\alpha |2\pi h|^{-2j}
  \le
  \sum_{j \ge 0} (2\pi)^{-2j}
  =
  \left(1-(2\pi)^{-2}\right)^{-1}
  .
\end{align*}
Since $(1-(2\pi)^{-2})^{-1} > 1$ this proves the bound.
One can take the function $\exp(2\pi\rmi \, x)$ which has Fourier support $h=1$ and then the limit of the ratio gives exactly $(1-(2\pi)^{-2})^{-1}$ as above for $\alpha \to \infty$.
\end{proof}

Thus, for the relevant ranges of $\alpha$, the Fourier-based spaces covered by the lemma coincide as vector spaces and have equivalent norms.
For $\alpha > 0$, define their common underlying vector space by
\begin{align*}
  \widetilde{\calH}^\alpha([0,1])
  &:=
  \widetilde{H}'^\alpha([0,1])
  =
  \widetilde{H}''^\alpha([0,1])
  =
  \widetilde{H}^{\alpha,\rmu}([0,1])
  \\
  &\hphantom{:}=
  \widetilde{H}'_1{}^\alpha([0,1])
  =
  \widetilde{H}''_1{}^\alpha([0,1])
  =
  \widetilde{H}^{\alpha,\rmu}_1([0,1])
  ,
\end{align*}
where these are equalities of vector spaces, and the corresponding norms are equivalent.
Similarly, for $\alpha = 0$, we leave out the primed space and define
\begin{align*}
  \widetilde{\calH}^0([0,1])
  &:=
  \widetilde{H}''^0([0,1])
  =
  \widetilde{H}^{0,\rmu}([0,1])
  \\
  &\hphantom{:}=
  \widetilde{H}''_1{}^0([0,1])
  =
  \widetilde{H}^{0,\rmu}_1([0,1])
  =
  L_2([0,1])
  =
  \widetilde{H}^\alpha([0,1])
  .
\end{align*}
Obviously, for $\alpha \in \N_0$, $\widetilde{H}^\alpha([0,1]) = \widetilde{\calH}^\alpha([0,1])$ as well.

\paragraph{Continuity, nestedness and derivative regularity.}
Next we show that $\widetilde{\calH}^\alpha([0,1])$ is continuously embedded into the continuous functions for any of those norms when $\alpha > 1/2$.
For us this means that when point evaluations are required, e.g., for deterministic algorithms, we restrict to $\alpha > 1/2$ for any of the periodic Sobolev spaces defined here.

\begin{lemma}[Embedding into continuous functions (and Wiener algebra)]\label{lem:periodic-continuous-embeddings-d=1}
For $\alpha > 1/2$ we have $\widetilde{\calH}^\alpha([0,1]) \subset A([0,1]) \hookrightarrow C(\bbT([0,1])) \hookrightarrow C([0,1])$, and this inclusion is continuous for any of the equivalent norms.
Let $F$ denote any of the spaces $\widetilde{H}'^\alpha([0,1])$, $\widetilde{H}''^\alpha([0,1])$ and $\widetilde{H}^{\alpha,\rmu}([0,1])$, or $\widetilde{H}^\alpha([0,1])$ when $\alpha \in \N$, and $r_F(h)$ its Fourier multiplier such that $\|f\|_F^2 = \sum_{h \in \Z} |\hat{f}_h|^2 \, r_F^2(h)$, then
\begin{align*}
  \|f\|_{C([0,1])}
  &\le
  \|f\|_{A([0,1])}
  \le
  \widetilde{C}_{\alpha,F} \, \|f\|_F
  ,
\end{align*}
with
\begin{align*}
  (\widetilde{C}_{\alpha,F})^2
  &:=
  \sum_{h \in \Z} r_F^{-2}(h)
  \le
  1 + 2 \, (2\pi)^{-2\alpha} \, \zeta(2\alpha)
  <
  \infty
  .
\end{align*}
Let $F_1$ denote any of the spaces $\widetilde{H}'_1{}^\alpha([0,1])$, $\widetilde{H}''_1{}^\alpha([0,1])$, $\widetilde{H}^{\alpha,\rmu}_1([0,1])$, and $r_{F_1}(h)$ its Fourier multiplier, then $F_1$ is continuously embedded into $A([0,1]) \hookrightarrow C(\bbT([0,1]))$ $\hookrightarrow C([0,1])$ with
\begin{align*}
  \|f\|_{C([0,1])}
  &\le
  \|f\|_{A([0,1])}
  \le
  \widetilde{C}_{\alpha,F_1} \, \|f\|_{F_1}
  ,
\end{align*}
with
\begin{align*}
  (\widetilde{C}_{\alpha,F_1})^2
  &:=
  \sum_{h \in \Z} r_{F_1}^{-2}(h)
  \le
  1 + 2 \, \zeta(2\alpha)
  <
  \infty
  .
\end{align*}
In each case the continuous representative is given by the absolutely and uniformly convergent Fourier series
\begin{align*}
  f(x)
  &=
  \sum_{h \in \Z}
  \hat{f}_h \, \exp(2\pi\rmi \, h x)
  \qquad
  \text{for } x \in [0,1]
  .
\end{align*}
\end{lemma}
\begin{proof}
For every squared Fourier multiplier with frequency scaling $2\pi$ listed in the lemma, we have $r_F^2(0) = 1$ and $r_F^2(h) \ge |2\pi h|^{2\alpha}$ for $h \ne 0$.
Consequently,
\begin{align*}
  \sum_{h \in \Z} r_F^{-2}(h)
  \le
  1 + 2 \sum_{h \ge 1} |2\pi h|^{-2\alpha}
  =
  1 + 2 \, (2\pi)^{-2\alpha} \, \zeta(2\alpha)
  <
  \infty
  .
\end{align*}
For the squared Fourier multipliers with frequency scaling~$1$ instead of~$2\pi$, the same argument uses $r_F^2(h) \ge |h|^{2\alpha}$ for $h \ne 0$ and gives the second bound in the statement.
Both bounds are finite by the zeta-function properties stated above because $2\alpha > 1$.
For any of these Fourier multipliers, Cauchy--Schwarz gives
\begin{align*}
  \sum_{h \in \Z} |\hat{f}_h|
  &\le
  \left(\sum_{h \in \Z} |\hat{f}_h|^2 \, r_F^2(h)\right)^{1/2}
  \left(\sum_{h \in \Z} r_F^{-2}(h) \right)^{1/2}
  =
  \widetilde{C}_{\alpha,F} \, \|f\|_F
  <
  \infty
  .
\end{align*}
Hence the Fourier series converges absolutely and uniformly and defines a continuous representative satisfying, for every $x \in [0,1]$,
\begin{align*}
  |f(x)|
  &\le
  \sum_{h \in \Z} |\hat{f}_h|
  \le
  \widetilde{C}_{\alpha,F} \, \|f\|_F
  .
  \qedhere
\end{align*}
\end{proof}

The following result tells us that our spaces are nested with respect to the smoothness.
This follows immediately because their Fourier multipliers are nondecreasing with~$\alpha$.

\begin{lemma}[Nestedness]\label{lem:periodic-decreasing-smoothness-d=1}
For $\alpha \ge \beta \ge 0$ we have the inclusion
\begin{align*}
  \widetilde{\calH}^\alpha([0,1])
  &\subseteq
  \widetilde{\calH}^\beta([0,1])
  ,
\end{align*}
and, for any one of these families of spaces, denoted by $F^\alpha$ and $F^\beta$, such that they are the same spaces with possibly different smoothness,
\begin{align*}
  \|f\|_{F^\beta}
  &\le
  \|f\|_{F^\alpha}
  .
\end{align*}
\end{lemma}

Finally, as an extension to Lemma~\ref{lem:periodic-Fourier-characterisation-d=1}, we show for which derivative orders the Fourier series of $D^\tau f$ converges absolutely and uniformly and which $D^\tau f$ are absolutely continuous, since these ranges are slightly different for real~$\alpha$.

\begin{lemma}\label{lem:periodic-derivatives-d=1}
Let $\alpha \ge 0$ be real and $f \in \widetilde{\calH}^\alpha([0,1])$.
For every $\tau \in \N_0$ satisfying $0 \le \tau \le \alpha - 1$, i.e., $0 \le \tau \le \lfloor\alpha\rfloor - 1$, the derivative $D^\tau f$ is absolutely continuous.
For every $\tau \in \N_0$ satisfying $0 \le \tau < \alpha - 1/2$, i.e., $0 \le \tau \le \lceil \alpha - 1/2 \rceil - 1$, the Fourier series of $D^\tau f$ converges absolutely and uniformly, with
\begin{align*}
  D^\tau f(x)
  &=
  \sum_{h \in \Z}
  (2\pi\rmi \, h)^\tau \, \hat{f}_h \, \exp(2\pi\rmi \, h x)
  \qquad
  \text{for } x \in [0,1]
  ,
\end{align*}
and thus $D^\tau f \in A([0,1]) \subset C(\bbT([0,1]))$, i.e., $D^\tau f \in C([0,1])$ and $D^\tau f(0) = D^\tau f(1)$, for integer $0 \le \tau < \alpha - 1/2$.
\end{lemma}
\begin{proof}
For $\alpha \ge 1$, Lemma~\ref{lem:periodic-decreasing-smoothness-d=1} gives
\begin{align*}
  \widetilde{\calH}^\alpha([0,1])
  &\subseteq
  \widetilde{\calH}^{\lfloor \alpha \rfloor}([0,1])
  =
  \widetilde{H}^{\lfloor \alpha \rfloor}([0,1])
  .
\end{align*}
The assertion about absolute continuity follows from the definition of $\widetilde{H}^{\lfloor \alpha \rfloor}([0,1])$.

For $\alpha = 0$ the statement is trivially true.
Now let $\alpha > 0$.
For an integer $\tau$ satisfying $0 \le \tau < \alpha-1/2$, we use $f \in \widetilde{H}'^\alpha([0,1])$.
The definition of the norm~\eqref{eq:norm-tildeH-prime-alpha-d=1} shows that $f$ has a weak derivative $D^\tau f$ with Fourier coefficients $(2\pi\rmi \, h)^\tau \, \hat{f}_h$.
For $\tau = 0$ we already established that $f$ is given by its absolutely converging Fourier series, see Lemma~\ref{lem:periodic-continuous-embeddings-d=1}.
For the higher derivatives $D^\tau f$ with integer $1 \le \tau < \alpha-1/2$, we use $1 + |2\pi h|^{2\alpha} \ge |2\pi h|^{2\alpha}$ and apply Cauchy--Schwarz to obtain, see also the proof of Lemma~\ref{lem:periodic-Fourier-characterisation-d=1},
\begin{align*}
  \sum_{h \in \Z}
  |(2\pi\rmi \, h)^\tau \, \hat{f}_h|
  &\le
  \left(
    2 \sum_{h \ge 1}
    \frac{|2\pi h|^{2\tau}}{1+|2\pi h|^{2\alpha}}
  \right)^{1/2}
  \|f\|_{\widetilde{H}'^\alpha([0,1])}
  \\
  &\le
  \left(
    2 \, (2\pi)^{-2(\alpha-\tau)} \, \zeta(2(\alpha-\tau))
  \right)^{1/2}
  \|f\|_{\widetilde{H}'^\alpha([0,1])}
  <
  \infty
  ,
\end{align*}
which is finite since $2 (\alpha - \tau) > 1$.
Hence the Fourier series of $D^\tau f$ converges absolutely and uniformly and gives the stated representation.
\end{proof}

All these univariate results can be used trivially when constructing tensor product spaces.
E.g., for a $d$-fold tensor product space the same continuous embeddings are achieved with the constants raised to the power~$d$.
Raising a constant to the power $d$ might have unintended effects when one wants to prove tractability, so we will have to proceed slightly more carefully.
The techniques that we used in the proofs in this section will still be of use in Sections~\ref{sec:weighted-spaces} and~\ref{sec:periodic-spaces}.
Sometimes we will adjust them for product weights to get rid of the $d$ dependence.
An overview of the multivariate spaces that we will build can be found in Table~\ref{tbl:periodic-equivalent-norms}.

After all of this, you might still be wondering where all these different variants of the same space come from.
We will not withhold that from you and will now give some references, but there are of course many more.
E.g., all three spaces $\widetilde{H}^\alpha([0,1])$, $\widetilde{H}'^\alpha([0,1])$ and $\widetilde{H}''^\alpha([0,1])$ are treated, together with even a further variant, in~\cite{KSU2015}.
For $\widetilde{H}^{\alpha,\rmu}([0,1])$, see, e.g.,~\cite{NW2008,NW2010,NS2023}, and for $\widetilde{H}^{\alpha,\rmu}_1([0,1])$, see, e.g.,~\cite{NW2008,DNP2014,DHP2015,DKS2013,Nuy2014,KPV2015,DKP2022}.
(These references all treat these as multivariate spaces.)

\subsubsection{Equivalent norms for the non-periodic Sobolev space}\label{sec:non-periodic-equivalent-norms-d=1}

In the information-based complexity and quasi-Monte Carlo literature there are two commonly used equivalent norms in the non-periodic space.
We will describe these equivalent norms here first for~$d=1$.
It will turn out that they have again equivalent norms and so we could introduce the notation $\calH^\alpha([0,1])$ to denote the vector space without specifying the norm, but we will not do that this time.
In Section~\ref{sec:mixed-spaces} we will define their multivariate variants of dominating mixed smoothness spaces that are used in QMC and IBC.

\paragraph{The unanchored Sobolev norm.}
The first modification is to create what is called an \emph{unanchored} norm.
For this modification, instead of taking $L_2$-norms of the intermediate derivatives, we will just take the absolute value of the integral of the derivative.
That is, we define, for $\alpha \in \N_0$,
\begin{align}\label{eq:H-alpha-u-d=1}
\boxed{
  \|f\|_{H^{\alpha,\rmu}([0,1])}^2
  :=
  \sum_{\tau=0}^{\alpha-1} \left|\int_0^1 D^\tau f(x) \rd{x} \right|^2
  +
  \|D^\alpha f\|_{L_2([0, 1])}^2
  .
}
\end{align}
The superscript ``$\rmu$'' is used to denote that this is the \emph{unanchored Sobolev space norm} in contrast to the ``anchored'' norm which we will define next.
When $f$ is in the periodic space then this norm can be written in terms of Fourier multipliers and we have already defined that norm in~\eqref{eq:norm-tildeH-alpha-u-d=1}.
Again, the space $H^{\alpha,\rmu}([0,1])$ is equal to $H^\alpha([0,1])$ as a vector space and they have equivalent norms.
We show the embedding constants and the connection to the Fourier form in the next result.

\begin{lemma}\label{lem:H-alpha-u-norm-equivalence-d=1}
Let $\alpha \in \N_0$.
We have the continuous embedding $H^{\alpha,\rmu}([0,1]) \hookrightleftarrows H^\alpha([0,1])$ with
\begin{align*}
  \|f\|_{H^{\alpha,\rmu}([0,1])}
  \le
  \|f\|_{H^\alpha([0,1])}
  \le
  (1-\pi^{-2})^{-1/2} \, \|f\|_{H^{\alpha,\rmu}([0,1])}
  .
\end{align*}
The unanchored norm restricted to the periodic space agrees with the periodic unanchored norm,
\begin{align*}
  \|f\|_{H^{\alpha,\rmu}([0,1])}
  &=
  \|f\|_{\widetilde{H}^{\alpha,\rmu}([0,1])}
  \qquad
  \text{for all } f \in \widetilde{H}^\alpha([0,1])
  ,
\end{align*}
and in this case we have the stronger continuous embedding, already stated in Lemma~\textup{\ref{lem:periodic-equivalent-norms-d=1}},
\begin{align*}
  \|f\|_{\widetilde{H}^{\alpha,\rmu}([0,1])}
  &\le
  \|f\|_{\widetilde{H}^\alpha([0,1])}
  \le
  \left(1-(2\pi)^{-2}\right)^{-1/2} \, \|f\|_{\widetilde{H}^{\alpha,\rmu}([0,1])}
  ,
\end{align*}
for which the upper bound is sharp for $\alpha \to \infty$.
\end{lemma}
\begin{proof}
The results are immediate for $\alpha = 0$, so we can assume $\alpha \in \N$ for the remainder of the proof.

Let $f \in H^\alpha([0,1])$.
Using Cauchy--Schwarz and $\int_{[0,1]} \rd{x} = 1$, we immediately find
\begin{align*}
  \|f\|_{H^{\alpha,\rmu}([0,1])}^2
  &=
  \sum_{\tau = 0}^{\alpha-1} \left| \int_0^1 D^\tau f(x) \rd{x} \right|^2
  +
  \|D^\alpha f\|_{L_2([0, 1])}^2
  \\
  &\le
  \sum_{\tau = 0}^{\alpha-1} \int_0^1 |D^\tau f(x)|^2 \rd{x}
  +
  \|D^\alpha f\|_{L_2([0, 1])}^2
  =
  \|f\|_{H^\alpha([0,1])}^2
  .
\end{align*}
This proves the first inequality and shows $H^\alpha([0,1]) \hookrightarrow H^{\alpha,\rmu}([0,1])$.

The Poincaré--Wirtinger inequality for $g \in L_1([0,1])$ with weak derivative $Dg \in L_2([0,1])$ gives
\begin{align*}
  \left\| g - \int_0^1 g(x) \rd{x} \right\|_{L_2([0,1])}^2
  &\le
  \pi^{-2} \, \|Dg\|_{L_2([0,1])}^2
  ,
\end{align*}
hence
\begin{align*}
  \|g\|_{L_2([0,1])}^2
  &\le
  \left|\int_0^1 g(x) \rd{x}\right|^2
  +
  \pi^{-2} \, \|Dg\|_{L_2([0,1])}^2
  .
\end{align*}

Let $f \in H^{\alpha,\rmu}([0,1])$.
Applying this successively to $g = D^\tau f$, starting with $\tau = \alpha-1$ and proceeding downwards to $\tau = 0$, gives, for each $0 \le \tau \le \alpha-1$,
\begin{align*}
  \|D^\tau f\|_{L_2([0,1])}^2
  &\le
  \left|\int_0^1 D^\tau f(x) \rd{x}\right|^2
  +
  \pi^{-2} \, \|D^{\tau+1} f\|_{L_2([0,1])}^2
  \\
  &\le
  \sum_{k=\tau}^{\alpha-1} \pi^{-2(k-\tau)} \, \left|\int_0^1 D^k f(x) \rd{x}\right|^2
  +
  \pi^{-2(\alpha-\tau)} \, \|D^\alpha f\|_{L_2([0,1])}^2
  .
\end{align*}
Summing over $\tau$ and using $\sum_{j \ge 0} \pi^{-2j} = (1-\pi^{-2})^{-1}$ proves the second inequality, and shows that $H^{\alpha,\rmu}([0,1]) \hookrightarrow H^\alpha([0,1])$.

If $f \in \widetilde{H}^\alpha([0,1])$ and $1 \le \tau \le \alpha-1$, then, see, e.g., Lemma~\ref{lem:periodic-derivatives-d=1},
\begin{align*}
  \int_0^1 D^\tau f(x) \rd{x}
  &=
  D^{\tau-1}f(1) - D^{\tau-1}f(0)
  =
  0
  .
\end{align*}
Moreover, $\int_0^1 f(x) \rd{x} = \hat{f}_0$, and Parseval's identity gives
\begin{align*}
  \|D^\alpha f\|_{L_2([0,1])}^2
  &=
  \sum_{0 \ne h \in \Z}
  |\hat{f}_h|^2 \, |2\pi h|^{2\alpha}
  .
\end{align*}
Hence
\begin{align*}
  \|f\|_{H^{\alpha,\rmu}([0,1])}^2
  &=
  \sum_{h \in \Z}
  |\hat{f}_h|^2 \, \max\{1,|2\pi h|\}^{2\alpha}
  =
  \|f\|_{\widetilde{H}^{\alpha,\rmu}([0,1])}^2
  .
  \qedhere
\end{align*}
\end{proof}

\paragraph{The anchored Sobolev norm.}
A second modification is to create what is called an \emph{anchored} norm.
For this modification we need to choose an anchor $a \in [0,1]$ and we replace all intermediate derivatives by evaluating them at the anchor point.
That is, we define, for $\alpha \in \N_0$,
\begin{align}\label{eq:H-alpha-a-d=1}
\boxed{
  \|f\|_{H^{\alpha,a}([0,1])}^2
  :=
  \sum_{\tau = 0}^{\alpha-1} \left|D^\tau f(a)\right|^2
  +
  \|D^\alpha f\|_{L_2([0, 1])}^2
  .
}
\end{align}
In this case, the superscript ``$a$'' denotes the anchor.
Contrary to ``$\rmu$'' we denote ``$a$'' with the normal maths font for variables as it is supposed to represent a value $a \in [0,1]$.
As such the space for the specific choice of $a=1/2$ is denoted by $H^{\alpha,1/2}([0,1])$.
This space is often used in functional analysis, see, e.g.,~\cite[Chapter~9]{Moo85}.

Again, this space is equal to $H^\alpha([0,1])$ as a vector space and they have equivalent norms.
We show the embedding constants in the next statement.
The embedding constants between the anchored and unanchored spaces for $\alpha=1$ are also given in~\cite{GHHRW2017}.

\begin{lemma}\label{lem:H-alpha-a-norm-equivalence-d=1}
Let $\alpha \in \N$ and $a \in [0,1]$.
We have the continuous embeddings $H^{\alpha,a}([0,1]) \hookrightleftarrows H^\alpha([0,1])$ with
\begin{align*}
  \frac{1}{\sqrt{3}} \, \|f\|_{H^{\alpha,a}([0,1])}
  \le
  \|f\|_{H^\alpha([0,1])}
  \le
  \sqrt{2\alpha} \, \|f\|_{H^{\alpha,a}([0,1])}
  ,
\end{align*}
and $H^{\alpha,a}([0,1]) \hookrightleftarrows H^{\alpha,\rmu}([0,1])$ with
\begin{align*}
  \sqrt{\frac{1-\pi^{-2}}{3}} \, \|f\|_{H^{\alpha,a}([0,1])}
  &\le
  \|f\|_{H^{\alpha,\rmu}([0,1])}
  \le
  \sqrt{2\alpha} \, \|f\|_{H^{\alpha,a}([0,1])}
  .
\end{align*}
When restricted to the periodic subspace, where $\|\cdot\|_{H^{\alpha,\rmu}([0,1])} = \|\cdot\|_{\widetilde{H}^{\alpha,\rmu}([0,1])}$, this comparison can be improved to
\begin{align*}
  \frac{\sqrt{3}}{2} \, \|f\|_{H^{\alpha,a}([0,1])}
  &\le
  \|f\|_{\widetilde{H}^{\alpha,\rmu}([0,1])}
  \le
  \frac{2}{\sqrt{3}} \, \|f\|_{H^{\alpha,a}([0,1])}
  \qquad
  \text{for all } f \in \widetilde{H}^\alpha([0,1])
  .
\end{align*}
Consequently, on the periodic subspace we also have
\begin{align*}
  \frac{\sqrt{3}}{2} \, \|f\|_{H^{\alpha,a}([0,1])}
  \le
  \|f\|_{\widetilde{H}^\alpha([0,1])}
  \le
  \frac{2}{\sqrt{3}} \left(1-(2\pi)^{-2}\right)^{-1/2} \, \|f\|_{H^{\alpha,a}([0,1])}
  \journalonly{\qquad\\*}\qquad
  \text{for all } f \in \widetilde{H}^\alpha([0,1])
  .
\end{align*}
\end{lemma}
\begin{proof}
For $h \in L_2([0,1])$, Cauchy--Schwarz gives
\begin{align*}
  \int_0^1 \left|\int_a^x h(t) \rd{t}\right|^2 \rd{x}
  &\le
  \frac{a^2 + (1 - a)^2}2 \, \|h\|_{L_2([0,1])}^2
  \le
  \frac12 \, \|h\|_{L_2([0,1])}^2
  .
\end{align*}
Let $f \in H^{\alpha,a}([0,1])$.
Using absolute continuity~\eqref{eq:absolute-continuity-d=1} of $g = D^\tau f$ for $\tau = \alpha-1$ down to~$0$ gives
\begin{align*}
  D^\tau f(x)
  &=
  D^\tau f(a)
  +
  \int_a^x D^{\tau+1} f(t) \rd{t}
  .
\end{align*}
Using $(r+s)^2 \le 2 \, (r^2+s^2)$ and then the previous estimate with $h = D^{\tau+1} f$ gives
\begin{align*}
  \|D^\tau f\|_{L_2([0,1])}^2
  &\le
  2 \, |D^\tau f(a)|^2
  +
  2 \, \int_0^1 \left|\int_a^x D^{\tau+1} f(t) \rd{t}\right|^2 \rd{x}
  \\
  &\le
  2 \, |D^\tau f(a)|^2
  +
  \|D^{\tau+1} f\|_{L_2([0,1])}^2
  \\
  &\le
  2 \, \sum_{k=\tau}^{\alpha-1} |D^k f(a)|^2
  +
  \|D^\alpha f\|_{L_2([0,1])}^2
  .
\end{align*}
Summing over $\tau$ and adding $\|D^\alpha f\|_{L_2([0,1])}^2$ proves the second inequality, since $2(k+1) \le 2\alpha$ and $\alpha+1 \le 2\alpha$, and shows that $H^{\alpha,a}([0,1]) \subseteq H^\alpha([0,1])$.
Conversely, let $f \in H^\alpha([0,1])$.
The same argument gives, for $g \in H^1([0,1])$,
\begin{align*}
  |g(a)|^2
  &\le
  2 \, \|g\|_{L_2([0,1])}^2
  +
  \|Dg\|_{L_2([0,1])}^2
  .
\end{align*}
Applying this to $g = D^\tau f$, for $\tau = 0,\ldots,\alpha-1$, and summing gives
\begin{align*}
  \|f\|_{H^{\alpha,a}([0,1])}^2
  &=
  \sum_{\tau=0}^{\alpha-1} |D^\tau f(a)|^2
  +
  \|D^\alpha f\|_{L_2([0,1])}^2
  \\
  &\le
  2 \, \sum_{\tau=0}^{\alpha-1} \|D^\tau f\|_{L_2([0,1])}^2
  +
  \sum_{\tau=0}^{\alpha-1} \|D^{\tau+1} f\|_{L_2([0,1])}^2
  +
  \|D^\alpha f\|_{L_2([0,1])}^2
  \\
  &\le
  3 \, \|f\|_{H^\alpha([0,1])}^2
  .
\end{align*}
This proves the first inequality and shows that $H^\alpha([0,1]) \subseteq H^{\alpha,a}([0,1])$.
The second embedding now follows from Lemma~\ref{lem:H-alpha-u-norm-equivalence-d=1}.

It remains to prove the improved bounds on the periodic subspace.
Let $f \in \widetilde{H}^\alpha([0,1])$.
For $k \in \N$ we have $2 \, (2\pi)^{-2k} \, \zeta(2k) \le 1/12$ which is achieved for $k = 1$.
Note that
\begin{align*}
  |f(x) - \hat{f}_0|
  &=
  \bigg| \sum_{0 \ne h \in \Z} \hat{f}_h \, \exp(2\pi\rmi \, h x) \, |2\pi h|^\alpha \, |2\pi h|^{-\alpha} \bigg|
  \\
  &\le
  \bigg( \sum_{0 \ne h \in \Z} |\hat{f}_h|^2 \, |2\pi h|^{2\alpha} \bigg)^{1/2}
  \,
  \bigg( 2 \sum_{h \ge 1} |2\pi h|^{-2\alpha} \bigg)^{1/2}
  \\
  \numberthis{eq:periodic-anchor-mean-Fourier-bound}
  &=
  \|D^\alpha f\|_{L_2([0,1])}
  \,
  \sqrt{2 \, (2\pi)^{-2\alpha} \, \zeta(2\alpha)}
  \le
  \frac1{\sqrt{12}} \, \|D^\alpha f\|_{L_2([0,1])}
  ,
\end{align*}
and, for $1 \le \tau \le \alpha - 1$,
\begin{align*}
  |D^\tau f(x)|
  &=
  \bigg| \sum_{0 \ne h \in \Z} \hat{f}_h \, (2\pi\rmi h)^\tau \, \exp(2\pi\rmi \, h x) \, |2\pi h|^\alpha \, |2\pi h|^{-\alpha} \bigg|
  \\
  &\le
  \bigg( \sum_{0 \ne h \in \Z} |\hat{f}_h|^2 \, |2\pi h|^{2\alpha} \bigg)^{1/2}
  \,
  \bigg( 2 \sum_{h \ge 1} |2\pi h|^{-2(\alpha-\tau)} \bigg)^{1/2}
  \\
  &=
  \|D^\alpha f\|_{L_2([0,1])}
  \,
  \sqrt{2 \, (2\pi)^{-2(\alpha-\tau)} \, \zeta(2(\alpha-\tau))}
  \le
  \frac1{\sqrt{12}} \, \|D^\alpha f\|_{L_2([0,1])}
  .
\end{align*}
These bounds hold uniformly for~$x$ and thus also for $x = a$.
Then
\begin{align*}
  |\hat{f}_0|
  =
  |\hat{f}_0 - f(a) + f(a)|
  &\le
  |f(a)| + |f(a) - \hat{f}_0|
  \le
  |f(a)| + \frac1{\sqrt{12}} \, \|D^\alpha f\|_{L_2([0,1])}
  ,
\end{align*}
from which we obtain
\begin{align*}
  \|f\|_{\widetilde{H}^{\alpha,\rmu}([0,1])}^2
  &=
  |\hat{f}_0|^2 + \|D^\alpha f\|_{L_2([0,1])}^2
  \\
  &\le
  \Big( |f(a)| + \frac1{\sqrt{12}} \, \|D^\alpha f\|_{L_2([0,1])} \Big)^2 + \|D^\alpha f\|_{L_2([0,1])}^2
  \\
  &\le
  \frac43 \, \Big( |f(a)|^2 + \|D^\alpha f\|_{L_2([0,1])}^2 \Big)
  \\
  &\le
  \frac43 \, \|f\|_{H^{\alpha,a}([0,1])}^2
  ,
\end{align*}
where we used Jensen's inequality~\eqref{eq:jensen2} with $3/4 + 1/4 = 1$ on $(a + b)^2 = ((3/4) \, (4/3) \, a + (1/4) \, 4 \, b)^2 \le (3/4) \, (4/3)^2 \, a^2 + (1/4) \, 4^2 \, b^2 = (4/3) \, a^2 + 4 \, b^2$ for $a, b \ge 0$.
Likewise,
\begin{align*}
  |f(a)|
  &=
  |f(a) - \hat{f}_0 + \hat{f}_0|
  \le
  |\hat{f}_0| + |f(a) - \hat{f}_0|
  \journalbreak\le
  |\hat{f}_0| + \sqrt{2 \, (2\pi)^{-2\alpha} \, \zeta(2\alpha)} \, \|D^\alpha f\|_{L_2([0,1])}
  .
\end{align*}
Now, again, using the same Jensen's inequality, we obtain
\begin{align*}
  \journalonly{\MoveEqLeft[1]} \|f\|_{H^{\alpha,a}([0,1])}^2
  \journalonly{\\*}&=
  |f(a)|^2
  +
  \sum_{\tau = 1}^{\alpha-1} \left|D^\tau f(a)\right|^2
  +
  \|D^\alpha f\|_{L_2([0, 1])}^2
  \\
  &\le
  \Big( |\hat{f}_0| + \sqrt{2 \, (2\pi)^{-2\alpha} \, \zeta(2\alpha)} \, \|D^\alpha f\|_{L_2([0,1])} \Big)^2
  \\
  &\qquad+
  \Big(
  1
  +
  \sum_{\tau=1}^{\alpha-1}
  2 \, (2\pi)^{-2(\alpha-\tau)} \, \zeta(2(\alpha-\tau))
  \Big)
  \,
  \|D^\alpha f\|_{L_2([0,1])}^2
  \\
  &\le
  \frac43 \, |\hat{f}_0|^2
  \journalbreak\journalonly{\qquad}
  +
  \Big(
  1
  +
  \sum_{\tau=1}^{\alpha-1}
  2 \, (2\pi)^{-2(\alpha-\tau)} \, \zeta(2(\alpha-\tau))
  +
  8 \, (2\pi)^{-2\alpha} \, \zeta(2\alpha)
  \Big)
  \,
  \|D^\alpha f\|_{L_2([0,1])}^2
  \\
  &\le
  \frac43 \, \|f\|_{\widetilde{H}^{\alpha,\rmu}([0,1])}^2
  ,
\end{align*}
where we used
\begin{align*}
  \sum_{\tau=1}^{\alpha-1} 2 \, (2\pi)^{-2(\alpha-\tau)} \, \zeta(2(\alpha-\tau))
  +
  8 \, (2\pi)^{-2\alpha} \, \zeta(2\alpha)
  &\le
  \frac13
  .
\end{align*}
For $\alpha = 1$ this is an equality, while for $\alpha \ge 2$ it follows from $2 \, (2\pi)^{-2k} \, \zeta(2k) \le (1/12) \, (2\pi)^{-2(k-1)}$ for $k \ge 1$ and $2 \, (2\pi)^{-2\alpha} \, \zeta(2\alpha) \le 1/720$, since $1/(12 \, (1-(2\pi)^{-2})) + 1/180 < 1/3$.
The constants for the last embedding follow from using Lemma~\ref{lem:H-alpha-u-norm-equivalence-d=1}.
\end{proof}

Again these different spaces occur all over the literature.
The three spaces $H^\alpha([0,1])$, $H^{\alpha,\rmu}([0,1])$ and $H^{\alpha,a}([0,1])$ are, e.g., all treated in~\cite{GHHR2017}.
Further, for $H^{\alpha,\rmu}([0,1])$, see, e.g.,~\cite{KSWW2010b,DNP2014,GSY2017,NS2023,DKP2022}, and for $H^{\alpha,a}([0,1])$, see, e.g.,~\cite{KSWW2010b,DHP2015,NN2021uniform,NN2021lognormal}; for $\alpha=1$, both spaces are also treated in~\cite{NW2008,DKS2013,DKP2022}.
(Again, these references all treat these as multivariate spaces.)

As a final remark, we started with defining the Sobolev norm in terms of an $\ell_2$-sum of $L_2$-norms of derivatives, see~\eqref{eq:H-alpha-d=1}.
In light of all the equivalent norms that we just discussed, it should be clear that this $\ell_2$-sum could easily be replaced by any other $\ell_p$-norm, since, with $m = \alpha + 1$ and $\bsx \in \bbC^m$, $1 \le p \le q \le \infty$,
\begin{align*}
  \|\bsx\|_{\ell_\infty(\bbC^m)}
  \le
  \|\bsx\|_{\ell_q(\bbC^m)}
  &\le
  \|\bsx\|_{\ell_p(\bbC^m)}
  \le
  m^{(1/p) - (1/q)} \, \|\bsx\|_{\ell_q(\bbC^m)}
  .
\end{align*}
Thus moving from any such norm with an $\ell_p$-sum to an $\ell_q$-sum will just adjust the embedding constants.
For a multivariate tensor product space this will however be raised to the power~$d$.

\subsection{Multivariate mixed smoothness spaces}\label{sec:non-periodic-spaces}

Generalising Sobolev space norms for use in high-dimensional settings will be a two-step procedure for us.
The notation for these spaces will be introduced precisely below but we will first try to explain the two steps.

\textbf{(1)} In Section~\ref{sec:mixed-spaces} we move from the classical space $H^\alpha([0,1]^d)$ to the space of dominating mixed smoothness $H^\alpha_\mix([0,1]^d)$.
This has the effect of replacing an upper error bound for good algorithms of the form $C_{d,\alpha} \, n^{-\alpha/d}$ by one of the form $C^\mix_{d,\alpha} \, n^{-\alpha}$.
So we go from a dimension-dependent rate $-\alpha/d$ to a dimension-independent rate $-\alpha$.
The constants $C_{d,\alpha} = C_{d,\alpha}(\calQ)$ and $C^\mix_{d,\alpha} = C^\mix_{d,\alpha}(\calQ)$ will depend on the choice of algorithms~$\calQ = \{Q_{d,n}\}_{d,n}$.
We will not prove these upper bounds, or show these good algorithms, but we will show lower bounds in Section~\ref{sec:lower-bounds} that hold for any algorithm (from a certain class).
That is, we will show that no algorithm that uses function values (sampling) can do better than $c_{d,\alpha} \, n^{-\alpha/d}$ and $c^\mix_{d,\alpha} \, n^{-\alpha}$ for some constants $c_{d,\alpha}$ and $c^\mix_{d,\alpha}$.
See Theorems~\ref{thm:Halpha-lower-bound} and~\ref{thm:Hmixalpha-lower-bound}, which specify the setting in detail.

\textbf{(2)} In Section~\ref{sec:weighted-spaces} we move from $H^\alpha_\mix([0,1]^d)$ to the weighted space $H^{\alpha,\bsgamma}_\mix([0,1]^d)$.
Here, $\bsgamma = \{ \gamma_\setu \}_{\setu}$ is a collection of weights $\gamma_\setu$ with $\gamma_\setu > 0$ for each subset of dimensions~$\setu$ to express the importance of each subset.
This will give us upper bounds for good algorithms of the form $C_{d,\alpha,\bsgamma}^\mix \, n^{-\alpha}$ where again $C_{d,\alpha,\bsgamma}^\mix = C_{d,\alpha,\bsgamma}^\mix(\calQ)$ depends on the choice of algorithms~$\calQ = \{Q_{d,n}\}_{d,n}$.
For suitable weights~$\bsgamma$ we have $\sup_{d \in \N} C^\mix_{d,\alpha,\bsgamma} \le C^\mix_{\alpha,\bsgamma} < \infty$ with a completely dimension-independent constant $C^\mix_{\alpha,\bsgamma} = C^\mix_{\alpha,\bsgamma}(\calQ)$.
We use the supremum when monotonicity in~$d$ has not been established.
For the explicit subset sums and products below, the quantities are nondecreasing in~$d$, and hence their supremum equals their limit as $d \to \infty$.
We will prove such an upper bound specifically with lattice rules as the choice of $\calQ$ later in Section~\ref{sec:upper-bound-lattice-rules} in multivariate periodic mixed function spaces.

Together, these two steps separate the convergence rate in the upper error bound from the dimension~$d$ and then control the dimension dependence of its constant.
The ``curse of dimensionality'' is an important concept when talking about high-dimensional problems.
Having an upper bound of the form $C_{\alpha,\bsgamma}^\mix \, n^{-\alpha}$ for a dimension-independent constant $C_{\alpha,\bsgamma}^\mix < \infty$ shows directly that there is no curse of dimensionality, since there is no~$d$ in the rate nor in the constant, and we then have what is called ``strong tractability'', see, e.g., \cite{SW98,SKJ2002,SWW2004,SW2001,SW2002,NW2010}.
Whether this is possible depends on space and on the weights~$\bsgamma$.
On the other hand, just having a lower bound of the form $c_{d,\alpha}^\mix \, n^{-\alpha}$ for the unweighted space does not mean we have the curse of dimensionality.
The lower bound constant $c_{d,\alpha}^\mix$ in Theorem~\ref{thm:Hmixalpha-lower-bound} will in fact be superexponentially small.
We will properly define what we mean by the curse of dimensionality in Section~\ref{sec:curse-definition} and will then show in Theorem~\ref{thm:curse} that the unweighted space $H^\alpha_\mix([0,1]^d)$ does suffer from the curse of dimensionality.

We are about to introduce several variants of these multivariate spaces.
Most of the references given above for the univariate norms already treat the corresponding multivariate spaces with dominating mixed smoothness, so naturally many of the same references reappear here.
We give a few pointers to the literature.
For $H^\alpha([0,1]^d)$, see, e.g.,~\cite{Eva2010,Nguyen2017thesis}.
For $H^\alpha_\mix([0,1]^d)$, see, e.g.,~\cite{Vybiral2005thesis,Nguyen2017thesis,Markhasin2013,NUU2017,UU2016}.
Another often used convention is the $\ell_1$-sum instead of the $\ell_2$-sum, e.g., used for both the classical and mixed norms in~\cite{Nguyen2017thesis} and for the mixed norm in~\cite{Vybiral2005thesis}.
Spaces similar to our increasing smoothness space $H^\bsalpha_\mix([0,1]^d)$ are treated in, e.g.,~\cite{Vybiral2005thesis,PW2010,DG2016,GHHRW2019}.
The three weighted spaces $H^{\alpha,\bsgamma}_\mix([0,1]^d)$, $H^{\alpha,\bsgamma,\bsa}_\mix([0,1]^d)$ and $H^{\alpha,\bsgamma,\rmu}_\mix([0,1]^d)$ are, e.g., treated together in~\cite{GHHR2017}, while further examples of the anchored and unanchored spaces can be found in~\cite{NW2008,NW2010,KSWW2010b,DP2010,DHP2015,DKS2013,DG2014,DNP2014,GSY2017,GHHRW2017,NN2021uniform,NN2021lognormal,NS2023,DKP2022}.
Further related spaces can be found, e.g., in~\cite{NUU2017,UU2016}.

\subsubsection{From classical smoothness to dominating mixed smoothness}\label{sec:mixed-spaces}

The Sobolev space $H^\alpha([0,1])$ introduced in Section~\ref{sec:typical-Sobolev-d=1} on $[0,1]$ is often directly defined in the multivariate setting.
Such spaces appear prominently in the numerical analysis of solutions of PDEs, see, e.g.,~\cite{Eva2010}.
For simplicity, let us still assume that the domain is~$[0, 1]^d$.
The most typical norm one encounters is the order-$1$ smoothness norm
\begin{align*}
  \|f\|_{H^1([0,1]^d)}^2
  \journalonly{&}:=
  \|f\|_{L_2([0,1]^d)}^2 + \|\nabla f\|_{L_2([0,1]^d)}^2
  \journalonly{\\*}&\journalonly{\hphantom{:}}=
  \int_{[0,1]^d} |f(\bsx)|^2 \rd\bsx + \int_{[0,1]^d} |\nabla f(\bsx)|^2 \rd\bsx
  \\
  &\journalonly{\hphantom{:}}=
  \int_{[0,1]^d} |f(\bsx)|^2 \rd\bsx + \int_{[0,1]^d} \sum_{j=1}^d |\partial_j f(\bsx)|^2 \rd\bsx
  \\
  &\journalonly{\hphantom{:}}=
  \|f\|_{L_2([0,1]^d)}^2 + \sum_{\substack{\bstau \in \{0,1\}^d \\ \|\bstau\|_1 = 1}} \|D^\bstau f\|_{L_2([0,1]^d)}^2
  \\
  \numberthis{eq:H-1-norm}
  &\journalonly{\hphantom{:}}=
  \sum_{\substack{\bstau \in \{0,1\}^d \\ \|\bstau\|_1 \le 1}} \|D^\bstau f\|_{L_2([0,1]^d)}^2
  .
\end{align*}
This norm is usually defined by the first expression after the $:=$~sign as the $\ell_2$-sum of the $L_2$-norm of~$f$ and its gradient~$\nabla f$; interpreting $|\nabla f|^2$ as $\sum_{j=1}^d |\partial_j f|^2$ as another squared $\ell_2$-sum.
We expanded the norm until the last expression for two reasons: (\textit{i}) we can easily generalise the last expression to larger order of the derivatives, (\textit{ii}) we will mimic this definition when defining our ``mixed smoothness''.

For $\alpha \in \N_0$ we generalise the norm~\eqref{eq:H-1-norm} to the arbitrary classical smoothness norm which then defines the \emph{classical Sobolev space} of smoothness $\alpha$,
\begin{align}\label{eq:norm-Halpha}
\boxed{  \|f\|_{H^\alpha([0,1]^d)}^2
  :=
  \sum_{\substack{\bstau \in \{0,\ldots,\alpha\}^d \\ \|\bstau\|_1 \le \alpha}} \|D^\bstau f\|_{L_2([0,1]^d)}^2
  .
}
\end{align}
For $\alpha = 0$ we interpret this as the standard $L_2$ norm.
We will call these norms ``classical Sobolev space norms''.
What they do is take derivatives up to the ``total order'' $\|\bstau\|_1 = \tau_1 + \cdots + \tau_d \le \alpha$.

As in the univariate definition, $D^\bstau f$ denotes the weak derivative and its norm is understood to be infinite if it is not in $L_2([0,1]^d)$.
Again we define the space $H^\alpha([0,1]^d)$ by finiteness of the associated norm.
We will not prove the Sobolev embedding into the continuous functions but we refer to, e.g., \cite{Eva2010}, which states $H^\alpha([0,1]^d) \hookrightarrow C([0,1]^d)$ for $\alpha > d/2$.
For deterministic algorithms using function values on this space we therefore assume $\alpha > d/2$, similarly to the univariate case.
The best possible convergence deteriorates from $n^{-\alpha}$ in the one-dimensional case to $n^{-\alpha/d}$ in $d$~dimensions.
We will prove a lower bound, that holds for any sampling algorithm, in Theorem~\ref{thm:Halpha-lower-bound}.
Later in Theorem~\ref{thm:curse} we will also show that this space suffers from the curse of dimensionality.

But, sure there must be functions, and hence function spaces, for which such bad behaviour can be avoided?
Yes indeed, we just have to ask ``a bit'' more from our functions.
It turns out that the ``divide by $d$'' in the rate can be avoided if we switch to ``dominating mixed smoothness spaces''.
The difference is that we add more derivatives to the norm.
We do this by switching from $\|\bstau\|_1 \le \alpha$ to $\|\bstau\|_\infty \le \alpha$, and define the \emph{Sobolev space of dominating mixed smoothness} of order $\alpha$ by the norm
\begin{align}\label{eq:norm-Halpha-mix}
  \boxed{
  \begin{aligned}
  \|f\|_{H^\alpha_\mix([0,1]^d)}^2
  &:=
  \sum_{\substack{\bstau \in \{0,\ldots,\alpha\}^d \\ \|\bstau\|_\infty \le \alpha}} \|D^\bstau f\|_{L_2([0,1]^d)}^2
  \journalbreak\journalonly{\hphantom{:}}=
  \sum_{\bstau \in \{0,\ldots,\alpha\}^d} \|D^\bstau f\|_{L_2([0,1]^d)}^2
  .
  \end{aligned}
  }
\end{align}
This space is exactly the completion of the tensor product of the one-dimensional spaces with respect to the Hilbert tensor-product norm,  $H^\alpha_\mix([0,1]^d) = \smash{\hat{\bigotimes}}_{j=1}^d H^\alpha([0,1])$.
Note that this does not mean that every function in the space is a product of univariate functions.
The space is the completion of finite linear combinations of such product functions.
The condition for the continuous embedding into the continuous functions however follows immediately from the tensor product form.
The next lemma shows that for $\alpha \in \N$, i.e., $\alpha \ge 1$, we have continuous representatives for all elements of $H^\alpha_\mix([0,1]^d)$, so there is no $d$-dependence in the condition like there is for $H^\alpha([0,1]^d)$, but there is in the constant for the norm embedding.
Later for the periodic spaces $\widetilde{H}^\alpha_\mix([0,1]^d)$ and the like in Section~\ref{sec:periodic-spaces} this will become $\alpha > 1/2$.

We will repeatedly use that if the embedding $J_j$ in coordinate $j$ has embedding constant $C_j = \|J_j\|$, then the corresponding tensor-product embedding satisfies
\begin{align}\label{eq:tensorisation}
  C^{(d)}
  &:=
  \Big\| \bigotimes_{j=1}^d J_j \Big\|
  =
  \prod_{j=1}^d \|J_j\|
  =
  \prod_{j=1}^d C_j
  .
\end{align}
For identical factors $C$ this gives $C^{(d)} = C^d$ and thus generally a dimension-dependent bound.
We will introduce ``product weights'' further below, and then the constants $C_j$ depend on numbers $\gamma_j$ and the product can remain uniformly bounded in~$d$ under suitable summability conditions, see Lemma~\ref{lem:H-1-mix-product-weight-embeddings}.

\begin{lemma}\label{lem:H-alpha-mix-C-embedding}
For $\alpha \in \N$, the continuous embedding $H^\alpha_\mix([0,1]^d) \hookrightarrow C([0,1]^d)$ satisfies
\begin{align}\label{eq:H-alpha-mix-C-embedding}
  \|f\|_{C([0,1]^d)}
  &\le
  2^{d/2} \, \|f\|_{H^\alpha_\mix([0,1]^d)}
  .
\end{align}
\end{lemma}
\begin{proof}
From Lemma~\ref{lem:H-alpha-C-embedding-d=1} we have a bound with constant~$\sqrt2$ for $d=1$, the tensorisation~\eqref{eq:tensorisation} then proves the claimed result.
\end{proof}

By taking the mixed norm instead of the classical norm we will be able to have a dimension-independent rate of convergence, going from $n^{-\alpha/d}$ to $n^{-\alpha}$.
We will however still have the curse of dimensionality, see Theorem~\ref{thm:curse}.
As one possible way to achieve tractability and avoid the curse, we will introduce weights in the next section.
Comparing the two norms, \eqref{eq:norm-Halpha} and~\eqref{eq:norm-Halpha-mix}, we notice that the mixed norm just added in some extra terms.
So we have $H^\alpha_\mix([0,1]^d) \subseteq H^\alpha([0,1]^d)$ and if $f \in H^\alpha_\mix([0,1]^d) \subseteq H^\alpha([0,1]^d)$ then $\|f\|_{H^\alpha([0,1]^d)} \le \|f\|_{H^\alpha_\mix([0,1]^d)}$.
This means $H^\alpha_\mix([0,1]^d) \hookrightarrow H^\alpha([0,1]^d)$.
Finally, let us remark that for $d=1$ the classical norms and their corresponding mixed norms coincide as there is no choice to make with respect to which mixed derivatives to include.
I.e., there is no ``mixed'' subscript for~$d=1$.

One further extension is to define a different smoothness for each dimension.
Given $\bsalpha \in \N_0^d$ define
\begin{align}\label{eq:H-bsalpha-mixed-norm}
\boxed{
  \|f\|_{H^\bsalpha_\mix([0,1]^d)}^2
  :=
  \sum_{\substack{\bstau \in \N_0^d \\ \bstau \le \bsalpha}}
  \|D^\bstau f\|_{L_2([0,1]^d)}^2
  ,
}
\end{align}
where $\bstau \in \N_0^d,\; \bstau \le \bsalpha$ means that $0 \le \tau_j \le \alpha_j$ for all $1 \le j \le d$.
The previous mixed space can then be recovered by taking $\bsalpha = (\alpha, \ldots, \alpha)$.
When $\alpha_1 \le \cdots \le \alpha_d$ we call these spaces \emph{increasing smoothness spaces}; similar spaces are treated in~\cite{Vybiral2005thesis,PW2010,DG2016,GHHRW2019}.
We will not pursue increasing smoothness spaces here but will focus on weighted spaces to beat the curse of dimensionality.

\subsubsection{From unweighted to weighted spaces}\label{sec:weighted-spaces}

Our next step is to add ``weights'' to our norms to obtain the corresponding weighted function spaces.
The weights are used to moderate how much a function is allowed to vary in certain dimensions and subspaces.
Note that next to our previous aim of encoding the smoothness of the functions, the $L_2$-norms of the mixed derivatives in the norm are in some way measuring the variation of the function for the dimensions in which we take derivatives.
In full generality we consider all $2^d$ subsets $\setu \subseteq \{1, \ldots, d\}$ and associate a positive weight $\gamma_\setu > 0$ with terms of the form $\smash{\|D^\bstau f\|_{L_2([0,1]^d)}^2}$ for which $\setu = \supp(\bstau)$.
The smaller the weight, the less important this subspace is.
Specifically, we will scale the term $\smash{\|D^\bstau f\|_{L_2([0,1]^d)}^2}$ by $\gamma_\setu^{-1}$ for $\setu = \supp(\bstau)$ such that components involving ``less important'' variable interactions, i.e., those with smaller $\gamma_\setu$, are ``punished'', i.e., will contribute more to the norm.
The norm measures the ``size'' or ``difficulty'' of our function.

For our exposition we will assume $\gamma_\emptyset = 1$ and $\gamma_\setu > 0$ for all~$\setu$.
For the weighted mixed norm $\|\cdot\|_{H^{\alpha,\bsgamma}_\mix([0,1]^d)}$, that we introduce below, this leaves the $L_2$ term unscaled, while for the weighted unanchored norm $\|\cdot\|_{H^{\alpha,\bsgamma,\rmu}_\mix([0,1]^d)}$, also introduced below, it leaves the integral term unscaled.
This has the effect of having the ``initial error'' (the error of the optimal algorithm that uses no information) to be~$1$.
We will need this in Section~\ref{sec:curse-definition}.
For further background on weighted spaces, see, e.g.,~\cite{SW98,SW2002,DSWW2004,DKS2013,GHHR2017,GHHRW2017}.

We now modify the unweighted mixed norm~\eqref{eq:norm-Halpha-mix} to obtain the norm of the \emph{weighted dominating mixed smoothness space}, $\alpha \in \N_0$,
\begin{align}\label{eq:norm-Halpha-bsgamma-mix}
  \boxed{
  \begin{aligned}
  \|f\|_{H^{\alpha,\bsgamma}_\mix([0,1]^d)}^2
  &:=
  \sum_{\substack{\bstau \in \N_0^d \\ \|\bstau\|_\infty \le \alpha}} \gamma_{\supp(\bstau)}^{-1} \, \|D^\bstau f\|_{L_2([0,1]^d)}^2
  \journalbreak\journalonly{\hphantom{:}}=
  \sum_{\bstau \in \{0,\ldots,\alpha\}^d} \gamma_{\supp(\bstau)}^{-1} \, \|D^\bstau f\|_{L_2([0,1]^d)}^2
  .\end{aligned}
  }
\end{align}
As mentioned in the univariate section, we could also consider other $\ell_p$-norms instead of the $\ell_2$-norm over the derivatives.
In some references the weights $\gamma_\setu$ are taken to be squared in the definition of the norm; squared since it is an $\ell_2$-norm.
This simplifies scaling of the norm or using $\ell_p$-inequalities, but we will not do that here.

\paragraph{Weight families.}
For a particular $d$-variate function space it is of course enough to consider the finite set of weights $\{ \gamma_\setu \}_{\setu \subseteq \{1, \ldots, d\}}$.
However, when we want to study worst-case errors and the effect of increasing the dimension $d$, e.g., tractability and the curse of dimensionality, then we are interested in the behaviour when $d \to \infty$.
Define
\begin{align*}
  \frakU &:=
  \big\{ \setu \subset \N : |\setu| < \infty \big\}
  \subset
  2^\N
  .
\end{align*}
We will use the notation $\bsgamma$ to define a family of weights as a mapping $\bsgamma : \frakU \to \R_{>0} : \setu \mapsto \gamma_\setu$ with the notation
\begin{align*}
  \bsgamma
  &=
  \{ \gamma_\setu \}_{\setu \in \frakU}
  .
\end{align*}
With this in place it should be not too surprising that worst-case errors as in~\eqref{eq:def-wce-integration} and~\eqref{eq:def-wce}, where we take the sup over the unit ball of the function space, might now be controlled with the weights.
In fact, we will be able to derive error bounds in which certain sums over the weights will appear.
Let's repeat again that they could be bounded by a constant $C_{d,\alpha,\bsgamma}^\mix(\calQ)$, and under the right summability conditions of the weights we will have that $\sup_{d \in \N} C_{d,\alpha,\bsgamma}^\mix(\calQ) \le C_{\alpha,\bsgamma}^\mix(\calQ) < \infty$, depending on the choice of algorithms~$\calQ = \{Q_{d,n}\}_{d,n}$.
We will make this precise for lattice rules in Section~\ref{sec:upper-bound-lattice-rules} in our periodic weighted Sobolev spaces of dominating mixed smoothness.

We first show that for fixed~$d$ these spaces are continuously embedded into each other.
The straightforward comparison in the next lemma gives bounds that depend on~$d$.
Lemma~\ref{lem:H-1-mix-product-weight-embeddings} below will give better bounds that are independent of~$d$ under suitable conditions.

\begin{lemma}[Straightforward embedding]\label{lem:H-alpha-C-embedding-d}
Given $d \in \N$ and a family of weights $\bsgamma = \{ \gamma_\setu \}_{\setu \in \frakU}$ with $\gamma_\setu > 0$ for every~$\setu \in \frakU$, let
\begin{align}\label{eq:def-gamma-max-min-d}
  \gamma_{\min,d}
  &:=
  \min_{\setu \subseteq \{1,\ldots,d\}} \gamma_\setu
  ,
  &
  \gamma_{\max,d}
  &:=
  \max_{\setu \subseteq \{1,\ldots,d\}} \gamma_\setu
  ,
\end{align}
then we have the continuous embedding $H^{\alpha,\bsgamma}_\mix([0,1]^d) \hookrightleftarrows H^\alpha_\mix([0,1]^d)$ with
\begin{align}\label{eq:H-alpha-mix-weighted-equivalence}
  \frac1{\sqrt{\gamma_{\max,d}}} \, \|f\|_{H^\alpha_\mix([0,1]^d)}
  &\le
  \|f\|_{H^{\alpha,\bsgamma}_\mix([0,1]^d)}
  \le
  \frac1{\sqrt{\gamma_{\min,d}}} \, \|f\|_{H^\alpha_\mix([0,1]^d)}
  .
\end{align}
For $\alpha \in \N$ we also have the continuous embedding $H^{\alpha,\bsgamma}_\mix([0,1]^d) \hookrightarrow C([0,1]^d)$ with
\begin{align*}
  \|f\|_{C([0,1]^d)}
  &\le
  2^{d/2} \, \sqrt{\gamma_{\max,d}} \, \|f\|_{H^{\alpha,\bsgamma}_\mix([0,1]^d)}
  .
\end{align*}
\end{lemma}
\begin{proof}
This follows immediately from the definitions of the two norms, \eqref{eq:norm-Halpha-mix} and~\eqref{eq:norm-Halpha-bsgamma-mix}, as well as combining this with Lemma~\ref{lem:H-alpha-mix-C-embedding}.
\end{proof}

\paragraph{Weighted anchored and unanchored Sobolev spaces of dominating mixed smoothness.}
We now build multivariate versions of the equivalent univariate norms in Section~\ref{sec:non-periodic-equivalent-norms-d=1}.
Based on the univariate anchored norm~\eqref{eq:H-alpha-a-d=1} we define the \emph{weighted anchored mixed Sobolev space} with anchor $\bsa \in [0,1]^d$, by finiteness of its norm, with the norm defined by, $\alpha \in \N_0$,
\begin{align}\label{eq:norm-H-alpha-bsgamma-a}
  \boxed{
  \|f\|_{H^{\alpha,\bsgamma,\bsa}_\mix([0,1]^d)}^2
  :=
  \sum_{\substack{\bstau \in \N_0^d \\ \|\bstau\|_\infty \le \alpha \\ \setv := \{ j : \tau_j \ne \alpha \}}} \gamma_{\supp(\bstau)}^{-1} \, \|D^\bstau f(\cdot_{-\setv}; \bsa_\setv)\|_{L_2([0,1]^d)}^2
  .
  }
\end{align}
Here the notation $f(\cdot_{-\setv}; \bsa_\setv)$ for a set $\setv \subseteq \{1,\ldots,d\}$ means to fix (i.e., ``anchor'') all components in $\setv$ to their respective anchor value (remember that $D^\bstau f(\bsx) = (D^\bstau f)(\bsx)$).
So this means $\bsx \mapsto f(\bsx_{-\setv}; \bsa_\setv) = f(\bsy)$ with $y_j = x_j$ for $j \in -\setv$ and $y_j = a_j$ for $j \in \setv$.
This results in $\bsx \mapsto f(\bsx_{-\setv}; \bsa_\setv)$ being a function of $\bsx_{-\setv} = (x_j)_{j\in -\setv}$, i.e., of those $x_j$ with $j$ in the complement set $-\setv := \{1,\ldots,d\} \setminus \setv$.

For a set $\setv \subseteq \{1,\ldots,d\}$ we define
\begin{align*}
  I_\setv f
  &:=
  \int_{[0,1]^{|\setv|}} f(\bsx_\setv; \bsx_{-\setv}) \rd\bsx_\setv
  ,
\end{align*}
and let $I_\emptyset$ be the identity operator.
Then, likewise, based on the univariate unanchored norm~\eqref{eq:H-alpha-u-d=1}, we obtain the \emph{weighted unanchored mixed Sobolev space} with norm, $\alpha \in \N_0$,
\begin{align}\label{eq:norm-H-alpha-bsgamma-u}
  \boxed{
  \|f\|_{H^{\alpha,\bsgamma,\rmu}_\mix([0,1]^d)}^2
  :=
  \sum_{\substack{\bstau \in \N_0^d \\ \|\bstau\|_\infty \le \alpha \\ \setv := \{ j : \tau_j \ne \alpha \}}} \gamma_{\supp(\bstau)}^{-1} \, \|I_\setv D^\bstau f\|_{L_2([0,1]^d)}^2
  .
  }
\end{align}
Written like this, this norm is used in, e.g.,~\cite{NN2021lognormal,NS2023}.
For $\alpha \in \N$, this norm also appears in the literature in a slightly more complicated written form, e.g., in~\cite{DKS2013,GSY2017,GSY2019,DKP2022}, as
\begin{align*}
  \numberthis{eq:H-alpha-u-d-alternative-indexing}
  &\|f\|_{H^{\alpha,\bsgamma,\rmu}_\mix([0,1]^d)}^2
  \\
  &\quad=
  \sum_{\setu \subseteq \{1,\ldots,d\}} \gamma_\setu^{-1}
  \sum_{\substack{\setv \subseteq \setu \\ \bstau_{\setu \setminus \setv} \in \{1,\ldots,\alpha-1\}^{|\setu \setminus \setv|}}}
  \journalbreak\journalonly{\qquad\qquad}
  \int_{[0,1]^{|\setv|}}
  \left|
  \int_{[0,1]^{d-|\setv|}}
  D^{(\alpha_\setv,\bstau_{\setu \setminus \setv},\bszero_{-\setu})} f(\bsx) \rd\bsx_{-\setv}
  \right|^2
  \rd\bsx_\setv
  ,
\end{align*}
where $\bstau = (\bsalpha_\setv,\bstau_{\setu \setminus \setv},\bszero_{-\setu}) \in \N_0^d$ denotes the multi-index with $\tau_j = \alpha$ for $j \in \setv$, $\tau_j \in \{1,\ldots,\alpha-1\}$ on $\setu \setminus \setv$, and $\tau_j = 0$ for $j \notin \setu$.

Note again that the $\bsa$ for the anchored norm is a mathematical variable, while $\rmu$ for the unanchored norm is a normal letter to denote ``unanchored''.
We remark that for both~\eqref{eq:norm-H-alpha-bsgamma-a} and~\eqref{eq:norm-H-alpha-bsgamma-u} the integrals in the $L_2$-norms over $[0,1]^d$ could be reduced to integrals over $[0,1]^{d - |\setv|}$ since the remaining function is constant in the components in~$\setv$.

The functional \emph{ANOVA decomposition} of a function is orthogonal with respect to the inner product associated with the unanchored norm~\eqref{eq:norm-H-alpha-bsgamma-u}, which is why this norm is also called an \emph{ANOVA norm} to say that it has the ANOVA property (although some references write ``the'' ANOVA norm, different norms can have the ANOVA orthogonality property too).
Likewise, for $\alpha \ge 1$, the \emph{anchored decomposition} is orthogonal with respect to the inner product associated with the anchored norm~\eqref{eq:norm-H-alpha-bsgamma-a}, see~\cite{KSWW2010b}.
The functional ANOVA decomposition requires the computation of multiple integrals up to dimension~$d$, which makes it generally impractical.
On the other hand, the anchored components are explicitly computable from function values and form the basis of the \emph{multivariate decomposition method} (MDM) for finite- and infinite-dimensional integration~\cite{KSWW2010,KSWW2010b,DG2014,NN2021uniform,NN2021lognormal}.

For fixed $d$ and all weights $\gamma_\setu > 0$ all these spaces are again continuously embedded into each other and so they are the same vector space of functions.
We will state this in the next lemma.
Since any of the unweighted spaces could be seen as the same space with all $\gamma_\setu \equiv 1$, e.g., $H^\alpha_\mix([0,1]^d) = H^{\alpha,\bsone}_\mix([0,1]^d)$ (as spaces and norms), it follows that, also for two different families of weights $\bsgamma$ and $\bsgamma'$, we have the same vector space and we can express the norm equivalences in terms of the maximum and minimum ratio of these weights.
In the lemma we take the most straightforward approach to obtain embedding constants.
This is not a good idea and we get embedding constants that depend exponentially on~$d$.
One can get better results by assuming some structure on the weights.
We will introduce several common types of weights a little further below and for ``product weights'' and $\alpha = 1$, we will give sharp embedding constants below in Lemma~\ref{lem:H-1-mix-product-weight-embeddings} that are independent of~$d$ under certain conditions on the weights.
So the following is an example of what not to do if one wants good constants!

\begin{lemma}[Straightforward embeddings (with (bad) dependence on~$d$)]\label{lem:H-alpha-mix-weighted-equivalent-norms}
Let $d \in \N$, $\alpha \in \N_0$, $\bsa \in [0,1]^d$ and $\bsgamma$ be a family of weights with $\gamma_\setu > 0$ for all~$\setu$.
With $\gamma_{\min,d}$ and $\gamma_{\max,d}$ as defined in~\eqref{eq:def-gamma-max-min-d}, set
\begin{align}\label{eq:def-rho-d}
  \rho_d(\bsgamma)
  &:=
  \frac{\sqrt{\gamma_{\max,d}}}{\sqrt{\gamma_{\min,d}}}
  .
\end{align}
For $\alpha = 0$ we have
\begin{align*}
  \|f\|_{H^{0,\bsgamma}_\mix([0,1]^d)}
  =
  \|f\|_{H^{0,\bsgamma,\rmu}_\mix([0,1]^d)}
  =
  \|f\|_{H^{0,\bsgamma,\bsa}_\mix([0,1]^d)}
  =
  \gamma_\emptyset^{-1/2} \, \|f\|_{L_2([0,1]^d)}
  .
\end{align*}
For $\alpha \in \N$ we have the continuous embeddings $H^{\alpha,\bsgamma,\rmu}_\mix([0,1]^d) \hookrightleftarrows H^{\alpha,\bsgamma}_\mix([0,1]^d) \hookrightleftarrows H^{\alpha,\bsgamma,\bsa}_\mix([0,1]^d)$ with
\begin{align}
  \label{eq:H-alpha-gamma-mix-equivalence-unanchored}
  \journalonly{&}
  \rho_d(\bsgamma)^{-1} \, (1-\pi^{-2})^{d/2} \, \|f\|_{H^{\alpha,\bsgamma}_\mix([0,1]^d)}
  \articleonly{&}\le
  \|f\|_{H^{\alpha,\bsgamma,\rmu}_\mix([0,1]^d)}
  \le
  \|f\|_{H^{\alpha,\bsgamma}_\mix([0,1]^d)}
  ,
  \\
  \label{eq:H-alpha-gamma-mix-equivalence-anchored}
  \journalonly{&\begin{aligned}[t]}
  \rho_d(\bsgamma)^{-1} \, (2\alpha)^{-d/2} \, \|f\|_{H^{\alpha,\bsgamma}_\mix([0,1]^d)}
  \articleonly{&}\le \journalonly{{}&}
  \|f\|_{H^{\alpha,\bsgamma,\bsa}_\mix([0,1]^d)}
  \journalonly{\\*&\qquad\qquad}
  \le
  \rho_d(\bsgamma) \, 3^{d/2} \, \|f\|_{H^{\alpha,\bsgamma}_\mix([0,1]^d)}
  ,
  \journalonly{\end{aligned}}
\end{align}
and the continuous embedding $H^{\alpha,\bsgamma,\rmu}_\mix([0,1]^d) \hookrightleftarrows H^{\alpha,\bsgamma,\bsa}_\mix([0,1]^d)$ with
\begin{align*}
  \numberthis{eq:H-alpha-gamma-mix-equivalence-unanchored-anchored}
  \rho_d(\bsgamma)^{-1} \, (2\alpha)^{-d/2} \, \|f\|_{H^{\alpha,\bsgamma,\rmu}_\mix([0,1]^d)}
  \le {}&
  \|f\|_{H^{\alpha,\bsgamma,\bsa}_\mix([0,1]^d)}
  \journalonly{\\* &}
  \le
  \rho_d(\bsgamma) \, \left(\frac{3}{1-\pi^{-2}}\right)^{d/2} \, \|f\|_{H^{\alpha,\bsgamma,\rmu}_\mix([0,1]^d)}
  ,
\end{align*}
and the continuous embeddings $H^{\alpha,\bsgamma,\rmu}_\mix([0,1]^d) \hookrightarrow C([0,1]^d)$ and $H^{\alpha,\bsgamma,\bsa}_\mix([0,1]^d) \hookrightarrow C([0,1]^d)$ with
\begin{align*}
  \|f\|_{C([0,1]^d)}
  &\le
  \sqrt{\gamma_{\max,d}} \, \left(\frac{2}{1-\pi^{-2}}\right)^{d/2} \, \|f\|_{H^{\alpha,\bsgamma,\rmu}_\mix([0,1]^d)}
  ,
  \\
  \|f\|_{C([0,1]^d)}
  &\le
  \sqrt{\gamma_{\max,d}} \, (4\alpha)^{d/2} \, \|f\|_{H^{\alpha,\bsgamma,\bsa}_\mix([0,1]^d)}
  .
\end{align*}
For $\bsgamma$ and $\bsgamma'$ two families of weights with $\gamma_\setu, \gamma'_\setu > 0$ for all~$\setu$, define
\begin{align*}
  c_d(\bsgamma,\bsgamma')
  &:=
  \max_{\setu \subseteq \{1,\ldots,d\}} \frac{\sqrt{\gamma_\setu}}{\sqrt{\gamma'_\setu}}
  ,
\end{align*}
then we have the continuous embeddings
\begin{align*}
  H^{\alpha,\bsgamma'}_\mix([0,1]^d)
  &\hookrightleftarrows
  H^{\alpha,\bsgamma}_\mix([0,1]^d)
  ,
  \\
  H^{\alpha,\bsgamma',\rmu}_\mix([0,1]^d)
  &\hookrightleftarrows
  H^{\alpha,\bsgamma,\rmu}_\mix([0,1]^d)
  ,
  \\
  H^{\alpha,\bsgamma',\bsa}_\mix([0,1]^d)
  &\hookrightleftarrows
  H^{\alpha,\bsgamma,\bsa}_\mix([0,1]^d)
  ,
\end{align*}
with
\begin{align*}
  c_d(\bsgamma',\bsgamma)^{-1} \, \|f\|_{H^{\alpha,\bsgamma}_\mix([0,1]^d)}
  &\le
  \|f\|_{H^{\alpha,\bsgamma'}_\mix([0,1]^d)}
  \le
  c_d(\bsgamma,\bsgamma') \, \|f\|_{H^{\alpha,\bsgamma}_\mix([0,1]^d)}
  ,
  \\
  c_d(\bsgamma',\bsgamma)^{-1} \, \|f\|_{H^{\alpha,\bsgamma,\rmu}_\mix([0,1]^d)}
  &\le
  \|f\|_{H^{\alpha,\bsgamma',\rmu}_\mix([0,1]^d)}
  \le
  c_d(\bsgamma,\bsgamma') \, \|f\|_{H^{\alpha,\bsgamma,\rmu}_\mix([0,1]^d)}
  ,
  \\
  c_d(\bsgamma',\bsgamma)^{-1} \, \|f\|_{H^{\alpha,\bsgamma,\bsa}_\mix([0,1]^d)}
  &\le
  \|f\|_{H^{\alpha,\bsgamma',\bsa}_\mix([0,1]^d)}
  \le
  c_d(\bsgamma,\bsgamma') \, \|f\|_{H^{\alpha,\bsgamma,\bsa}_\mix([0,1]^d)}
  .
\end{align*}
\end{lemma}
\begin{proof}
The case $\alpha = 0$ is clear, hence we assume $\alpha \in \N$ for the rest of the proof.
Clearly for any of the spaces we have an inequality as in~\eqref{eq:H-alpha-mix-weighted-equivalence} from Lemma~\ref{lem:H-alpha-C-embedding-d} to go from a weighted version to a non-weighted version.
Tensorising the one-dimensional inequalities in Lemmas~\ref{lem:H-alpha-u-norm-equivalence-d=1} and~\ref{lem:H-alpha-a-norm-equivalence-d=1} to go from unanchored or anchored to just the plain mixed form gives
\begin{align*}
  \|f\|_{H^{\alpha,\rmu}_\mix([0,1]^d)}
  &\le
  \|f\|_{H^\alpha_\mix([0,1]^d)}
  \le
  (1-\pi^{-2})^{-d/2} \, \|f\|_{H^{\alpha,\rmu}_\mix([0,1]^d)}
  ,
  \\
  3^{-d/2} \, \|f\|_{H^{\alpha,\bsa}_\mix([0,1]^d)}
  &\le
  \|f\|_{H^\alpha_\mix([0,1]^d)}
  \le
  (2\alpha)^{d/2} \, \|f\|_{H^{\alpha,\bsa}_\mix([0,1]^d)}
  .
\end{align*}
Now combining these proves~\eqref{eq:H-alpha-gamma-mix-equivalence-unanchored} and~\eqref{eq:H-alpha-gamma-mix-equivalence-anchored}, the upper bound in~\eqref{eq:H-alpha-gamma-mix-equivalence-unanchored} follows directly from using Cauchy--Schwarz, for all $\setv \subseteq \{1,\ldots,d\}$ and all $\bstau \in \N_0^d$,
\begin{align*}
  \|I_\setv D^\bstau f\|_{L_2([0,1]^d)}^2
  &=
  \int_{[0,1]^{d-|\setv|}} \left|
  \int_{[0,1]^{|\setv|}} D^\bstau f(\bsx_{-\setv}; \bsx_\setv) \rd\bsx_\setv
  \right|^2 \rd\bsx_{-\setv}
  \\
  &\le
  \int_{[0,1]^{d-|\setv|}}
  \int_{[0,1]^{|\setv|}} |D^\bstau f(\bsx_{-\setv}; \bsx_\setv)|^2 \rd\bsx_\setv
  \rd\bsx_{-\setv}
  \journalbreak=
  \|D^\bstau f\|_{L_2([0,1]^d)}^2
  ,
\end{align*}
and hence $\|\cdot\|_{H^{\alpha,\bsgamma,\rmu}_\mix([0,1]^d)} \le \|\cdot\|_{H^{\alpha,\bsgamma}_\mix([0,1]^d)}$, independent of the weights~$\bsgamma$.
Combining the two unweighted estimates first and then adding in the weights gives~\eqref{eq:H-alpha-gamma-mix-equivalence-unanchored-anchored}.
Combining with the embedding for $H^\alpha_\mix([0,1]^d)$ in Lemma~\ref{lem:H-alpha-mix-C-embedding} gives the embeddings into $C([0,1]^d)$.
The comparison between different weight families follows by comparing the norms term by term.
\end{proof}

\paragraph{Three common types of weights.}
The weight families that have been used in the literature often are not of such ``general'' form as above but have structure.
We give three common forms.
The original introduction of these weights considered \emph{product weights} \cite{SW98,SW2001,SW2002,Hic98a}, which take the form
\begin{align}\label{eq:ptw}
  \gamma_\setu
  =
  \prod_{j \in \setu} \gamma_j
  ,
\end{align}
for a positive sequence $\{ \gamma_j \}_{j \ge 1}$ and by definition of the empty product we have $\gamma_\emptyset = 1$.
These are probably the most popular kind of weights.

Next we have \emph{order-dependent weights} which take the form
\begin{align}\label{eq:odw}
  \gamma_\setu = \Gamma_{|\setu|}
  ,
\end{align}
for a positive sequence $\{ \Gamma_\ell \}_{\ell \ge 0}$ and we set $\gamma_\emptyset = \Gamma_0 = 1$, see, e.g.,~\cite{DSWW2006,DKS2013,Nuy2014}.
These are not used too often.

A combination of the two previous weights is called \emph{product and order-dependent weights} (POD) and takes the form
\begin{align}\label{eq:podw}
  \gamma_\setu = \Gamma_{|\setu|} \, \prod_{j \in \setu} \gamma_j
  ,
\end{align}
for two positive sequences $\{\Gamma_\ell\}_{\ell \ge 0}$ and $\{ \gamma_j \}_{j \ge 1}$, see, e.g.,~\cite{KSS2012,DKGNS2014,KN2016,DKS2013,Nuy2014}.
To have $\gamma_\emptyset = 1$ we need to set $\Gamma_0 = 1$ here as well.
Note that order-dependent weights are a special case of POD weights with $\gamma_j = 1$ for all $j \ge 1$.
The POD weights warrant special attention since they arise naturally when analysing the regularity of the solution of a PDE with random-field coefficients.
E.g., if we represent the solution of the PDE in terms of an expansion of the random field with coefficients $\bsy$, which we would like to sample from a QMC point set, then bounds on the norm of the mixed first order partial derivatives w.r.t.\ the random parameters typically take the form
\begin{align}\label{eq:PDE-bound}
  \left\|\frac{\partial^{|\setu|} u(\bsy)}{\partial \bsy_\setu}\right\|_V
  &\le
  C \, |\setu|! \, \prod_{j \in \setu} \gamma_j
  ,
\end{align}
where $C > 0$ is independent of $d$, $\setu$ and $\bsy$ and $\|\cdot\|_V$ is the norm on the solution space of the PDE in the physical variable~$\bsx$, see, e.g.,~\cite{CDS2010,CDS2011,KSS2012,DKGNS2014,KN2016}.
For more QMC and weighted function spaces, including higher order derivatives, see the overview in \cite{KN2016} which is accompanied by the QMC4PDE software package~\cite{NuyQMC4PDE}.
The bound~\eqref{eq:PDE-bound} motivates POD weights~\eqref{eq:podw} with $\Gamma_\ell = \ell!$.
Higher-order regularity of the solution of a PDE is normally analysed using \emph{smoothness driven product and order-dependent weights} (SPOD), see, e.g., \cite{DKGNS2014,KN2016}, but we will not discuss those here.

It will turn out that the sum of the weights over all subsets of $\{1,\ldots,d\}$ will become important in analysing worst-case errors and other quantities.
We hence formulate the following lemma for product weights.
This is used all over the literature inside many proofs.

\begin{lemma}[Sum of product weights]\label{lem:product-weight-sum}
Let $\bsgamma$ be product weights as in~\eqref{eq:ptw}.
For an arbitrary constant $\kappa > 0$ we have
\begin{align*}
  \sum_{\setu \subseteq \{1,\ldots,d\}} \gamma_\setu \, \kappa^{|\setu|}
  &=
  \sum_{\setu \subseteq \{1,\ldots,d\}} \Big( \prod_{j \notin \setu} 1 \Big) \, \Big( \prod_{j \in \setu} \gamma_j \, \kappa \Big)
  \journalbreak=
  \prod_{j=1}^d (1 + \gamma_j \, \kappa)
  \le
  \exp\left(\kappa \sum_{j=1}^d \gamma_j\right)
  .
\end{align*}
Hence, if $\sum_{j \ge 1} \gamma_j < \infty$ then
\begin{align*}
  \lim_{d \to \infty} \sum_{\setu \subseteq \{1,\ldots,d\}} \gamma_\setu \, \kappa^{|\setu|}
  &\le
  \exp\left(\kappa \sum_{j \ge 1} \gamma_j\right)
  <
  \infty
  .
\end{align*}
\end{lemma}
\begin{proof}
The identity follows by expanding the product, and the inequality follows from $1+x \le \exp(x)$ for $x \ge 0$ which can easily be seen by the Taylor expansion of $\exp(x) = \sum_{k\ge0} x^k / k!$.
\end{proof}

In Lemma~\ref{lem:H-1-mix-product-weight-embeddings} below we will give embeddings for product weights and we will not go further than product weights in these lecture notes.
Although theory can, has, and is, developed directly for POD and SPOD weights, here we only record straightforward bounds for POD weights obtained either by comparison with product weights or by directly estimating sums over POD weights.
This is shown in the next lemma.
Related estimates for sums over weighted subsets can be found, e.g., in~\cite{KSS2012,DG2014,GMR2014,KN2016,KNPSW2017,GKNW2018,NN2021uniform,NN2021lognormal,Pan2026}.

\begin{lemma}[Bounds for POD weights]\label{lem:POD-product-weight-bound}
Let $\bsgamma$ be POD weights as in~\eqref{eq:podw}, $\gamma_\setu = \Gamma_{|\setu|} \prod_{j\in\setu} \gamma_j$.

(i) If $\Gamma_\ell$ grows at most exponentially, i.e., there is a constant
\begin{align*}
  C_\Gamma
  &:=
  \sup_{\ell \ge 1} \Gamma_\ell^{1/\ell}
  <
  \infty
  ,
\end{align*}
then, uniformly in~$d$, we have the product-weight bound
\begin{align*}
  \gamma_\setu
  &\le
  \prod_{j \in \setu} (C_\Gamma \, \gamma_j)
  \qquad
  \text{for all}\quad \setu \subseteq \{1,\ldots,d\}
  ,
\end{align*}
and $\lim_{d\to\infty} \sum_{\setu \subseteq \{1,\ldots,d\}} \gamma_\setu < \infty$ if $\sum_{j \ge 1} \gamma_j < \infty$.

(ii) We have, the sum of POD weights,
\begin{align*}
  \sum_{\setu \subseteq \{1,\ldots,d\}} \gamma_\setu
  =
  \sum_{\setu \subseteq \{1,\ldots,d\}} \Gamma_{|\setu|} \, \prod_{j \in \setu} \gamma_j
  &\le
  \sum_{\ell=0}^d \frac{\Gamma_\ell}{\ell!} \left(\sum_{j=1}^d \gamma_j\right)^{\!\!\ell}
  .
\end{align*}

(iii) Suppose $\Gamma_\ell \le B^\ell \, \ell!$ for some $B > 0$ and $\sum_{j \ge 1} \gamma_j < \infty$, fix $m \in \N$ such that
\begin{align*}
  T_m
  &:=
  \sum_{j > m} \gamma_j
  <
  \frac{1}{B}
  ,
\end{align*}
then
\begin{align*}
  \lim_{d\to\infty} \sum_{\setu \subseteq \{1,\ldots,d\}} \gamma_\setu
  &\le
  \sum_{\setv \subseteq \{1,\ldots,m\}}
  \frac{B^{|\setv|} \, |\setv|!}{(1 - B \, T_m)^{|\setv|+1}} \, \prod_{j \in \setv} \gamma_j
  <
  \infty
  .
\end{align*}
\end{lemma}
\begin{proof}
The first claim follows immediately from $\Gamma_{|\setu|} \le C_\Gamma^{|\setu|}$ and Lemma~\ref{lem:product-weight-sum}.
For the second claim we have
\begin{align*}
  \sum_{\setu \subseteq \{1,\ldots,d\}} \Gamma_{|\setu|} \, \prod_{j\in\setu} \gamma_j
  =
  \sum_{\ell=0}^d
  \Gamma_\ell
  \sum_{\substack{\setu \subseteq \{1,\ldots,d\} \\ |\setu|=\ell}} \prod_{j \in \setu} \gamma_j
  \le
  \sum_{\ell=0}^d
  \frac{\Gamma_\ell}{\ell!}
  \left(\sum_{j=1}^d \gamma_j\right)^{\!\!\ell}
  ,
\end{align*}
since, for $\ell \in \N$,
\begin{align*}
  \left(\sum_{j=1}^d \gamma_j\right)^{\!\!\ell}
  &=
  \sum_{\substack{a_1,\ldots,a_d \ge 0 \\ a_1 + \cdots + a_d = \ell}}
  \binom{\ell}{a_1,\ldots,a_d}
  \prod_{j=1}^d \gamma_j^{a_j}
  \journalbreak\ge
  \sum_{\substack{a_1,\ldots,a_d \in \{0, 1\} \\ a_1 + \cdots + a_d = \ell}}
  \ell! \, \prod_{j=1}^d \gamma_j^{a_j}
  \journalbreak=
  \ell! \, \sum_{\substack{\setu \subseteq \{1,\ldots,d\} \\ |\setu|=\ell}} \prod_{j \in \setu} \gamma_j
  .
\end{align*}
For the final claim, take $d > m$ and write $\setu = \setv \cup \setw$ with $\setv \subseteq \{1,\ldots,m\}$ and $\setw \subseteq \{m+1,\ldots,d\}$.
Grouping the inner sum according to $\ell = |\setw|$ and using the second claim gives
\begin{align*}
  \sum_{\setu \subseteq \{1,\ldots,d\}} \Gamma_{|\setu|} \, \prod_{j \in \setu} \gamma_j
  &\le
  \sum_{\setv \subseteq \{1,\ldots,m\}}
  B^{|\setv|} \, \Big[ \prod_{j \in \setv} \gamma_j \Big]
  \sum_{\ell=0}^{d-m} B^\ell \, (|\setv|+\ell)!
  \journalonly{\mkern-10mu}\sum_{\substack{\setw \subseteq \{m+1,\ldots,d\} \\ |\setw|=\ell}} \prod_{j \in \setw} \gamma_j
  \\
  &\le
  \sum_{\setv \subseteq \{1,\ldots,m\}}
  B^{|\setv|} \, \Big[ \prod_{j \in \setv} \gamma_j \Big]
  \sum_{\ell=0}^{d-m} \frac{(|\setv|+\ell)!}{\ell!}
  \left(B \sum_{j=m+1}^d \gamma_j\right)^{\!\!\ell}
  \\
  &\le
  \sum_{\setv \subseteq \{1,\ldots,m\}}
  B^{|\setv|} \, \Big[ \prod_{j \in \setv} \gamma_j \Big]
  \sum_{\ell \ge 0} \frac{(|\setv|+\ell)!}{\ell!} \, (B \, T_m)^\ell
  .
\end{align*}
The sum over $\ell$ in the last inequality converges by the ratio test since we have $B \, T_m < 1$.
In fact
\begin{align*}
  \sum_{\ell \ge 0} \frac{(|\setv|+\ell)!}{\ell!} \, (B \, T_m)^\ell
  &=
  |\setv|! \, \sum_{\ell \ge 0} \binom{|\setv| + \ell}{|\setv|} \, (B \, T_m)^\ell
  =
  \frac{|\setv|!}{(1 - B \, T_m)^{|\setv|+1}}
  ,
\end{align*}
where we used $(1 + x)^{-s} = \sum_{k \ge 0} \binom{-s}{k} \, x^k = \sum_{k\ge0} \binom{s+k-1}{k} \, (-1)^k \, x^k$ for $|x| < 1$.
This finishes the proof.
\end{proof}

We now specify the previous embeddings from Lemmas~\ref{lem:H-alpha-C-embedding-d} and~\ref{lem:H-alpha-mix-weighted-equivalent-norms}, which were done for general weights, to product weights.
We also fix $\alpha = 1$ to simplify things and get nice expressions.
(For higher-order smoothness $\alpha > 1$ we would probably consider SPOD weights anyway.)
The embedding constants obtained here can now be independent of~$d$ under certain conditions on the weights.
The proof reuses the techniques that were used for the univariate spaces in Section~\ref{sec:non-periodic-equivalent-norms-d=1} with the addition of product weights.

\begin{lemma}[Embeddings specifically for product weights with $\alpha = 1$]\label{lem:H-1-mix-product-weight-embeddings}
Let $d \in \N$, $\alpha = 1$, $\bsa \in [0,1]^d$ and let $\bsgamma$ be product weights $\gamma_\setu = \prod_{j \in \setu} \gamma_j$ for given $\gamma_j > 0$.
Then we have the continuous embedding $H^{1,\bsgamma,\rmu}_\mix([0,1]^d) \hookrightleftarrows H^{1,\bsgamma}_\mix([0,1]^d)$ with
\begin{align}\label{eq:H-1-mix-product-weight-standard-unanchored}
  \|f\|_{H^{1,\bsgamma,\rmu}_\mix([0,1]^d)}
  &\le
  \|f\|_{H^{1,\bsgamma}_\mix([0,1]^d)}
  \le
  \prod_{j=1}^d \left(1 + \frac{\gamma_j}{\pi^2}\right)^{1/2} \, \|f\|_{H^{1,\bsgamma,\rmu}_\mix([0,1]^d)}
  ,
\end{align}
and $H^{1,\bsgamma,\rmu}_\mix([0,1]^d) \hookrightleftarrows H^{1,\bsgamma,\bsa}_\mix([0,1]^d)$ with
\begin{align}\label{eq:H-1-mix-product-weight-anchored-unanchored}
  \kappa_d(\bsgamma,\bsa)^{-1} \, \|f\|_{H^{1,\bsgamma,\rmu}_\mix([0,1]^d)}
  &\le
  \|f\|_{H^{1,\bsgamma,\bsa}_\mix([0,1]^d)}
  \le
  \kappa_d(\bsgamma,\bsa) \, \|f\|_{H^{1,\bsgamma,\rmu}_\mix([0,1]^d)}
  ,
\end{align}
where
\begin{align}\label{eq:def-kappa-d}
  \kappa_d(\bsgamma,\bsa)
  &:=
  \prod_{j=1}^d \frac{\sqrt{4 + \gamma_j \, m(a_j)} + \sqrt{\gamma_j \, m(a_j)}}{2}
  ,
  &
  m(a)
  &:=
  \frac{a^3 + (1-a)^3}{3}
  .
\end{align}
The function $m(a)$ attains its minimum $1/12$ at $a = 1/2$ and its maximum $1/3$ at $a = 0$ or $a = 1$.
The constants in~\eqref{eq:H-1-mix-product-weight-standard-unanchored} and~\eqref{eq:H-1-mix-product-weight-anchored-unanchored} are sharp.
As $d \to \infty$, the constants in~\eqref{eq:H-1-mix-product-weight-standard-unanchored} are uniformly bounded if and only if $\sum_{j \ge 1} \gamma_j < \infty$, while the constants in~\eqref{eq:H-1-mix-product-weight-anchored-unanchored} are uniformly bounded if and only if $\sum_{j \ge 1} \sqrt{\gamma_j} < \infty$.

We have the continuous embeddings into $C([0,1]^d)$ with
\begin{align}\label{eq:H-1-mix-product-weight-C-embeddings}
  \|f\|_{C([0,1]^d)}
  &\le
  \prod_{j=1}^d \left(\sqrt{\gamma_j} \, \coth(\sqrt{\gamma_j})\right)^{1/2} \, \|f\|_{H^{1,\bsgamma}_\mix([0,1]^d)}
  \\
  \|f\|_{C([0,1]^d)}
  &\le
  \prod_{j=1}^d \left(1 + \frac{\gamma_j}{3}\right)^{1/2} \, \|f\|_{H^{1,\bsgamma,\rmu}_\mix([0,1]^d)}
  \label{eq:H-1-mix-product-weight-C-embedding-unanchored}
  \\
  \|f\|_{C([0,1]^d)}
  &\le
  \prod_{j=1}^d \left(1 + \gamma_j \, \max\{a_j,1-a_j\}\right)^{1/2} \, \|f\|_{H^{1,\bsgamma,\bsa}_\mix([0,1]^d)}
  \label{eq:H-1-mix-product-weight-C-embedding-anchored}
  .
\end{align}
Each of these embedding constants is sharp, and they are uniformly bounded as $d \to \infty$ if and only if $\sum_{j \ge 1} \gamma_j < \infty$.
\end{lemma}
\begin{proof}
We repurpose the proofs for the univariate case with the addition of the corresponding weight $\gamma_j$ for each dimension.
The multivariate constants then come by simply multiplying all weighted univariate constants.

Our first one-dimensional estimate extends the argument in the proof of Lemma~\ref{lem:H-alpha-u-norm-equivalence-d=1} to the weighted norm.
We use the Poincaré--Wirtinger inequality, with its optimal constant,
\begin{align*}
  \journalonly{\MoveEqLeft[1]} \left|\int_0^1 f(x) \rd{x}\right|^2 + \gamma^{-1} \, \|f'\|_{L_2([0,1])}^2
  \journalonly{\\*} &\le
  \|f\|_{L_2([0,1])}^2 + \gamma^{-1} \, \|f'\|_{L_2([0,1])}^2
  \\
  &=
  \left|\int_0^1 f(x) \rd{x}\right|^2 + \gamma^{-1} \, \|f'\|_{L_2([0,1])}^2 + \left\|f-\int_0^1 f(x) \rd{x}\right\|_{L_2([0,1])}^2
  \\
  &\le
  \left|\int_0^1 f(x) \rd{x}\right|^2 + \gamma^{-1} \, \|f'\|_{L_2([0,1])}^2 + \frac1{\pi^2} \, \|f'\|_{L_2([0,1])}^2
  \\
  &\le
  \left(1 + \frac{\gamma}{\pi^2}\right) \left(\left|\int_0^1 f(x) \rd{x}\right|^2 + \gamma^{-1} \, \|f'\|_{L_2([0,1])}^2\right)
  .
\end{align*}
Tensorising this constant proves~\eqref{eq:H-1-mix-product-weight-standard-unanchored}.

To prove the norm inequalities between the anchored and unanchored norms we start with $f(x) = f(a) + \int_a^x f'(t) \rd{t}$ and then integrate $f(x) - f(a)$ over $x$ to obtain
\begin{align*}
  \journalonly{\MoveEqLeft[1]} \int_0^1 (f(x) - f(a)) \rd{x}
  =
  \int_0^1 f(x) \rd{x} - f(a)
  \journalonly{\\* } &=
  - \int_0^1 \int_x^a f'(t) \rd{t} \rd{x}
  \\
  &=
  - \int_0^a \int_x^a f'(t) \rd{t} \rd{x} + \int_a^1 \int_a^x f'(t) \rd{t} \rd{x}
  \\
  &=
  - \int_0^a f'(t) \int_0^t \rd{x} \rd{t} + \int_a^1 f'(t) \int_t^1 \rd{x} \rd{t}
  \\
  &=
  - \int_0^a t \, f'(t) \rd{t} + \int_a^1 (1-t) \, f'(t) \rd{t}
  \numberthis{eq:H-1-mean-anchor-representation}
  .
\end{align*}
Cauchy--Schwarz then gives
\begin{align*}
  \left| \int_0^1 f(x) \rd{x} - f(a) \right|
  &\le
  \left(\int_0^a t^2 \rd{t} + \int_a^1 (1-t)^2 \rd{t}\right)^{1/2} \, \|f'\|_{L_2([0,1])}
  \journalonly{\\&}=
  \sqrt{m(a)} \, \|f'\|_{L_2([0,1])}
  ,
\end{align*}
with
\begin{align*}
  m(a)
  &=
  \int_0^a t^2 \rd{t} + \int_a^1 (1-t)^2 \rd{t}
  =
  \frac{a^3 + (1-a)^3}{3}
  .
\end{align*}
The reverse triangle inequality then gives
\begin{align}
  \left|\int_0^1 f(x) \rd{x}\right|
  &\le
  |f(a)| + \sqrt{m(a)} \, \|f'\|_{L_2([0,1])}
  \label{eq:H-1-mean-from-anchor}
  ,
  \\
  |f(a)|
  &\le
  \left|\int_0^1 f(x) \rd{x}\right| + \sqrt{m(a)} \, \|f'\|_{L_2([0,1])}
  \label{eq:H-1-anchor-from-mean}
  .
\end{align}

The following inequality for $u,v,r \ge 0$ can easily be checked by completing the square and rearranging,
\begin{align*}
  \numberthis{eq:scalar-quadratic-inequality}
  \big((u+r v)^2+v^2\big)^{1/2}
  &\le
  \frac{\sqrt{4+r^2}+r}{2} \, \big(u^2+v^2\big)^{1/2}
  .
\end{align*}
Equality holds for $v = (\sqrt{4+r^2}+r) \, u/2$.
Applying the scalar inequality first with $u = |f(a)|$, $v = \gamma^{-1/2} \, \|f'\|_{L_2([0,1])}$ and $r = \sqrt{\gamma \, m(a)}$, and then with $u = |\int_0^1 f(x) \rd{x}|$ and the same $v$ and $r$, and setting $\kappa_1 = (\sqrt{4+\gamma \, m(a)}+\sqrt{\gamma \, m(a)})/2$, we obtain from~\eqref{eq:H-1-mean-from-anchor} and~\eqref{eq:H-1-anchor-from-mean}, respectively,
\begin{align*}
  \journalonly{&}\left(\left|\int_0^1 f(x) \rd{x}\right|^2 + \gamma^{-1} \, \|f'\|_{L_2([0,1])}^2\right)^{1/2}
  \articleonly{&}\le
  \kappa_1 \,
  \left(|f(a)|^2 + \gamma^{-1} \, \|f'\|_{L_2([0,1])}^2\right)^{1/2}
  ,
  \\
  \journalonly{&}\left(|f(a)|^2 + \gamma^{-1} \, \|f'\|_{L_2([0,1])}^2\right)^{1/2}
  \articleonly{&}\le
  \kappa_1 \,
  \left(\left|\int_0^1 f(x) \rd{x}\right|^2 + \gamma^{-1} \, \|f'\|_{L_2([0,1])}^2\right)^{1/2}
  .
\end{align*}
Equality in the first and second estimates is obtained, respectively, by taking $f'(t) = \mp \sqrt{\gamma/m(a)} \, t$ for $t<a$ and $f'(t) = \pm \sqrt{\gamma/m(a)} \, (1-t)$ for $t \ge a$, and choosing $f(a) = \kappa_1^{-1}$ for the upper signs and $f(a) = \kappa_1$ for the lower signs, so both constants are sharp.
Tensorising gives~\eqref{eq:H-1-mix-product-weight-anchored-unanchored}.

Next we show weighted univariate embeddings into the continuous functions.
Every continuous representative of an element of the one-dimensional Sobolev space of smoothness~$1$ has the absolutely and uniformly convergent cosine expansion, see, e.g.,~\cite{DNP2014},
\begin{align*}
  f(x)
  &=
  \hat{f}^{\cos}_0
  +
  \sum_{k \ge 1} \hat{f}^{\cos}_k \, \sqrt2 \cos(\pi k x)
  ,\journalonly{\mkern-8mu}
  &
  \hat{f}^{\cos}_k
  &:=
  \begin{cases}
  \int_0^1 f(t) \rd{t}
  ,
  &
  \text{if } k = 0
  ,
  \\[2mm]
  \int_0^1 f(t) \, \sqrt2 \cos(\pi k t) \rd{t}
  ,
  &
  \text{if } k \in \N
  .
  \end{cases}
\end{align*}
By Parseval
\begin{align*}
  \|f\|_{L_2([0,1])}^2 + \gamma^{-1} \, \|f'\|_{L_2([0,1])}^2
  &=
  |\hat{f}^{\cos}_0|^2
  +
  \sum_{k \ge 1} \left(1 + \frac{\pi^2 k^2}{\gamma}\right) |\hat{f}^{\cos}_k|^2
  .
\end{align*}
Hence Cauchy--Schwarz gives
\begin{align*}
  |f(x)|^2
  &=
  \Big| \hat{f}^{\cos}_0 + \sum_{k \ge 1} \hat{f}^{\cos}_k \, \sqrt2 \cos(\pi k x) \Big|^2
  \\
  &\le
  \left(1 + 2 \sum_{k \ge 1} \frac{\gamma \, \cos^2(\pi k x)}{\gamma + \pi^2 k^2}\right)
  \left(|\hat{f}^{\cos}_0|^2
  +
  \sum_{k \ge 1} \left(1 + \frac{\pi^2 k^2}{\gamma}\right) |\hat{f}^{\cos}_k|^2\right)
  \\
  &=
  \left(1 + 2 \sum_{k \ge 1} \frac{\gamma \, \cos^2(\pi k x)}{\gamma + \pi^2 k^2}\right)
  \left(\|f\|_{L_2([0,1])}^2 + \gamma^{-1} \, \|f'\|_{L_2([0,1])}^2\right)
  .
\end{align*}
The factor is maximised at $x = 0$, and $x = 1$, and
\begin{align*}
  1 + 2 \sum_{k \ge 1} \frac{\gamma}{\gamma + \pi^2 k^2}
  &=
  \sqrt{\gamma} \coth(\sqrt{\gamma})
  .
\end{align*}
Alternatively this same constant can be obtained by rescaling the result for $H^1(0,s)$ from~\cite{Marti1983}, with $s=\sqrt{\gamma}$.

For the unanchored norm, rearranging~\eqref{eq:H-1-mean-anchor-representation} and setting $a = x$ gives
\begin{align}\label{eq:H-1-unanchored-point-evaluation-representation}
  f(x)
  &=
  \int_0^1 f(t) \rd{t}
  +
  \int_0^x t \, f'(t) \rd{t} - \int_x^1 (1-t) \, f'(t) \rd{t}
  .
\end{align}
With
\begin{align*}
  h(t)
  &:=
  \begin{cases}
  t,      & 0 \le t \le x, \\
  -(1-t), & x < t \le 1 ,
  \end{cases}
\end{align*}
the triangle inequality applied to~\eqref{eq:H-1-unanchored-point-evaluation-representation}, followed by Cauchy--Schwarz in $L_2([0,1])$ and then in $\R^2$, gives
\begin{align*}
  |f(x)|^2
  &\le
  \left(\left|\int_0^1 f(t) \rd{t}\right| + \left|\int_0^1 h(t) \, f'(t) \rd{t}\right|\right)^2
  \\
  &\le
  \left(\left|\int_0^1 f(t) \rd{t}\right| + \|h\|_{L_2([0,1])} \, \|f'\|_{L_2([0,1])}\right)^2
  \\
  &=
  \left(\left|\int_0^1 f(t) \rd{t}\right| \cdot 1
  + \big(\gamma^{-1/2} \, \|f'\|_{L_2([0,1])}\big) \big(\gamma^{1/2} \, \|h\|_{L_2([0,1])}\big)\right)^2
  \\
  &\le
  \left(1 + \gamma \, \|h\|_{L_2([0,1])}^2\right)
  \left(\left|\int_0^1 f(t) \rd{t}\right|^2 + \gamma^{-1} \, \|f'\|_{L_2([0,1])}^2\right)
  \\
  &=
  \big(1 + \gamma \, m(x)\big)
  \left(\left|\int_0^1 f(t) \rd{t}\right|^2 + \gamma^{-1} \, \|f'\|_{L_2([0,1])}^2\right)
  .
\end{align*}
Since the maximum of $m(x)$ is $1/3$, the maximum squared point-evaluation norm is $1 + \gamma/3$.

Finally, for the anchored norm, we have $f(x) = f(a) + \int_a^x f'(t) \rd{t}$, and the same two applications of Cauchy--Schwarz give
\begin{align*}
  |f(x)|^2
  &\le
  \left(|f(a)| + \left|\int_a^x f'(t) \rd{t}\right|\right)^2
  \\
  &\le
  \left(|f(a)| + |x-a|^{1/2} \, \|f'\|_{L_2([0,1])}\right)^2
  \\
  &=
  \left(|f(a)| \cdot 1 + \big(\gamma^{-1/2} \, \|f'\|_{L_2([0,1])}\big) \big(\gamma^{1/2} \, |x-a|^{1/2}\big)\right)^2
  \\
  &\le
  \big(1 + \gamma \, |x-a|\big)
  \left(|f(a)|^2 + \gamma^{-1} \, \|f'\|_{L_2([0,1])}^2\right)
  .
\end{align*}
The maximum factor is therefore $1+ \gamma \max\{a, 1-a\}$.

Equality can be attained in each of these Cauchy--Schwarz estimates, and tensorising the point-evaluation functionals, see~\eqref{eq:tensorisation}, proves all of the constants.

Finally, the assertions about uniform boundedness for $d \to \infty$ follow from $1/12 \le m(a) \le 1/3$, the Taylor expansion $\sqrt{\gamma} \coth(\sqrt{\gamma}) = 1 + \gamma/3 + O(\gamma^2)$ as $\gamma \to 0$, and the product estimate in Lemma~\ref{lem:product-weight-sum}, after comparing the individual factors with $1+c\gamma$ or $1+c\sqrt{\gamma}$ for small~$\gamma$.
\end{proof}

We now compare these results with some of the literature.
The anchored--unanchored equivalence~\eqref{eq:H-1-mix-product-weight-anchored-unanchored} and its summability condition agree with the Hilbert product weight specialisation of~\cite{GHHRW2017}, where the weights enter squared in the norm.
For the anchor $\bsa = \bszero$, the constant in~\eqref{eq:H-1-mix-product-weight-anchored-unanchored} appears in~\cite[Theorem~8]{KPW2016}, which likewise uses squared weights in the norm, see also~\cite[Theorem~7.35]{DKP2022}.
The standard and unanchored norms in~\eqref{eq:H-1-mix-product-weight-standard-unanchored} correspond to the standard and ANOVA norms considered in~\cite{GHHR2017}.
The unanchored and anchored point-evaluation constants in~\eqref{eq:H-1-mix-product-weight-C-embedding-unanchored} and~\eqref{eq:H-1-mix-product-weight-C-embedding-anchored} follow directly from the reproducing kernels in~\cite[Section~7.1]{DKP2022}, while the standard constant in~\eqref{eq:H-1-mix-product-weight-C-embeddings} follows by tensorising the one-dimensional result from~\cite{Marti1983} after rescaling.

\paragraph{Transplanting upper bounds and lower bounds.}
Let us now demonstrate how these embedding bounds can be used to transfer upper bounds and lower bounds on the worst-case error.
Suppose we have two spaces, $F_d$ and $G_d$, and the continuous embedding $F_d \hookrightarrow G_d$ with
\begin{align*}
  \|f\|_{G_d}
  \le
  C_d \, \|f\|_{F_d}
  .
\end{align*}
Let $\frakQ_{d,n}$ be a class of admissible algorithms for numerical integration in $d$ dimensions using at most $n$ admissible information evaluations and let $\calQ = \{Q_{d,n}\}_{d,n}$, where $Q_{d,n} \in \frakQ_{d,n}$, be a particular choice of such algorithms.
For example, $\frakQ_{d,n}$ could be the class of all deterministic algorithms using at most $n$ function values, while $\calQ$ could be a specified family of lattice rules.

We can now transplant an upper bound from $G_d$ to~$F_d$.
Suppose that the selected algorithms $Q_{d,n}$ in $\calQ$ are linear and satisfy
\begin{align*}
  e(Q_{d,n},G_d)
  \le
  B_{d,n}
  ,
\end{align*}
for some nice expression $B_{d,n}$.
Then, for every $f \in F_d$, we have $f \in G_d$, and
\begin{align*}
  |I_d(f) - Q_{d,n}(f)|
  &\le
  \|f\|_{G_d} \, e(Q_{d,n},G_d)
  \le
  C_d \, B_{d,n} \, \|f\|_{F_d}
  ,
\end{align*}
and we thus obtain
\begin{align}\label{eq:transplant-upper-bound}
  e(Q_{d,n},F_d)
  \le
  C_d \, B_{d,n}
  .
\end{align}

We can also transplant a lower bound in the opposite direction, from $F_d$ to~$G_d$.
Here we assume that $\frakQ_{d,n}$ is closed under the rescaling $Q_{d,n}(f) \mapsto C_d \, Q_{d,n}(f/C_d)$, as is the case for the classes of linear rules and of deterministic nonlinear adaptive algorithms considered below.
Suppose we have
\begin{align*}
  \inf_{Q_{d,n} \in \frakQ_{d,n}} e(Q_{d,n},F_d)
  \ge
  b_{d,n}
  ,
\end{align*}
then
\begin{align*}
  \inf_{Q_{d,n} \in \frakQ_{d,n}} e(Q_{d,n},G_d)
  &\ge
  C_d^{-1} \, \inf_{Q_{d,n} \in \frakQ_{d,n}} e(Q_{d,n},F_d)
  \ge
  C_d^{-1} \, b_{d,n}
  .
\end{align*}

\subsubsection{Information complexity and the curse of dimensionality}\label{sec:curse-definition}

We now formally define the curse of dimensionality.
For this we first need to define what information complexity is.
For general background, see~\cite{TWW88,NW2008,NW2010,NW2012}.

For a sequence of normed function spaces $\{ F_d \}_{d \ge 1}$ (e.g., $H^\alpha([0,1]^d)$, or $H^\alpha_\mix([0,1]^d)$, for increasing~$d$) define \emph{the information complexity for the absolute error criterion} in terms of the worst-case error for numerical integration, cf.~\eqref{eq:def-wce-integration}, by
\begin{align}\label{eq:information-complexity}
  n(\varepsilon, F_d)
  &:=
  \min \Big\{ n \in \N_0 : \inf_{Q_{d,n} \in \frakQ_{d,n}} \sup_{\substack{f \in F_d \\ \|f\|_{F_d} \le 1}} |I_d(f) - Q_{d,n}(f)| \le \varepsilon \Big\}
  ,
\end{align}
For the information complexity, $\frakQ_{d,n}$ denotes the class of admissible algorithms under consideration in dimension~$d$ that use at most $n$ admissible information evaluations, and $Q_{d,n}$ denotes an individual algorithm in this class.
We do not annotate the symbol $n(\varepsilon, F_d)$ with the class of algorithms and allowed information evaluations since we only use it here locally and now describe them.
For our lower bounds we consider deterministic algorithms using at most $n$ admissible information evaluations; these algorithms are allowed to be linear or nonlinear, and the information evaluations are allowed to be chosen adaptively based on the previously observed values.
The admissible information evaluations that we consider are function values or sometimes, when mentioned, even derivative values (when well defined).
The same definition of information complexity can also be used for a smaller class, such as deterministic linear algorithms using function values only, e.g., lattice rules, by restricting the infimum accordingly.
Since restricting the class of algorithms can only increase the information complexity, any lower bound proved for the larger class automatically remains valid for every such smaller class.
The most important keyword here was ``deterministic''.
(A randomised algorithm can have improved bounds for a properly defined error criterion.)

We could also define the \emph{normalised error criterion}, which instead replaces $\varepsilon$ in the preceding definition by $\varepsilon \; e(\emptyset_{F_d}, F_d)$, for $0 < \varepsilon < 1$, where $e(\emptyset_{F_d}, F_d)$ denotes the initial worst-case error without information ($\emptyset_{F_d}$~represents the optimal ``zero algorithm'', i.e., $n = 0$, for the space~$F_d$).
The unit balls of the spaces $F_d$ that we will consider are symmetric, and hence the optimal algorithm without information returns zero, cf.~the radius of information in Section~\ref{sec:lower-bounds} and in particular~\eqref{eq:triangle-fp-fpp}.
For $F_d$ any of the spaces we consider, except the anchored space $H^{\alpha,\bsgamma,\bsa}_\mix([0,1]^d)$, and with $\gamma_\emptyset = 1$, we have
\begin{align*}
  e(\emptyset_{F_d}, F_d)
  &=
  \sup_{\substack{f \in F_d \\ \|f\|_{F_d} \le 1}} |I_d(f)|
  =
  1
  ,
\end{align*}
since $|I_d(f)| \le \|f\|_{F_d}$ and the constant function $f = 1$ has $F_d$-norm and integral equal to~$1$.
For the norms containing the $L_2$-norm the inequality follows from $\|f\|_{L_2([0,1]^d)} \le \|f\|_{F_d}$, while for $H^{\alpha,\bsgamma,\rmu}_\mix([0,1]^d)$ the integral is directly one of the terms in the norm.
Thus the absolute and normalised error criteria coincide for these spaces.
If $\gamma_\emptyset \ne 1$, then the initial error of $H^{\alpha,\bsgamma}_\mix([0,1]^d)$ and $H^{\alpha,\bsgamma,\rmu}_\mix([0,1]^d)$ equals $\sqrt{\gamma_\emptyset}$ instead.
For the anchored space $H^{\alpha,\bsgamma,\bsa}_\mix([0,1]^d)$ the preceding argument does not apply and the initial error can depend on the anchor and the weights.

Given the sequence of normed function spaces $F_d$ and the admissible information evaluations, following, e.g.,~\cite{Nov2016,NW2010}, we then define the integration problem to suffer from the \emph{curse of dimensionality} if there exist fixed constants $c, \epsilon_0, \eta > 0$ such that
\begin{align}\label{eq:curse}
  n(\varepsilon, F_d)
  &\ge
  c \, (1 + \eta)^d
  \qquad
  \text{for infinitely many } d \in \N
  \text{ and every } 0 < \varepsilon \le \varepsilon_0
  .
\end{align}
In particular, an algebraic convergence rate such as $n^{-\alpha/d}$ does not by itself imply the curse of dimensionality, since the dimension dependence of the constant in the error bound also matters.

\subsubsection{Lower bounds and the curse of dimensionality}\label{sec:lower-bounds}

We first establish that no deterministic algorithm using function values can converge faster than the lower bound $c_{d,\alpha} \, n^{-\alpha/d}$ for the space $H^\alpha([0,1]^d)$.
We state this result in the next theorem, Theorem~\ref{thm:Halpha-lower-bound}, and provide a proof.
The proof is maybe atypical w.r.t.\ other current references, but similar to the 1959 proof of Bakhvalov~\cite{Bak59}, adjusted to the $L_2$ norm and our function space.
It is easy to extend this result to general $L_p$-based Sobolev spaces, and although we said we were not interested in these more general $L_p$-based Sobolev spaces, we will handle this extension exceptionally in Theorem~\ref{thm:Walpha-p-lower-bound}.
Afterwards we will show a similar proof for the space $H^\alpha_\mix([0,1]^d)$ in Theorem~\ref{thm:Hmixalpha-lower-bound}, showing the lower bound $c_{d,\alpha}^\mix \, n^{-\alpha}$ (or, more refined, $c_{d,\alpha}^\mix \, n^{-\alpha} \, (\log_2 n)^{(d-1)/2}$).

However, while these constants $c_{d,\alpha}$ and $c_{d,\alpha}^\mix$ are independent of~$n$, they still depend on $\alpha$ and~$d$.
So it is possible that the constant is exponentially small in~$d$ and then this is not a proof of the curse of dimensionality.
To fulfil the definition of the curse of dimensionality~\eqref{eq:curse} we still need a separate argument.
Finally, Theorem~\ref{thm:curse} then shows that the unweighted spaces $H^\alpha([0,1]^d)$, $H^\alpha_\mix([0,1]^d)$ and $H^{\alpha,\rmu}_\mix([0,1]^d)$ do suffer from the curse of dimensionality.

\paragraph{A lower bound for $H^\alpha([0,1]^d)$.}
Our aim is to prove a lower bound on the worst-case error for integrating functions $f \in H^\alpha([0,1]^d)$ with $\|f\|_{H^\alpha([0,1]^d)} \le 1$ for \emph{any deterministic algorithm} $Q_{d,n}$ that uses information on at most $n$ function values.
We use $Q_{d,n}$ to denote such a general deterministic algorithm.
We will use the \emph{diameter of information}, see, e.g., \cite{TWW88}, to argue that, since our algorithm only has partial information on the function $f \in H^\alpha([0,1]^d)$, it cannot distinguish between all functions sharing the same information.

``Sharing the same information'' here means that the algorithm has access to function values of $f$ on at most $n$ points $\bsx_k$ chosen by the algorithm.
Since for the deterministic worst-case error an algorithm is a deterministic map, the same information leads to the same outcome.
Therefore, to show such a lower bound, we want to find two functions $f^\fmarkone \in H^\alpha([0,1]^d)$ and $f^\fmarktwo \in H^\alpha([0,1]^d)$, for which the algorithm uses the same $n$ sample points and hence sees the same function values, but for which the difference between $I_d(f^\fmarkone)$ and $I_d(f^\fmarktwo)$ is as large as possible, yet we would have $Q_{d,n}(f^\fmarkone) = Q_{d,n}(f^\fmarktwo) = q$ for any such algorithm.
The triangle inequality gives $|I_d(f^\fmarkone) - I_d(f^\fmarktwo)| \le |I_d(f^\fmarkone) - q| + |I_d(f^\fmarktwo) - q| \le 2 \max\{|I_d(f^\fmarkone) - q|, |I_d(f^\fmarktwo) - q|\}$, hence
\begin{align}\label{eq:triangle-fp-fpp}
  \max\{|I_d(f^\fmarkone) - q|, |I_d(f^\fmarktwo) - q|\}
  &\ge
  \frac12 \, |I_d(f^\fmarkone) - I_d(f^\fmarktwo)|
  .
\end{align}
The supremum of the distances $|I_d(f^\fmarkone) - I_d(f^\fmarktwo)|$ over functions producing the same information is called the \emph{diameter of information} for the given information and function class.
The worst-case error~\eqref{eq:def-wce-integration} is hence bounded \emph{from below} by half the diameter.
The best an algorithm could do is to return the ``centroid'', i.e., the midpoint of $I_d(f^\fmarkone)$ and $I_d(f^\fmarktwo)$.
Figure~\ref{fig:diameter-of-information} illustrates the diameter of information.
See \cite{TWW88} for a more thorough discussion of these concepts.

If we want to use this argument to show a lower bound on the worst-case error~\eqref{eq:def-wce-integration} in the unit ball of a certain space $H$, then it is sufficient to consider a subclass $F \subseteq H$, and the reasoning from above translates into
\begin{align*}
  \sup_{\substack{f \in H \\ \|f\|_H \le 1}} |I_d(f) - Q_{d,n}(f)|
  &\ge
  \sup_{\substack{f \in F \subseteq H \\ \|f\|_H \le 1}} |I_d(f) - Q_{d,n}(f)|
  \journalbreak\ge
  \frac12 \hspace*{-4mm} \sup_{\substack{f^\fmarkone, f^\fmarktwo \in F \subseteq H \\ \|f^\fmarkone\|_H, \|f^\fmarktwo\|_H \le 1 \\ N(Q_{d,n}, f^\fmarkone) = N(Q_{d,n}, f^\fmarktwo)}} \hspace*{-4mm} |I_d(f^\fmarkone) - I_d(f^\fmarktwo)|
  .
\end{align*}
Here $N(Q_{d,n}, f)$ denotes the complete information consumed by the algorithm (possibly adaptively) on~$f$.
Equality of the information implies equality of the deterministic algorithm's output; the converse need not hold.
In the terminology of~\cite{TWW88}, the worst-case error is defined over a ``function class''~$F$.
In our setting this class is usually the unit ball $F = \{ f \in H : \|f\|_H \le 1 \}$ of a function space~$H$.
Describing $F$ directly is useful when the admissible functions are specified by additional conditions rather than by a norm.

\ifdefined\galleryview
\else
  \begin{figure}[htb]
  \centering
\fi
\tikzsetnextfilename{figure_diameter_of_information}
\begin{tikzpicture}[
  x=1cm,y=1cm,font=\small,
  infoarrow/.style={-{Stealth[length=3mm,width=2mm]},thick},
]
  \fill[softblue] (0.2,0) ellipse (2 and 2.35);
  \draw[deepblue,very thick] (0.2,0) ellipse (2 and 2.35);
  \node[deepblue,anchor=base] at (0.2,2.9) {\strut Function class $F$};
  \node at (0.2,-2.9) {$F \subseteq \{g \in H : \|g\|_H \le 1\}$};

  \coordinate (fp) at (0.5,1.4);
  \coordinate (f) at (-0.2,0.2);
  \coordinate (fpp) at (0.55,-1.6);
  \coordinate (q) at (-7.8,-0.55);
  \coordinate (ifp) at (3.2,1.9);
  \coordinate (if) at (3.2,0.55);
  \coordinate (ifpp) at (3.2,-1.9);

  \draw[deepgreen,very thick]
    (0.8,2.05) .. controls (0.75,1.9) and (0.65,1.65) .. (fp)
    .. controls (0.35,1.15) and (-0.2,0.65) .. (f)
    .. controls (-0.2,-0.25) and (0.4,-1.35) .. (fpp)
    .. controls (0.7,-1.85) and (0.8,-2.05) .. (0.8,-2.15);
  \node[deepgreen,anchor=west] at (0.02,-0.25) {$N(Q_{d,n},f)$};

  \node[deepblue,anchor=base] at (-4.8,2.9) {\strut Sampled information};
  \node[deepblue,anchor=base] at (-4.8,2.43) {$N(Q_{d,n},f)$};
  \node[draw=deepblue,thick,fill=softblue!35,inner sep=7pt] (data) at (-4.8,0) {%
    \renewcommand{\arraystretch}{1.4}%
    \begin{tabular}{@{}l@{\hspace{0.5em}}|@{\hspace{0.5em}}c@{\quad}c@{\quad}c@{}}
      & $\bsx_1$ & $\cdots$ & $\bsx_m$ \\
      \hline
      $f$ & $y_1$ & $\cdots$ & $y_m$ \\
      $f^\fmarkone$ & $y_1$ & $\cdots$ & $y_m$ \\
      $f^\fmarktwo$ & $y_1$ & $\cdots$ & $y_m$
    \end{tabular}%
  };
  \node[below=5pt,deepblue] at (data.south) {$m \le n$};

  \draw[infoarrow,deepblue] (fp) -- ([yshift=-5pt]data.east);
  \draw[infoarrow,deepblue] (f) -- ([yshift=7pt]data.east);
  \draw[infoarrow,deepblue] (fpp) -- ([yshift=-17pt]data.east);
  \draw[infoarrow,deepblue] (data.west) -- (q);

  \draw[infoarrow,deeporange] (fp) -- (ifp);
  \draw[infoarrow,deeporange] (f) -- (if);
  \draw[infoarrow,deeporange] (fpp) -- (ifpp);

  \foreach \point/\label in {fp/f^\fmarkone,f/f} {
    \fill[black] (\point) circle (2.2pt);
    \node[above=5pt] at (\point) {$\label$};
  }
  \fill[black] (fpp) circle (2.2pt);
  \node[above=3pt,xshift=3pt] at (fpp) {$f^\fmarktwo$};

  \draw[very thick,-{Stealth[length=3mm,width=2mm]}] (-7.8,-2.35) -- (-7.8,2.35);
  \node[deepblue,anchor=base] at (-7.8,2.9) {\strut $Q_{d,n}$};
  \fill[deepblue] (q) circle (2.5pt) node[left=6pt] {$q$};
  \draw[very thick,-{Stealth[length=3mm,width=2mm]}] (3.2,-2.35) -- (3.2,2.35);
  \node[deeporange,anchor=base] at (3.2,2.9) {\strut $I_d$};
  \fill[deeporange] (ifp) circle (2.2pt) node[right=5pt] {$I_d(f^\fmarkone)$};
  \fill[deeporange] (if) circle (2.2pt) node[right=5pt] {$I_d(f)$};
  \fill[deeporange] (ifpp) circle (2.2pt) node[right=5pt] {$I_d(f^\fmarktwo)$};
  \fill[deepblue] (q -| if) circle (2.5pt) node[right=5pt] {$q$};
  \foreach \axisx in {-7.8,3.2} {
    \draw[thick] ({\axisx-0.1},-2.2) -- ({\axisx+0.1},-2.2);
  }

  \draw[deeporange,thick,{Stealth[length=3mm,width=2mm]}-{Stealth[length=3mm,width=2mm]}]
    (4.85,-1.9) -- (4.85,1.9);
  \node[rotate=90,deeporange] at (5.35,0)
    {$|I_d(f^\fmarkone) - I_d(f^\fmarktwo)|$};

  \node at (0,-3.6) {$N(Q_{d,n},f) = N(Q_{d,n},f^\fmarkone) = N(Q_{d,n},f^\fmarktwo)$};
  \node at (0,-4.2) {$Q_{d,n}(f) = Q_{d,n}(f^\fmarkone) = Q_{d,n}(f^\fmarktwo) = q$};
\end{tikzpicture}
\ifdefined\galleryview
\else
\caption{Diameter of information.
We pick functions $f$, $f^\fmarkone$ and $f^\fmarktwo$ that belong to the admissible class $F$, are sampled by the deterministic algorithm $Q_{d,n}$ at the same nodes $\bsx_1,\ldots,\bsx_m$ and have the same values $f(\bsx_k) = f^\fmarkone(\bsx_k) = f^\fmarktwo(\bsx_k) = y_k$ for every~$k$.
Thus $N(Q_{d,n},f) = N(Q_{d,n},f^\fmarkone) = N(Q_{d,n},f^\fmarktwo)$ and the algorithm returns the same output $q$ for all three functions.
The common output $q$ is also marked on the integral axis.
Their integrals can differ, and at least one of $f^\fmarkone$ and $f^\fmarktwo$ has error at least $\tfrac12 |I_d(f^\fmarkone) - I_d(f^\fmarktwo)|$.
The diameter of information is the supremum of this integral difference over pairs of functions giving the same information.}
\label{fig:diameter-of-information}
\end{figure}
\fi

Since the algorithms considered below use function values, point evaluation must be well defined on the function space.
The Sobolev embedding theorem provides a unique continuous representative for every $f \in H^\alpha([0,1]^d)$ when $\alpha > d/2$, which we therefore assume in the following theorem.

\begin{theorem}[{Lower bound for $H^\alpha([0,1]^d)$}]\label{thm:Halpha-lower-bound}
Let $d, \alpha, n \in \N$ with $\alpha>d/2$.
For every deterministic algorithm $Q_{d,n}(f)$ that tries to approximate the integral $I_d(f)$ using information on at most $n$ function values at points $\bsx_k$, which are up to the algorithm to decide, possibly adaptively, there is a function $f \in H^\alpha([0,1]^d)$ with $\|f\|_{H^\alpha([0,1]^d)} \le 1$ that has error bounded from below by $c_{d,\alpha} \, n^{-\alpha/d}$ with
\begin{align}\label{eq:def-c-d-alpha}
  c_{d,\alpha}
  &:=
  \frac{I_d(\phi_{d,\alpha})}{2^{1+\alpha+\alpha/d} \, \|\phi_{d,\alpha}\|_{H^\alpha([0,1]^d)}}
  ,
\end{align}
where the function $\phi_{d,\alpha} = \phi_\alpha^{\otimes d}$ is defined by~\eqref{eq:phidalpha} and~\eqref{eq:phialpha} in the proof below.
Hence
\begin{align*}
  \inf_{Q_{d,n} \in \frakQ_{d,n}} \sup_{\substack{f \in H^\alpha([0,1]^d) \\ \|f\|_{H^\alpha([0,1]^d)} \le 1}} |I_d(f) - Q_{d,n}(f)|
  \ge
  c_{d,\alpha} \, n^{-\alpha/d}
  .
\end{align*}
\end{theorem}
\begin{proof}
We divide $[0,1]^d$ into $m^d$ equal subcubes and place properly scaled versions of a nice smooth function on these subcubes.
We need to ensure that the overall function still has the required smoothness where we glue these functions together.
With $n \in \N$ the number of function evaluations given, define $m \in \N$ by
\begin{align}\label{eq:def-m}
  m = m(d, n)
  &:=
  \lceil (2n)^{1/d} \rceil
  ,
\end{align}
so that $n \le m^d / 2$.
(We note that $(2n)^{1/d} > 1$ for all $n, d \in \N$ and $m(d, 1) = \lceil 2^{1/d} \rceil = 2$ for all $d \in \N$, and hence $m^d \ge 2^d$.
Furthermore $m$ keeps on taking the value $2$ up to and including $n = 2^{d-1}$.)
We partition the unit cube $[0, 1]^d$ into $m^d$ equal subcubes of volume $m^{-d}$.
The location of such a subcube can be denoted by a vector $\bsa \in \Z_m^d$, such that for $\bsx$ in that subcube we have $a_j/m \le x_j \le (a_j+1)/m$.
Define, for $\alpha \in \N$, $x \in \R$,
\begin{align}\label{eq:phialpha}
  \phi_\alpha(x)
  &:=
  \begin{cases}
    (1 - x)^{\alpha+1} \, x^{\alpha+1}, & \text{when } 0 < x < 1 , \\
    0 , & \text{when } x \in \R \setminus (0, 1).
  \end{cases}
\end{align}
Figure~\ref{fig:phi-alpha} shows $\phi_\alpha$ and its derivatives up to order~$\alpha$.
This function is non-negative.
This function and its derivatives up to order $\alpha$ are continuous and go to zero when approaching $0$ or~$1$.
This guarantees that if we scale this function and glue it together, we have the required smoothness.
(In fact the exponent $\alpha + 1$ is stronger than we need, but will make some bounds easier.)

\ifdefined\galleryview
\else
  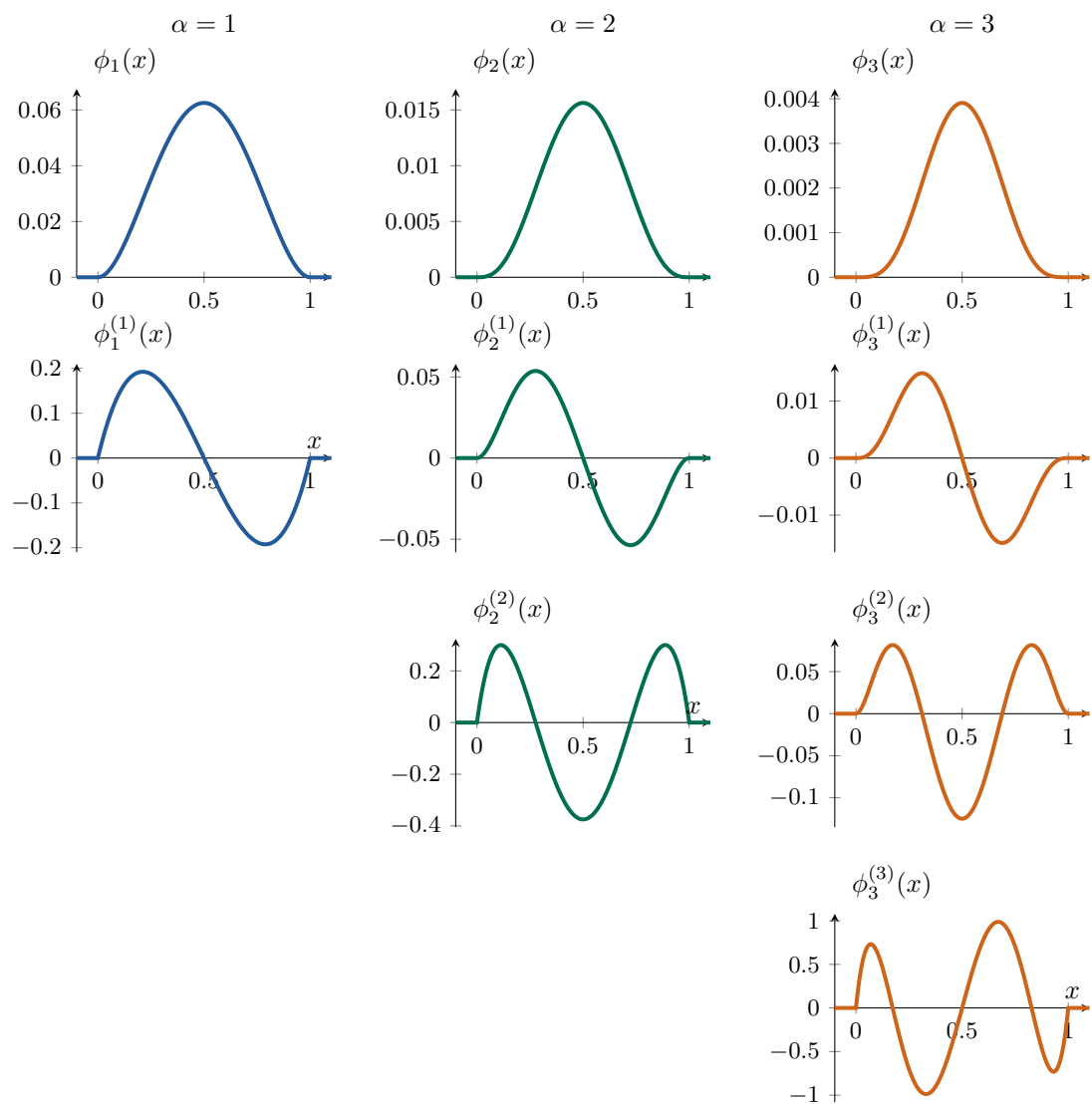
\begin{figure}
  \centering
\fi
\tikzsetnextfilename{figure_bump_functions}
\begin{tikzpicture}
\begin{groupplot}[
  group style={group size=3 by 4,horizontal sep=1.65cm,vertical sep=1.15cm},
  width=0.23\textwidth,
  height=0.17\textwidth,
  scale only axis,
  clip=false,
  xmin=-0.1,xmax=1.1,
  enlarge y limits=0.08,
  axis x line=middle,
  axis y line=left,
  xlabel={},
  xtick={0,0.5,1},
  scaled y ticks=false,
  yticklabel style={/pgf/number format/fixed,/pgf/number format/precision=5},
  tick label style={font=\small},
  samples=193,
  domain=-0.1:1.1,
  every axis plot/.append style={line width=1.5pt},
]
  \nextgroupplot[title={$\alpha = 1$},title style={yshift=12pt},ymin=0,ymax=0.0675]
  \addplot[deepblue] {(x >= 0 && x <= 1) ? ((x*(1-x))^2) : 0};
  \node[anchor=south west,xshift=3pt,yshift=2pt] at (axis description cs:0,1) {$\phi_1(x)$};
  \nextgroupplot[title={$\alpha = 2$},title style={yshift=12pt},ymin=0,ymax=0.016875]
  \addplot[deepgreen] {(x >= 0 && x <= 1) ? ((x*(1-x))^3) : 0};
  \node[anchor=south west,xshift=3pt,yshift=2pt] at (axis description cs:0,1) {$\phi_2(x)$};
  \nextgroupplot[title={$\alpha = 3$},title style={yshift=12pt},ymin=0,ymax=0.00421875]
  \addplot[deeporange] {(x >= 0 && x <= 1) ? ((x*(1-x))^4) : 0};
  \node[anchor=south west,xshift=3pt,yshift=2pt] at (axis description cs:0,1) {$\phi_3(x)$};

  \nextgroupplot[xlabel={$x$},ymin=-0.21,ymax=0.21]
  \addplot[deepblue] {(x >= 0 && x <= 1) ? (2*x*(1-x)*(1-2*x)) : 0};
  \node[anchor=south west,xshift=3pt,yshift=2pt] at (axis description cs:0,1) {$\phi_1^{(1)}(x)$};
  \nextgroupplot[ymin=-0.058,ymax=0.058]
  \addplot[deepgreen] {(x >= 0 && x <= 1) ? (3*x^2*(1-x)^2*(1-2*x)) : 0};
  \node[anchor=south west,xshift=3pt,yshift=2pt] at (axis description cs:0,1) {$\phi_2^{(1)}(x)$};
  \nextgroupplot[ymin=-0.0165,ymax=0.0165]
  \addplot[deeporange] {(x >= 0 && x <= 1) ? (4*x^3*(1-x)^3*(1-2*x)) : 0};
  \node[anchor=south west,xshift=3pt,yshift=2pt] at (axis description cs:0,1) {$\phi_3^{(1)}(x)$};

  \nextgroupplot[group/empty plot,ymin=0,ymax=1]
  \nextgroupplot[xlabel={$x$},ymin=-0.405,ymax=0.324]
  \addplot[deepgreen] {(x >= 0 && x <= 1) ? (6*x*(1-x)*(1-5*x+5*x^2)) : 0};
  \node[anchor=south west,xshift=3pt,yshift=2pt] at (axis description cs:0,1) {$\phi_2^{(2)}(x)$};
  \nextgroupplot[ymin=-0.135,ymax=0.089]
  \addplot[deeporange] {(x >= 0 && x <= 1) ? (4*x^2*(1-x)^2*(3-14*x+14*x^2)) : 0};
  \node[anchor=south west,xshift=3pt,yshift=2pt] at (axis description cs:0,1) {$\phi_3^{(2)}(x)$};

  \nextgroupplot[group/empty plot,ymin=0,ymax=1]
  \nextgroupplot[group/empty plot,ymin=0,ymax=1]
  \nextgroupplot[xlabel={$x$},ymin=-1.08,ymax=1.08]
  \addplot[deeporange] {(x >= 0 && x <= 1) ? (24*x*(1-x)*(1-2*x)*(1-7*x+7*x^2)) : 0};
  \node[anchor=south west,xshift=3pt,yshift=2pt] at (axis description cs:0,1) {$\phi_3^{(3)}(x)$};
\end{groupplot}
\end{tikzpicture}
\ifdefined\galleryview
\else
\caption{The function $\phi_\alpha(x) = (x \, (1-x))^{\alpha+1} \, \bbone\{x \in [0,1]\}$ for $\alpha = 1, 2, 3$ and its derivatives through order $\alpha$.}
\label{fig:phi-alpha}
\end{figure}
\fi

Let us denote the $d$-fold tensor product of this function by
\begin{align}\label{eq:phidalpha}
  \phi_{d,\alpha}(\bsx)
  =
  \phi_\alpha^{\otimes d}(\bsx)
  &:=
  \prod_{j=1}^d \phi_\alpha(x_j)
  .
\end{align}
We easily see that $\phi_{d,\alpha} \in H^\alpha([0,1]^d)$.
We can compute the integral exactly
\begin{align}\label{eq:int-phialpha}
  \int_0^1 \phi_\alpha(x) \rd{x}
  &=
  \frac{((\alpha + 1)!)^2}{(2\alpha+3)!}
  ,
  &
  I_d(\phi_{d,\alpha})
  =
  \int_{[0,1]^d} \phi_{d,\alpha}(\bsx) \rd\bsx
  &=
  \frac{((\alpha + 1)!)^{2d}}{((2\alpha+3)!)^d}
  .
\end{align}

Next we define a $d$-dimensional scaled product version of this function which has only non-zero values inside the subcube $[\bsa/m,(\bsa+\bsone)/m]$, $\bsa \in \Z_m^d$,
\begin{align}\label{eq:phi-d-alpha-m-bsa}
  \phi_{d,\alpha,m,\bsa}(\bsx)
  :=
  m^{-\alpha} \, \prod_{j=1}^d \phi_\alpha(m \, x_j - a_j)
  .
\end{align}
The extra scaling $m^{-\alpha}$ is for convenience and such that for any $\bstau \in \N_0^d$ with $\|\bstau\|_1 \le \alpha$ we have
\begin{align*}
  |D^\bstau \phi_{d,\alpha,m,\bsa}(\bsx)|
  &=
  \Big|m^{-\alpha} \, m^{\|\bstau\|_1} \, \prod_{j=1}^d \phi_\alpha^{(\tau_j)}(m \, x_j - a_j)\Big|
  \le
  \Big|\prod_{j=1}^d \phi_\alpha^{(\tau_j)}(m \, x_j - a_j)\Big|
  .
\end{align*}
Now consider putting this function on each of the subcubes by taking a linear combination with coefficients $c_\bsa$, for $\bsa \in \Z_m^d$, collected in the vector~$\bsc$,
\begin{align*}
  f_{d,\alpha,m,\bsc}(\bsx)
  &:=
  \sum_{\bsa \in \Z_m^d} c_\bsa \, \phi_{d,\alpha,m,\bsa}(\bsx)
  .
\end{align*}
This function belongs to $H^\alpha([0,1]^d)$ since the function and its derivatives continuously match with zero on the boundaries of the subcubes, see Figure~\ref{fig:neighbouring-bumps} for an illustration.
\ifdefined\galleryview
\else
  \begin{figure}
  \centering
\fi
\tikzsetnextfilename{figure_glued_bumps}
\begin{tikzpicture}
\begin{axis}[
  width=0.7\textwidth,
  height=0.32\textwidth,
  scale only axis,
  clip=false,
  xmin=0,xmax=1,
  ymin=-0.013,ymax=0.009,
  axis x line=middle,
  axis y line=left,
  xlabel={$x$},
  xtick={0,0.25,0.5,0.75,1},
  ytick={-0.01,-0.005,0,0.005},
  scaled y ticks=false,
  yticklabel style={/pgf/number format/fixed,/pgf/number format/precision=3},
  tick label style={font=\small},
  samples=161,
  every axis plot/.append style={very thick},
]
  \addplot[deepblue,domain=0:0.5] {2*2^(-2)*(2*x*(1-2*x))^3};
  \addplot[deeporange,domain=0.5:1] {-3*2^(-2)*((2*x-1)*(2-2*x))^3};
  \node[deepblue,anchor=south] at (axis description cs:0.25,1) {$c_0 = 2$};
  \node[deeporange,anchor=south] at (axis description cs:0.75,1) {$c_1 = -3$};
\end{axis}
\end{tikzpicture}
\ifdefined\galleryview
\else
\caption{Gluing scaled copies of $\phi_{d,\alpha} = \phi_\alpha^{\otimes d}$ on neighbouring subcubes gives a function in $H^\alpha([0,1]^d)$.
The plot shows $2^{-\alpha} \, \big(2 \phi_\alpha(2x) - 3 \phi_\alpha(2x-1)\big)$ for $d = 1$ and $\alpha = 2$.}
\label{fig:neighbouring-bumps}
\end{figure}
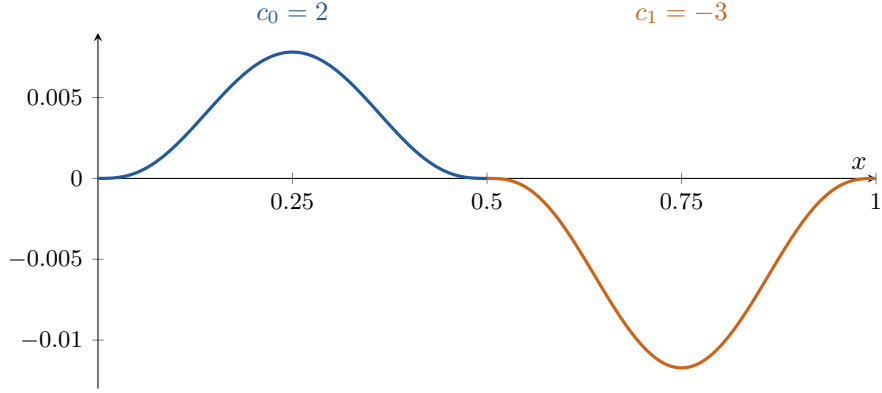
\fi
Since the functions $\phi_{d,\alpha,m,\bsa}$ are supported on their respective subcubes and they do not overlap, for $\|\bstau\|_1 \le \alpha$,
\begin{align*}
  \journalonly{\MoveEqLeft[1]} \|D^\bstau f_{d,\alpha,m,\bsc}\|_{L_2([0,1]^d)}^2
  \journalonly{\\*}&=
  m^{-2\alpha} \sum_{\bsa \in \Z_m^d} |c_\bsa|^2 \, \prod_{j=1}^d \int_{a_j/m}^{(a_j+1)/m} |m^{\tau_j} \, \phi_\alpha^{(\tau_j)}(m x_j - a_j)|^2 \rd{x_j}
  \\
  &=
  m^{-2\alpha+2\|\bstau\|_1} \sum_{\bsa \in \Z_m^d} |c_\bsa|^2 \, \prod_{j=1}^d m^{-1} \int_0^1 |\phi_\alpha^{(\tau_j)}(x_j)|^2 \rd{x_j}
  \\\numberthis{eq:pull-out-c-equality}
  &=
  m^{-2\alpha+2\|\bstau\|_1} \|D^\bstau \phi_{d,\alpha}\|_{L_2([0,1]^d)}^2 \, m^{-d} \sum_{\bsa \in \Z_m^d} |c_\bsa|^2
  \\
  &\le
  \left(\max_{\bsa \in \Z_m^d} |c_\bsa|^2\right) \,
  \|D^\bstau \phi_{d,\alpha}\|_{L_2([0,1]^d)}^2
  .
\end{align*}
We have
\begin{align*}
  \|f_{d,\alpha,m,\bsc}\|_{H^\alpha([0,1]^d)}
  &=
  \Bigg(\sum_{\substack{\bstau \in \N_0^d \\ \|\bstau\|_1 \le \alpha}} \|D^\bstau f_{d,\alpha,m,\bsc}\|_{L_2([0,1]^d)}^2\Bigg)^{1/2}
  \journalbreak\le
  \left(\max_{\bsa \in \Z_m^d} |c_\bsa|\right) \,
  \|\phi_{d,\alpha}\|_{H^\alpha([0,1]^d)}
  .
\end{align*}
We also have
\begin{align*}
  \int_{[0,1]^d} f_{d,\alpha,m,\bsc}(\bsx) \rd{\bsx}
  &=
  m^{-\alpha} \sum_{\bsa \in \Z_m^d} c_\bsa \, \prod_{j=1}^d \int_{a_j/m}^{(a_j+1)/m} \phi_\alpha(m x_j - a_j) \rd{x_j}
  \\
  &=
  m^{-\alpha} \sum_{\bsa \in \Z_m^d} c_\bsa \, \prod_{j=1}^d m^{-1} \int_0^1 \phi_\alpha(x_j) \rd{x_j}
  \\
  &=
  m^{-\alpha} \, m^{-d} \Bigg( \sum_{\bsa \in \Z_m^d} c_\bsa \Bigg) \int_{[0,1]^d} \phi_{d,\alpha}(\bsx) \rd{\bsx}
  \\
  &=
  m^{-\alpha} \, m^{-d} \, I_d(\phi_{d,\alpha}) \sum_{\bsa \in \Z_m^d} c_\bsa
  .
\end{align*}

We now specify a subclass of functions in the unit ball of $H^\alpha([0,1]^d)$
\begin{align*}
  F
  &:=
  \big\{ f_{d,\alpha,m,\bsc} / \|\phi_{d,\alpha}\|_{H^\alpha([0,1]^d)} \in H^\alpha([0,1]^d) : c_\bsa \in \{\pm1\} \text{ for all } \bsa \in \Z_m^d \big\}
  .
\end{align*}
Note that all $f \in F$ satisfy $\|f\|_{H^\alpha([0,1]^d)} \le 1$ and in fact all have the same norm, cf.~\eqref{eq:pull-out-c-equality},
\begin{align*}
   \|f\|_{H^\alpha([0,1]^d)}^2
   &=
   \sum_{\substack{\bstau \in \N_0^d \\ \|\bstau\|_1 \le \alpha}} \frac{m^{-2\alpha+2\|\bstau\|_1}}{\|\phi_{d,\alpha}\|_{H^\alpha([0,1]^d)}^2} \, \|D^\bstau \phi_{d,\alpha}\|_{L_2([0,1]^d)}^2
  .
\end{align*}
Note that we have a very specific set of functions, but the behaviour on the different subcubes is independent of any of the other subcubes.
If our algorithm only has access to $n$ function values then we have no information on at least half of the subcubes since $m^d - n \ge m^d/2$.
(Note that the function (and its derivatives) vanish on the boundaries of the subcubes.)

This means we can construct two functions $f^\fmarkone, f^\fmarktwo \in F$ which take the same values on all the subcubes where we have information and differ maximally in those where we do not.
So, let us fix $f^\fmarkone = f_{d,\alpha,m,\bsc^\fmarkone} / \|\phi_{d,\alpha}\|_{H^\alpha([0,1]^d)}$ with $c^\fmarkone_\bsa = +1$ for all $\bsa \in \Z_m^d$ and $f^\fmarktwo = f_{d,\alpha,m,\bsc^\fmarktwo} / \|\phi_{d,\alpha}\|_{H^\alpha([0,1]^d)}$ with $c^\fmarktwo_\bsa = c^\fmarkone_\bsa = +1$ on every subcube where we have information and $c^\fmarktwo_\bsa = -1$ on every subcube where we do not have information.
Note that, since the algorithm is deterministic, we can first run the algorithm on the function $f^\fmarkone$ and keep track which subcubes were sampled.
Afterwards we can construct $f^\fmarktwo$.
Even if the algorithm is adaptive, running it on $f^\fmarktwo$ will result in exactly the same sampling being done, since it sees the same information, and hence it computes the same result.
Figure~\ref{fig:subcubes} illustrates the idea.

\ifdefined\galleryview
\else
  \begin{figure}[t]
  \centering
\fi
\tikzsetnextfilename{figure_subcube_construction}
\begin{tikzpicture}[x=5.3cm,y=5.3cm,font=\small]
  \foreach \panel in {0,1} {
    \begin{scope}[shift={({1.5*\panel},0)}]
      \foreach \cellx in {0,...,3} {
        \foreach \celly in {0,...,3} {
          \pgfmathtruncatemacro{\sampledcell}{
            (\cellx == 0 || \cellx == 2) ? (\celly < 2) : (\celly >= 2)
          }
          \ifnum\panel=0
            \def\cellcolour{softgreen}
            \def\signcolour{deepgreen!60!black}
            \def\cellsign{+1}
          \else
            \ifnum\sampledcell=1
              \def\cellcolour{softgreen}
              \def\signcolour{deepgreen!60!black}
              \def\cellsign{+1}
            \else
              \def\cellcolour{softorange}
              \def\signcolour{deeporange!60!black}
              \def\cellsign{-1}
            \fi
          \fi
          \fill[\cellcolour] ({\cellx/4},{\celly/4})
            rectangle ({(\cellx+1)/4},{(\celly+1)/4});
          \node[\signcolour,font=\boldmath] at ({(\cellx+0.5)/4},{(\celly+0.5)/4}) {$\cellsign$};
        }
      }
      \foreach \k in {1,2,3} {
        \draw[gray] ({\k/4},0) -- ({\k/4},1);
        \draw[gray] (0,{\k/4}) -- (1,{\k/4});
      }
      \draw[very thick] (0,0) rectangle (1,1);
      \pic {construction samples};
      \node[below left=2pt] at (0,0) {$0$};
      \node[below=2pt] at (1,0) {$1$};
      \node[left=2pt] at (0,1) {$1$};
      \node[below=13pt] at (0.5,0) {$x_1$};
      \node[left=13pt] at (0,0.5) {$x_2$};
    \end{scope}
  }
  \node at (0.5,1.13) {$f^\fmarkone$: all coefficients $+1$};
  \node at (2,1.13) {$f^\fmarktwo$: signs reversed in unsampled cells};
  \node at (1.25,-0.20) {$d = 2,\quad n = 8,\quad m = \lceil\sqrt{2n}\rceil = 4,\quad m^d = 16$};
  \node at (1.25,-0.30) {$N(Q_{d,n},f^\fmarkone) = N(Q_{d,n},f^\fmarktwo)$};
\end{tikzpicture}
\ifdefined\galleryview
\else
\caption{Construction of $f^\fmarkone$ and $f^\fmarktwo$.
Each cell has $\phi_{d,\alpha,m,\bsa}$ with the indicated coefficient.
The samples are marked by black dots.
The function $f^\fmarkone$ has coefficient $+1$ in every cell, while $f^\fmarktwo$ keeps $+1$ in the sampled cells and uses $-1$ in the unsampled cells.
After the common norm normalisation, the two functions agree at every sample point but have different integrals.}
\label{fig:subcubes}
\end{figure}
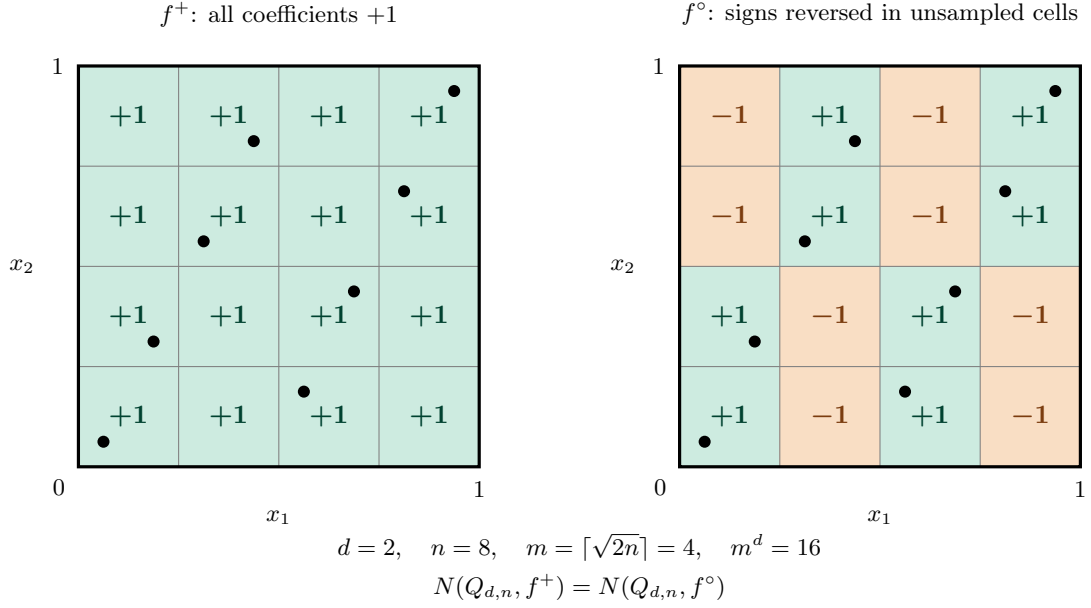
\fi

We have
\begin{align*}
  \int_{[0,1]^d} f^\fmarkone(\bsx) \rd{\bsx}
  -
  \int_{[0,1]^d} f^\fmarktwo(\bsx) \rd{\bsx}
  &=
  m^{-\alpha} \, m^{-d} \, \frac{I_d(\phi_{d,\alpha})}{\|\phi_{d,\alpha}\|_{H^\alpha([0,1]^d)}} \sum_{\substack{\bsa \in \Z_m^d \\ c^\fmarktwo_\bsa = -1}} 2
  .
\end{align*}
Hence
\begin{align*}
  \sup_{\substack{f \in H^\alpha([0,1]^d) \\ \|f\|_{H^\alpha([0,1]^d)} \le 1}} |I_d(f) - Q_{d,n}(f)|
  &\ge
  \frac12 \frac{m^{-\alpha} \, m^{-d} \, 2 \, (m^d - n) \, I_d(\phi_{d,\alpha})}{\|\phi_{d,\alpha}\|_{H^\alpha([0,1]^d)}}
  \ge
  c_{d,\alpha} \, n^{-\alpha/d}
  ,
\end{align*}
with
\begin{align*}
  c_{d,\alpha}
  &=
  \frac{I_d(\phi_{d,\alpha})}{2^{1+\alpha+\alpha/d} \, \|\phi_{d,\alpha}\|_{H^\alpha([0,1]^d)}}
  .
\end{align*}
To obtain this bound, first note from~\eqref{eq:def-m} that
$n \le m^d/2$ and therefore $m^d - n \ge m^d/2$ and thus
\begin{align*}
  m^{-d} \, 2 \, (m^d - n)
  &=
  \frac{m^d-n}{m^d / 2}
  \ge
  1
  .
\end{align*}
Furthermore
\begin{align*}
  \frac1{m^\alpha}
  =
  \frac1{(m^d)^{\alpha/d}}
  >
  \frac1{(2^d \, 2n)^{\alpha/d}}
  =
  \frac1{2^{\alpha+\alpha/d} \, n^{\alpha/d}}
  ,
\end{align*}
since $\lceil x \rceil < 1 + x < 2 x$ for $x > 1$ and with $x = (2 n)^{1/d} > 1$ this yields $m^d = \lceil (2n)^{1/d} \rceil^d < 2^d \, 2 n$.
This finishes the proof.
\end{proof}

Let us remark that, in the proof of the previous theorem, we only assumed function values, but we could as well have assumed  derivative values at those sample points (if the function is sufficiently regular); this does not help since the subcubes do not leak information to other subcubes.
Even more importantly, nonlinear algorithms or adaptive algorithms will not help either, as information on the sampled subcubes does not reveal information on the other subcubes.

\paragraph{Generalisation to $L_p$-based Sobolev space norms.}
From the proof of the previous theorem it is also easy to extend the result to Sobolev spaces which take $L_p$ norms instead of $L_2$ norms.
We denote those Sobolev spaces by $W^\alpha_p([0,1]^d)$ where for $p = 2$ we have $H^\alpha([0,1]^d) = W^\alpha_2([0,1]^d)$.
The norms are defined by
\begin{align}\label{eq:W-alpha-p-norms}
  \|f\|_{W^\alpha_p([0,1]^d)}
  &:=
  \begin{cases}
    \displaystyle
    \Bigg(
    \sum_{\substack{\bstau \in \N_0^d \\ \|\bstau\|_1 \le \alpha}}\|D^\bstau f\|_{L_p([0,1]^d)}^p
    \Bigg)^{1/p}
    ,
    &
    1 \le p < \infty
    ,
    \\[10mm]
    \displaystyle
    \max_{\substack{\bstau \in \N_0^d \\ \|\bstau\|_1 \le \alpha}} \|D^\bstau f\|_{L_\infty([0,1]^d)}
    ,
    &
    p = \infty
    .
  \end{cases}
\end{align}
For function values to be well defined, we need to assume $\alpha > d / p$ here.
We formulate this as the next theorem which includes the previous theorem as a special case.

\begin{theorem}[{Lower bound for $W^\alpha_p([0,1]^d)$ for $1 \le p \le \infty$}]\label{thm:Walpha-p-lower-bound}
Let $d, \alpha, n \in \N$, let $1 \le p \le \infty$, and assume $\alpha > d / p$.
Then for every deterministic algorithm $Q_{d,n}$ which approximates $I_d$ using at most $n$ function values, at points chosen possibly adaptively, there is a function $f \in W^\alpha_p([0,1]^d)$ with $\|f\|_{W^\alpha_p([0,1]^d)} \le 1$ that has error bounded from below by $c_{d,\alpha,p} \, n^{-\alpha/d}$ with
\begin{align*}
  c_{d,\alpha,p}
  &:=
  \frac{I_d(\phi_{d,\alpha})}{2^{1+\alpha+\alpha/d} \, \|\phi_{d,\alpha}\|_{W^\alpha_p([0,1]^d)}}
  ,
\end{align*}
where the function $\phi_{d,\alpha} = \phi_\alpha^{\otimes d}$ is defined by~\eqref{eq:phidalpha} and~\eqref{eq:phialpha}.
Hence,
\begin{align*}
  \inf_{Q_{d,n} \in \frakQ_{d,n}}
  \sup_{\substack{f \in W^\alpha_p([0,1]^d) \\ \|f\|_{W^\alpha_p([0,1]^d)} \le 1}}
  |I_d(f) - Q_{d,n}(f)|
  &\ge
  c_{d,\alpha,p} \, n^{-\alpha/d}.
\end{align*}
\end{theorem}
\begin{proof}
One can easily modify the proof of Theorem~\ref{thm:Halpha-lower-bound} to obtain this result.
\end{proof}

\paragraph{A lower bound for $H^\alpha_\mix([0,1]^d)$ and $H^{\alpha,\rmu}_\mix([0,1]^d)$.}
We now use a similar proof to show a lower bound for the spaces $H^\alpha_\mix([0,1]^d)$ and $H^{\alpha,\rmu}_\mix([0,1]^d)$.
Point evaluation is well defined in these spaces for $\alpha \in \N$ by Lemmas~\ref{lem:H-alpha-mix-C-embedding} and~\ref{lem:H-alpha-mix-weighted-equivalent-norms}, similarly for $D^\bstau f$ for $\bstau \in \{0,\ldots,\alpha-1\}^d$.
We demand $n \ge 2$ such that we can simply write $\log_2 n$ instead of having a form with $1 + \log_2 n$.
A similar lower bound can be obtained for the anchored space $H^{\alpha,\bsa}_\mix([0,1]^d)$ using the same construction in the proof.
We just need to estimate the norms, which is easiest if the anchor $\bsa \in [0,1]^d$ is on the boundary of the cube, e.g., $\bsa = \bszero$ or $\bsa = \bsone$.
We will not give that proof.
For $\alpha = 1$, a lower bound (using the simpler function $\phi(x) = (1 - x) \, x$ for $0 < x < 1$ and $0$ elsewhere in the proof) can be found in~\cite[Section~2.7]{DHP2015}.

\begin{theorem}[{Lower bound for $H^\alpha_\mix([0,1]^d)$ and $H^{\alpha,\rmu}_\mix([0,1]^d)$}]\label{thm:Hmixalpha-lower-bound}
Let $d, \alpha, n \in \N$ with $n \ge 2$.
For every deterministic algorithm $Q_{d,n}$ which approximates $I_d(f)$ using point values of the function, or its mixed derivatives $D^\bstau f$ with $\bstau \in \{0,\ldots,\alpha-1\}^d$, at no more than $n$ points, chosen possibly adaptively, there is an $f \in H^\alpha_\mix([0,1]^d)$ with $\|f\|_{H^\alpha_\mix([0,1]^d)} \le 1$ that has error bounded from below by $c^\mix_{d,\alpha} \, n^{-\alpha} \, (\log_2 n)^{(d-1)/2}$ with $c^\mix_{d,\alpha}$ given by
\begin{align*}
  \numberthis{eq:def-c-mix-d-alpha}
  c^\mix_{d,\alpha}
  &:=
  \frac{I_d(\phi_{d,\alpha})}{2^{2\alpha+1} \, 3^{d/2} \, C_\alpha^{d/2} \, (\alpha + 1)^{d/2} \, \sqrt{(d-1)!}}
  \\
  &\hphantom{:}=
  \frac{((\alpha+1)!)^d \, (2\alpha+3)^{d/2}}{2^{2\alpha+1} \, 3^{d/2} \, (\alpha+1)^{d/2} \, ((2\alpha+3)!)^d \, \sqrt{(d-1)!}}
  ,
\end{align*}
where the function $\phi_{d,\alpha} = \phi_\alpha^{\otimes d}$ is defined by~\eqref{eq:phidalpha} and~\eqref{eq:phialpha} and $C_\alpha$ is defined by~\eqref{eq:Calpha}.
Hence,
\begin{align*}
  \inf_{Q_{d,n} \in \frakQ_{d,n}} \sup_{\substack{f \in H^\alpha_\mix([0,1]^d) \\ \|f\|_{H^\alpha_\mix([0,1]^d)} \le 1}} |I_d(f) - Q_{d,n}(f)|
  &\ge
  c^\mix_{d,\alpha} \, n^{-\alpha} \, (\log_2 n)^{(d-1)/2}
  .
\end{align*}
The same lower bound holds with $H^{\alpha,\rmu}_\mix([0,1]^d)$ in place of $H^\alpha_\mix([0,1]^d)$,
\begin{align*}
  \inf_{Q_{d,n} \in \frakQ_{d,n}} \sup_{\substack{f \in H^{\alpha,\rmu}_\mix([0,1]^d) \\ \|f\|_{H^{\alpha,\rmu}_\mix([0,1]^d)} \le 1}} |I_d(f) - Q_{d,n}(f)|
  &\ge
  c^\mix_{d,\alpha} \, n^{-\alpha} \, (\log_2 n)^{(d-1)/2}
  .
\end{align*}
We also have $c^\mix_{d,\alpha} \, n^{-\alpha} \, (\log_2 n)^{(d-1)/2} \ge c^\mix_{d,\alpha} \, n^{-\alpha}$ as a simpler lower bound for both $H^\alpha_\mix([0,1]^d)$ and $H^{\alpha,\rmu}_\mix([0,1]^d)$.
\end{theorem}
\begin{proof}
We follow a similar strategy as in the proof of Theorem~\ref{thm:Halpha-lower-bound}.
There we constructed a function that we could control on any of the $m^d$ subcubes for a well chosen~$m$, here we use a superposition of functions indexed by $\bsv \in \N_0^d$ with $\|\bsv\|_1 = t$ for some well chosen~$t$.
Each such partition consists of $2^t$ boxes which will be independent, but over the partitions there can be overlap.
We choose $t$ so that $2^t / 4 < n \le 2^t / 2$, namely
\begin{align}\label{eq:def-t}
  t
  =
  t(n)
  &:=
  1 + \lceil \log_2 n \rceil
  .
\end{align}

Define, for $\bsv \in \N_0^d$ and $\bsa \in \Z_{2^{v_1}} \oplus \cdots \oplus \Z_{2^{v_d}}$,
\begin{align*}
  \psi_{d,\alpha,\bsv,\bsa}(\bsx)
  &:=
  \prod_{j=1}^d 2^{-v_j \alpha} \, \phi_\alpha(2^{v_j} x_j - a_j)
  ,
\end{align*}
where $\phi_\alpha$ is defined by~\eqref{eq:phialpha} in the proof of Theorem~\ref{thm:Halpha-lower-bound}.
This mimics the functions $\phi_{d,\alpha,m,\bsa}$ defined in~\eqref{eq:phi-d-alpha-m-bsa} and we have that $\psi_{d,\alpha,\bsv,\bsa}$ is non-zero only inside a particular dyadic box.
A specific $\bsv \in \N_0^d$ will partition the unit cube $[0, 1]^d$ in $\prod_{j=1}^d 2^{v_j} = 2^{v_1+\cdots+v_d} = 2^{\|\bsv\|_1}$ pieces.
We restrict $\bsv$ to have $\|\bsv\|_1 = v_1 + \cdots + v_d = t$.
The scaling with $\prod_{j=1}^d 2^{-v_j \alpha}$ is such that for any $\bstau \in \N_0^d$ with $\|\bstau\|_\infty \le \alpha$ we have
\begin{align*}
  |D^\bstau \psi_{d,\alpha,\bsv,\bsa}(\bsx)|
  =
  \Big| \prod_{j=1}^d 2^{-v_j \alpha} \, 2^{v_j \tau_j} \, \phi^{(\tau_j)}_\alpha(2^{v_j} x_j - a_j) \Big|
  \le
  \Big| \prod_{j=1}^d \, \phi^{(\tau_j)}_\alpha(2^{v_j} x_j - a_j) \Big|
  .
\end{align*}

With the same purpose as the function $f_{d,\alpha,m,\bsc}$ in the proof of Theorem~\ref{thm:Halpha-lower-bound} we now define a function
\begin{align*}
  g_{d,\alpha,t,\bsc}(\bsx)
  :=
  \sum_{\substack{\bsv \in \N_0^d \\ \|\bsv\|_1 = t}}
  \sum_{\substack{a_1 \in \Z_{2^{v_1}} \\ \cdots \\ a_d \in \Z_{2^{v_d}}}}
  c_{\bsv,\bsa} \, \psi_{d,\alpha,\bsv,\bsa}(\bsx)
  ,
\end{align*}
where the coefficient vector $\bsc$ is the collection of all the $c_{\bsv,\bsa}$.
Note that the sum over $\bsa$ is a particular partition of the unit cube and so inside the first sum, when $\bsv$ is fixed, we have a non-overlapping superposition of functions supported on $2^{\|\bsv\|_1} = 2^t$ smaller boxes.
As in the proof of Theorem~\ref{thm:Halpha-lower-bound}, each $\psi_{d,\alpha,\bsv,\bsa}$ and all the mixed derivatives needed here vanish on the boundary of its box.
Hence each glued function for fixed $\bsv$, and therefore their finite sum $g_{d,\alpha,t,\bsc}$, belongs to $H^\alpha_\mix([0,1]^d)$.
See Figure~\ref{fig:dyadic-partition} for all five partitions for $d = 2$ and $t = 4$.

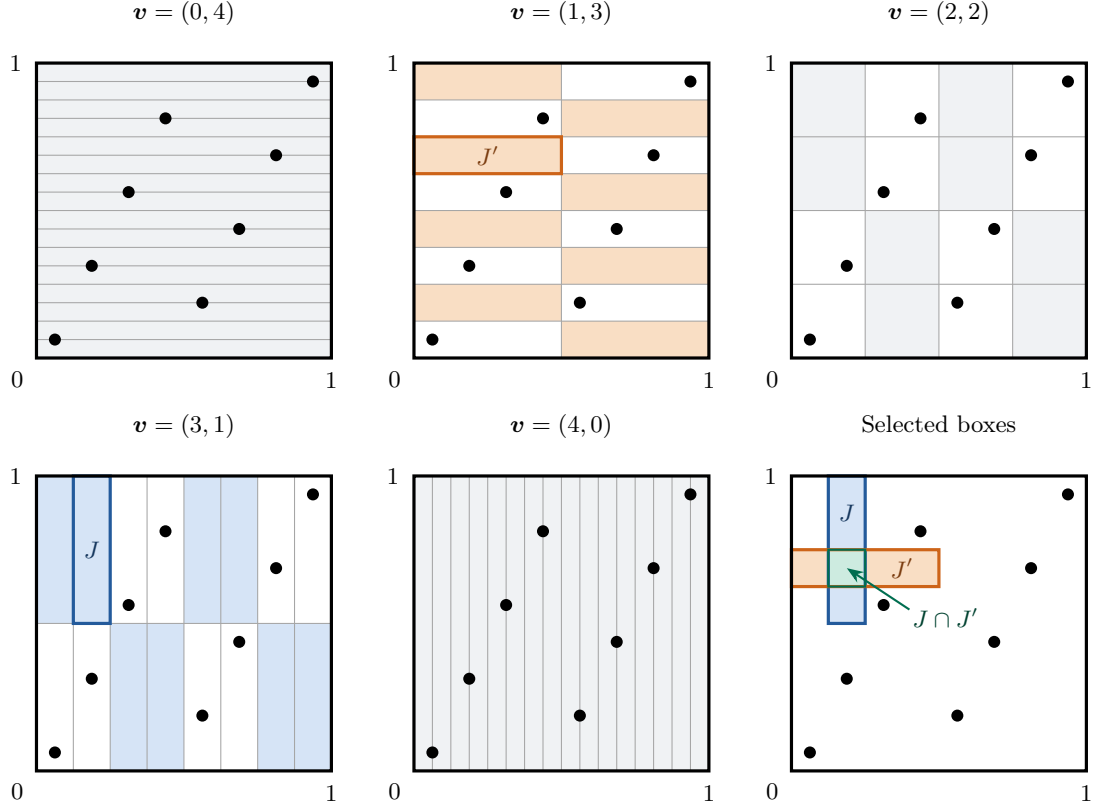
\ifdefined\galleryview
\else
  \begin{figure}[htb]
  \centering
\fi
\tikzsetnextfilename{figure_dyadic_partitions}
\begin{tikzpicture}[x=3.9cm,y=3.9cm,font=\small]
  \foreach \level/\posx/\posy in {0/0/1.4,1/1.28/1.4,2/2.56/1.4,3/0/0,4/1.28/0} {
    \begin{scope}[shift={(\posx,\posy)}]
      \pgfmathtruncatemacro{\otherlevel}{4-\level}
      \pgfmathtruncatemacro{\columns}{2^\level}
      \pgfmathtruncatemacro{\rows}{2^\otherlevel}
      \pgfmathtruncatemacro{\lastcolumn}{\columns-1}
      \pgfmathtruncatemacro{\lastrow}{\rows-1}
      \def\cellcolour{softgray}
      \ifnum\level=1
        \def\cellcolour{softorange}
      \fi
      \ifnum\level=3
        \def\cellcolour{softblue}
      \fi
      \fill[\cellcolour] (0,0) rectangle (1,1);
      \ifcase\level
        \def\occupiedcells{}
      \or
        \def\occupiedcells{0/0,1/1,0/2,1/3,0/4,1/5,0/6,1/7}
      \or
        \def\occupiedcells{0/0,0/1,1/2,1/3,2/0,2/1,3/2,3/3}
      \or
        \def\occupiedcells{0/0,1/0,2/1,3/1,4/0,5/0,6/1,7/1}
      \or
        \def\occupiedcells{}
      \fi
      \ifx\occupiedcells\empty
      \else
        \foreach \cellx/\celly in \occupiedcells {
          \fill[white] ({\cellx/\columns},{\celly/\rows})
            rectangle ({(\cellx+1)/\columns},{(\celly+1)/\rows});
        }
      \fi
      \ifnum\columns>1
        \foreach \k in {1,...,\lastcolumn} {
          \draw[gray!70] ({\k/\columns},0) -- ({\k/\columns},1);
        }
      \fi
      \ifnum\rows>1
        \foreach \k in {1,...,\lastrow} {
          \draw[gray!70] (0,{\k/\rows}) -- (1,{\k/\rows});
        }
      \fi
      \ifnum\level=1
        \draw[deeporange,very thick] (0,0.625) rectangle (0.5,0.75);
        \node[deeporange!60!black] at (0.25,0.6875) {$J'$};
        \node[anchor=base] at (0.5,1.15) {\strut $\bsv = (1,3)$};
      \else
        \node[anchor=base] at (0.5,1.15) {\strut $\bsv = (\level,\otherlevel)$};
      \fi
      \ifnum\level=3
        \draw[deepblue,very thick] (0.125,0.5) rectangle (0.25,1);
        \node[deepblue!60!black] at (0.1875,0.75) {$J$};
      \fi
      \draw[very thick] (0,0) rectangle (1,1);
      \pic {construction samples};
      \node[below left=2pt] at (0,0) {$0$};
      \node[below=2pt] at (1,0) {$1$};
      \node[left=2pt] at (0,1) {$1$};
    \end{scope}
  }

  \begin{scope}[shift={(2.56,0)}]
    \fill[softblue] (0.125,0.5) rectangle (0.25,1);
    \fill[softorange] (0,0.625) rectangle (0.5,0.75);
    \fill[softgreen] (0.125,0.625) rectangle (0.25,0.75);
    \draw[deepblue,very thick] (0.125,0.5) rectangle (0.25,1);
    \draw[deeporange,very thick] (0,0.625) rectangle (0.5,0.75);
    \draw[deepgreen,thick] (0.125,0.625) rectangle (0.25,0.75);
    \node[deepblue!60!black] at (0.1875,0.875) {$J$};
    \node[deeporange!60!black] at (0.375,0.6875) {$J'$};
    \draw[very thick] (0,0) rectangle (1,1);
    \pic {construction samples};
    \node[anchor=base] at (0.5,1.15) {\strut Selected boxes};
    \node[below left=2pt] at (0,0) {$0$};
    \node[below=2pt] at (1,0) {$1$};
    \node[left=2pt] at (0,1) {$1$};
    \node[deepgreen!60!black,anchor=west,inner sep=1pt] (overlaplabel) at (0.40,0.52) {$J \cap J'$};
    \draw[infoarrow,deepgreen] ([yshift=3pt]overlaplabel.west) -- (0.1875,0.6875);
  \end{scope}
  \node at (1.78,-0.27) {$d = 2,\quad t = 4,\quad n = 8,\quad 2^t = 16$ boxes in each partition};
\end{tikzpicture}
\ifdefined\galleryview
\else
\caption{Example of dyadic partitions for the functions $g_{d,\alpha,t,\bsc}$.
In each of the boxes of each partition we have $\psi_{d,\alpha,\bsv,\bsa}$ with its coefficient.
In a particular partition these functions do not overlap, but in the superposition they can overlap, as illustrated in the final panel for two unsampled boxes $J$ and~$J'$.}
\label{fig:dyadic-partition}
\end{figure}
\fi

We can express its integral in terms of the integral of $\phi_{d,\alpha}$ as follows
\begin{align*}
  \int_{[0,1]^d} g_{d,\alpha,t,\bsc}(\bsx) \rd\bsx
  &=
  \sum_{\substack{\bsv \in \N_0^d \\ \|\bsv\|_1 = t}}
  \int_{[0,1]^d}
  \sum_{\substack{a_1 \in \Z_{2^{v_1}} \\ \cdots \\ a_d \in \Z_{2^{v_d}}}}\journalonly{\mkern-12mu}
  c_{\bsv,\bsa} \,
  \psi_{d,\alpha,\bsv,\bsa}(\bsx)
  \rd{\bsx}
  \\
  &=
  \sum_{\substack{\bsv \in \N_0^d \\ \|\bsv\|_1 = t}}
  \sum_{\substack{a_1 \in \Z_{2^{v_1}} \\ \cdots \\ a_d \in \Z_{2^{v_d}}}}\journalonly{\mkern-12mu}
  c_{\bsv,\bsa}
  \prod_{j=1}^d \int_{a_j/2^{v_j}}^{(a_j+1)/2^{v_j}} 2^{-v_j \alpha} \, \phi_\alpha(2^{v_j} x_j - a_j) \rd{x_j}
  \\
  &=
  \sum_{\substack{\bsv \in \N_0^d \\ \|\bsv\|_1 = t}}
  \sum_{\substack{a_1 \in \Z_{2^{v_1}} \\ \cdots \\ a_d \in \Z_{2^{v_d}}}}\journalonly{\mkern-12mu}
  c_{\bsv,\bsa}
  \prod_{j=1}^d \int_0^1 2^{-v_j \alpha} \, 2^{-v_j} \, \phi_\alpha(x_j) \rd{x_j}
  \\
  &=
  I_d(\phi_{d,\alpha})
  \sum_{\substack{\bsv \in \N_0^d \\ \|\bsv\|_1 = t}}
  2^{-\|\bsv\|_1 \alpha} \; 2^{-\|\bsv\|_1}
  \sum_{\substack{a_1 \in \Z_{2^{v_1}} \\ \cdots \\ a_d \in \Z_{2^{v_d}}}}\journalonly{\mkern-12mu}
  c_{\bsv,\bsa}
  \\
  \numberthis{eq:integral-g}
  &=
  2^{-t \alpha} \, 2^{-t} \, I_d(\phi_{d,\alpha})
  \sum_{\substack{\bsv \in \N_0^d \\ \|\bsv\|_1 = t}}
  \sum_{\substack{a_1 \in \Z_{2^{v_1}} \\ \cdots \\ a_d \in \Z_{2^{v_d}}}}\journalonly{\mkern-12mu}
  c_{\bsv,\bsa}
  .
\end{align*}
Note that $|\{ \bsv \in \N_0^d : \|\bsv\|_1 = t \}| = \binom{t + d - 1}{d - 1}$ and the sum over $\bsa \in \Z_{2^{v_1}} \oplus \cdots \oplus \Z_{2^{v_d}}$ has $2^t$ terms.

We now want to estimate the norm $\|g_{d,\alpha,t,\bsc}\|_{H^\alpha_\mix([0,1]^d)}$.
We will show two ways of estimating this norm.
The first one is the easier bound and uses the Minkowski inequality, but gives a slightly inferior result where we would obtain a factor $(\log_2 n)^{d-1}$ in the norm bound instead of $(\log_2 n)^{(d-1)/2}$.
After that we show a more technically involved bound which will give the claimed $(\log_2 n)^{(d-1)/2}$ factor in the lower bound.
The first estimate uses the Minkowski inequality $\|\sum_i f_i\|_{L_p} \le \sum_i \|f_i\|_{L_p}$ for $1 \le p \le \infty$, which gives
\begin{align*}
  \| g_{d,\alpha,t,\bsc} \|_{H^\alpha_\mix([0,1]^d)}^2
  &=
  \sum_{\bstau \in \{0,\ldots,\alpha\}^d}
  \| D^\bstau g_{d,\alpha,t,\bsc} \|_{L_2([0,1]^d)}^2
  \\
  &=
  \sum_{\bstau \in \{0,\ldots,\alpha\}^d}
  \Bigg\|
  \sum_{\substack{\bsv \in \N_0^d \\ \|\bsv\|_1 = t}}
  \sum_{\substack{a_1 \in \Z_{2^{v_1}} \\ \cdots \\ a_d \in \Z_{2^{v_d}}}}
  c_{\bsv,\bsa} \, D^\bstau \psi_{d,\alpha,\bsv,\bsa}
  \Bigg\|_{L_2([0,1]^d)}^2
  \\
  \numberthis{eq:Halphamix-norm-bound}
  &\le
  \binom{t+d-1}{d-1}^{\!2} \, \|\bsc\|_\infty^2 \,
  \| \phi_{d,\alpha} \|_{H^\alpha_\mix([0,1]^d)}^2
  ,
\end{align*}
where $\|\bsc\|_\infty$ is taken over all coefficients
$c_{\bsv,\bsa}$.
We now show the second estimate.
We use the forthcoming Lemma~\ref{lem:phialpha-derivative-inner-product} to estimate the $L_2$ inner products
\begin{align*}
  \big| \langle
    2^{-v \, (\alpha-\tau)} \,
    \phi_\alpha^{(\tau)}(2^v \cdot - a)
    ,
    2^{-v' \, (\alpha-\tau)} \,
    \phi_\alpha^{(\tau)}(2^{v'} \cdot - a')
    \rangle \big|
  &\le
  C_\alpha \,
  2^{-\max\{v, v'\}} \,
  2^{-|v - v'|}
  ,
\end{align*}
and zero if the supports $J$ and $J'$ of $\phi_\alpha(2^v \cdot - a)$ and $\phi_\alpha(2^{v'} \cdot - a')$ do not overlap, with $C_\alpha$ given by~\eqref{eq:Calpha}.
Then, using the notation of $J$ and $J'$ for those supports,
\begin{align*}
  &\| g_{d,\alpha,t,\bsc} \|_{H^\alpha_\mix([0,1]^d)}^2
  =
  \sum_{\bstau \in \{0,\ldots,\alpha\}^d}
  \| D^\bstau g_{d,\alpha,t,\bsc} \|_{L_2([0,1]^d)}^2
  \\
  &\quad=
  \sum_{\bstau \in \{0,\ldots,\alpha\}^d}
  \Bigg\|
  \sum_{\substack{\bsv \in \N_0^d \\ \|\bsv\|_1 = t}}
  \sum_{\substack{a_1 \in \Z_{2^{v_1}} \\ \cdots \\ a_d \in \Z_{2^{v_d}}}}
  c_{\bsv,\bsa} \, D^\bstau \psi_{d,\alpha,\bsv,\bsa}
  \Bigg\|_{L_2([0,1]^d)}^2
  \\
  &\quad=
  \sum_{\bstau \in \{0,\ldots,\alpha\}^d}
  \sum_{\substack{\bsv, \bsv' \in \N_0^d \\ \|\bsv\|_1 = \|\bsv'\|_1 = t}}
  \sum_{\substack{\bsa \in \oplus_{j=1}^d \Z_{2^{v_j}} \\ \bsa' \in \oplus_{j=1}^d \Z_{2^{v'_j}}}}
  c_{\bsv,\bsa} \,
  c_{\bsv',\bsa'}
  \\*
  &\hspace{\ifjournalstyle 1cm\else 4cm\fi}
  \prod_{j=1}^d
  \langle
  2^{-v_j \, (\alpha-\tau_j)} \,
  \phi_\alpha^{(\tau_j)}(2^{v_j} \cdot - a_j)
  ,
  2^{-v'_j \, (\alpha-\tau_j)} \,
  \phi_\alpha^{(\tau_j)}(2^{v_j'} \cdot - a_j')
  \rangle
  \\
  &\quad\le
  \|\bsc\|_\infty^2
  \sum_{\bstau \in \{0,\ldots,\alpha\}^d}
  \sum_{\substack{\bsv, \bsv' \in \N_0^d \\ \|\bsv\|_1 = \|\bsv'\|_1 = t}}
  \sum_{\substack{\bsa \in \oplus_{j=1}^d \Z_{2^{v_j}} \\ \bsa' \in \oplus_{j=1}^d \Z_{2^{v'_j}} \\ J \cap J' \ne \emptyset}}
  \prod_{j=1}^d
  C_\alpha \,
  2^{-\max\{v_j, v_j'\}} \,
  2^{-|v_j - v_j'|}
  \\
  &\quad\le
  C_\alpha^d \,
  \|\bsc\|_\infty^2 \,
  (\alpha + 1)^d \,
  \sum_{\substack{\bsv, \bsv' \in \N_0^d \\ \|\bsv\|_1 = \|\bsv'\|_1 = t}}
  2^{-\|\bsv - \bsv'\|_1}
  \\
  &\quad=
  C_\alpha^d \,
  \|\bsc\|_\infty^2 \,
  (\alpha + 1)^d \,
  \sum_{\substack{\bsv \in \N_0^d \\ \|\bsv\|_1 = t}}
  \;
  \sum_{\substack{\bsk \in \Z^d \\ \sum_{j=1}^d k_j = 0 \\ \bsv - \bsk \in \N_0^d}}
  2^{-\|\bsk\|_1}
  \\
  &\quad\le
  C_\alpha^d \,
  \|\bsc\|_\infty^2 \,
  (\alpha + 1)^d \,
  \sum_{\substack{\bsv \in \N_0^d \\ \|\bsv\|_1 = t}}
  \sum_{\bsk \in \Z^d}
  2^{-\|\bsk\|_1}
  \\
  &\quad\le
  C_\alpha^d \,
  \|\bsc\|_\infty^2 \,
  (\alpha + 1)^d \,
  \binom{t + d - 1}{d - 1}
  \prod_{j=1}^d \Big( 1 + 2 \sum_{k \ge 1}
  2^{-k} \Big)
  \\
  \numberthis{eq:Halphamix-norm-bound2}
  &\quad=
  \underbrace{3^d \, C_\alpha^d \, (\alpha + 1)^d \, \binom{t + d - 1}{d - 1}}_{=: M_{d,\alpha,t}^2} \,
  \|\bsc\|_\infty^2
  ,
\end{align*}
where the norm $\|\bsc\|_\infty$ is taken over all coefficients~$c_{\bsv,\bsa}$ and we defined the square of $M_{d,\alpha,t}$ in the underbrace notation.
Note that~\eqref{eq:Halphamix-norm-bound2} now only has $\binom{t+d-1}{d-1}$ instead of its square as in~\eqref{eq:Halphamix-norm-bound}.

As before we now specify a subclass of functions in the unit ball of $H^\alpha_\mix([0,1]^d)$
\begin{align*}
  G
  &:=
  \big\{ g_{d,\alpha,t,\bsc} / M_{d,\alpha,t} \in H^\alpha_\mix([0,1]^d) : c_{\bsv,\bsa} \in \{\pm1\} \big\}
  .
\end{align*}
Run the algorithm on $g^\fmarkone = g_{d,\alpha,t,\bsc^\fmarkone} / M_{d,\alpha,t}$, where all $c^\fmarkone_{\bsv,\bsa} = +1$, and let $\bsx_1,\ldots,\bsx_n$ be the resulting, possibly adaptively chosen, sample points.
We define $g^\fmarktwo = g_{d,\alpha,t,\bsc^\fmarktwo} / M_{d,\alpha,t}$ by setting $c^\fmarktwo_{\bsv,\bsa} = +1$ for every open dyadic box $(\bsv, \bsa)$ containing at least one sample point and $c^\fmarktwo_{\bsv,\bsa} = -1$ for all other boxes.
The functions $g^\fmarkone$ and $g^\fmarktwo$, together with all mixed derivatives of coordinatewise order at most~$\alpha$, agree at every sample point, so the algorithm obtains the same information and produces the same output for both functions.
For each fixed $\bsv$, at most $n$ of its $2^t$ dyadic boxes contain a sample point.

We obtain
\begin{align*}
  I_d(g^\fmarkone) - I_d(g^\fmarktwo)
  &=
  2^{-t \alpha} \, 2^{-t} \, \frac{I_d(\phi_{d,\alpha})}{M_{d,\alpha,t}}
  \sum_{\substack{\bsv \in \N_0^d \\ \|\bsv\|_1 = t}}
  \sum_{\substack{a_1 \in \Z_{2^{v_1}} \\ \cdots \\ a_d \in \Z_{2^{v_d}} \\ c^\fmarktwo_{\bsv,\bsa} = -1}}
  2
  .
\end{align*}
Hence
\begin{align*}
  \sup_{\substack{f \in H^\alpha_\mix([0,1]^d) \\ \|f\|_{H^\alpha_\mix([0,1]^d)} \le 1}} |I_d(f) - Q_{d,n}(f)|
  &\ge
  \frac12
  \frac{2^{-t\alpha} \, 2^{-t} \, I_d(\phi_{d,\alpha}) \, 2 \, (2^t - n) \, \binom{t+d-1}{d-1}}{3^{d/2} \, C_\alpha^{d/2} \, (\alpha + 1)^{d/2} \, \sqrt{\binom{t+d-1}{d-1}}}
  \\
  &\ge
  c^\mix_{d,\alpha} \, n^{-\alpha} \, (\log_2 n)^{(d-1)/2}
  ,
\end{align*}
with
\begin{align*}
  c^\mix_{d,\alpha}
  &=
  \frac{I_d(\phi_{d,\alpha})}{2^{2\alpha+1} \, 3^{d/2} \, C_\alpha^{d/2} \, (\alpha + 1)^{d/2} \, \sqrt{(d-1)!}}
  .
\end{align*}
This follows since \eqref{eq:def-t} implies $n \le 2^t / 2$ and hence $2^t - n \ge 2^t / 2$; \eqref{eq:def-t}~also implies $2^t < 4 n$, and therefore $2^{-t \alpha} > 2^{-2 \alpha} \, n^{-\alpha}$.
Finally,
\begin{align*}
  \binom{t+d-1}{d-1}
  &=
  \frac{(t+1) \cdots (t+d-1)}{(d-1)!}
  \ge
  \frac{t^{d-1}}{(d-1)!}
  \ge
  \frac{(\log_2 n)^{d-1}}{(d-1)!}
  .
\end{align*}
This proves the claimed lower bound for $H^\alpha_\mix([0,1]^d)$.
The embedding~\eqref{eq:H-alpha-gamma-mix-equivalence-unanchored} with $\bsgamma \equiv \bsone$, i.e., $\|f\|_{H^{\alpha,\rmu}_\mix([0,1]^d)} \le \|f\|_{H^\alpha_\mix([0,1]^d)}$, shows that the unit ball of $H^\alpha_\mix([0,1]^d)$ is contained in the unit ball of $H^{\alpha,\rmu}_\mix([0,1]^d)$, so the same lower bound holds for the latter.
\end{proof}


\paragraph{The curse of dimensionality for $H^\alpha([0,1]^d)$ and $H^\alpha_\mix([0,1]^d)$.}
The lower bounds in Theorems~\ref{thm:Halpha-lower-bound} and~\ref{thm:Hmixalpha-lower-bound} do not establish the curse of dimensionality.
Indeed, a general lower bound of the form
\begin{align*}
  \inf_{Q_{d,n} \in \frakQ_{d,n}} e(Q_{d,n}, F_{d,\alpha})
  &\ge
  c(F_{d,\alpha}) \, n^{-r(F_{d,\alpha})}
  ,
\end{align*}
only implies
\begin{align*}
  n(\varepsilon, F_{d,\alpha})
  &\ge
  \left(\frac{c(F_{d,\alpha})}{\varepsilon}\right)^{1/r(F_{d,\alpha})}
  .
\end{align*}
Although they give the dependence on $n$ for every fixed $d$, the constant $c_{d,\alpha}^\mix$ tends to zero superexponentially as $d \to \infty$.
Consequently, this bound does not imply exponential information complexity in $d$ at any fixed positive accuracy.
A separate lower-bound argument is needed which we give in the next result.
We state this result for $H^\alpha([0,1]^d)$, $H^\alpha_\mix([0,1]^d)$ and $H^{\alpha,\rmu}_\mix([0,1]^d)$.
The anchored space $H^{\alpha,\bsa}_\mix([0,1]^d)$ also suffers from the curse of dimensionality but we will not show that here, see, e.g.,~\cite{NW2010,NP2025}.
For the space $H^\alpha([0,1]^d)$ we act as if we can take $d \to \infty$, but we need to remind the reader that bounded point evaluation requires $\alpha > d/2$.
Since we do not want to send $\alpha$ to infinity the result needs to be read as holding for functions in $H^\alpha([0,1]^d) \cap C([0,1]^d)$ with the norm from $H^\alpha([0,1]^d)$.
This makes point evaluation well defined, but does not guarantee a finite worst-case error, the following lower-bound argument however remains valid.

\begin{theorem}[{Curse of dimensionality for $H^\alpha([0,1]^d)$, $H^\alpha_\mix([0,1]^d)$ and $H^{\alpha,\rmu}_\mix([0,1]^d)$}]\label{thm:curse}
Let $d, \alpha, n \in \N$, and consider cubature rules of the form
\begin{align*}
  Q_{d,n}(f)
  &=
  \sum_{i=1}^n w_i \, f(\bsx_i)
  ,
  \qquad
  w_i \in \R
  .
\end{align*}
Define $\eta_\alpha := 2 / (\sum_{\tau=0}^{\alpha} (2\pi)^{2\tau}) \in (0, 1)$, then, for $n \le (1 + \eta_\alpha)^d$,
\begin{align}\label{eq:curse-H-alpha-mix}
  \sup_{\substack{f \in H^\alpha_\mix([0,1]^d) \\ \|f\|_{H^\alpha_\mix([0,1]^d)} \le 1}}
  |I_d(f) - Q_{d,n}(f)|
  &\ge
  \left(1 - \frac{n}{(1 + \eta_\alpha)^d}\right)^{\!1/2}
  .
\end{align}
Consequently, reducing the error to at most some $\varepsilon \in (0, 1)$
requires
\begin{align*}
  n
  &\ge
  (1 - \varepsilon^2) \, (1 + \eta_\alpha)^d
  ,
\end{align*}
which grows exponentially with $d$.
The same conclusion holds for $H^\alpha([0,1]^d)$ and $H^{\alpha,\rmu}_\mix([0,1]^d)$ by the constant-one embeddings $H^\alpha_\mix([0,1]^d) \hookrightarrow H^\alpha([0,1]^d)$ and $H^\alpha_\mix([0,1]^d) \hookrightarrow H^{\alpha,\rmu}_\mix([0,1]^d)$, see~\eqref{eq:H-alpha-gamma-mix-equivalence-unanchored} with $\bsgamma \equiv \bsone$,
\begin{align*}
  \sup_{\substack{f \in H^\alpha([0,1]^d) \\ \|f\|_{H^\alpha([0,1]^d)} \le 1}}
  |I_d(f) - Q_{d,n}(f)|
  \ge
  \sup_{\substack{f \in H^\alpha_\mix([0,1]^d) \\ \|f\|_{H^\alpha_\mix([0,1]^d)} \le 1}}
  |I_d(f) - Q_{d,n}(f)|
  ,
\end{align*}
and
\begin{align*}
  \sup_{\substack{f \in H^{\alpha,\rmu}_\mix([0,1]^d) \\ \|f\|_{H^{\alpha,\rmu}_\mix([0,1]^d)} \le 1}}
  |I_d(f) - Q_{d,n}(f)|
  \ge
  \sup_{\substack{f \in H^\alpha_\mix([0,1]^d) \\ \|f\|_{H^\alpha_\mix([0,1]^d)} \le 1}}
  |I_d(f) - Q_{d,n}(f)|
  .
\end{align*}
\end{theorem}
\begin{proof}
For the proof set $\eta = \eta_\alpha$.
Consider the three univariate functions
\begin{align*}
  e_0(x)
  &:=
  1
  ,
  &
  e_1(x)
  &:=
  \sqrt{\eta} \, \cos(2\pi x)
  ,
  &
  e_2(x)
  &:=
  \sqrt{\eta} \, \sin(2\pi x)
  .
\end{align*}
They are orthonormal with respect to the univariate Sobolev inner product
\begin{align*}
  \langle f, g \rangle_{H^\alpha([0,1])}
  &:=
  \sum_{r=0}^{\alpha}
  \int_0^1 f^{(r)}(x) \, \overline{g^{(r)}(x)} \rd{x}
  .
\end{align*}
(Note that this is for $d=1$ hence we just wrote $H^\alpha$.)
Furthermore, $I_1(e_0) = 1$ and $I_1(e_1) = I_1(e_2) = 0$.
The $d$-variate functions
\begin{align*}
  e_{\bsk}(\bsx)
  &:=
  \prod_{j=1}^d e_{k_j}(x_j)
  ,
  \qquad
  \bsk \in \{0,1,2\}^d
  ,
\end{align*}
form an orthonormal family in $H^\alpha_\mix([0,1]^d)$.
Only $e_\bszero = 1$ has a non-zero integral, being $I_d(e_\bszero) = 1$.

Define the error functional $E_{d,n} := I_d - Q_{d,n}$ and consider the function
\begin{align}\label{eq:def-cubature-rule-fooling-function}
  h_{d,n}(\bsx)
  &:=
  \sum_{\bsk \in \{0,1,2\}^d} E_{d,n}(e_\bsk) \, e_\bsk(\bsx)
  .
\end{align}
By linearity of $E_{d,n}$,
\begin{align*}
  E_{d,n}(h_{d,n})
  &=
  \sum_{\bsk \in \{0,1,2\}^d}
  E_{d,n}(e_\bsk) \, E_{d,n}(e_\bsk)
  =
  \sum_{\bsk \in \{0,1,2\}^d} (E_{d,n}(e_\bsk))^2
  .
\end{align*}
It turns out the squared norm is also equal to this expression
\begin{align*}
  \|h_{d,n}\|_{H^\alpha_\mix([0,1]^d)}^2
  &=
  \langle h_{d,n}, h_{d,n} \rangle_{H^\alpha_\mix([0,1]^d)}
  \\
  &=
  \sum_{\bsk \in \{0,1,2\}^d}
  \sum_{\bsk' \in \{0,1,2\}^d}
  E_{d,n}(e_\bsk) \, E_{d,n}(e_{\bsk'}) \,
  \langle e_\bsk, e_{\bsk'} \rangle_{H^\alpha_\mix([0,1]^d)}
  \\
  &=
  \sum_{\bsk \in \{0,1,2\}^d}
  (E_{d,n}(e_\bsk))^2
  ,
\end{align*}
where the last equality follows from $\langle e_\bsk, e_{\bsk'} \rangle_{H^\alpha_\mix([0,1]^d)} = \bbone\{\bsk = \bsk'\}$.
If $h_{d,n} \ne 0$ then $h_{d,n} / \|h_{d,n}\|_{H^\alpha_\mix([0,1]^d)}$ is in the unit ball and has integration error $\|h_{d,n}\|_{H^\alpha_\mix([0,1]^d)}$.
If $h_{d,n} = 0$ the resulting lower bound below is trivially valid.
Thus, in either case,
\begin{align*}
  \sup_{\substack{f \in H^\alpha_\mix([0,1]^d) \\ \|f\|_{H^\alpha_\mix([0,1]^d)} \le 1}} |I_d(f) - Q_{d,n}(f)|
  &=
  \sup_{\substack{f \in H^\alpha_\mix([0,1]^d) \\ \|f\|_{H^\alpha_\mix([0,1]^d)} \le 1}} |E_{d,n}(f)|
  \\
  &\ge
  \|h_{d,n}\|_{H^\alpha_\mix([0,1]^d)}
  =
  \sqrt{E_{d,n}(h_{d,n})}
  .
\end{align*}

We find, using $I_d(e_\bsk) = \bbone\{\bsk = \bszero\}$ and with $s := \sum_{i=1}^n w_i$,
\begin{align*}
  \journalonly{\MoveEqLeft[1]} E_{d,n}(h_{d,n})
  \journalonly{\\*}&=
  \sum_{\bsk \in \{0,1,2\}^d}
  (E_{d,n}(e_\bsk))^2
  \\
  &=
  \sum_{\bsk \in \{0,1,2\}^d}
  (I_d(e_\bsk) - Q_{d,n}(e_\bsk))^2
  \\
  &=
  \sum_{\bsk \in \{0,1,2\}^d}
  \left((I_d(e_\bsk))^2 - 2 \, I_d(e_\bsk) \, Q_{d,n}(e_\bsk) + (Q_{d,n}(e_\bsk))^2\right)
  \\
  &=
  1 - 2 \sum_{i=1}^n w_i
  +
  \sum_{i,i'=1}^n \articleonly{\!} w_i \, w_{i'} \sum_{\bsk \in \{0,1,2\}^d} e_\bsk(\bsx_i) \, e_\bsk(\bsx_{i'})
  \\
  &=
  1 - 2 s
  +
  \sum_{i,i'=1}^n \articleonly{\!} w_i \, w_{i'} \sum_{\bsk \in \{0,1,2\}^d} \prod_{j=1}^d e_{k_j}(x_{i,j}) \, e_{k_j}(x_{i',j})
  \\
  &=
  1 - 2 s
  +
  \sum_{i,i'=1}^n \articleonly{\!} w_i \, w_{i'} \prod_{j=1}^d \sum_{k_j \in \{0, 1, 2\}} e_{k_j}(x_{i,j}) \, e_{k_j}(x_{i',j})
  \\
  &=
  1 - 2 s
  +
  \sum_{i,i'=1}^n \articleonly{\!} w_i \, w_{i'} \prod_{j=1}^d \Big(1 + \eta \, \cos(2\pi x_{i,j}) \, \cos(2\pi x_{i',j})
  \\*&\hphantom{{}= 1 - 2 s + \sum_{i,i'=1}^n \articleonly{\!} w_i \, w_{i'} \prod_{j=1}^d \Big(1}
  + \eta \, \sin(2\pi x_{i,j}) \, \sin(2\pi x_{i',j}) \Big)
  \\
  &=
  1 - 2 s
  +
  \sum_{i,i'=1}^n \articleonly{\!} w_i \, w_{i'} \prod_{j=1}^d \left(1 + \eta \, \cos(2\pi (x_{i,j} - x_{i',j})) \right)
  \\
  &\ge
  1 - 2 s + \frac{(1 + \eta)^d}{n} \, s^2
  \\
  &\ge
  1 - \frac{n}{(1 + \eta)^d}
  ,
\end{align*}
where the inequalities follow from Lemma~\ref{lem:double-sum} and minimising the quadratic.
This proves the lower bound for $H^\alpha_\mix([0,1]^d)$.
In Section~\ref{sec:rkhs} we will introduce the topic of reproducing kernel Hilbert spaces.
It is interesting to note that the expression above is exactly the expression for the squared worst-case error in a reproducing kernel Hilbert space with reproducing kernel $K(\bsx, \bsy) = \sum_{\bsk \in \{0,1,2\}^d} e_\bsk(\bsx) \, e_\bsk(\bsy) = \prod_{j=1}^d (1 + \eta \, \cos(2\pi(x_j-y_j)))$ for $\operatorname{span}\{ e_\bsk : \bsk \in \{0,1,2\}^d \}$.
The worst-case error on this subspace then forms a lower bound for the worst-case error on the full space.

The right-hand side is positive as long as $n < (1 + \eta)^d$, which implies that then $h_{d,n}$ is not the zero function since its norm is positive.
Furthermore, if we want to reduce the squared error to be at most $0 < \varepsilon^2 \le \varepsilon_0^2 < 1$ then we need $n \ge (1 - \varepsilon^2) \, (1 + \eta)^d \ge (1 - \varepsilon_0^2) \, (1 + \eta)^d$ function values.
Since $\eta = \eta(\alpha) > 0$ this is exponential in~$d$.
We have the curse of dimensionality~\eqref{eq:curse} with $c = 1 - \varepsilon_0^2$.

Finally, $\|f\|_{H^\alpha([0,1]^d)} \le \|f\|_{H^\alpha_\mix([0,1]^d)}$, the unit ball of $H^\alpha_\mix([0,1]^d)$ is therefore contained in the unit ball of $H^\alpha([0,1]^d)$, so the same lower bound holds for the latter.
Similarly, the unit ball of $H^\alpha_\mix([0,1]^d)$ is contained in the unit ball of $H^{\alpha,\rmu}_\mix([0,1]^d)$, see~\eqref{eq:H-alpha-gamma-mix-equivalence-unanchored}.
\end{proof}

The next theorem shows that it does not matter if you have an adaptive algorithm and/or a nonlinear algorithm, you will still have the curse of dimensionality in the unweighted spaces $H^\alpha([0,1]^d)$, $H^\alpha_\mix([0,1]^d)$ and $H^{\alpha,\rmu}_\mix([0,1]^d)$.
In the proof we will use a result on optimal weights, Theorem~\ref{thm:RKHS-optimal-fixed-nodes}, that we will prove later in Section~\ref{sec:rkhs} on reproducing kernel Hilbert spaces.
Alternatively we refer to~\cite[Theorem~4.7]{NW2008} which states this in more generality.

\begin{theorem}[{Extension of the curse for deterministic nonlinear and adaptive cubature for $H^\alpha([0,1]^d)$, $H^\alpha_\mix([0,1]^d)$ and $H^{\alpha,\rmu}_\mix([0,1]^d)$}]\label{thm:curse-nonlinear-adaptive}
Let $\alpha \in \N$ and $F_d$ be a normed function space on $[0,1]^d$ such that $H^\alpha_\mix([0,1]^d) \hookrightarrow F_d$ with an embedding constant $C > 0$ independent of $d$, i.e., $\|f\|_{F_d} \le C \, \|f\|_{H^\alpha_\mix([0,1]^d)}$ for all $f \in H^\alpha_\mix([0,1]^d)$, and let $Q_{d,n}$ be an arbitrary deterministic linear or nonlinear cubature algorithm that uses at most $n$ function values, where the sample points may be chosen adaptively based on previously observed function values.
Then, with $\eta_\alpha$ as in Theorem~\textup{\ref{thm:curse}}, for $n \le (1 + \eta_\alpha)^d$,
\begin{align*}
  \sup_{\substack{f \in F_d \\ \|f\|_{F_d} \le 1}}
  |I_d(f) - Q_{d,n}(f)|
  &\ge
  \frac1C
  \left(1 - \frac{n}{(1 + \eta_\alpha)^d}\right)^{\!1/2}
  .
\end{align*}
For $H^\alpha([0,1]^d)$, $H^\alpha_\mix([0,1]^d)$ and $H^{\alpha,\rmu}_\mix([0,1]^d)$ we can take $C = 1$ and so all suffer from the curse of dimensionality even if we allow deterministic nonlinear and/or adaptive cubature algorithms.
\end{theorem}
\begin{proof}
For any such algorithm $Q_{d,n}$, run $Q_{d,n}$ on the zero function and denote the resulting fixed set of nodes by $X^0 = \{\bsx_1^0, \ldots, \bsx_m^0\}$ with $m \le n$.
Set $q := Q_{d,n}(0)$.
Take any $f \in H^\alpha_\mix([0,1]^d)$ with $\|f\|_{H^\alpha_\mix([0,1]^d)} \le 1$ and $f(\bsx_i^0) = 0$ for all $i = 1, \ldots, m$.
The norm inequality $\|f\|_{F_d} \le C \, \|f\|_{H^\alpha_\mix([0,1]^d)}$ implies that both $f/C$ and $-f/C$ belong to the unit ball of~$F_d$.
For both functions we have $Q_{d,n}(f/C) = Q_{d,n}(-f/C) = q$.
Applying~\eqref{eq:triangle-fp-fpp} with $f^\fmarkone = f/C$ and $f^\fmarktwo = -f/C$ gives
\begin{align*}
  e(Q_{d,n}, F_d)
  &\ge
  \max\{|I_d(f/C) - q|, |I_d(-f/C) - q|\}
  \ge
  \frac1C \, |I_d(f)|
  .
\end{align*}
Taking the supremum over all such $f$ gives
\begin{align*}
  e(Q_{d,n}, F_d)
  &\ge
  \frac1C \,
  \sup_{\substack{f \in H^\alpha_\mix([0,1]^d) \\ \|f\|_{H^\alpha_\mix([0,1]^d)} \le 1 \\ f(\bsx_i^0)=0,\ i=1,\ldots,m}}
  |I_d(f)|
  =
  \frac1C \, e(Q^{X^0,\bsw^*}_{d,m},H^\alpha_\mix([0,1]^d))
  ,
\end{align*}
where $Q^{X^0,\bsw^*}_{d,m}$ is a nonadaptive linear algorithm using the nodes in $X^0$ and using optimal weights $\bsw^*$ in the sense of Theorem~\ref{thm:RKHS-optimal-fixed-nodes} in Section~\ref{sec:rkhs} on reproducing kernel Hilbert spaces.
(The space $H^\alpha_\mix([0,1]^d)$ for $\alpha \in \N$ is a reproducing kernel Hilbert space, see also Section~\ref{sec:rkhs}.)
Alternatively, the existence of such an optimal nonadaptive linear rule follows from~\cite[Theorem~4.7]{NW2008}, since the Sobolev unit ball is convex and symmetric.
Thus, even if $Q_{d,n}$ is nonlinear and adaptive, its worst-case error on $F_d$ is at least $1/C$ times that of the fixed nonadaptive linear rule on $H^\alpha_\mix([0,1]^d)$.
We can now invoke Theorem~\ref{thm:curse} applied to $Q^{X^0,\bsw^*}_{d,m}$ with $m \le n$.
\end{proof}

\paragraph{Final thoughts on unweighted spaces and transfer to periodic spaces.}
We finish by stating that $n^{-\alpha}$ is the optimal algebraic rate in one dimension.
Mixed smoothness cannot improve this power of~$n$ in higher dimensions.
A recent paper with lower bounds is~\cite{NP2025}.
This includes lower bounds for different variations of anchored function spaces.

Putting more restrictions on functions or increasing the norm reduces the ``size'' of the unit ball.
One such reduction of the unit ball is by means of the weighted norms that we introduced in Section~\ref{sec:weighted-spaces} and we will prove upper bounds in Section~\ref{sec:upper-bound-lattice-rules} to indeed show that we can get rid of the curse of dimensionality under appropriate summability conditions on these weights.
In Section~\ref{sec:upper-bound-lattice-rules} we do this for periodic functions since the topic of these lecture notes is lattice rules, but also weighted non-periodic function spaces are able to break the curse of dimensionality under appropriate summability conditions on the weights.

One could wonder if just switching to unweighted periodic function spaces, which is an extra condition on the functions, not just a rescaling of the norms, could also get rid of the curse?
Unfortunately, unweighted periodic spaces also suffer the curse of dimensionality.
This is actually immediate by looking at the proofs of Theorem~\ref{thm:Hmixalpha-lower-bound} and Theorem~\ref{thm:curse}.
The functions $g_{d,\alpha,t,\bsc} \in H^\alpha_\mix([0,1]^d)$ we construct for the lower bound in Theorem~\ref{thm:Hmixalpha-lower-bound}, and all their derivatives through the required order, agree at the opposite faces of the unit cube, since they vanish there.
Hence $g_{d,\alpha,t,\bsc} \in \widetilde{H}^\alpha_\mix([0,1]^d) \subset H^\alpha_\mix([0,1]^d)$.
The same is true for the function $h_{d,n}$ in Theorem~\ref{thm:curse}, we have $h_{d,n} \in \widetilde{H}^\alpha_\mix([0,1]^d) \subset H^\alpha_\mix([0,1]^d)$.
We will formally state this later in Theorem~\ref{thm:periodic-lower-bound-curse}.

\paragraph{Two technical lemmas.}
We now show two lemmas that were used in the proofs of Theorems~\ref{thm:Hmixalpha-lower-bound} and~\ref{thm:curse}.
The following lemma was used in the proof of Theorem~\ref{thm:Hmixalpha-lower-bound}.

\begin{lemma}\label{lem:phialpha-derivative-inner-product}
Let $\alpha \in \N$, $v, v' \in \N_0$, $a \in \Z_{2^v}$, and $a' \in \Z_{2^{v'}}$, $\tau \in \{0, \ldots, \alpha\}$ and let $\phi_\alpha$ be defined by~\eqref{eq:phialpha}.
Define the dyadic open intervals $J := (a/2^v, (a+1)/2^v)$ and $J' := (a'/2^{v'}, (a'+1)/2^{v'})$ such that they are the support of $\phi_\alpha^{(\tau)}(2^v x - a)$ and $\phi_\alpha^{(\tau)}(2^{v'} x - a')$.
Then
\begin{align*}
  \numberthis{eq:phialpha-derivative-inner-product}
  \journalonly{\MoveEqLeft[1]} \left|
  \int_0^1
  2^{-v \, (\alpha-\tau)} \, \phi_\alpha^{(\tau)}(2^v x - a) \,
  2^{-v' \, (\alpha-\tau)} \, \phi_\alpha^{(\tau)}(2^{v'} x - a')
  \rd{x}
  \right|
  \journalonly{\\* }&\le
  C_\alpha \,
  2^{-\max\{v, v'\}} \,
  2^{-|v - v'|}
  ,
\end{align*}
with
\begin{align}\label{eq:Calpha}
  C_\alpha
  &:=
  \frac{((\alpha+1)!)^2}{2\alpha+3}
  .
\end{align}
If the open intervals $J$ and $J'$ are disjoint, then the integral vanishes.
\end{lemma}
\begin{proof}
From the definition~\eqref{eq:phialpha} of $\phi_\alpha$ we note that its derivatives up to order $\alpha$ vanish at the endpoints of its support; for the $\alpha$th derivative this is a consequence of us defining $\phi_\alpha$ with a power $\alpha+1$ instead of $\alpha$.
Thus, for our choice of $\phi_\alpha$,
\begin{align}\label{eq:phialpha-alpha-zero-mean}
  \int_0^1 \phi_\alpha^{(\tau)}(x) \rd{x}
  &=
  \phi_\alpha^{(\tau-1)}(x) \Big|_0^1
  =
  0
  ,
  \qquad
  \text{for } 1 \le \tau \le \alpha
  .
\end{align}
The $(\alpha+1)$st derivative can be written in terms of the Legendre polynomial of degree $\alpha+1$ as follows
\begin{align*}
  \phi_\alpha^{(\alpha+1)}(x)
  &=
  (\alpha+1)! \, P_{\alpha+1}(1-2x)
  ,
  \qquad
  0 < x < 1
  .
\end{align*}
Using properties of the Legendre polynomials we know that
\begin{align*}
  \phi_\alpha^{(\alpha)}(x)
  &=
  \frac{(\alpha+1)!}{2(2\alpha+3)}
  \left(P_\alpha(1-2x) - P_{\alpha+2}(1-2x)\right)
  .
\end{align*}
Further using $|P_\ell(x)| \le 1$ for $-1 \le x \le 1$ and $\ell \in \N_0$ we can conclude that
\begin{align}\label{eq:phialpha-Linfty-bounds}
  \|\phi_\alpha^{(\alpha)}\|_{L_\infty([0,1])}
  &\le
  \frac{(\alpha+1)!}{2\alpha+3}
  ,
  &
  \|\phi_\alpha^{(\alpha+1)}\|_{L_\infty([0,1])}
  &\le
  (\alpha+1)!
  ,
\end{align}
where the latter bound gives us a bound on the Lipschitz constant for~$\phi_\alpha^{(\alpha)}$ on~$(0, 1)$.
Since $\phi_\alpha^{(\tau)}$ vanishes at both endpoints for all $0 \le \tau \le \alpha$, the mean value theorem gives $\|\phi_\alpha^{(\tau)}\|_{L_\infty([0,1])} \le 2^{-1} \, \|\phi_\alpha^{(\tau+1)}\|_{L_\infty([0,1])}$, which together with the bound $\|\phi_\alpha^{(\alpha)}\|_{L_\infty([0,1])} \le (\alpha+1)!/(2\alpha+3)$ gives
\begin{align}\label{eq:phialpha-Linfty-bounds2}
  \|\phi_\alpha\|_{L_\infty([0,1])}
  \le
  2^{-1} \,
  \|\phi^{(1)}_\alpha\|_{L_\infty([0,1])}
  \le
  \cdots
  \le
  2^{-\alpha} \,
  \|\phi^{(\alpha)}_\alpha\|_{L_\infty([0,1])}
  \le
  2^{-\alpha} \, \frac{(\alpha+1)!}{2\alpha+3}
  .
\end{align}

If the interiors of the two dyadic intervals are disjoint, then the integral~\eqref{eq:phialpha-derivative-inner-product} vanishes, otherwise the intervals are nested.
The final panel of Figure~\ref{fig:dyadic-partition} illustrates this nesting that happens separately in each coordinate even though neither box contains the other.
First assume $1 \le \tau \le \alpha$.
By symmetry, suppose that $v' \ge v$, i.e., $J' \subseteq J$, put $k := v' - v$ and let $r \in \{0, \ldots, 2^k - 1\}$ be such that $a' = 2^k a + r$.
The substitution $y = 2^{v'} x - a' \in [0,1]$ gives
\begin{align*}
  \journalonly{\MoveEqLeft[1]} \int_0^1
  \phi_\alpha^{(\tau)}(2^{v'} x - a') \,
  \phi_\alpha^{(\tau)}(2^v x - a)
  \rd{x}
  \journalonly{\\*}&=
  2^{-v'}
  \int_0^1
  \phi_\alpha^{(\tau)}(y) \,
  \phi_\alpha^{(\tau)}(2^{-k} (r + y))
  \rd{y}
  \\
  &=
  2^{-v'}
  \int_0^1
  \phi_\alpha^{(\tau)}(y)
  \left[
    \phi_\alpha^{(\tau)}(2^{-k} (r + y))
    -
    \phi_\alpha^{(\tau)}(2^{-k} r)
  \right]
  \rd{y}
  ,
\end{align*}
where in the last equality we subtracted the constant $\phi_\alpha^{(\tau)}(2^{-k}r)$, which does not change the integral because $\phi_\alpha^{(\tau)}(y)$ has integral zero over $[0,1]$ by~\eqref{eq:phialpha-alpha-zero-mean} since we assumed $1 \le \tau \le \alpha$.
We now have
\begin{align*}
  \left|
    \phi_\alpha^{(\tau)}(2^{-k} (r + y))
    -
    \phi_\alpha^{(\tau)}(2^{-k} r)
  \right|
  &\le
  \|\phi_\alpha^{(\tau+1)}\|_{L_\infty([0,1])} \articleonly{\,} 2^{-k} \articleonly{\,} y
  \le
  2^{-k} \articleonly{\,} \|\phi_\alpha^{(\tau+1)}\|_{L_\infty([0,1])}
  .
\end{align*}
For $1 \le \tau < \alpha$, iterating~\eqref{eq:phialpha-Linfty-bounds2} bounds both $\|\phi_\alpha^{(\tau)}\|_{L_\infty([0,1])}$ and $\|\phi_\alpha^{(\tau+1)}\|_{L_\infty([0,1])}$ in terms of $\|\phi_\alpha^{(\alpha)}\|_{L_\infty([0,1])}$.
For $\tau = \alpha$, we instead use both bounds in~\eqref{eq:phialpha-Linfty-bounds}.
Thus, in either case,
\begin{align*}
  \journalonly{\MoveEqLeft[1]} \left|
  \int_0^1
  \phi_\alpha^{(\tau)}(2^{v'} x - a') \,
  \phi_\alpha^{(\tau)}(2^v x - a)
  \rd{x}
  \right|
  \journalonly{\\*}&\le
  2^{-\max\{v, v'\}} \, 2^{-|v - v'|} \,
  \|\phi^{(\tau)}_\alpha\|_{L_\infty([0,1])} \,
  \|\phi^{(\tau+1)}_\alpha\|_{L_\infty([0,1])}
  \\
  &\le
  \frac{((\alpha+1)!)^2}{2\alpha+3} \,
  2^{-\max\{v, v'\}} \, 2^{-|v - v'|} \,
  .
\end{align*}
The result for $1 \le \tau \le \alpha$ now follows by using $2^{-v (\alpha - \tau)} \le 1$, and similarly for~$v'$.

For $\tau = 0$, again assume $v' \ge v$, set $k := v' - v$, and let $r \in \{0, \ldots, 2^k - 1\}$ be such that $a' = 2^k a + r$.
We then have
\begin{align*}
  \left|
  \int_0^1
  \phi_\alpha(2^{v'} x - a') \,
  \phi_\alpha(2^v x - a)
  \rd{x}
  \right|
  &=
  2^{-v'}
  \left|
  \int_0^1
  \phi_\alpha(y) \,
  \phi_\alpha(2^{-k} (r + y))
  \rd{y}
  \right|
  \\
  &\le
  2^{-v'} \, \|\phi_\alpha\|_{L_\infty([0,1])}^2
  \\
  &\le
  C_\alpha \, 2^{-\max\{v, v'\}}
  .
\end{align*}
Here the last inequality follows from~\eqref{eq:phialpha-Linfty-bounds2} and the fact that clearly $2^{-2\alpha} / (2\alpha+3) \le 1$.
We finish the proof by noticing that $v + v' \ge |v - v'|$ and hence, for $\alpha \ge 1$, we have $2^{-v \, \alpha} \, 2^{-v' \, \alpha} \le 2^{-|v - v'|}$.
\end{proof}

The following lemma was used in the proof of Theorem~\ref{thm:curse}.

\begin{lemma}\label{lem:double-sum}
Let $d, n \in \N$, let $\eta \in (0,1)$, and let $\bsx_1, \ldots, \bsx_n \in [0,1]^d$ and $w_1, \ldots, w_n \in \R$.
Then
\begin{align}\label{eq:double-sum-lower-bound}
  \sum_{i,i'=1}^n w_i \, w_{i'}
  \prod_{j=1}^d \left(1 + \eta \, \cos(2\pi(x_{i,j} - x_{i',j}))\right)
  &\ge
  \frac{(1 + \eta)^d}{n}
  \left(\sum_{i=1}^n w_i\right)^2
  .
\end{align}
\end{lemma}
\begin{proof}
Let $C \in \R^{n \times n}$ be a real symmetric positive semidefinite matrix with ones on its diagonal.
Then its eigenvalue decomposition $C = V \, \Lambda \, V^\mathsf{T}$ has non-negative eigenvalues and $V$ is an orthogonal matrix.
Define $Z := V \, \Lambda^{1/2}$, then $C = Z \, Z^\mathsf{T}$.
It follows that $C_{i,i'} = \sum_{r=1}^n Z_{i,r} \, Z_{i',r}$ and $C_{i,i} = \sum_{r=1}^n Z_{i,r}^2 = 1$.
Then
\begin{align*}
  \sum_{i,i'=1}^n \journalonly{\!} w_i \, w_{i'} \, C_{i,i'}^2
  &=
  \sum_{i,i'=1}^n \journalonly{\!} w_i \, w_{i'} \, \left(\sum_{r=1}^n Z_{i,r} \, Z_{i',r}\right)^2
  \\
  &=
  \sum_{i,i'=1}^n \journalonly{\!} w_i \, w_{i'} \, \sum_{r,s=1}^n Z_{i,r} \, Z_{i',r} \, Z_{i,s} \, Z_{i',s}
  =
  \sum_{r,s=1}^n
  \left(\sum_{i=1}^n w_i \, Z_{i,r} \, Z_{i,s}\right)^2
  \\
  &\ge
  \sum_{r=1}^n
  \left(\sum_{i=1}^n w_i \, Z_{i,r}^2\right)^2
  \ge
  \frac1n
  \left(\sum_{r=1}^n \sum_{i=1}^n w_i \, Z_{i,r}^2\right)^2
  =
  \frac1n \left(\sum_{i=1}^n w_i\right)^2
  ,
\end{align*}
where the first inequality follows by dropping all the off-diagonal terms, the second inequality is Cauchy--Schwarz and the last equality uses $\sum_{r=1}^n Z_{i,r}^2 = C_{i,i} = 1$.

For $\setu \subseteq \{1,\ldots,d\}$, define $C_\setu \in \R^{n \times n}$ by
\begin{align*}
  (C_\setu)_{i,i'}
  &:=
  \prod_{j \in \setu}
  \cos(\pi (x_{i,j} - x_{i',j}))
  .
\end{align*}
Each $C_\setu$ is positive semidefinite and has ones on its diagonal: every individual cosine matrix is positive semidefinite by $\cos(a - b) = \cos(a) \cos(b) + \sin(a) \sin(b)$, and their entrywise product is again positive semidefinite by the Schur product theorem.
Using $1 + \eta \cos(2x) = 1 - \eta + 2 \eta \cos^2(x)$ and expanding the product gives
\begin{align*}
  \prod_{j=1}^d
  \left(1 + \eta \, \cos(2\pi (x_{i,j} - x_{i',j}))\right)
  &=
  \sum_{\setu \subseteq \{1,\ldots,d\}}
  (1 - \eta)^{d-|\setu|} \, (2\eta)^{|\setu|} \, (C_\setu)_{i,i'}^2
  .
\end{align*}
Applying the first part of the proof to every $C_\setu$ therefore gives
\begin{align*}
  \journalonly{\MoveEqLeft[1]} \sum_{i,i'=1}^n w_i \, w_{i'}
  \prod_{j=1}^d (1 + \eta \, \cos(2\pi (x_{i,j} - x_{i',j})))
  \journalonly{\\*}&\ge
  \frac1n \left(\sum_{i=1}^n w_i\right)^2
  \sum_{\setu \subseteq \{1,\ldots,d\}}
  (1 - \eta)^{d-|\setu|} \, (2\eta)^{|\setu|}
  \\
  &=
  \frac{(1 + \eta)^d}{n}
  \left(\sum_{i=1}^n w_i\right)^2
  ,
\end{align*}
where the last equality is the binomial theorem.
\end{proof}

\subsection{Multivariate periodic mixed smoothness spaces}\label{sec:periodic-spaces}

Since we again introduce several function spaces, we first give a few pointers to the literature.
The spaces that we define are the multivariate versions of the univariate spaces from Section~\ref{sec:periodic-spaces-d=1}.
We give a sneak peek in Table~\ref{tbl:periodic-equivalent-norms}.
This table gives an overview of the notation and also of the Fourier multipliers that we use to define their norms.
Again, depending on aim, authors and topic, one or the other function space is chosen in the literature to prove things.
Sometimes this makes comparison hard.
Hopefully these notes make it easier to understand how all these spaces connect, and maybe serve as a citable reference should such a comparison be necessary at some point.

Most of the references given above for the univariate periodic norms in Section~\ref{sec:periodic-spaces-d=1} already treat these corresponding multivariate periodic spaces with dominating mixed smoothness, so naturally many of the same references reappear here.
E.g., the spaces $\widetilde{H}^\alpha_\mix([0,1]^d)$, $\widetilde{H}'^\alpha_\mix([0,1]^d)$ and $\widetilde{H}''^\alpha_\mix([0,1]^d)$ are treated in~\cite{KSU2015}, see also~\cite{KUV2021,KV2019}, while the unanchored spaces $\widetilde{H}^{\alpha,\rmu}_\mix([0,1]^d)$ and $\widetilde{H}^{\alpha,\rmu}_{\mix,1}([0,1]^d)$, particularly the scale-$1$ version commonly called a Korobov space, can be found in, e.g.,~\cite{SJ94,NW2008,NW2010,DHP2015,NS2023,DKP2022}.
For their $\bsgamma$-weighted counterparts listed in Table~\ref{tbl:periodic-equivalent-norms}, see, e.g.,~\cite{Kuo2003,DSWW2004,DSWW2006,DP2010,DNP2014,DKS2013,Nuy2014,KPV2015,GHHR2017,DKP2022}.

\begin{table}[htbp]
  \centering
  \journalonly{\setlength{\abovecaptionskip}{8pt}}
  \renewcommand{\arraystretch}{1.4}
  \setlength{\tabcolsep}{4pt}
  \ifjournalstyle\begin{adjustbox}{max width=\linewidth}\fi
  \begin{tabular}{llll}
    \hline
    \multicolumn{4}{c}{unweighted} \\
    norm & & space $F$ & squared Fourier multiplier $r_F^2(\bsh)$ \\
    \hline
    standard & $\alpha \in \N_0$ & $\widetilde{H}^\alpha_\mix([0,1]^d)$ & $\prod_{j=1}^d \left(1+\sum_{\tau=1}^\alpha |2\pi h_j|^{2\tau}\right)$ \\
    first and last & $\alpha > 0$ & $\widetilde{H}'^\alpha_\mix([0,1]^d)$ & $\prod_{j=1}^d \left(1+|2\pi h_j|^{2\alpha}\right)$ \\
    first and last, scale $1$ & $\alpha > 0$ & $\widetilde{H}'^\alpha_{\mix,1}([0,1]^d)$ & $\prod_{j=1}^d \left(1+|h_j|^{2\alpha}\right)$ \\
    Bessel potential & $\alpha \ge 0$ & $\widetilde{H}''^\alpha_\mix([0,1]^d)$ & $\prod_{j=1}^d \left(1+|2\pi h_j|^2\right)^\alpha$ \\
    Bessel potential, scale $1$ & $\alpha \ge 0$ & $\widetilde{H}''^\alpha_{\mix,1}([0,1]^d)$ & $\prod_{j=1}^d \left(1+|h_j|^2\right)^\alpha$ \\
    unanchored & $\alpha \ge 0$ & $\widetilde{H}^{\alpha,\rmu}_\mix([0,1]^d)$ & $\prod_{j=1}^d \max\{1,|2\pi h_j|\}^{2\alpha}$ \\
    unanchored, scale $1$ & $\alpha \ge 0$ & $\widetilde{H}^{\alpha,\rmu}_{\mix,1}([0,1]^d)$ & $\prod_{j=1}^d \max\{1,|h_j|\}^{2\alpha}$ \\
    \hline
    \noalign{\vskip 0.5em}
    \hline
    \multicolumn{4}{c}{general weights $\bsgamma$} \\
    norm & & space $F$ & squared Fourier multiplier $r_F^2(\bsh)$ \\
    \hline
    standard & $\alpha \in \N_0$ & $\widetilde{H}^{\alpha,\bsgamma}_\mix([0,1]^d)$ & $\sum_{\setu \subseteq \{1,\ldots,d\}} \gamma_\setu^{-1} \prod_{j \in \setu} \left(\sum_{\tau=1}^\alpha |2\pi h_j|^{2\tau}\right)$ \\
    first and last & $\alpha > 0$ & $\widetilde{H}'^{\alpha,\bsgamma}_\mix([0,1]^d)$ & $\sum_{\setu \subseteq \{1,\ldots,d\}} \gamma_\setu^{-1} \prod_{j \in \setu} |2\pi h_j|^{2\alpha}$ \\
    first and last, scale $1$ & $\alpha > 0$ & $\widetilde{H}'^{\alpha,\bsgamma}_{\mix,1}([0,1]^d)$ & $\sum_{\setu \subseteq \{1,\ldots,d\}} \gamma_\setu^{-1} \prod_{j \in \setu} |h_j|^{2\alpha}$ \\
    Bessel potential & $\alpha \ge 0$ & $\widetilde{H}''^{\alpha,\bsgamma}_\mix([0,1]^d)$ & $\sum_{\setu \subseteq \{1,\ldots,d\}} \gamma_\setu^{-1} \prod_{j \in \setu} \left(\left(1+|2\pi h_j|^2\right)^\alpha-1\right)$ \\
    Bessel potential, scale $1$ & $\alpha \ge 0$ & $\widetilde{H}''^{\alpha,\bsgamma}_{\mix,1}([0,1]^d)$ & $\sum_{\setu \subseteq \{1,\ldots,d\}} \gamma_\setu^{-1} \prod_{j \in \setu} \left(\left(1+|h_j|^2\right)^\alpha-1\right)$ \\
    unanchored & $\alpha \ge 0$ & $\widetilde{H}^{\alpha,\bsgamma,\rmu}_\mix([0,1]^d)$ & $\gamma_{\supp(\bsh)}^{-1} \prod_{j \in \supp(\bsh)} |2\pi h_j|^{2\alpha}$ \\
    unanchored, scale $1$ & $\alpha \ge 0$ & $\widetilde{H}^{\alpha,\bsgamma,\rmu}_{\mix,1}([0,1]^d)$ & $\gamma_{\supp(\bsh)}^{-1} \prod_{j \in \supp(\bsh)} |h_j|^{2\alpha}$ \\
    \hline
    \noalign{\vskip 0.5em}
    \hline
    \multicolumn{4}{c}{product weights $\gamma_\setu = \prod_{j \in \setu} \gamma_j$} \\
    norm & & space $F$ & squared Fourier multiplier $r_F^2(\bsh)$ \\
    \hline
    standard & $\alpha \in \N_0$ & $\widetilde{H}^{\alpha,\bsgamma}_\mix([0,1]^d)$ & $\prod_{j=1}^d \left(1+\gamma_j^{-1}\sum_{\tau=1}^\alpha |2\pi h_j|^{2\tau}\right)$ \\
    first and last & $\alpha > 0$ & $\widetilde{H}'^{\alpha,\bsgamma}_\mix([0,1]^d)$ & $\prod_{j=1}^d \left(1+\gamma_j^{-1}|2\pi h_j|^{2\alpha}\right)$ \\
    first and last, scale $1$ & $\alpha > 0$ & $\widetilde{H}'^{\alpha,\bsgamma}_{\mix,1}([0,1]^d)$ & $\prod_{j=1}^d \left(1+\gamma_j^{-1}|h_j|^{2\alpha}\right)$ \\
    Bessel potential & $\alpha \ge 0$ & $\widetilde{H}''^{\alpha,\bsgamma}_\mix([0,1]^d)$ & $\prod_{j=1}^d \left(1+\gamma_j^{-1}\left(\left(1+|2\pi h_j|^2\right)^\alpha-1\right)\right)$ \\
    Bessel potential, scale $1$ & $\alpha \ge 0$ & $\widetilde{H}''^{\alpha,\bsgamma}_{\mix,1}([0,1]^d)$ & $\prod_{j=1}^d \left(1+\gamma_j^{-1}\left(\left(1+|h_j|^2\right)^\alpha-1\right)\right)$ \\
    unanchored & $\alpha \ge 0$ & $\widetilde{H}^{\alpha,\bsgamma,\rmu}_\mix([0,1]^d)$ & $\prod_{j \in \supp(\bsh)} \gamma_j^{-1} \, |2\pi h_j|^{2\alpha}$ \\
    unanchored, scale $1$ & $\alpha \ge 0$ & $\widetilde{H}^{\alpha,\bsgamma,\rmu}_{\mix,1}([0,1]^d)$ & $\prod_{j \in \supp(\bsh)} \gamma_j^{-1} \, |h_j|^{2\alpha}$ \\
    \hline
  \end{tabular}%
  \ifjournalstyle\end{adjustbox}\fi
  \caption{Squared Fourier multipliers for the unweighted and weighted equivalent norms on the (weighted) dominating mixed smoothness spaces $\widetilde{\calH}^{\alpha,\bsgamma}_\mix([0,1]^d)$. For each space $F$ in the table, the norm has the Fourier form $\|f\|_F^2 = \sum_{\bsh \in \Z^d} |\hat{f}_\bsh|^2 \, r_F^2(\bsh)$.}
  \label{tbl:periodic-equivalent-norms}
\end{table}

We proceed in the same two steps as for the multivariate non-periodic spaces in Section~\ref{sec:non-periodic-spaces}: we first pass from univariate smoothness to dominating mixed smoothness by taking tensor products, and then introduce weights for subsets of coordinates as in Section~\ref{sec:weighted-spaces} where we did this for the non-periodic spaces already.
The aim is to obtain convergence upper bounds of the form $C_{d,\alpha,\bsgamma}^\mix(\epsilon) \, n^{-\alpha + \epsilon}$, for arbitrarily small $\epsilon > 0$, where the constant $C_{d,\alpha,\bsgamma}^\mix(\epsilon)$ might blow up for $\epsilon \to 0$ and, under certain summability conditions on the weights $\bsgamma$, we have $\sup_{d \in \N} C_{d,\alpha,\bsgamma}^\mix(\epsilon) \le C_{\alpha,\bsgamma}^\mix(\epsilon) < \infty$ independent of~$d$.
Here and below we retain the convention $\gamma_\emptyset = 1$ and assume $\gamma_\setu > 0$ for every~$\setu$.
We need $\gamma_\emptyset = 1$ so that the initial error, the error of the optimal algorithm that uses no information, is~$1$, see Section~\ref{sec:curse-definition}.
By assuming $\gamma_\setu > 0$ we do not have to consider special cases when writing $\gamma_\setu^{-1}$, and this also allows us to show equivalence between function spaces that differ only in the choice of the weights.

\subsubsection{Fourier coefficients and absolutely converging Fourier series}
For $f \in L_1([0,1]^d)$ and $\bsh \in \Z^d$, we define the multivariate Fourier coefficients by
\begin{align}\label{eq:Fourier-coeff-d}
  \hat{f}_\bsh
  &:=
  \int_{[0,1]^d} f(\bsx) \, \exp(-2\pi\rmi \, \bsh \cdot \bsx) \rd\bsx
  .
\end{align}
As in the univariate setting, it is useful to collect the functions with absolutely summable Fourier coefficients in the Wiener algebra
\begin{align}\label{eq:def-Wiener-algebra-d}
\boxed{
  A([0,1]^d)
  :=
  \Big\{
    f \in L_1([0,1]^d)
    :
    \|f\|_{A([0,1]^d)}
    :=
    \sum_{\bsh \in \Z^d} |\hat{f}_\bsh|
    <
    \infty
  \Big\}
  .
}
\end{align}
We identify $f \in A([0,1]^d)$ with its continuous representative given by its absolutely and uniformly convergent Fourier series.
We have $A([0,1]^d) \hookrightarrow C(\bbT([0,1]^d))$ with embedding constant~$1$ since, for all $\bsx \in [0,1]^d$,
\begin{align}\label{eq:A-embedded-in-C}
  |f(\bsx)|
  &=
  \left| \sum_{\bsh \in \Z^d} \hat{f}_\bsh \, \exp(2\pi\rmi \, \bsh\cdot\bsx)\right|
  \le
  \sum_{\bsh \in \Z^d} |\hat{f}_\bsh|
  <
  \infty.
\end{align}

\subsubsection{Weighted multivariate periodic spaces}
All the norms, and corresponding function spaces, in Table~\ref{tbl:periodic-equivalent-norms} are multivariate generalisations of Fourier multiplier norms of the univariate spaces from Section~\ref{sec:periodic-spaces-d=1}.
As before, a space $F$ in the table is defined by finiteness of its norm, and for these Fourier-based spaces they all take the form
\begin{align}\label{eq:periodic-mixed-Fourier-norm}
  \|f\|_F^2
  &=
  \sum_{\bsh \in \Z^d} |\hat{f}_\bsh|^2 \, r_F^2(\bsh)
  ,
\end{align}
where $r_F^2$ is the squared Fourier multiplier of that space.
We will use $r(\bsh)$ as a general symbol for Fourier multipliers and it takes a specific form depending on the function space.
Sometimes we write $r_F$ to make clear it is associated with a function space~$F$.
In case the space is a plain tensor product space, e.g., for the unweighted case (which also corresponds to the equal weight case $\gamma_\setu \equiv 1$), or when we have product weights $\gamma_\setu = \prod_{j\in\setu} \gamma_j$, then the multiplier is just the product of the (weighted) multipliers of the univariate spaces, while for general weights we have a weighted sum of the product of multipliers.

We will explain this in more detail shortly, but let us first look at some examples of Fourier multipliers from the table.
For the ``first and last, scale~$1$'' space, which we defined for $d=1$ in~\eqref{eq:norm-tildeH-prime-alpha-1-d=1}, with product weights the table shows
\begin{align}\label{eq:multiplier-tildeH-prime-alpha-1-ptw}
  r^2(\bsh)
  &=
  \prod_{j=1}^d (1 + \gamma_j^{-1} \, |h_j|^{2\alpha})
  ,
\end{align}
and for the same space with general weights the table shows
\begin{align}\label{eq:multiplier-tildeH-prime-alpha-1-gw}
  r^2(\bsh)
  &=
  \sum_{\setu \subseteq \{1,\ldots,d\}} \gamma_\setu^{-1} \prod_{j\in\setu} |h_j|^{2\alpha}
  .
\end{align}
As another example, for the ``unanchored, scale 1'' space, which we defined for $d=1$ in~\eqref{eq:norm-tildeH-alpha-u-1-d=1}, with product weights the table shows
\begin{align*}
  r^2(\bsh)
  &=
  \prod_{j \in \supp(\bsh)} \gamma_j^{-1} \, |h_j|^{2\alpha}
  ,
\end{align*}
and for the same space with general weights the table shows
\begin{align*}
  r^2(\bsh)
  &=
  \gamma_{\supp(\bsh)}^{-1} \, \prod_{j\in\supp(\bsh)} |h_j|^{2\alpha}
  .
\end{align*}

Clearly when we have product weights $\gamma_\setu = \prod_{j\in\setu} \gamma_j$ then for such choice of $\gamma_\setu$ we must have that~\eqref{eq:multiplier-tildeH-prime-alpha-1-gw} reduces to~\eqref{eq:multiplier-tildeH-prime-alpha-1-ptw}.
Let us quickly verify this.
We start by writing~\eqref{eq:multiplier-tildeH-prime-alpha-1-gw} as follows
\begin{align*}
  \sum_{\setu \subseteq \{1,\ldots,d\}} \gamma_\setu^{-1} \Big(\prod_{j \notin \setu} 1\Big) \Big(\prod_{j\in\setu} |h_j|^{2\alpha}\Big)
  &=
  \sum_{\setu \subseteq \{1,\ldots,d\}} \Big(\prod_{j \notin \setu} 1\Big) \Big(\prod_{j\in\setu} \gamma_j^{-1} \, |h_j|^{2\alpha}\Big)
  \\
  \numberthis{eq:sum-incl-excl}
  &=
  \prod_{j=1}^d (1 + \gamma_j^{-1} \, |h_j|^{2\alpha} )
  .
\end{align*}
Such identities occur quite often in the literature and so it is useful to be familiar with them.

Looking at Table~\ref{tbl:periodic-equivalent-norms} you should be quite confident that the unweighted and product weight versions, except maybe the Bessel potential ones, are indeed the correct generalisations of those univariate ones in Section~\ref{sec:periodic-spaces-d=1}, but the general weights versions might look slightly puzzling, except for the unanchored ones.
Therefore, let us take the ``Bessel potential, scale~$1$'' norm with general weights as an example and then also derive the product weights form from it.
The univariate Bessel potential norm with scaling~$1$ was introduced in~\eqref{eq:norm-tildeH-double-prime-alpha-1-d=1}.
Recall that the idea of the weight $\gamma_\setu$ is to be able to punish movement of the function for the subset of variables in~$\setu$.
For the non-periodic functions we did this by weighting the $L_2$-norms of the derivatives $D^\bstau f$ with $\supp(\bstau) = \setu$, see, \eqref{eq:norm-Halpha-bsgamma-mix}, \eqref{eq:norm-H-alpha-bsgamma-a} and~\eqref{eq:norm-H-alpha-bsgamma-u}.
Here we need to generalise in such a way that norms which can be expressed in terms of derivatives, i.e., when $\alpha \in \N$, will have exactly this effect.
For Fourier-based norms this means that the weights should act on the terms containing $h_j$, since these terms arise from derivatives of $\exp(2\pi\rmi \, h_j x_j)$.
For a moment, assume $\alpha \in \N$, and let us use the original Bessel potential norm with scaling $2\pi$, see~\eqref{eq:norm-tildeH-double-prime-alpha-d=1}, such that we can interpret the multipliers directly in terms of derivatives of $\exp(2\pi\rmi \, h x)$, then $(1 + |2\pi h|^2)^\alpha = \sum_{\tau=0}^\alpha \binom{\alpha}{\tau} \, |2\pi h|^{2\tau} = 1 + \sum_{\tau=1}^\alpha \binom{\alpha}{\tau} \, |2\pi h|^{2\tau}$.
Apart from the binomial coefficients this now matches~\eqref{eq:norm-tildeH-alpha-d=1} and it is clear that the part with the sum is where the norms of the derivatives come in and so this is the part where we want the weights to act.
We could leave it in this form, or we could rewrite $\sum_{\tau=1}^\alpha \binom{\alpha}{\tau} \, |2\pi h|^{2\tau} = -1 + \sum_{\tau=0}^\alpha \binom{\alpha}{\tau} \, |2\pi h|^{2\tau} = (1 + |2\pi h|^2)^\alpha - 1$.
You should now recognise the form in the table.
For $\alpha \in \N$, and back to scale~$1$, we have the following equivalent forms with general weights,
\begin{align*}
  \journalonly{\MoveEqLeft[3]} \sum_{\setu \subseteq \{1,\ldots,d\}} \gamma_\setu^{-1} \, \Big( \prod_{j\notin\setu} 1 \Big) \Big( \prod_{j\in\setu} \sum_{\tau=1}^\alpha \binom{\alpha}{\tau} \, |h_j|^{2\tau} \Big)
  \journalonly{\\*}&=
  \sum_{\setu \subseteq \{1,\ldots,d\}} \gamma_\setu^{-1} \,\prod_{j\in\setu} \Big( (1 + |h_j|^2)^\alpha - 1 \Big)
  .
\end{align*}
The right-hand side can now be interpreted for general $\alpha \ge 0$.
The product weights form can then be obtained by doing the same as in~\eqref{eq:sum-incl-excl}.

Let us explicitly highlight some of the norms from Table~\ref{tbl:periodic-equivalent-norms}.
For $\alpha \in \N_0$ we have the weighted periodic norm
\begin{align}\label{eq:norm-Htilde-alpha-bsgamma-mix}
\boxed{
  \|f\|_{\widetilde{H}^{\alpha,\bsgamma}_\mix([0,1]^d)}^2
  :=
  \sum_{\bsh \in \Z^d}
  |\hat{f}_\bsh|^2 \,
  \sum_{\setu \subseteq \{1,\ldots,d\}}
  \gamma_\setu^{-1} \,
  \prod_{j \in \setu}
  \left(
    \sum_{\tau=1}^\alpha |2\pi h_j|^{2\tau}
  \right)
  .
}
\end{align}
For real $\alpha \ge 0$ we have the weighted periodic unanchored norm
\begin{align}\label{eq:norm-Htilde-alpha-bsgamma-u}
\boxed{
  \|f\|_{\widetilde{H}^{\alpha,\bsgamma,\rmu}_\mix([0,1]^d)}^2
  :=
  \sum_{\bsh \in \Z^d}
  \gamma_{\supp(\bsh)}^{-1} \,
  |\hat{f}_\bsh|^2
  \prod_{j \in \supp(\bsh)} |2\pi h_j|^{2\alpha}
  ,
}
\end{align}
and its scale-$1$ variant
\begin{align}\label{eq:norm-Htilde-alpha-bsgamma-u-1}
\boxed{
  \|f\|_{\widetilde{H}^{\alpha,\bsgamma,\rmu}_{\mix,1}([0,1]^d)}^2
  :=
  \sum_{\bsh \in \Z^d}
  \gamma_{\supp(\bsh)}^{-1} \,
  |\hat{f}_\bsh|^2
  \prod_{j \in \supp(\bsh)} |h_j|^{2\alpha}
  .
}
\end{align}
Here, as in much of the QMC literature, a weighted Korobov space means this weighted unanchored periodic Sobolev space of dominating mixed smoothness, usually with the scale-$1$ norm, although the name is also used for closely related variants.
Korobov's 1959 paper~\cite{Kor59} actually uses an unweighted $\ell_\infty$-condition on the Fourier coefficients rather than the $\ell_2$-based norm that is now commonly called a Korobov space norm and that we use here.

We can now show in the next lemma that all spaces from Table~\ref{tbl:periodic-equivalent-norms} are equal as vector spaces and have equivalent norms.
As for $d=1$ we will denote the collection of all these spaces from Table~\ref{tbl:periodic-equivalent-norms} as vector spaces, without specifying the particular norm, by
\begin{align*}
\boxed{
  \widetilde{\calH}^{\alpha,\bsgamma}_\mix([0,1]^d)
  .
}
\end{align*}

\begin{lemma}\label{lem:periodic-equivalent-norms-d}
For $d \in \N$, $\alpha \ge 0$ in the respective ranges and weights $\bsgamma$ with $\gamma_\setu > 0$ for all~$\setu$, let $F_{d,\alpha,\bsgamma}$ be any of the spaces in Table~\textup{\ref{tbl:periodic-equivalent-norms}}.
For fixed $d$ and $\alpha$, all these spaces are continuously embedded into each other, also when using different weight families $\bsgamma$ and $\bsgamma'$ with $\gamma_\setu, \gamma'_\setu > 0$ for all~$\setu$.
When $\alpha > 1/2$ then we have the continuous embedding
\begin{align*}
  F_{d,\alpha,\bsgamma}
  &\hookrightarrow
  A([0,1]^d)
  \hookrightarrow
  C(\bbT([0,1]^d))
  .
\end{align*}
\end{lemma}
\begin{proof}
The pointwise univariate multiplier comparisons from Lemma~\ref{lem:periodic-equivalent-norms-d=1} easily extend coordinatewise to the unweighted and product-weight multipliers.
For general weights we can apply the same pointwise comparisons to every product in the sums over $\setu \subseteq \{1,\ldots,d\}$.
For different weight families we can follow the strategy from Lemma~\ref{lem:H-alpha-mix-weighted-equivalent-norms}.
This shows they are all continuously embedded into each other.

To show the continuous embedding into the Wiener algebra we can hence just pick one of these spaces.
For clarity let us write $r$ as the Fourier multiplier of our particular space, but let us also denote its dependence on $\alpha$ and $\bsgamma$ explicitly by writing $r_{\alpha,\bsgamma}$.
Consider the scale-$1$ unanchored multiplier
\begin{align*}
  r_{\alpha,\bsgamma}^2(\bsh)
  &=
  \gamma_{\supp(\bsh)}^{-1}
  \prod_{j \in \supp(\bsh)} |h_j|^{2\alpha}
  .
\end{align*}
When $\alpha > 1/2$ we have $\zeta(2\alpha) < \infty$ and hence
\begin{align*}
  \sum_{\bsh \in \Z^d} r_{\alpha,\bsgamma}^{-2}(\bsh)
  &=
  \sum_{\bsh \in \Z^d}
  \gamma_{\supp(\bsh)}
  \prod_{j \in \supp(\bsh)} |h_j|^{-2\alpha}
  \\
  &=
  \sum_{\setu \subseteq \{1,\ldots,d\}}
  \gamma_\setu
  \sum_{\substack{\bsh \in \Z^d \\ \supp(\bsh) = \setu}}
  \prod_{j \in \setu} |h_j|^{-2\alpha}
  \\
  &=
  \sum_{\setu \subseteq \{1,\ldots,d\}}
  \gamma_\setu \,
  \prod_{j \in \setu}
  2 \sum_{h_j \ge 1} |h_j|^{-2\alpha}
  \\
  &=
  \sum_{\setu \subseteq \{1,\ldots,d\}}
  \gamma_\setu \, (2 \, \zeta(2\alpha))^{|\setu|}
  <
  \infty
  .
\end{align*}
Hence Cauchy--Schwarz gives, for any $\bsx \in [0,1]^d$,
\begin{align}\label{eq:embedding-r}
  |f(\bsx)|
  \le
  \sum_{\bsh \in \Z^d} |\hat{f}_\bsh|
  &\le
  \left(
    \sum_{\bsh \in \Z^d} |\hat{f}_\bsh|^2 \, r_{\alpha,\bsgamma}^2(\bsh)
  \right)^{\!1/2}
  \left(
    \sum_{\bsh \in \Z^d} r_{\alpha,\bsgamma}^{-2}(\bsh)
  \right)^{\!1/2}
  <
  \infty
  .
\end{align}
This finishes the proof.
\end{proof}

The previous proof used the commonly occurring chain of inequalities~\eqref{eq:embedding-r} for our Fourier-based Hilbert norms; we just multiply and divide by a Fourier multiplier $r_{\alpha,\bsgamma}(\bsh) > 0$ and then apply Cauchy--Schwarz.
We stress the dependence on $\alpha$ and $\bsgamma$ by using those as subscripts for the multiplier, but the particular space is left implicit from the context to not overload the notation.
Repeating the argument with Hölder's inequality we obtain, for any Fourier multiplier $r_{\alpha,\bsgamma}$,
\begin{align}\label{eq:embedding-r-holder}
  |f(\bsx)|
  \le
  \sum_{\bsh \in \Z^d} |\hat{f}_\bsh|
  &\le
  \left(
    \sum_{\bsh \in \Z^d} |\hat{f}_\bsh|^p \, r_{\alpha,\bsgamma}^p(\bsh)
  \right)^{\!1/p}
  \left(
    \sum_{\bsh \in \Z^d} r_{\alpha,\bsgamma}^{-q}(\bsh)
  \right)^{\!1/q}
  ,
\end{align}
with $\nicefrac{1}{p} + \nicefrac{1}{q} = 1$ and $1 \le p, q \le \infty$ (and the proper adjustment in case one of $p$ or $q$ is infinity).
We could then take the first factor as an $\ell_p$-norm defining a function space $F^{\alpha,\bsgamma}_{r,p}([0,1]^d)$ with norm
\begin{align}\label{eq:Fourier-ell-p}
  \|f\|_{F^{\alpha,\bsgamma}_{r,p}([0,1]^d)}
  &=
  \begin{cases}
  \left( \sum_{\bsh \in \Z^d} |\hat{f}_\bsh|^p \, r_{\alpha,\bsgamma}^p(\bsh) \right)^{\!1/p}, & \text{if } 1 \le p < \infty, \\[4mm]
  \sup_{\bsh \in \Z^d} |\hat{f}_\bsh| \, r_{\alpha,\bsgamma}(\bsh), & \text{if } p = \infty .
  \end{cases}
\end{align}
This is, e.g., discussed in~\cite{Nuy2014}.
For $1 \le q \le \infty$, the condition for continuous embedding into $A([0,1]^d)$ and then $C(\bbT([0,1]^d))$, becomes $r_{\alpha,\bsgamma}^{-1} \in \ell_q(\Z^d)$.
For $1 \le q < \infty$, this means
\begin{align*}
  \left(
    \sum_{\bsh \in \Z^d} r_{\alpha,\bsgamma}^{-q}(\bsh)
  \right)^{\!1/q}
  <
  \infty
  ,
\end{align*}
and for all Fourier multipliers in Table~\ref{tbl:periodic-equivalent-norms} this holds when $\alpha > 1/q$ (for the same reasoning as in the previous proof as we will have $\zeta(q \alpha)$ popping up).
For $q = \infty$ the condition is $\sup_{\bsh \in \Z^d} r_{\alpha,\bsgamma}^{-1}(\bsh) < \infty$ which holds for $\alpha \ge 0$ for the norms defined there.
In fact $p = 1$ ($q = \infty$) and $\alpha = 0$ results in the Wiener algebra for any of the Fourier multipliers in Table~\ref{tbl:periodic-equivalent-norms} that allow $\alpha = 0$.
Note however that for the weighted unanchored norms the weights will still show up even when $\alpha = 0$.
This is also the case for the Hilbert space setting with $p = 2$.

Since our Fourier multiplier function includes the weights $\bsgamma$, this could be a good reason to square the weights $\gamma_\setu$ appearing in the squared multipliers in Table~\ref{tbl:periodic-equivalent-norms}, which is done in some of the literature, but not in others, and we also do not do it here.

\subsubsection{``Periodicity'' in the multivariate spaces}
We now clarify what we mean by ``periodic'' functions in all the spaces $\widetilde{\calH}^{\alpha,\bsgamma}_\mix([0,1]^d)$.
We extend Lemma~\ref{lem:periodic-derivatives-d=1}.

\begin{lemma}[Periodicity implied by the Fourier norms]\label{lem:periodic-derivatives-d}
Given $d \in \N$, $\alpha \ge 0$, weights $\bsgamma$ with $\gamma_\setu > 0$ for all~$\setu$ and $f \in \widetilde{\calH}^{\alpha,\bsgamma}_\mix([0,1]^d)$, as a vector space, and any of the equivalent norms from Table~\textup{\ref{tbl:periodic-equivalent-norms}} with $\alpha$ restricted to the associated range.
For $\bstau \in \N_0^d$ with all $0 \le \tau_j \le \alpha$, we have $D^\bstau f \in L_2([0,1]^d)$.
For $\bstau \in \N_0^d$ with all $0 \le \tau_j \le \alpha-1$, for $i \in \{1,\ldots,d\}$ the derivative $D^\bstau f$, regarded as a function of $x_i$, has an $L_2([0,1]^{d-1})$-valued absolutely continuous representative.
For $\bstau \in \N_0^d$ with all $0 \le \tau_j < \alpha-1/2$, the Fourier series of $D^\bstau f$ converges absolutely and uniformly, with
\begin{align}\label{eq:periodic-derivatives-Fourier-d}
  D^\bstau f(\bsx)
  &=
  \sum_{\bsh \in \Z^d}
  \hat{f}_\bsh \,
  \Big[ \prod_{j=1}^d (2\pi\rmi \, h_j)^{\tau_j} \Big] \,
  \exp(2\pi\rmi \, \bsh \cdot \bsx)
  \qquad
  \text{for } \bsx \in [0,1]^d
  .
\end{align}
I.e., for $\bstau \in \N_0^d$ with all $0 \le \tau_j < \alpha-1/2$, we have $D^\bstau f \in A([0,1]^d) \subset C(\bbT([0,1]^d))$, i.e., $D^\bstau f$ is continuous and its values agree when any coordinate is set to~$0$ and~$1$.
\end{lemma}
\begin{proof}
We can immediately assume $\alpha > 0$.
By Lemma~\ref{lem:periodic-equivalent-norms-d}, we may work with the unweighted product norm $\widetilde{H}'^\alpha_\mix([0,1]^d)$, with squared Fourier multiplier
\begin{align*}
  r_\alpha^2(\bsh)
  &=
  \prod_{j=1}^d (1 + |2\pi h_j|^{2\alpha})
  .
\end{align*}
For $\bstau \in \N_0^d$ with all $0 \le \tau_j \le \alpha$, we have, $\bsh \in \Z^d$,
\begin{align*}
  \prod_{j=1}^d |2\pi h_j|^{2\tau_j}
  &\le
  \prod_{j=1}^d (1 + |2\pi h_j|^{2\alpha})
  =
  r_\alpha^2(\bsh)
  ,
\end{align*}
hence, using Parseval and $f \in \widetilde{H}'^\alpha_\mix([0,1]^d)$,
\begin{align*}
  \|D^\bstau f\|_{L_2([0,1]^d)}^2
  &=
  \sum_{\bsh \in \Z^d}
  |\hat{f}_\bsh|^2 \, \prod_{j=1}^d |2\pi h_j|^{2\tau_j}
  \\
  &\le
  \sum_{\bsh \in \Z^d}
  |\hat{f}_\bsh|^2 \, r_\alpha^2(\bsh)
  =
  \|f\|_{\widetilde{H}'^\alpha_\mix([0,1]^d)}^2
  <
  \infty
  ,
\end{align*}
and hence $D^\bstau f \in L_2([0,1]^d)$.
It follows that if all $0 \le \tau_j \le \alpha-1$, then both $D^\bstau f$ and $D^{\bstau+\bse_i}f$ belong to $L_2([0,1]^d)$ for $i \in \{1,\ldots,d\}$.
Regarding $D^\bstau f$ as a function of $x_i$ with values in $L_2([0,1]^{d-1})$, its weak derivative is $D^{\bstau+\bse_i}f$, so it belongs to $H^1((0,1);L_2([0,1]^{d-1}))$ and therefore has an $L_2([0,1]^{d-1})$-valued absolutely continuous representative.

Now let all $0 \le \tau_j < \alpha-1/2$.
Applying the Cauchy--Schwarz argument from~\eqref{eq:embedding-r} to the Fourier coefficients of $D^\bstau f$ gives
\begin{align*}
  \sum_{\bsh \in \Z^d}
  |\hat{f}_\bsh| \, \prod_{j=1}^d |2\pi h_j|^{\tau_j}
  &\le
  \|f\|_{\widetilde{H}'^\alpha_\mix([0,1]^d)}
  \left(
    \sum_{\bsh \in \Z^d}
    \prod_{j=1}^d
    \frac{|2\pi h_j|^{2\tau_j}}{1 + |2\pi h_j|^{2\alpha}}
  \right)^{1/2}
  \\
  &=
  \|f\|_{\widetilde{H}'^\alpha_\mix([0,1]^d)}
  \prod_{j=1}^d
  \left(
    \sum_{h_j \in \Z}
    \frac{|2\pi h_j|^{2\tau_j}}{1 + |2\pi h_j|^{2\alpha}}
  \right)^{1/2}
  <
  \infty
  .
\end{align*}
The last product is finite since $\zeta(2(\alpha-\tau_j)) < \infty$ when $2(\alpha-\tau_j) > 1 \Leftrightarrow \tau_j < \alpha - 1/2$, by the same zeta-function estimate as in Lemma~\ref{lem:periodic-derivatives-d=1}.
Thus the Fourier series in~\eqref{eq:periodic-derivatives-Fourier-d} converges absolutely and uniformly and defines a continuous representative whose values agree when any coordinate is set to~$0$ and~$1$.
\end{proof}

\subsubsection{Relation with the non-periodic weighted mixed spaces for integer smoothness}
For integer $\alpha \in \N$ and $\bstau \in \N_0^d$, the condition $0 \le \tau_j < \alpha-1/2$ is equivalent to $0 \le \tau_j \le \alpha-1$, hence the preceding lemma shows that every mixed derivative $D^\bstau f$ with $\bstau \in \{0,\ldots,\alpha-1\}^d$ has a continuous representative whose values agree pointwise when $x_j$ is set to~$0$ and~$1$.
Thus, for every $j \in \{1,\ldots,d\}$, $\bstau \in \{0,\ldots,\alpha-1\}^d$ and $\bsx \in [0,1]^d$,
\begin{align}\label{eq:periodic-mixed-boundary-conditions}
  D^\bstau f(\bsx)\big|_{x_j=0}
  &=
  D^\bstau f(\bsx)\big|_{x_j=1}
  .
\end{align}
More generally, if $\bstau \in \{0,\ldots,\alpha\}^d$ and for some $j \in \{1,\ldots,d\}$ we have the stricter condition $\tau_j \le \alpha-1$, then the traces for $x_j=0$ and $x_j=1$ exist in $L_2([0,1]^{d-1})$ and satisfy
\begin{align}\label{eq:periodic-mixed-boundary-conditions-L2}
  D^\bstau f(\bsx)\big|_{x_j=0}
  &=
  D^\bstau f(\bsx)\big|_{x_j=1}
  \qquad\text{in}\quad L_2([0,1]^{d-1})
  .
\end{align}
This follows by regarding $D^\bstau f$ as an $L_2([0,1]^{d-1})$-valued $H^1$ function of~$x_j$ as in the proof of Lemma~\ref{lem:periodic-derivatives-d} and approximating it by its periodic Fourier partial sums in that space.
This equality holds almost everywhere, but it need not hold pointwise when one of the $\tau_i = \alpha$ for some $i \in \{1,\ldots,d\} \setminus \{j\}$.
If $\tau_j=\alpha$, then these traces are not guaranteed to exist and no boundary equality is asserted.
For $\alpha=0$, the boundary conditions are obviously void.

We now make exactly the same link as in the univariate section.
For $f \in \widetilde{H}^{\alpha,\bsgamma}_\mix([0,1]^d)$, repeated application of Parseval's identity as in Lemma~\ref{lem:periodic-Fourier-characterisation-d=1}, gives
\begin{align*}
  \|f\|_{\widetilde{H}^{\alpha,\bsgamma}_\mix([0,1]^d)}^2
  &=
  \sum_{\bstau \in \{0,\ldots,\alpha\}^d}
  \gamma_{\supp(\bstau)}^{-1} \,
  \|D^\bstau f\|_{L_2([0,1]^d)}^2
  =
  \|f\|_{H^{\alpha,\bsgamma}_\mix([0,1]^d)}^2
  .
\end{align*}
The coordinatewise version of the same univariate Fourier characterisation as in Lemma~\ref{lem:periodic-Fourier-characterisation-d=1} also proves the converse, and hence, we can equivalently define, for integer $\alpha \in \N_0$,
\begin{align}\label{eq:periodic-mixed-space-restriction}
  \widetilde{H}^{\alpha,\bsgamma}_\mix([0,1]^d)
  &=
  \Big\{
    f \in H^{\alpha,\bsgamma}_\mix([0,1]^d)
    :
    f \text{ satisfies~\eqref{eq:periodic-mixed-boundary-conditions} (or~\eqref{eq:periodic-mixed-boundary-conditions-L2})}
  \Big\}
  .
\end{align}
Compare with the univariate definitions~\eqref{eq:space-tildeH-alpha-d=1} and~\eqref{eq:space-tildeH-alpha-d=1-b}.
Thus $\widetilde{H}^{\alpha,\bsgamma}_\mix([0,1]^d)$ is exactly the restriction of $H^{\alpha,\bsgamma}_\mix([0,1]^d)$ and its Fourier-multiplier based norm agrees exactly with the inherited norm.

For integer $\alpha \in \N$, the periodic unanchored norm~\eqref{eq:norm-Htilde-alpha-bsgamma-u} is likewise exactly the restriction of the non-periodic unanchored norm~\eqref{eq:norm-H-alpha-bsgamma-u} to the periodic subspace conditions as in~\eqref{eq:periodic-mixed-space-restriction}.
Indeed, a term in~\eqref{eq:norm-H-alpha-bsgamma-u} vanishes whenever an integrated coordinate carries a derivative of order between $1$ and $\alpha-1$, because the traces of the preceding derivative agree when that coordinate is set to~$0$ and~$1$.
See also the discussion in~\cite{NS2023}.
Only the terms indexed by $\bstau \in \{0,\alpha\}^d$ remain, and Parseval's identity gives
\begin{align*}
  \|f\|_{\widetilde{H}^{\alpha,\bsgamma,\rmu}_\mix([0,1]^d)}^2
  &=
  \sum_{\bstau \in \{0, \alpha\}^d}
  \gamma_{\supp(\bstau)}^{-1} \,
  \|I_{-\supp(\bstau)} D^\bstau f\|_{L_2([0,1]^d)}^2
  ,
\end{align*}
where $-\supp(\bstau) = \{1,\ldots,d\} \setminus \supp(\bstau)$.
The scale-$1$ variant of this unanchored norm is obtained by rescaling the general weights according to
\begin{align*}
  \gamma_\setu^*
  &:=
  \frac{\gamma_\setu}{(2\pi)^{2\alpha|\setu|}}
  ,
\end{align*}
for which
\begin{align*}
  \|f\|_{\widetilde{H}^{\alpha,\bsgamma,\rmu}_\mix([0,1]^d)}
  &=
  \|f\|_{\widetilde{H}^{\alpha,\bsgamma^*,\rmu}_{\mix,1}([0,1]^d)}
  .
\end{align*}
For product weights this is simply $\gamma_j^* = \gamma_j/(2\pi)^{2\alpha}$.

\subsubsection{Transferring lower bounds and upper bounds}

Now that we have defined our ``periodic'' spaces and, for $\alpha \in \N$, made some connections to their corresponding non-periodic variants, we are ready to show that the lower bounds proved for the non-periodic spaces also transfer to the periodic spaces.

\begin{theorem}[Lower bound and curse of dimensionality for the unweighted periodic mixed spaces]\label{thm:periodic-lower-bound-curse}
Let $d, \alpha \in \N$ and let $F_{d,\alpha}$ be any of the seven unweighted spaces $\widetilde{H}^\alpha_\mix([0,1]^d)$, $\widetilde{H}'^\alpha_\mix([0,1]^d)$, $\widetilde{H}'^\alpha_{\mix,1}([0,1]^d)$, $\widetilde{H}''^\alpha_\mix([0,1]^d)$, $\widetilde{H}''^\alpha_{\mix,1}([0,1]^d)$, $\widetilde{H}^{\alpha,\rmu}_\mix([0,1]^d)$ and $\widetilde{H}^{\alpha,\rmu}_{\mix,1}([0,1]^d)$ listed in the first part of Table~\textup{\ref{tbl:periodic-equivalent-norms}}.
By Lemma~\textup{\ref{lem:periodic-equivalent-norms-d}}, there exist embedding constants $C_{d,\alpha,F} > 0$ such that
\begin{align*}
  \|f\|_{F_{d,\alpha}}
  &\le
  C_{d,\alpha,F} \, \|f\|_{\widetilde{H}^\alpha_\mix([0,1]^d)}
  .
\end{align*}
(In fact $C_{d,\alpha,F} \ge 1$.)
For $n \ge 2$, the lower bound from Theorem~\textup{\ref{thm:Hmixalpha-lower-bound}} then gives
\begin{align}\label{eq:periodic-mixed-lower-bound}
  \inf_{Q_{d,n} \in \frakQ_{d,n}}
  \sup_{\substack{f \in F_{d,\alpha} \\ \|f\|_{F_{d,\alpha}} \le 1}}
  |I_d(f) - Q_{d,n}(f)|
  &\ge
  \frac{c^\mix_{d,\alpha}}{C_{d,\alpha,F}} \, n^{-\alpha} \, (\log_2 n)^{(d-1)/2}
  .
\end{align}
Let $r_{1,\alpha,F}^2$ denote the corresponding univariate squared Fourier multiplier from Table~\textup{\ref{tbl:periodic-equivalent-norms}}, such that $r_{d,\alpha,F}^2(\bsh) = \prod_{j=1}^d r_{1,\alpha,F}^2(h_j)$ is the squared Fourier multiplier for the space $F_{d,\alpha}$, and set
\begin{align*}
  \eta_{\alpha,F}
  &:=
  \min\Big\{\frac12,\frac{2}{r_{1,\alpha,F}^2(1)}\Big\}
  \in
  (0,1)
  .
\end{align*}
For every cubature rule $Q_{d,n}(f) = \sum_{i=1}^n w_i \, f(\bsx_i)$ and $n \le (1 + \eta_{\alpha,F})^d$ we have
\begin{align}\label{eq:periodic-mixed-curse}
  \sup_{\substack{f \in F_{d,\alpha} \\ \|f\|_{F_{d,\alpha}} \le 1}}
  |I_d(f) - Q_{d,n}(f)|
  &\ge
  \left(1 - \frac{n}{(1 + \eta_{\alpha,F})^d}\right)^{1/2}
  .
\end{align}
Consequently, all seven unweighted periodic mixed spaces suffer from the curse of dimensionality.
\end{theorem}
\begin{proof}
The functions $g_{d,\alpha,t,\bsc}$ used in the proof of Theorem~\ref{thm:Hmixalpha-lower-bound}, together with all required derivatives, vanish on opposite faces of the cube, so taking care of the scaling of the unit ball by dividing by $C_{d,\alpha,F}$ proves~\eqref{eq:periodic-mixed-lower-bound}.
To prove the curse, we take the same steps as in the proof of Theorem~\ref{thm:curse} with $e_0(x) = 1$, $e_1(x) = \sqrt{\eta_{\alpha,F}} \, \cos(2\pi x)$ and $e_2(x) = \sqrt{\eta_{\alpha,F}} \, \sin(2\pi x)$, which are periodic and have norm at most one by the definition of $\eta_{\alpha,F}$.
\end{proof}

The previous theorem shows that the unweighted spaces in $\widetilde{\calH}^\alpha_\mix([0,1]^d)$ are tightly coupled and we could easily transfer lower bounds and the statement on the curse of dimensionality from the non-periodic spaces.
We have introduced weights as a possible route to tractability and avoiding the curse of dimensionality.
So it makes sense to also investigate if we can transfer upper bounds between these spaces.
We use the fact that the squared Fourier multiplier for $\widetilde{H}^{\alpha,\bsgamma,\rmu}_{\mix,1}([0,1]^d)$ is always the smallest for all $\bsh \in \Z^d$ (with the exception of the Bessel potential scale-$1$ when $1/2 < \alpha < 1$).
This enables us to transfer upper bounds for $\widetilde{H}^{\alpha,\bsgamma,\rmu}_{\mix,1}([0,1]^d)$ directly to those other spaces, as the next theorem shows.
If we assume a worst-case error bound for a lattice rule in the space $\widetilde{H}^{\alpha,\bsgamma,\rmu}_{\mix,1}([0,1]^d)$ of the form $C^\mix_{d,\alpha,\bsgamma}(\epsilon) \, n^{-\alpha+\epsilon}$, for $\epsilon > 0$, and the weights are such that we actually have a dimension-independent constant $\sup_{d \in \N} C_{d,\alpha,\bsgamma}^\mix(\epsilon) \le C_{\alpha,\bsgamma}^\mix(\epsilon) < \infty$, then this dimension-independent bound transfers immediately to those other spaces.
(The Bessel potential scale-$1$ space can be handled by rescaling the weights and, in the case of rescaled product weights, also permits transfer of the dimension-independent bound.)
We also show how to transfer such upper bounds to the periodic anchored space.
We show a result for general weights with bad dependence on~$d$, cf.\ Lemma~\ref{lem:H-alpha-mix-weighted-equivalent-norms}, and a result specialised for product weights, cf.\ Lemma~\ref{lem:H-1-mix-product-weight-embeddings}, that can be independent of~$d$ if and only if $\sum_{j \ge 1} \sqrt{\gamma_j} < \infty$.

\begin{theorem}[Transfer of upper bounds between weighted periodic mixed spaces]\label{thm:periodic-upper-bound-transfer}
Let $d, n \in \N$, $\alpha > 1/2$, $\bsgamma$ be a family of weights with $\gamma_\setu > 0$ for all~$\setu$ and let $Q_n$ be a cubature rule.
Suppose, we know an upper bound on the worst-case error for integration in the weighted unanchored space $\widetilde{H}^{\alpha,\bsgamma,\rmu}_{\mix,1}([0,1]^d)$ of the form, for $\epsilon > 0$,
\begin{align}\label{eq:assumed-bound}
  e(Q_n, \widetilde{H}^{\alpha,\bsgamma,\rmu}_{\mix,1}([0,1]^d))
  &\le
  C_{d,\alpha,\bsgamma}^\mix(\epsilon) \, n^{-\alpha+\epsilon}
  ,
\end{align}
then, for $F_{d,\alpha,\bsgamma}$ any weighted space from Table~\textup{\ref{tbl:periodic-equivalent-norms}}, and corresponding valid~$\alpha$, we have the same worst-case error upper bound
\begin{align}\label{eq:periodic-upper-bound-transfer}
  e(Q_n, F_{d,\alpha,\bsgamma})
  &\le
  C_{d,\alpha,\bsgamma}^\mix(\epsilon) \, n^{-\alpha+\epsilon}
  ,
\end{align}
where for $\widetilde{H}''^{\alpha,\bsgamma}_{\mix,1}([0,1]^d)$ we restrict this to $\alpha \ge 1$.
This means dimension-independent bounds transfer with the same constant.
For $\widetilde{H}''^{\alpha,\bsgamma}_{\mix,1}([0,1]^d)$ with $1/2 < \alpha < 1$, set $c_\alpha := 2^\alpha-1$ and define
\begin{align*}
  \gamma_\setu^*
  &:=
  \frac{\gamma_\setu}{c_\alpha^{|\setu|}}
  ,
\end{align*}
if~\eqref{eq:assumed-bound} holds with $\bsgamma$ replaced by $\bsgamma^*$, then
\begin{align}\label{eq:periodic-Bessel-upper-bound-transfer}
  e(Q_n, \widetilde{H}''^{\alpha,\bsgamma}_{\mix,1}([0,1]^d))
  &\le
  e(Q_n, \widetilde{H}^{\alpha,\bsgamma^*,\rmu}_{\mix,1}([0,1]^d))
  \le
  C_{d,\alpha,\bsgamma^*}^\mix(\epsilon) \, n^{-\alpha+\epsilon}
  ,
\end{align}
and in case $\bsgamma$ are product weights $\gamma_\setu = \prod_{j\in\setu} \gamma_j$ then also $\bsgamma^*$ are product weights with $\gamma_j^* = \gamma_j / c_\alpha$ and $\sum_{j \ge 1} \gamma_j^p < \infty \Leftrightarrow \sum_{j \ge 1} (\gamma_j^*)^p < \infty$, and dimension-independent bounds transfer whenever the assumed bound for $\bsgamma^*$ is uniform in~$d$.

For $\alpha \in \N$ and $\bsa \in [0,1]^d$, let $\widetilde{H}^{\alpha,\bsgamma,\bsa}_\mix([0,1]^d)$ denote the periodic subspace of the anchored Sobolev space $H^{\alpha,\bsgamma,\bsa}_\mix([0,1]^d)$ determined by the conditions in~\eqref{eq:periodic-mixed-boundary-conditions}, or~\eqref{eq:periodic-mixed-boundary-conditions-L2}, equipped with the inherited anchored norm~\eqref{eq:norm-H-alpha-bsgamma-a}, then the assumed bound~\eqref{eq:assumed-bound}, with $\rho_d(\bsgamma)$ defined in~\eqref{eq:def-rho-d}, gives
\begin{align}\label{eq:periodic-anchored-upper-bound-transfer}
  e(Q_n, \widetilde{H}^{\alpha,\bsgamma,\bsa}_\mix([0,1]^d))
  &\le
  \rho_d(\bsgamma) \, (2\alpha)^{d/2} \, C_{d,\alpha,\bsgamma}^\mix(\epsilon) \, n^{-\alpha+\epsilon}
  ,
\end{align}
which scales badly with~$d$.
However, in case of product weights $\gamma_\setu = \prod_{j\in\setu} \gamma_j$ we have
\begin{align}\label{eq:periodic-anchored-product-upper-bound-transfer}
  e(Q_n, \widetilde{H}^{\alpha,\bsgamma,\bsa}_\mix([0,1]^d))
  &\le
  \kappa_{d,\alpha}(\bsgamma) \, C_{d,\alpha,\bsgamma}^\mix(\epsilon) \, n^{-\alpha+\epsilon}
  ,
\end{align}
with $m_\alpha := 2 \, (2\pi)^{-2\alpha} \, \zeta(2\alpha)$ and
\begin{align*}
  \numberthis{eq:def-kappa-d-alpha}
  \kappa_{d,\alpha}(\bsgamma)
  &:=
  \prod_{j=1}^d
  \frac{\sqrt{4 + \gamma_j \, m_\alpha} + \sqrt{\gamma_j \, m_\alpha}}{2}
  \le
  \prod_{j=1}^d
  \frac{\sqrt{4 + \gamma_j/12} + \sqrt{\gamma_j/12}}{2}
  ,
\end{align*}
where $\kappa_{d,\alpha}(\bsgamma)$ is independent of the anchor and is uniformly bounded in~$d$ if and only if $\sum_{j \ge 1} \sqrt{\gamma_j} < \infty$ (cf.\ $\kappa_d(\bsgamma,\bsa)$ in~\eqref{eq:def-kappa-d} for $\alpha = 1$) and hence dimension-independent bounds transfer.
\end{theorem}
\begin{proof}
For each of the spaces $F_{d,\alpha,\bsgamma}$, its squared Fourier multiplier in Table~\ref{tbl:periodic-equivalent-norms} is pointwise at least the squared multiplier of $\widetilde{H}^{\alpha,\bsgamma,\rmu}_{\mix,1}([0,1]^d)$, with the exception of $\widetilde{H}''^{\alpha,\bsgamma}_{\mix,1}([0,1]^d)$ when $1/2 < \alpha < 1$.
For $\widetilde{H}''^{\alpha,\bsgamma}_{\mix,1}([0,1]^d)$ with $\alpha \ge 1$ we use~\eqref{eq:fake-jensen-2} to obtain
\begin{align*}
  (1 + |h|^2)^\alpha - 1
  &\ge
  |h|^{2\alpha}
  ,
\end{align*}
and for $\widetilde{H}''^{\alpha,\bsgamma}_\mix([0,1]^d)$ with $\alpha > 1/2$ we use, for $0 \ne h \in \Z$,
\begin{align*}
  (1 + |2\pi h|^2)^\alpha - 1
  &\ge
  |h|^{2\alpha}
  ,
\end{align*}
which follows from $1 + |2\pi h|^2 \ge 4 \, |h|^2$, $|h|^{2\alpha} \ge 1$ and $4^\alpha - 1 > 1$.
This implies for all those cases the embedding $F_{d,\alpha,\bsgamma} \hookrightarrow \widetilde{H}^{\alpha,\bsgamma,\rmu}_{\mix,1}([0,1]^d)$ with constant~$1$, since
\begin{align*}
  \|f\|_{\widetilde{H}^{\alpha,\bsgamma,\rmu}_{\mix,1}([0,1]^d)}
  &\le
  \|f\|_{F_{d,\alpha,\bsgamma}}
  .
\end{align*}
We can now apply~\eqref{eq:transplant-upper-bound} to transplant the worst-case upper bound for $\widetilde{H}^{\alpha,\bsgamma,\rmu}_{\mix,1}([0,1]^d)$ to $F_{d,\alpha,\bsgamma}$.

For the special case $\widetilde{H}''^{\alpha,\bsgamma}_{\mix,1}([0,1]^d)$ with $1/2 < \alpha < 1$ we use, for $0 \ne h \in \Z$,
\begin{align*}
  (1 + |h|^2)^\alpha - 1
  &\ge
  c_\alpha \, |h|^{2\alpha}
  ,
\end{align*}
which follows from~\eqref{eq:jensen1} with $\lambda = 1/\alpha$ since $(1 + x)^\alpha - 1 \ge 2^{\alpha-1} (1 + x^\alpha) - 1 \ge (2^\alpha - 1) \, x^\alpha$ with $x = |h|^2 \ge 1$ and the second inequality follows from $2^{\alpha-1} (1 + x^\alpha) - 1 - (2^\alpha - 1) \, x^\alpha = (1 - 2^{\alpha-1}) \, (x^\alpha - 1) \ge 0$.
We now compare the multiplier for each $\bsh \in \Z^d$ between the norm of $\widetilde{H}''^{\alpha,\bsgamma}_{\mix,1}([0,1]^d)$ and the norm of $\widetilde{H}^{\alpha,\bsgamma^*,\rmu}_{\mix,1}([0,1]^d)$.
For $\bsh \in \Z^d$ we get
\begin{align*}
  \sum_{\setu \subseteq \{1,\ldots,d\}} \gamma_\setu^{-1} \prod_{j \in \setu} ((1 + |h_j|^2)^\alpha - 1)
  &\ge
  \gamma_{\supp(\bsh)}^{-1} \prod_{j \in \supp(\bsh)} ((1 + |h_j|^2)^\alpha - 1)
  \\
  &\ge
  \gamma_{\supp(\bsh)}^{-1} \, c_\alpha^{|\supp(\bsh)|} \prod_{j \in \supp(\bsh)} |h_j|^{2\alpha}
  ,
\end{align*}
hence, with $\gamma_\setu^* = \gamma_\setu / c_\alpha^{|\setu|}$, we also have the embedding with constant~$1$,
\begin{align*}
  \|f\|_{\widetilde{H}^{\alpha,\bsgamma^*,\rmu}_{\mix,1}([0,1]^d)}
  &\le
  \|f\|_{\widetilde{H}''^{\alpha,\bsgamma}_{\mix,1}([0,1]^d)}
  .
\end{align*}
So in this case we can apply~\eqref{eq:transplant-upper-bound} to transplant the worst-case upper bound from $\widetilde{H}^{\alpha,\bsgamma^*,\rmu}_{\mix,1}([0,1]^d)$ to $\widetilde{H}''^{\alpha,\bsgamma}_{\mix,1}([0,1]^d)$.

We now look at the periodic anchored space $\widetilde{H}^{\alpha,\bsgamma,\bsa}_\mix([0,1]^d)$ with $\alpha \in \N$.
We have
\begin{align*}
  \|f\|_{\widetilde{H}^{\alpha,\bsgamma,\rmu}_{\mix,1}([0,1]^d)}
  &\le
  \|f\|_{\widetilde{H}^{\alpha,\bsgamma,\rmu}_\mix([0,1]^d)}
  \le
  \rho_d(\bsgamma) \, (2\alpha)^{d/2} \, \|f\|_{\widetilde{H}^{\alpha,\bsgamma,\bsa}_\mix([0,1]^d)}
  ,
\end{align*}
where the first inequality follows directly from comparing~\eqref{eq:norm-Htilde-alpha-bsgamma-u} with~\eqref{eq:norm-Htilde-alpha-bsgamma-u-1}, and the second follows by restricting~\eqref{eq:H-alpha-gamma-mix-equivalence-unanchored-anchored} to the periodic subspace.
Using~\eqref{eq:transplant-upper-bound} then gives the bound on the worst-case error for $\widetilde{H}^{\alpha,\bsgamma,\bsa}_\mix([0,1]^d)$ with the bad $d$-dependence.

We now switch to product weights $\gamma_\setu = \prod_{j\in\setu} \gamma_j$.
For $\alpha \in \N$ and univariate $f \in \widetilde{H}^\alpha([0,1])$, in particular $f \in C(\bbT([0,1]))$ and $D^\alpha f \in L_2([0,1])$, and for $m_\alpha$ as defined above, the bound in~\eqref{eq:periodic-anchor-mean-Fourier-bound} from Lemma~\ref{lem:H-alpha-a-norm-equivalence-d=1}, applied with $x = a$, gives
\begin{align*}
  |f(a) - \hat{f}_0|
  &\le
  \sqrt{m_\alpha} \, \|D^\alpha f\|_{L_2([0,1])}
  .
\end{align*}
Using this estimate, we apply the inequality~\eqref{eq:scalar-quadratic-inequality} from the proof of Lemma~\ref{lem:H-1-mix-product-weight-embeddings} with
\begin{align*}
  u
  &=
  |f(a)|
  ,
  &
  v
  &=
  \gamma^{-1/2} \, \|D^\alpha f\|_{L_2([0,1])}
  ,
  &
  r
  &=
  \sqrt{\gamma \, m_\alpha}
  .
\end{align*}
Thus,
\begin{align*}
  \journalonly{\MoveEqLeft[1]} \left(|\hat{f}_0|^2 + \gamma^{-1} \, \|D^\alpha f\|_{L_2([0,1])}^2\right)^{1/2}
  \journalonly{\\*}&\le
  \frac{\sqrt{4 + \gamma \, m_\alpha} + \sqrt{\gamma \, m_\alpha}}{2}
  \left(|f(a)|^2 + \gamma^{-1} \, \|D^\alpha f\|_{L_2([0,1])}^2\right)^{1/2}
  .
\end{align*}
Tensorising the resulting inequality over the coordinates, as in the proof of Lemma~\ref{lem:H-1-mix-product-weight-embeddings}, gives, with $\kappa_{d,\alpha}(\bsgamma)$ as defined in~\eqref{eq:def-kappa-d-alpha},
\begin{align*}
  \|f\|_{\widetilde{H}^{\alpha,\bsgamma,\rmu}_{\mix,1}([0,1]^d)}
  &\le
  \|f\|_{\widetilde{H}^{\alpha,\bsgamma,\rmu}_\mix([0,1]^d)}
  \le
  \kappa_{d,\alpha}(\bsgamma) \, \|f\|_{\widetilde{H}^{\alpha,\bsgamma,\bsa}_\mix([0,1]^d)}
  .
\end{align*}
We can now again use~\eqref{eq:transplant-upper-bound} to get the worst-case error bound.
The uniform boundedness assertion for $\kappa_{d,\alpha}(\bsgamma)$ follows from the behaviour of each factor as $\gamma_j \to 0$ as in Lemma~\ref{lem:H-1-mix-product-weight-embeddings}.
\end{proof}

A good conclusion from here is that it is a good idea to construct lattice rules for the unanchored space $\widetilde{H}^{\alpha,\bsgamma,\rmu}_{\mix,1}([0,1]^d)$ since these results transfer easily to the other same-weighted Fourier-multiplier spaces from Table~\ref{tbl:periodic-equivalent-norms}, with the exceptional case of the range $1/2 < \alpha < 1$ for the space $\widetilde{H}''^{\alpha,\bsgamma}_{\mix,1}([0,1]^d)$.
But also for this space and even for the anchored periodic space $\widetilde{H}^{\alpha,\bsgamma,\bsa}_\mix([0,1]^d)$ we still have such claims when considering product weights.

Although it turns out that upper bounds seem to be transferable without ruining the $d$-dependence in these cases, at least for product weights, we must warn that tractability can nevertheless be sensitive to the particular equivalent norm.
E.g., for the special case of fully permutation-invariant subspaces, where tractability can be achieved without having weights, \cite[Section~3.2]{NSW2016} compares the squared Fourier multipliers for the unanchored space $\max\{1, |2\pi h|\}^{2\alpha}$ and the Bessel potential space with $(1 + |h|^2)^\alpha$, and their upper-bound argument yields strong polynomial tractability for $\alpha \ge \alpha_1^* \approx 0.61769976$ and $\alpha \ge \alpha_2^* \approx 1.521196$ respectively.

\subsection{An upper bound for lattice rules in periodic spaces}\label{sec:upper-bound-lattice-rules}

The whole equivalent function space business from the previous sections might have been a bit too much if we just want to understand error upper bounds for lattice rules.
So, if you are not too worried where our fancy norms come from, then this section could be read more or less independently.

\subsubsection{Error and worst-case error for lattice rules}\label{sec:lattice-worst-case-error}

The next lemma shows an explicit expression for the error of using a lattice rule to integrate a function that can be represented by an absolutely converging Fourier series.
We introduced this as the Wiener algebra~\eqref{eq:def-Wiener-algebra-d} in the previous section and we write $f \in A([0,1]^d)$.
We recall that for $f \in A([0,1]^d)$ we have, for all $\bsx \in [0,1]^d$,
\begin{align}\label{eq:Fourier-series-abs-conv}
  f(\bsx)
  &=
  \sum_{\bsh \in \Z^d} \hat{f}_\bsh \, \exp(2\pi\rmi \, \bsh \cdot \bsx)
  ,
  &
  \sum_{\bsh \in \Z^d} |\hat{f}_\bsh|
  &<
  \infty
  .
\end{align}

With the \emph{character property}, for $a \in \Z$ and $n \in \N$, we mean
\begin{align*}
  \numberthis{eq:character-prop}
  \frac1n \sum_{k=0}^{n-1} \exp(2\pi\rmi\, k a / n)
  &=
  \bbone\{a \equiv 0 \tpmod{n}\}
  \\
  &=
  \begin{cases}
    1, & \text{if $a$ is a multiple of $n$}, \\
    0, & \text{otherwise}.
  \end{cases}
\end{align*}
Using this on the Fourier basis functions with our lattice points~\eqref{eq:R1L-points} then gives us the \emph{lattice character property}, for $\bsh \in \Z^d$, $\bsz \in \Z^d$ and $n \in \N$,
\begin{align}\label{eq:lattice-character-prop}
  \frac1n \sum_{k=0}^{n-1} \exp(2\pi\rmi\, k \, (\bsz \cdot \bsh) / n)
  &=
  \begin{cases}
    1 & \text{if } \bsh \cdot \bsz \equiv 0 \pmod{n} ,\\
    0 & \text{otherwise} .
  \end{cases}
\end{align}
Using our notation for a lattice rule $Q_{d,n,\bsz}$, see~\eqref{eq:R1L-rule}, this can be written as
\begin{align*}
  Q_{d,n,\bsz}(\bsx \mapsto \exp(2\pi\rmi\, \bsh \cdot \bsx))
  &=
  \bbone\{\bsh \cdot \bsz \equiv 0 \tpmod{n}\}
  ,
\intertext{while}
  I_d(\bsx \mapsto \exp(2\pi\rmi\, \bsh \cdot \bsx))
  &=
  \bbone\{ \bsh = \bszero \}
  ,
\end{align*}
i.e., for all $\bsh \ne \bszero$ the integral is zero.
This hence gives us a way of looking at the error of integrating the Fourier basis functions.

This allows us to get an explicit expression for the error of integrating $f \in A([0,1]^d)$ which we show in the next lemma.
This result can be found in any standard reference, e.g., \cite{SJ94,Nie92,Nuy2014,DKS2013,DKP2022}.
We obtain the nice form in the lemma since the basis functions of our absolutely converging series representation match up perfectly with the nodes we use in the cubature rule: the lattice rule either integrates exactly a basis function or it has error~$1$.
Similar expressions exist for Walsh series expansions in combinations with digital nets, see, e.g., \cite{Nuy2014,DP2010}.
One often defines the \emph{dual lattice} by
\begin{align*}
  \big\{ \bsh \in \Z^d : \bsh \cdot \bsz \equiv 0 \pmod{n} \big\}
  .
\end{align*}
This is the set of Fourier indices of the basis functions where the lattice rule returns~$1$.
(For the equivalent concept with digital nets one sometimes leaves out the Walsh basis function corresponding to the $\bszero$ index and then this is the set of indices where the rule does not integrate the respective basis function correctly.)

\begin{lemma}[Exact error expression]\label{lem:fourier-sum-error-expression}
  Suppose $f$ has an absolutely converging Fourier series~\eqref{eq:Fourier-series-abs-conv}, i.e., $f \in A([0,1]^d)$, then the error of approximating its $d$-dimensional integral $I_d$, defined in~\eqref{eq:Id}, with a lattice rule $Q_{d,n,\bsz}$, as defined in~\eqref{eq:R1L-rule}, can be written as
  \begin{align*}
    Q_{d,n,\bsz}(f) - I_d(f)
    &=
    \sum_{\substack{\bszero \ne \bsh \in \Z^d \\ \bsh\cdot\bsz \equiv 0 \tpmod{n}}} \hat{f}_\bsh
    .
  \end{align*}
\end{lemma}
\begin{proof}
We expand $f$ in its absolutely converging Fourier series and make use of the lattice character property~\eqref{eq:lattice-character-prop} to obtain
\begin{align*}
  Q_{d,n,\bsz}(f) - I_d(f)
  &=
  \frac1n \sum_{k=0}^{n-1} f((\bsz k \bmod n) / n) - \int_{[0,1]^d} f(\bsx) \rd\bsx
  \\
  &=
  \frac1n \sum_{k=0}^{n-1} \sum_{\bsh \in \Z^d} \hat{f}_\bsh \, \exp(2\pi\rmi \, \bsh \cdot (\bsz k \bmod n) / n) - \int_{[0,1]^d} f(\bsx) \rd\bsx
  \\
  &=
  \frac1n \sum_{k=0}^{n-1} \sum_{\bszero \ne \bsh \in \Z^d} \hat{f}_\bsh \, \exp(2\pi\rmi \, \bsh \cdot \bsz k / n)
  \\
  &=
  \sum_{\bszero \ne \bsh \in \Z^d} \hat{f}_\bsh \, \frac1n \sum_{k=0}^{n-1} \exp(2\pi\rmi \, \bsh \cdot \bsz k / n)
  \\
  &=
  \sum_{\bszero \ne \bsh \in \Z^d} \hat{f}_\bsh \, \bbone\{\bsh \cdot \bsz \equiv 0 \tpmod{n}\}
  \\
  &=
  \sum_{\substack{\bszero \ne \bsh \in \Z^d \\ \bsh\cdot\bsz \equiv 0 \tpmod{n}}} \hat{f}_\bsh
  .
  \qedhere
\end{align*}
\end{proof}

Using the previous lemma we are now in a position to derive a very easy error bound using the Cauchy--Schwarz inequality.
We could also go ahead and use Hölder's inequality as we did in~\eqref{eq:embedding-r-holder} and from where we defined a general $\ell_p$-based Fourier norm~\eqref{eq:Fourier-ell-p} for a space we denoted by $F^{\alpha,\bsgamma}_{r,p}([0,1]^d)$.
This is done in~\cite{Nuy2014}, but we will not pursue that here.
However, we do want to remark that the original articles~\cite{Kor59,Kor59-b,Kor60} introducing lattice rules by Korobov in 1959 are similar to using an $\ell_\infty$-norm and not an $\ell_2$-norm as we pursue here.
There Korobov introduces a class $E_\alpha(C)$ and membership in the class is based on conditions like $|\hat{f}_\bsh| \le C \prod_{j \in \supp(\bsh)} |h_j|^{-\alpha} = C \prod_{j=1}^d \max\{1, |h_j|\}^{-\alpha}$ or $|\hat{f}_\bsh| \le C \prod_{j=1}^d (1 + |h_j|)^{-\alpha}$ (which could be seen as an equivalent norm as in Sections~\ref{sec:periodic-spaces-d=1} and~\ref{sec:periodic-spaces}).
We could interpret Korobov's condition as an $\ell_\infty$-based Fourier norm.

To state the next theorem we define the norm in terms of Fourier multipliers.
We denote the multiplier for a space $F^{\alpha,\bsgamma}_r([0,1]^d) = F^{\alpha,\bsgamma}_{r,2}([0,1]^d)$ by $r_{\alpha,\bsgamma} : \Z^d \to \R_{>0}$, where $\alpha$ is the smoothness and $\bsgamma$ will be a family of weights.
This makes the statement much more general than doing this for one specific space/norm and we know from Section~\ref{sec:periodic-spaces} and Table~\ref{tbl:periodic-equivalent-norms} that all of the norms we are interested in can be expressed in this way.
From the discussion on general Fourier $\ell_p$-norms~\eqref{eq:Fourier-ell-p} we already know that for $p = 2$ and the Fourier multipliers given in Table~\ref{tbl:periodic-equivalent-norms} we need $\alpha > 1/2$ and $r_{\alpha,\bsgamma}^{-1} \in \ell_2(\Z^d)$, i.e., $\sum_{\bsh \in \Z^d} r_{\alpha,\bsgamma}^{-2}(\bsh) < \infty$ for the embedding into the Wiener algebra $A([0,1]^d)$ which is then also embedded into the space of the continuous functions on the torus $C(\bbT([0,1]^d))$.
The parameters $\alpha$ and $\bsgamma$ are just placeholders in the theorem, where $\alpha$ denotes the smoothness and $\bsgamma$ is a weight family.
We write it in terms of $\alpha$ and $\bsgamma$ just to make the connection to our spaces in Table~\ref{tbl:periodic-equivalent-norms}.

\begin{theorem}[Worst-case error for Fourier series based norms]\label{thm:wce-F-d-r}
  Consider the function space $F^{\alpha,\bsgamma}_r([0,1]^d)$ with norm defined in terms of a given Fourier multiplier $r_{\alpha,\bsgamma} : \Z^d \to \R_{>0}$ as
  \begin{align}\label{eq:def-F-d-r-norm}
    \| f \|_{F^{\alpha,\bsgamma}_r([0,1]^d)}^2
    &:=
    \sum_{\bsh \in \Z^d} |\hat{f}_\bsh|^2 \, r_{\alpha,\bsgamma}^2(\bsh)
    ,
  \end{align}
  and $f \in F^{\alpha,\bsgamma}_r([0,1]^d)$ whenever $\|f\|_{F^{\alpha,\bsgamma}_r([0,1]^d)} < \infty$.
  If $r_{\alpha,\bsgamma}^{-1} \in \ell_2(\Z^d)$ then $F^{\alpha,\bsgamma}_r([0,1]^d) \hookrightarrow A([0,1]^d)$, i.e., $f$ has an absolutely converging Fourier series, and for every $f \in F^{\alpha,\bsgamma}_r([0,1]^d)$,
  \begin{align*}
    | Q_{d,n,\bsz}(f) - I_d(f) |
    &\le
    \| f \|_{F^{\alpha,\bsgamma}_r([0,1]^d)}
    \,
    \Bigg( \sum_{\substack{\bszero \ne \bsh \in \Z^d \\ \bsh\cdot\bsz \equiv 0 \tpmod{n}}} r_{\alpha,\bsgamma}^{-2}(\bsh) \Bigg)^{\!1/2}
    .
  \end{align*}
  Moreover, this error bound is attainable for some $f \in F^{\alpha,\bsgamma}_r([0,1]^d)$, and hence
  \begin{align*}
    e(Q_{d,n,\bsz}, F^{\alpha,\bsgamma}_r([0,1]^d))
    &=
    \sup_{\substack{f \in F^{\alpha,\bsgamma}_r([0,1]^d) \\ \|f\|_{F^{\alpha,\bsgamma}_r([0,1]^d)} \le 1}} |I_d(f) - Q_{d,n,\bsz}(f)|
    \\
    \numberthis{eq:F-r-wce}
    &=
    \Bigg( \sum_{\substack{\bszero \ne \bsh \in \Z^d \\ \bsh\cdot\bsz \equiv 0 \tpmod{n}}} r_{\alpha,\bsgamma}^{-2}(\bsh) \Bigg)^{\!1/2}
    .
  \end{align*}
  This worst-case error is ``computable'' for a given triplet $(d, n, \bsz)$ by calculating the square root of the error for the function $\xi_{d,r,\alpha,\bsgamma} \in F^{\alpha,\bsgamma}_r([0,1]^d)$, which is defined independently of $n$ and $\bsz$ by
  \begin{align}\label{eq:def-xi-d-r}
    \xi_{d,r,\alpha,\bsgamma}(\bsx)
    &:=
    \sum_{\bszero \ne \bsh \in \Z^d} r_{\alpha,\bsgamma}^{-2}(\bsh) \, \exp(\twopii \bsh\cdot\bsx)
    ,
  \end{align}
  such that
  \begin{align}\label{eq:R1L-wce}
    e(Q_{d,n,\bsz}, F^{\alpha,\bsgamma}_r([0,1]^d))
    =
    \Bigg( \sum_{\substack{\bszero \ne \bsh \in \Z^d \\ \bsh\cdot\bsz \equiv 0 \tpmod{n}}} r_{\alpha,\bsgamma}^{-2}(\bsh) \Bigg)^{\!1/2}
    &=
    \Big( Q_{d,n,\bsz}(\xi_{d,r,\alpha,\bsgamma}) \Big)^{1/2}
    .
  \end{align}
\end{theorem}
\begin{proof}
We first use~\eqref{eq:embedding-r} for our specific $r_{\alpha,\bsgamma}$ to show the embedding based on $r_{\alpha,\bsgamma}^{-1} \in \ell_2(\Z^d)$.
Next we work in exactly the same spirit as in the proof of Lemma~\ref{lem:fourier-sum-error-expression}.
We multiply and divide by the Fourier multiplier $r_{\alpha,\bsgamma}$, and apply Cauchy--Schwarz
\begin{align*}
  | Q_{d,n,\bsz}(f) - I_d(f) |
  &=
  \Bigg| \sum_{\substack{\bszero \ne \bsh \in \Z^d \\ \bsh\cdot\bsz \equiv 0 \tpmod{n}}} \hat{f}_\bsh \Bigg|
  \\
  &\le
  \sum_{\substack{\bszero \ne \bsh \in \Z^d \\ \bsh\cdot\bsz \equiv 0 \tpmod{n}}} |\hat{f}_\bsh|
  \\
  &=
  \sum_{\substack{\bszero \ne \bsh \in \Z^d \\ \bsh\cdot\bsz \equiv 0 \tpmod{n}}} |\hat{f}_\bsh| \, \frac{r_{\alpha,\bsgamma}(\bsh)}{r_{\alpha,\bsgamma}(\bsh)}
  \\
  &\le
  \Bigg( \sum_{\substack{\bszero \ne \bsh \in \Z^d \\ \bsh\cdot\bsz \equiv 0 \tpmod{n}}} |\hat{f}_\bsh|^2 \, r_{\alpha,\bsgamma}^2(\bsh) \Bigg)^{1/2}
  \journalonly{\!} \articleonly{\,}
  \Bigg( \sum_{\substack{\bszero \ne \bsh \in \Z^d \\ \bsh\cdot\bsz \equiv 0 \tpmod{n}}} r_{\alpha,\bsgamma}^{-2}(\bsh) \Bigg)^{1/2}
  \\
  \numberthis{eq:CS-bound}
  &\le
  \Bigg( \sum_{\bsh \in \Z^d} |\hat{f}_\bsh|^2 \, r_{\alpha,\bsgamma}^2(\bsh) \Bigg)^{\!1/2}
  \journalonly{\!} \articleonly{\,}
  \Bigg( \sum_{\substack{\bszero \ne \bsh \in \Z^d \\ \bsh\cdot\bsz \equiv 0 \tpmod{n}}} r_{\alpha,\bsgamma}^{-2}(\bsh) \Bigg)^{\!1/2}
  .
\end{align*}
If we take the function $\xi = \xi_{d,r,\alpha,\bsgamma,n,\bsz}$ (which is not yet the function $\xi_{d,r,\alpha,\bsgamma}$ from the statement) with Fourier coefficients $\hat{\xi}_\bsh = r_{\alpha,\bsgamma}^{-2}(\bsh)$ for $\bszero \ne \bsh \in \Z^d$ whenever $\bsh\cdot\bsz \equiv 0 \pmod{n}$ and $\hat{\xi}_\bsh = 0$ otherwise, then this upper bound is attained, since we have
\begin{align*}
  &Q_{d,n,\bsz}(\xi) - I_d(\xi)
  =
  \sum_{\substack{\bszero \ne \bsh \in \Z^d \\ \bsh\cdot\bsz \equiv 0 \tpmod{n}}} r_{\alpha,\bsgamma}^{-2}(\bsh)
  \\
  &\journalonly{\!}\articleonly{\quad}=
  \Bigg( \sum_{\bsh \in \Z^d} \frac{\bbone\{\bsh\cdot\bsz \equiv 0 \tpmod{n}\} \, \bbone\{\bsh \ne \bszero\}}{r_{\alpha,\bsgamma}^4(\bsh)} \, r_{\alpha,\bsgamma}^2(\bsh) \Bigg)^{\!1/2}
  \,
  \Bigg( \journalonly{\mkern-12mu}\sum_{\substack{\bszero \ne \bsh \in \Z^d \\ \bsh\cdot\bsz \equiv 0 \tpmod{n}}}\journalonly{\mkern-12mu} r_{\alpha,\bsgamma}^{-2}(\bsh) \Bigg)^{\!1/2}
  ,
\end{align*}
which matches~\eqref{eq:CS-bound} above.
We need to check that this function is also part of the function space, and this is immediate in this case since $\|\xi\|_{F^{\alpha,\bsgamma}_r([0,1]^d)} < \infty$ because $r_{\alpha,\bsgamma}^{-1} \in \ell_2(\Z^d)$ and there are no further conditions for membership in~$F^{\alpha,\bsgamma}_r([0,1]^d)$.
Moreover, the similar looking function $\xi_{d,r,\alpha,\bsgamma}$, defined in~\eqref{eq:def-xi-d-r}, and which does not depend on $n$ and $\bsz$, has
\begin{align*}
  Q_{d,n,\bsz}(\xi_{d,r,\alpha,\bsgamma}) - I_d(\xi_{d,r,\alpha,\bsgamma})
  &=
  Q_{d,n,\bsz}(\xi_{d,r,\alpha,\bsgamma})
  \\
  &=
  \sum_{\substack{\bszero \ne \bsh \in \Z^d \\ \bsh\cdot\bsz \equiv 0 \tpmod{n}}} r_{\alpha,\bsgamma}^{-2}(\bsh)
  \\
  &=
  \Big( e(Q_{d,n,\bsz}, F^{\alpha,\bsgamma}_r([0,1]^d)) \Big)^2
  .
\end{align*}
This finishes the proof.
\end{proof}

We wrote ``computable'', in quotes, in the statement of the theorem to suggest that you can actually compute that expression in a practical way in many cases.
This means we need to be able to somehow handle the $d$-dimensional infinite Fourier series~\eqref{eq:def-xi-d-r}.
This could either be done numerically or, if it exists, by a closed-form expression.

Quite often the function $r_{\alpha,\bsgamma}$ is a product over the dimensions $r(\bsh) = r_{\alpha,\bsgamma}(\bsh) = \prod_{j=1}^d r_j(h_j)$, for some functions~$r_j$.
E.g., this is the case for product weights $\gamma_\setu = \prod_{j\in\setu} \gamma_j$ in Table~\ref{tbl:periodic-equivalent-norms}, including the unweighted case since that corresponds to $\gamma_j = 1$.
This results in a tensor-product space and the $d$-dimensional Fourier series then factorises into one-dimensional Fourier series as follows
\begin{align*}
  \journalonly{\MoveEqLeft[1]} \xi_{d,r,\alpha,\bsgamma}(\bsx)
  =
  \sum_{\bszero \ne \bsh \in \Z^d} r_{\alpha,\bsgamma}^{-2}(\bsh) \, \exp(\twopii \bsh\cdot\bsx)
  \journalonly{\\*}&=
  -r_{\alpha,\bsgamma}^{-2}(\bszero) + \sum_{\bsh \in \Z^d} r_{\alpha,\bsgamma}^{-2}(\bsh) \, \exp(\twopii \bsh\cdot\bsx)
  \\
  &=
  -\prod_{j=1}^d r_j^{-2}(0) + \prod_{j=1}^d \sum_{h_j \in \Z} r_j^{-2}(h_j) \, \exp(\twopii h_j x_j)
  \\
  \numberthis{eq:xi-d-r-form-1}
  &=
  -\prod_{j=1}^d r_j^{-2}(0) + \prod_{j=1}^d \psi_j(x_j)
  ,
\end{align*}
for some one-dimensional functions~$\psi_j = \sum_{h_j \in \Z} r_j^{-2}(h_j) \, \exp(\twopii h_j x_j)$.
These one-dimensional functions $\psi_j$ might be explicitly computable in closed form depending on what $r_j^{-2}(h)$ is.
Even if there is no closed form, we could numerically approximate $\psi_j$ for all $x = k / n$, $k = 0, \ldots, n-1$, to sufficient precision (possibly using an FFT), since every projection of our lattice points consists only of $k / n$ coordinates.

There is one common form that appears a lot in the literature in case $r_j(0) = 1$ (which can also be obtained by normalising).
Again with $r_{\alpha,\bsgamma}(\bsh) = \prod_{j=1}^d r_j(h_j)$ and with $r_j(0) = 1$, we can write
\begin{align*}
  \journalonly{\MoveEqLeft[1]} \xi_{d,r,\alpha,\bsgamma}(\bsx)
  =
  \sum_{\bszero \ne \bsh \in \Z^d} r_{\alpha,\bsgamma}^{-2}(\bsh) \, \exp(\twopii \bsh\cdot\bsx)
  \journalonly{\\*}&=
  -1 + \prod_{j=1}^d \Big( 1 + \sum_{0 \ne h_j \in \Z} r_j^{-2}(h_j) \, \exp(\twopii h_j x_j) \Big)
  \\
  \numberthis{eq:xi-d-r-form-2}
  &=
  -1 + \prod_{j=1}^d \Big( 1 + \omega_j(x_j) \Big)
  .
\end{align*}
Obviously this is the same form as~\eqref{eq:xi-d-r-form-1}, using $r_j(0) = 1$, and writing $\psi_j(x_j) = 1 + \omega_j(x_j)$.
The same remarks about computability hold.
We might know a closed form expression for the Fourier series $\omega_j$ or we can approximate its values for all $x = k / n$, $k = 0,\ldots,n-1$, up to sufficient precision.

In particular, when we have product weights $\gamma_\setu = \prod_{j\in\setu} \gamma_j$ and the squared Fourier multiplier is that of the unanchored periodic Sobolev space of dominating mixed smoothness $\widetilde{H}^{\alpha,\bsgamma,\rmu}_{\mix,1}([0,1]^d)$, see Table~\ref{tbl:periodic-equivalent-norms}, then
\begin{align*}
  r_j^2(h_j)
  &=
  \begin{cases}
    \gamma_j^{-1} \, |h_j|^{2\alpha}, & \text{if } h_j \ne 0, \\
    1, &\text{if } h_j = 0,
  \end{cases}
  .
\end{align*}
For $\alpha \in \N$ we then have
\begin{align}\label{eq:unanchored-omega-j}
  \omega_j(x_j)
  &=
  \gamma_j \, \sum_{0 \ne h_j \in \Z} \frac{\exp(\twopii h_j x_j)}{|h_j|^{2\alpha}}
  =
  \gamma_j \, \frac{(-1)^{\alpha+1} \, (2\pi)^{2\alpha}}{(2\alpha)!} \, B_{2\alpha}(x_j)
  ,
\end{align}
where $B_{2\alpha}$ is the Bernoulli polynomial of even degree $2\alpha$, see, \cite[(24.8.1)]{DLMF} or~\cite[(23.1.18)]{AS64}.
This form is hence particularly interesting for computer implementation since $\omega_j$ reduces to just a simple polynomial.

For non-integer $\alpha > 1/2$ we can still find kind of a closed form in terms of the polylogarithm, see~\cite[(25.12.10)]{DLMF},
\begin{align*}
  \Li_s(z)
  &=
  \sum_{h \ge 1} \frac{z^h}{h^s}
  ,
\end{align*}
which converges for $z \in \bbC$, $|z| \le 1$ and $\Re s > 1$.
We can take $z = \exp(\twopii x)$ and $s = 2\alpha$ and compute
\begin{align*}
  \sum_{0 \ne h \in \Z} \frac{\exp(\twopii h x)}{|h|^{2\alpha}}
  &=
  \sum_{h \ge 1} \frac{\exp(\twopii h x)}{h^{2\alpha}} + \sum_{h \ge 1} \frac{\exp(-\twopii h x)}{h^{2\alpha}}
  \\
  &=
  \sum_{h \ge 1} \frac{z^h}{h^{2\alpha}} + \sum_{h \ge 1} \frac{\overline{z}^h}{h^{2\alpha}}
  \\
  &=
  \Li_{2\alpha}(\exp(\twopii x)) + \Li_{2\alpha}(\exp(-\twopii x))
  \\
  &=
  2 \, \Re\Big\{ \Li_{2\alpha}(\exp(\twopii x)) \Big\}
  \numberthis{eq:unanchored-polylog}
  ,
\end{align*}
which holds for any $\alpha > 1/2$.
Implementations of the polylogarithm function are typically available in different programming languages.
Again, we only need the values at $x = k / n$, $k = 0, \ldots, n-1$.

\subsubsection{Component-by-component construction}\label{sec:cbc-construction}

At this point we are ready to show a constructive upper bound for lattice rules in function spaces whose squared Fourier multiplier has the form
\begin{align}\label{eq:r-general-weights-prod}
  r_{\alpha,\bsgamma}^2(\bsh)
  &=
  \gamma_{\supp(\bsh)}^{-1} \, \prod_{j\in\supp(\bsh)} r_\alpha^2(h_j)
  ,
\end{align}
and additionally satisfies that if any of the components of $\bsh$ is a non-zero multiple of a number $n \in \N$ then it behaves like multiplication with~$n^{2\alpha}$, see condition~\eqref{eq:r-multiplier-scaling} below.
This is a very specific restriction which only the unanchored weighted periodic Sobolev spaces $\widetilde{H}^{\alpha,\bsgamma,\rmu}_\mix([0,1]^d)$ and $\widetilde{H}^{\alpha,\bsgamma,\rmu}_{\mix,1}([0,1]^d)$ in Table~\ref{tbl:periodic-equivalent-norms} satisfy.
The condition~\eqref{eq:r-multiplier-scaling} below is written in an exact form but could be relaxed up to a constant as in~\cite{NSW2016}.
However, the result would be similar to using embedding constants on the final result, see~\eqref{eq:transplant-upper-bound}, which is the strategy we are following, and so we do not add this complication.
The assumptions~\eqref{eq:r-general-weights-prod} and~\eqref{eq:r-multiplier-scaling} enable us to do (rather crude) estimates of the form
\begin{align*}
  \sum_{\substack{\bszero \ne \bsh \in \Z^d \\ \bsh \equiv \bszero \tpmod{n}}} r_{\alpha,\bsgamma}^{-2}(\bsh)
  &\le
  n^{-2\alpha} \, \sum_{\bszero \ne \bsh \in \Z^d} r_{\alpha,\bsgamma}^{-2}(\bsh)
  .
\end{align*}

The theorem below shows the component-by-component result due to Sloan and collaborators~\cite{SR2002,SKJ2002}, in which we inductively extend a generating vector component-by-component minimising the worst-case error in each step.
A similar idea was in fact already used by Korobov in~\cite[Theorem~1]{Kor59}.
We first show the result for $n$ prime as this is traditionally easier to prove.
Afterwards we extend to composite~$n$ using a technique from~\cite{KMN2025} which works more generally than the direct techniques that are most often used.
We repeat again that among the spaces in Table~\ref{tbl:periodic-equivalent-norms}, the conditions of the theorem hold only for the unanchored variants $\widetilde{H}^{\alpha,\bsgamma,\rmu}_\mix([0,1]^d)$ and $\widetilde{H}^{\alpha,\bsgamma,\rmu}_{\mix,1}([0,1]^d)$.
Afterwards we will use norm embeddings to extend the result to the other spaces in Theorem~\ref{thm:lattice-rule-upper-bounds}.

\begin{theorem}[Component-by-component construction for lattice rules with prime~$n$]\label{thm:CBC-prime}
Let $\alpha > 1/2$ and let $r_\alpha : \Z \to \R_{>0}$ be a one-dimensional Fourier multiplier whose inverse satisfies
\begin{align}\label{eq:r-multiplier-scaling}
  r_\alpha^{-1}(n \, h)
  &\le
  n^{-\alpha} \, r_\alpha^{-1}(h)
  &&
  \text{for all } n \in \N \text{ and } 0 \ne h \in \Z
  .
\end{align}
This implies $r_\alpha^{-1} \in \ell_{1/\lambda}(\Z)$ for every $1/2 \le \lambda < \alpha$.
Let $\bsgamma = \{\gamma_\setu\}_{\setu \in \frakU}$ with $\gamma_\emptyset = 1$ and $\gamma_\setu > 0$ for all~$\setu$.
For any $d \in \N_0$ and $\bsh \in \Z^d$, define the squared Fourier multiplier
\begin{align}\label{eq:r-form-CBC}
  r_{\alpha,\bsgamma}^2(\bsh)
  &:=
  \gamma_{\supp(\bsh)}^{-1}
  \prod_{j\in\supp(\bsh)} r_\alpha^2(h_j)
  .
\end{align}
Note that, e.g., $r_{\alpha,\bsgamma}^2((\bsh', 0)) = r_{\alpha,\bsgamma}^2(\bsh')$ and this definition works for any $d \in \N_0$.

For $d = 1, 2, \ldots$, let $F^{\alpha,\bsgamma}_r([0,1]^d)$ be the corresponding $d$-dimensional space as in~\eqref{eq:def-F-d-r-norm}.
Let $n$ be prime, pick any $z_1 \in \Z_n^\times$, and, successively for $d = 2, 3, \ldots$, choose
\begin{align}\label{eq:CBC-argmin-Zntimes}
  z_d
  &\in
  \underset{z\in\Z_n^\times}{\arg\min}\;
  e(Q_{d,n,(z_1,\ldots,z_{d-1},z)}, F^{\alpha,\bsgamma}_r([0,1]^d))
  .
\end{align}
Then, for every $1/2 \le \lambda < \alpha$ and every $d = 1, 2, \ldots$, the resulting generating vector satisfies
\begin{align}\label{eq:CBC-prime-upper-bound}
  e(Q_{d,n,\bsz_{\{1,\ldots,d\}}}, F^{\alpha,\bsgamma}_r([0,1]^d))
  &\le
  \left(
    \frac1{\varphi(n)}
    \sum_{\emptyset\ne\setu\subseteq\{1,\ldots,d\}}
    \gamma_\setu^{1/(2\lambda)} \,
    \mu_\alpha^{|\setu|}(\lambda)
  \right)^{\!\lambda}
  ,
\end{align}
with
\begin{align}\label{eq:def-mu-alpha-lambda}
  \mu_\alpha(\lambda)
  &:=
  \sum_{0\ne h\in\Z} r_\alpha^{-1/\lambda}(h)
  <
  \infty
  .
\end{align}
\end{theorem}
\begin{proof}
We have that $\mu_\alpha(\lambda) < \infty$ for $1/2 \le \lambda < \alpha$ since $r_\alpha^{-1} \in \ell_{1/\lambda}(\Z)$ which follows from the scaling condition~\eqref{eq:r-multiplier-scaling} which says $r_\alpha^{-1}(h) \le |h|^{-\alpha} \, r_\alpha^{-1}(\pm 1) \in O(|h|^{-\alpha})$ and so $r_\alpha^{-1/\lambda}(h) \in O(|h|^{-\alpha/\lambda})$ which is summable if $\alpha / \lambda > 1$.
From the assumed form of $r_{\alpha,\bsgamma} : \Z^d \to \R_{>0}$ in~\eqref{eq:r-form-CBC} it now follows that $r_{\alpha,\bsgamma}^{-1} \in \ell_2(\Z^d)$ and we can thus apply the previous theorem, Theorem~\ref{thm:wce-F-d-r}, to write the squared worst-case error of a lattice rule $Q_{d,n,\bsz}$ in our space $F^{\alpha,\bsgamma}_r([0,1]^d)$ as, using the shorthand notation for the sum as in~\eqref{eq:fake-0-dimensional-sum},
\begin{align*}
  \journalonly{\MoveEqLeft[1]} e^2(Q_{d,n,\bsz}, F^{\alpha,\bsgamma}_r([0,1]^d))
  \journalonly{\\*}&=
  \sum_{\emptyset \ne \setu \subseteq \{1,\ldots,d\}}
  \gamma_\setu
  \sum_{\substack{\bsh_\setu \in (\Z \setminus \{0\})^{|\setu|} \\ \bsh_\setu \cdot \bsz_\setu \equiv 0 \tpmod{n}}}
  \prod_{j \in \setu} r_\alpha^{-2}(h_j)
  \\
  &=
  \sum_{\emptyset \ne \setu \subseteq \{1,\ldots,d-1\}}
  \gamma_\setu
  \sum_{\substack{\bsh_\setu \in (\Z \setminus \{0\})^{|\setu|} \\ \bsh_\setu \cdot \bsz_\setu \equiv 0 \tpmod{n}}}
  \prod_{j \in \setu} r_\alpha^{-2}(h_j)
  \\
  &\qquad+
  \sum_{d \in \setu \subseteq \{1,\ldots,d\}}
  \gamma_\setu
  \sum_{\substack{\bsh_\setu \in (\Z \setminus \{0\})^{|\setu|} \\ \bsh_\setu \cdot \bsz_\setu \equiv 0 \tpmod{n}}}
  \prod_{j \in \setu} r_\alpha^{-2}(h_j)
  \\
  &=
  e^2(Q_{d-1,n,\bsz_{\{1,\ldots,d-1\}}}, F^{\alpha,\bsgamma}_r([0,1]^{d-1}))
  \journalbreak\journalonly{\qquad} +
  \theta_d(z_d; z_1,\ldots,z_{d-1})
  ,
\end{align*}
with
\begin{align*}
  \journalonly{\MoveEqLeft[1]} \theta_d(z_d; z_1,\ldots,z_{d-1})
  =
  \theta_d(z_d; z_1,\ldots,z_{d-1}, r, \alpha, \bsgamma)
  \journalonly{\\*}&:=
  \sum_{d \in \setu \subseteq \{1,\ldots,d\}}
  \gamma_\setu
  \sum_{\substack{\bsh_\setu \in (\Z \setminus \{0\})^{|\setu|} \\ \bsh_\setu \cdot \bsz_\setu \equiv 0 \tpmod{n}}}
  \prod_{j \in \setu} r_\alpha^{-2}(h_j)
  .
\end{align*}
This forms the basis for an inductive proof since the squared worst-case error can be written as the squared worst-case error for $d-1$ and an additional term $\theta_d$, which we consider as a function of~$z_d$ while keeping all previous components $z_1$, \ldots, $z_{d-1}$ fixed.
We raise the squared worst-case error to a power $1/(2\lambda)$ with $2\lambda \ge 1$, for which, from Section~\ref{sec:Jensen} equation~\eqref{eq:fake-jensen}, we know that
\begin{align*}
  [ e^2(Q_{d,n,\bsz}, F^{\alpha,\bsgamma}_r([0,1]^d)) ]^{1/(2\lambda)}
  &=
  e^{1/\lambda}(Q_{d,n,\bsz}, F^{\alpha,\bsgamma}_r([0,1]^d))
  \\
  &\le
  e^{1/\lambda}(Q_{d-1,n,\bsz_{\{1,\ldots,d-1\}}}, F^{\alpha,\bsgamma}_r([0,1]^{d-1}))
  \\
  &\qquad+
  \theta_d^{1/(2\lambda)}(z_d; z_1,\ldots,z_{d-1})
  .
\end{align*}

For $d = 1$ we can pick any $z_1 \in \Z_n^\times$ to obtain
\begin{align*}
  e^2(Q_{1,n,z_1}, F^{\alpha,\bsgamma}_r([0,1]))
  &=
  \gamma_{\{1\}} \, \sum_{\substack{0 \ne h_1 \in \Z \\ h_1 z_1 \equiv 0 \tpmod{n}}} r_\alpha^{-2}(h_1)
  \\
  &=
  \gamma_{\{1\}} \, \sum_{\substack{0 \ne h_1 \in \Z \\ h_1 \equiv 0 \tpmod{n}}} r_\alpha^{-2}(h_1)
  \\
  &\le
  \frac{\gamma_{\{1\}}}{n^{2\alpha}} \, \sum_{0 \ne h_1 \in \Z} r_\alpha^{-2}(h_1)
  \\
  &=
  \frac{\gamma_{\{1\}} \, \mu_\alpha(1/2)}{n^{2\alpha}}
  .
\end{align*}
To bring this into the form for our inductive proof, we raise it to the power $1/(2\lambda)$.
We have, using~\eqref{eq:fake-jensen} again, with $2\lambda \ge 1$,
\begin{align*}
  \mu_\alpha^{1/(2\lambda)}(1/2)
  &=
  \Bigg( \sum_{0 \ne h \in \Z} r_\alpha^{-2}(h) \Bigg)^{1/(2\lambda)}
  \le
  \sum_{0 \ne h \in \Z} r_\alpha^{-1/\lambda}(h)
  =
  \mu_\alpha(\lambda)
  ,
\end{align*}
where $\mu_\alpha(\lambda) < \infty$ for $\lambda < \alpha$, hence
\begin{align}\label{eq:induction-d=1}
  e^{1/\lambda}(Q_{1,n,z_1}, F^{\alpha,\bsgamma}_r([0,1]))
  &\le
  \frac{\gamma_{\{1\}}^{1/(2\lambda)} \, \mu_\alpha(\lambda)}{n^{\alpha/\lambda}}
  <
  \frac{\gamma_{\{1\}}^{1/(2\lambda)} \, \mu_\alpha(\lambda)}{n-1}
  .
\end{align}
We take the following bound as our induction hypothesis
\begin{align*}
  e^{1/\lambda}(Q_{d-1,n,\bsz_{\{1,\ldots,d-1\}}}, F^{\alpha,\bsgamma}_r([0,1]^{d-1}))
  &\le
  \frac1{n-1} \sum_{\emptyset \ne \setu \subseteq \{1,\ldots,d-1\}} \gamma_\setu^{1/(2\lambda)} \, \mu_\alpha^{|\setu|}(\lambda)
  .
\end{align*}
This is clearly satisfied for~\eqref{eq:induction-d=1}.

Next we show that if the induction hypothesis is satisfied for the previous choice $z_1$, \ldots, $z_{d-1}$, then we can find a good $z_d \in \Z_n^\times$ such that the hypothesis is satisfied for the $d$-dimensional error.
We pick $z_d \in \Z_n^\times$ which achieves the smallest value possible for $\theta_d(z; z_1, \ldots, z_{d-1})$ over all choices $z \in \Z_n^\times$.
There might be multiple valid choices.
In fact if $r_\alpha^2(-h) = r_\alpha^2(h) = r_\alpha^2(|h|)$, which is always the case for our Fourier multipliers, and $n > 2$ such that $|\Z_n^\times| \ge 2$, then there are always at least two equivalent choices due to the symmetry $\theta_d(-z; z_1, \ldots, z_{d-1}) = \theta_d(z; z_1, \ldots, z_{d-1})$ where $-z$ is the additive inverse $-z \equiv n - z \pmod{n}$.
These minimisers do not change when considering $\theta_d^{1/(2\lambda)}$ instead of~$\theta_d$.
Explicitly, we have $z_d \in Z_d(r, \alpha, \bsgamma, \bsz_{\{1,\ldots,d-1\}})$.
We could denote this with $z_d = z_d(r,\alpha,\bsgamma,\bsz_{\{1,\ldots,d-1\}})$.
This remark is important, since the selection does not depend on the choice of~$\lambda$ and in an implementation we hence just look for a minimiser of $\theta_d(z; z_1, \ldots, z_{d-1})$.
In the proof we bound $\theta_d^{1/(2\lambda)}(z_d) = \theta_d^{1/(2\lambda)}(z_d(r,\alpha,\bsgamma,\bsz_{\{1,\ldots,d-1\}}))$, and the resulting bound holds for any $\lambda \in [1/2, \alpha)$.
To obtain our bound we use the trick that ``the best choice is at least as good as the average'', which results in, again using~\eqref{eq:fake-jensen}, since $2\lambda \ge 1$, and the shorthand notation for the sum as in~\eqref{eq:fake-0-dimensional-sum},
\begin{align*}
  &
  \theta_d^{1/(2\lambda)}(z_d; z_1,\ldots,z_{d-1})
  \\
  &\quad\le
  \frac1{|\Z_n^\times|} \sum_{z \in \Z_n^\times} \theta_d^{1/(2\lambda)}(z; z_1,\ldots,z_{d-1})
  \\
  &\quad\le
  \frac1{|\Z_n^\times|} \sum_{z \in \Z_n^\times}
  \sum_{\setv \subseteq \{1,\ldots,d-1\}} \gamma_{\setv\cup\{d\}}^{1/(2\lambda)}
  \sum_{\substack{\bsh_\setv \in (\Z \setminus \{0\})^{|\setv|} \\ 0 \ne h_d \in \Z \\ \bsh_\setv \cdot \bsz_\setv + h_d z \equiv 0 \tpmod{n}}}
  r_\alpha^{-1/\lambda}(h_d) \,
  \prod_{j \in \setv} r_\alpha^{-1/\lambda}(h_j)
  \\
  \numberthis{eq:split-c}
  &\quad=
  \sum_{c \in \Z_n}
  \Bigg[
  \sum_{\setv \subseteq \{1,\ldots,d-1\}} \gamma_{\setv\cup\{d\}}^{1/(2\lambda)}
  \sum_{\substack{\bsh_\setv \in (\Z \setminus \{0\})^{|\setv|} \\ \bsh_\setv \cdot \bsz_\setv \equiv c \tpmod{n}}}
  \prod_{j \in \setv} r_\alpha^{-1/\lambda}(h_j)
  \Bigg]
  \\*
  &\qquad\qquad\quad
  \Bigg[
  \frac1{|\Z_n^\times|} \sum_{z \in \Z_n^\times}
  \sum_{\substack{0 \ne h_d \in \Z \\ c + h_d z \equiv 0 \tpmod{n}}}
  r_\alpha^{-1/\lambda}(h_d)
  \Bigg]
  .
\end{align*}
Given $c \in \Z_n$, $0 \ne h \in \Z$ and prime $n$ we split the congruence $c + h z \equiv 0 \pmod{n}$ into two cases while considering solutions for $z \in \Z_n^\times$.
Remember that $z \in \Z_n^\times$ means that $z$ has a multiplicative inverse $z^{-1} \in \Z_n^\times$.
Also $\Z_n^\times = \{1, \ldots, n-1\}$ since $n$ is prime.
If $c = 0$ then $h z \equiv 0 \pmod{n}$ reduces to the condition $h \equiv 0 \pmod{n}$ and all choices of $z \in \Z_n^\times$ are solutions of the congruence.
If $c \ne 0$ then, since $n$ is prime, $c \in \Z_n^\times$, and for $z \in \Z_n^\times$ this necessarily means we have the condition $h \bmod{n} \in \Z_n^\times$.
It follows that there is a unique solution $z \in \Z_n^\times$ provided $h \not\equiv 0 \pmod{n}$.
This results in the following bound
\begin{align*}
  \frac1{|\Z_n^\times|} \sum_{z \in \Z_n^\times}
  \sum_{\substack{0 \ne h \in \Z \\ c + h z \equiv 0 \tpmod{n}}}
  r_\alpha^{-1/\lambda}(h)
  &=
  \begin{cases}
    \displaystyle
    \sum_{\substack{0 \ne h \in \Z \\ h \equiv 0 \tpmod{n}}} r_\alpha^{-1/\lambda}(h),
    & \text{if } c = 0,
    \\[8mm]
    \displaystyle
    \frac1{|\Z_n^\times|} \sum_{\substack{0 \ne h \in \Z \\ h \not\equiv 0 \tpmod{n}}} r_\alpha^{-1/\lambda}(h),
    & \text{if } c \ne 0,
  \end{cases}
  \\
  &\le
  \begin{cases}
    \displaystyle
    n^{-\alpha/\lambda} \, \mu_\alpha(\lambda),
    & \text{if } c = 0,
    \\[2mm]
    \displaystyle
    \frac{\mu_\alpha(\lambda)}{|\Z_n^\times|},
    & \text{if } c \ne 0,
  \end{cases}
  \\
  &\le
  \frac{\mu_\alpha(\lambda)}{n-1}
  ,
\end{align*}
where we used $\lambda < \alpha$ and hence $n^{-\alpha/\lambda} < n^{-1} < (n-1)^{-1}$ and $|\Z_n^\times| = n - 1$.

Hence, for our choice $z_d = z_d(r,\alpha,\bsgamma,\bsz_{\{1,\ldots,d-1\}})$,
\begin{align*}
  \journalonly{\MoveEqLeft[1]} \theta_d^{1/(2\lambda)}(z_d; z_1,\ldots,z_{d-1})
  \journalonly{\\*}&\le
  \frac{\mu_\alpha(\lambda)}{n-1}
  \sum_{c \in \Z_n}
  \sum_{\setv \subseteq \{1,\ldots,d-1\}} \gamma_{\setv\cup\{d\}}^{1/(2\lambda)}
  \sum_{\substack{\bsh_\setv \in (\Z \setminus \{0\})^{|\setv|} \\ \bsh_\setv \cdot \bsz_\setv \equiv c \tpmod{n}}}
  \prod_{j \in \setv} r_\alpha^{-1/\lambda}(h_j)
  \\
  &=
  \frac{\mu_\alpha(\lambda)}{n-1}
  \sum_{\setv \subseteq \{1,\ldots,d-1\}} \gamma_{\setv\cup\{d\}}^{1/(2\lambda)}
  \sum_{\bsh_\setv \in (\Z \setminus \{0\})^{|\setv|}}
  \prod_{j \in \setv} r_\alpha^{-1/\lambda}(h_j)
  \\
  &=
  \frac1{n-1}
  \sum_{d \in \setu \subseteq \{1,\ldots,d\}} \gamma_{\setu}^{1/(2\lambda)}
  \sum_{\bsh_\setu \in (\Z \setminus \{0\})^{|\setu|}}
  \prod_{j \in \setu} r_\alpha^{-1/\lambda}(h_j)
  \\
  &\le
  \frac1{n-1}
  \sum_{d \in \setu \subseteq \{1,\ldots,d\}} \gamma_{\setu}^{1/(2\lambda)}
  \mu_\alpha^{|\setu|}(\lambda)
  .
\end{align*}

Combining this with our induction hypothesis we obtain, for all $\lambda \in [1/2, \alpha)$,
\begin{align*}
  e^{1/\lambda}(Q_{d,n,\bsz}, F^{\alpha,\bsgamma}_r([0,1]^d))
  &\le
  \frac1{n-1} \sum_{\emptyset \ne \setu \subseteq \{1,\ldots,d-1\}} \gamma_\setu^{1/(2\lambda)} \, \mu_\alpha^{|\setu|}(\lambda)
  \\
  &\qquad\qquad+
  \frac1{n-1}
  \sum_{d \in \setu \subseteq \{1,\ldots,d\}} \gamma_{\setu}^{1/(2\lambda)}
  \mu_\alpha^{|\setu|}(\lambda)
  \\
  &=
  \frac1{n-1} \sum_{\emptyset \ne \setu \subseteq \{1,\ldots,d\}} \gamma_\setu^{1/(2\lambda)} \, \mu_\alpha^{|\setu|}(\lambda)
  .
\end{align*}
Raising both sides to the power $\lambda$ finishes the proof.
\end{proof}

The minimisation in~\eqref{eq:CBC-argmin-Zntimes} is the \emph{component-by-component construction}.
We can stop the constructions in Theorem~\ref{thm:CBC-prime}, and the following Theorem~\ref{thm:CBC-prime-or-composite}, at any dimension, and we can resume later if more dimensions are needed.
Every leading subvector satisfies the corresponding bound in its number of dimensions.
Depending on the structure of the weights, one can obtain fast implementations based on the FFT, see, e.g.,~\cite{NC2006-prime,NC2006,CKN2006,Nuy2014}, and they are called \emph{fast component-by-component construction}.
An alternative proof that covers both lattice rules with prime~$n$ and polynomial lattice rules with an irreducible modulus is given in~\cite[Theorem~5.1]{Nuy2014}, which also describes how to formulate fast CBC algorithms using FFTs and convolutions.
We will not discuss fast constructions here.

The worst-case error bound~\eqref{eq:CBC-prime-upper-bound} is dimension-independent when we can bound it independently of~$d$.
Since $\mu_\alpha(\lambda)$ is a number, this all depends on the weight family~$\bsgamma$.
More precisely, the bound is dimension-independent, also said to be strongly tractable, see, e.g., \cite{NW2010}, if
\begin{align}\label{eq:tract-gen}
  \sup_{d \in \N} \;
  \sum_{\emptyset \ne \setu \subseteq \{1,\ldots,d\}}
  \gamma_\setu^{1/(2\lambda)} \,
  \mu_\alpha^{|\setu|}(\lambda)
  &<
  \infty
  .
\end{align}
For product weights $\gamma_\setu = \prod_{j\in\setu} \gamma_j$, Lemma~\ref{lem:product-weight-sum}, applied with $\kappa = \mu_\alpha(\lambda)$ and weights $\gamma_j^{1/(2\lambda)}$, shows that this condition holds if $\sum_{j\ge1} \gamma_j^{1/(2\lambda)} < \infty$.
For POD weights $\gamma_\setu = \Gamma_{|\setu|} \prod_{j\in\setu} \gamma_j$, the summability condition is
\begin{align}\label{eq:tract-pod}
  \sup_{d \in \N} \;
  \sum_{\emptyset \ne \setu \subseteq \{1,\ldots,d\}}
  \Gamma_{|\setu|}^{1/(2\lambda)}
  \,
  \mu_\alpha^{|\setu|}(\lambda)
  \,
  \prod_{j\in\setu} \gamma_j^{1/(2\lambda)}
  &<
  \infty
  .
\end{align}
We can use Lemma~\ref{lem:POD-product-weight-bound}(i): if $\Gamma_\ell^{1/(2\lambda)}$ grows at most exponentially, i.e., $\sup_{\ell \ge 1} \Gamma_\ell^{1/\ell} < \infty$, and if $\sum_{j \ge 1} \gamma_j^{1/(2\lambda)} < \infty$ then we have a dimension-independent bound.
We can also use Lemma~\ref{lem:POD-product-weight-bound}(iii): if $\Gamma_\ell^{1/(2\lambda)} \le B^\ell \, \ell!$ for some $B > 0$ and $\sum_{j \ge 1} \gamma_j^{1/(2\lambda)} < \infty$, then we again obtain a dimension-independent bound.
In particular, this works for the weights motivated by the PDE bound~\eqref{eq:PDE-bound} where $\Gamma_\ell = \ell!$ since $1/(2\lambda) \le 1$.

We now move to composite~$n$.
Existence results of good generating vectors for arbitrary~$n$ follow, for example, from the averaging results in~\cite[Theorem~5.10]{Nie92} and~\cite[Theorem~5.8]{DKS2013}.
Component-by-component constructions for composite~$n$ and their fast implementations are treated in~\cite{KJ2002a,NC2006}.
The corresponding arguments are more involved than for prime~$n$.
In the next theorem we use~\cite[Lemmas~4.1 and~4.2]{KMN2025} to count the number of solutions $z \in \Z_n^\times$ for $c + h z \equiv 0 \pmod{n}$ given $c, h$ and~$n$, together with partitioning sums on the gcd with~$n$ in some way.
For $n \in \N$ let
\begin{align*}
  \divisors(n)
  &:=
  \big\{ m \in \N : m \mid n \big\}
  ,
\end{align*}
denote its set of positive divisors.
We will use a couple of simple tricks.
Given $n \in \N$ and $\bsz \in \Z^d$, we can partition the following sums as follows
\begin{align}
  \label{eq:ip-partition}
  \sum_{\bsh \in \Z^d} F(\bsh)
  &=
  \sum_{c \in \Z_n}
  \sum_{\substack{\bsh \in \Z^d \\ \bsh \cdot \bsz \equiv c \tpmod{n}}}
  F(\bsh)
  ,
  \\
  \label{eq:gcd-partition}
  \sum_{c \in \Z_n} G(c)
  &=
  \sum_{g \in \divisors(n)}
  \sum_{\substack{c \in \Z_n \\ \gcd(c,n) = g}}
  G(c)
  .
\end{align}
Note that for $n \in \N = \{1, 2, \ldots\}$ we have $\gcd(0,n) = n$.
These can of course also be used in reverse to remove the extra sums if nothing else depends anymore on the extra variables we introduced.
Under these sums there could be additional conditions, like excluding zeros.
This does not matter.

\begin{theorem}[Component-by-component construction for lattice rules with prime or composite~$n$]\label{thm:CBC-prime-or-composite}
Let $\alpha$, $r_\alpha$, $\bsgamma$ and the spaces $F^{\alpha,\bsgamma}_r([0,1]^d)$ for $d = 1, 2, \ldots$ satisfy the corresponding assumptions of Theorem~\textup{\ref{thm:CBC-prime}}.
Let $n \in \N$ with $n \ge 2$.
Pick any $z_1 \in \Z_n^\times$ and, successively for $d = 2, 3, \ldots$, choose
\begin{align*}
  z_d
  &\in
  \underset{z \in \Z_n^\times}{\arg\min}\;
  e(Q_{d,n,(z_1,\ldots,z_{d-1},z)}, F^{\alpha,\bsgamma}_r([0,1]^d))
  .
\end{align*}
Then, for every $1/2 \le \lambda < \alpha$ and every $d = 1, 2, \ldots$, the resulting CBC generating vector satisfies
\begin{align}\label{eq:CBC-composite-upper-bound}
  e(Q_{d,n,\bsz_{\{1,\ldots,d\}}},F^{\alpha,\bsgamma}_r([0,1]^d))
  &\le
  \left(
    \frac1{\varphi(n)}
    \sum_{\emptyset \ne \setu \subseteq \{1,\ldots,d\}}
    \gamma_\setu^{1/(2\lambda)} \,
    \mu_\alpha^{|\setu|}(\lambda)
  \right)^{\!\lambda}
  ,
\end{align}
with $\mu_\alpha$ defined by~\eqref{eq:def-mu-alpha-lambda}.
\end{theorem}
\begin{proof}
We follow the previous proof until~\eqref{eq:split-c} where, using the shorthand notation for the sum as in~\eqref{eq:fake-0-dimensional-sum}, we need to bound
\begin{align*}
  \theta_d^{1/(2\lambda)}(z_d; z_1,\ldots,z_{d-1})
  &\le
  \sum_{\setv \subseteq \{1,\ldots,d-1\}} \gamma_{\setv\cup\{d\}}^{1/(2\lambda)}
  \sum_{c \in \Z_n}
  \sum_{\substack{\bsh_\setv \in (\Z \setminus \{0\})^{|\setv|} \\ \bsh_\setv \cdot \bsz_\setv \equiv c \tpmod{n}}}
  \prod_{j \in \setv} r_\alpha^{-1/\lambda}(h_j)
  \\*
  &\qquad\qquad\qquad\qquad
  \frac1{|\Z_n^\times|} \sum_{z \in \Z_n^\times}
  \sum_{\substack{0 \ne h_d \in \Z \\ c + h_d z \equiv 0 \tpmod{n}}}
  r_\alpha^{-1/\lambda}(h_d)
  .
\end{align*}

We will use the elementary fact, see also~\cite[Lemmas~4.1 and~4.2]{KMN2025}, that for fixed $c \in \Z_n$ and $h \in \Z$, the congruence $c + h z \equiv 0 \pmod{n}$ has exactly $\gcd(h, n)$ solutions $z \in \Z_n$ if $\gcd(h, n) \mid c$, and no solutions otherwise.
We pick the components of the generating vector from $\Z_n^\times$, so this serves as an upper bound.
Moreover, if $z \in \Z_n^\times$ solves $c + h z \equiv 0 \pmod{n}$ then we have $\gcd(h, n) = \gcd(z h, n) = \gcd(c, n)$.
We thus have, given $c \in \Z_n$ and $h \in \Z$,
\begin{align}\label{eq:CBC-composite-unit-count}
  \left|\big\{z \in \Z_n^\times : c + h z \equiv 0 \pmod{n}\big\}\right|
  &\le
  \begin{cases}
    \gcd(h, n), & \text{if } \gcd(h, n) = \gcd(c, n), \\
    0, & \text{otherwise}.
  \end{cases}
\end{align}
We will also use $\gcd(h, n) = g \Rightarrow g \mid h$ and $\gcd(h/g, n/g) = 1$.

We now combine the partitioning of~\eqref{eq:ip-partition} and~\eqref{eq:gcd-partition}, the congruence properties~\eqref{eq:CBC-composite-unit-count} and the decay property~\eqref{eq:r-multiplier-scaling}, to obtain, for $\setv \subseteq \{1, \ldots, d-1\}$,
\begin{align*}
  &
  \sum_{c \in \Z_n}
  \sum_{\substack{\bsh_\setv \in (\Z \setminus \{0\})^{|\setv|} \\ \bsh_\setv \cdot \bsz_\setv \equiv c \tpmod{n}}}
  \prod_{j \in \setv} r_\alpha^{-1/\lambda}(h_j)
  \;
  \frac1{|\Z_n^\times|} \sum_{z \in \Z_n^\times}
  \smashoperator[r]{\sum_{\substack{0 \ne h_d \in \Z \\ c + h_d z \equiv 0 \tpmod{n}}}}
  r_\alpha^{-1/\lambda}(h_d)
  \\
  &\;=
  \sum_{g \in \divisors(n)}
  \sum_{\substack{c \in \Z_n \\ \gcd(c, n) = g}}
  \sum_{\substack{\bsh_\setv \in (\Z \setminus \{0\})^{|\setv|} \\ \bsh_\setv \cdot \bsz_\setv \equiv c \tpmod{n}}}
  \prod_{j \in \setv} r_\alpha^{-1/\lambda}(h_j)
  \;
  \frac1{|\Z_n^\times|} \sum_{z \in \Z_n^\times}
  \smashoperator[r]{\sum_{\substack{0 \ne h_d \in \Z \\ \gcd(h_d, n) = g \\ c + h_d z \equiv 0 \tpmod{n}}}}
  r_\alpha^{-1/\lambda}(h_d)
  \\
  &\;\le
  \sum_{g \in \divisors(n)}
  \sum_{\substack{c \in \Z_n \\ \gcd(c, n) = g}}
  \sum_{\substack{\bsh_\setv \in (\Z \setminus \{0\})^{|\setv|} \\ \bsh_\setv \cdot \bsz_\setv \equiv c \tpmod{n}}}
  \prod_{j \in \setv} r_\alpha^{-1/\lambda}(h_j)
  \;
  \frac{g}{|\Z_n^\times|}
  \smashoperator[r]{\sum_{\substack{0 \ne h \in \Z \\ \gcd(h, n/g) = 1}}}
  r_\alpha^{-1/\lambda}(g \, h)
  \\
  &\;\le
  \sum_{g \in \divisors(n)}
  \sum_{\substack{c \in \Z_n \\ \gcd(c, n) = g}}
  \sum_{\substack{\bsh_\setv \in (\Z \setminus \{0\})^{|\setv|} \\ \bsh_\setv \cdot \bsz_\setv \equiv c \tpmod{n}}}
  \prod_{j \in \setv} r_\alpha^{-1/\lambda}(h_j)
  \;
  \frac{g}{|\Z_n^\times|}
  \sum_{0 \ne h \in \Z}
  g^{-\alpha/\lambda} \,
  r_\alpha^{-1/\lambda}(h)
  \\
  &\;\le
  \frac{\mu_\alpha(\lambda)}{\varphi(n)}
  \sum_{g \in \divisors(n)}
  \sum_{\substack{c \in \Z_n \\ \gcd(c, n) = g}}
  \sum_{\substack{\bsh_\setv \in (\Z \setminus \{0\})^{|\setv|} \\ \bsh_\setv \cdot \bsz_\setv \equiv c \tpmod{n}}}
  \prod_{j \in \setv} r_\alpha^{-1/\lambda}(h_j)
  \\
  &\;=
  \frac{\mu_\alpha(\lambda)}{\varphi(n)}
  \sum_{\bsh_\setv \in (\Z \setminus \{0\})^{|\setv|}}
  \prod_{j \in \setv} r_\alpha^{-1/\lambda}(h_j)
  \\
  &\;=
  \frac{\mu_\alpha^{|\setv|+1}(\lambda)}{\varphi(n)}
  ,
\end{align*}
where in the last inequality we used $g^{1 - \alpha/\lambda} \le 1$ since $g \ge 1$ and $\lambda < \alpha$.
Summation over~$\setv$ gives
\begin{align*}
  \theta_d^{1/(2\lambda)}(z_d; z_1,\ldots,z_{d-1})
  &\le
  \frac1{\varphi(n)}
  \sum_{d \in \setu \subseteq \{1,\ldots,d\}}
  \gamma_\setu^{1/(2\lambda)} \, \mu_\alpha^{|\setu|}(\lambda)
  .
\end{align*}
For $d = 1$ we already had the required bound from the previous proof in~\eqref{eq:induction-d=1} since $\varphi(n) \le n - 1$.
The induction then yields~\eqref{eq:CBC-composite-upper-bound}.
\end{proof}

Obtaining dimension-independent bounds follows exactly the same discussion as after Theorem~\ref{thm:CBC-prime}, see~\eqref{eq:tract-gen} and~\eqref{eq:tract-pod}.

The bounds in Theorems~\ref{thm:CBC-prime} and~\ref{thm:CBC-prime-or-composite} also hold if the components are chosen to minimise the error over all of~$\Z_n$, starting with any $z_1 \in \Z_n^\times$.
This is because the minimum over $\Z_n$ is at most the minimum over $\Z_n^\times$.
For prime $n$, we have $\varphi(n) = n - 1$, so the two theorems give the same bound.
For prime powers $n = p^m$, there is also the easy bound $1/\varphi(n) \le 2 / n$, see, e.g.,~\cite[Theorem~5.1]{KN2016}.

\subsubsection{Convergence rates and dimension-independent bounds}\label{sec:lattice-convergence-rates}

The setup of the previous theorem applies directly to the weighted unanchored periodic spaces $\widetilde{H}^{\alpha,\bsgamma,\rmu}_\mix([0,1]^d)$ and $\widetilde{H}^{\alpha,\bsgamma,\rmu}_{\mix,1}([0,1]^d)$.
Their Fourier multipliers satisfy~\eqref{eq:r-multiplier-scaling} with equality.
Theorem~\ref{thm:periodic-upper-bound-transfer} transfers results from $\widetilde{H}^{\alpha,\bsgamma,\rmu}_{\mix,1}([0,1]^d)$ to the other spaces in Table~\ref{tbl:periodic-equivalent-norms}, which is what we use in the following theorem.
We first state results for general weights and then specialise to product and POD weights, giving conditions for dimension-independent convergence.

\begin{theorem}[Component-by-component bounds for lattice rules in the weighted periodic mixed spaces from Table~\ref{tbl:periodic-equivalent-norms}]\label{thm:lattice-rule-upper-bounds}
Let $\alpha > 1/2$ and $\bsgamma = \{\gamma_\setu\}_{\setu \in \frakU}$ be general weights with $\gamma_\emptyset = 1$ and $\gamma_\setu > 0$ for all~$\setu$.
Let $n \in \N$ with $n \ge 2$.
Let $\bsz$ be constructed component-by-component for the unanchored space $\widetilde{H}^{\alpha,\bsgamma,\rmu}_{\mix,1}([0,1]^d)$ using Theorem~\textup{\ref{thm:CBC-prime-or-composite}}.
Then, for every $1/2 \le \lambda < \alpha$, we have
\begin{align}\label{eq:CBC-periodic-general-totient-bound}
  e(Q_{d,n,\bsz}, \widetilde{H}^{\alpha,\bsgamma,\rmu}_{\mix,1}([0,1]^d))
  &\le
  \Bigg(
  \frac1{\varphi(n)}
  \,
  \sum_{\emptyset \ne \setu \subseteq \{1,\ldots,d\}}
  \gamma_\setu^{1/(2\lambda)} \,
  (
    2\zeta(\alpha / \lambda)
  )^{|\setu|}
  \Bigg)^{\!\lambda}
  .
\end{align}
Substituting $\delta = \alpha - \lambda \in (0, \alpha - 1/2]$, for every $\epsilon > 0$, there is a constant $C_{\alpha,\delta,\epsilon} = C_{\lambda,\epsilon} > 0$, independent of $n$, $d$ and $\bsgamma$, such that
\begin{align*}
  \numberthis{eq:CBC-periodic-general-upper-bound}
  \journalonly{\MoveEqLeft[1]} e(Q_{d,n,\bsz}, \widetilde{H}^{\alpha,\bsgamma,\rmu}_{\mix,1}([0,1]^d))
  \journalonly{\\* }&\le
  C_{\alpha,\delta,\epsilon} \, n^{-\alpha+\delta+\epsilon} \,
  \Bigg(
    \sum_{\emptyset \ne \setu \subseteq \{1,\ldots,d\}}
    \gamma_\setu^{1/(2(\alpha - \delta))} \,
    \Big(
      2\zeta\big(\frac{\alpha}{\alpha - \delta}\big)
    \Big)^{|\setu|}
  \Bigg)^{\!\alpha - \delta}
  .
\end{align*}
Both bounds hold for $e(Q_{d,n,\bsz}, F^{\alpha,\bsgamma}_r([0,1]^d))$ with
\begin{align*}
  F^{\alpha,\bsgamma}_r([0,1]^d)
  &\in
  \Big\{
    \widetilde{H}^{\alpha,\bsgamma}_\mix([0,1]^d),
    \widetilde{H}'^{\alpha,\bsgamma}_\mix([0,1]^d),
    \widetilde{H}'^{\alpha,\bsgamma}_{\mix,1}([0,1]^d),
    \widetilde{H}''^{\alpha,\bsgamma}_\mix([0,1]^d),
  \\*
  &\qquad
    \widetilde{H}''^{\alpha,\bsgamma}_{\mix,1}([0,1]^d),
    \widetilde{H}^{\alpha,\bsgamma,\rmu}_\mix([0,1]^d),
    \widetilde{H}^{\alpha,\bsgamma,\rmu}_{\mix,1}([0,1]^d)
  \Big\}
  ,
\end{align*}
where $r = r_{\alpha,\bsgamma}$ is the corresponding Fourier multiplier from Table~\textup{\ref{tbl:periodic-equivalent-norms}} and $\alpha$ is restricted to the corresponding range, and with $\alpha \ge 1$ for $\widetilde{H}''^{\alpha,\bsgamma}_{\mix,1}([0,1]^d)$.

For fixed $\alpha$, $\delta$ and $\epsilon$, and every $c > 0$, there is an $n_0 = n_0(\alpha,\delta,\epsilon,c)$, independent of $d$ and $\bsgamma$, such that~\eqref{eq:CBC-periodic-general-upper-bound} holds with $c$ in place of $C_{\alpha,\delta,\epsilon}$ for all $n \ge n_0$.
Furthermore, if the sum in the bound is uniformly bounded in~$d$ then the convergence is dimension-independent and is $o(n^{-\alpha+\delta+\epsilon})$ for every $\epsilon > 0$.
In particular the sum is uniformly bounded in the following cases.
\begin{itemize}
  \item
  If $\bsgamma$ are product weights $\gamma_\setu = \prod_{j \in \setu} \gamma_j$ and $\sum_{j \ge 1} \gamma_j^{1/(2(\alpha - \delta))} < \infty$, the sum in~\eqref{eq:CBC-periodic-general-totient-bound} and~\eqref{eq:CBC-periodic-general-upper-bound} is uniformly bounded in~$d$.

  For the anchored periodic space $\widetilde{H}^{\alpha,\bsgamma,\bsa}_\mix([0,1]^d)$, which is not in the above list, with $\alpha \in \N$ and product weights, the bound also holds after multiplication by $\kappa_{d,\alpha}(\bsgamma)$ from~\eqref{eq:def-kappa-d-alpha}.
  The extra constant $\kappa_{d,\alpha}(\bsgamma)$ additionally needs $\sum_{j \ge 1} \sqrt{\gamma_j} < \infty$ to be uniformly bounded in~$d$.

  \item
  If $\bsgamma$ are POD weights $\gamma_\setu = \Gamma_{|\setu|} \, \prod_{j \in \setu} \gamma_j$, then the sum in~\eqref{eq:CBC-periodic-general-totient-bound} and~\eqref{eq:CBC-periodic-general-upper-bound} is uniformly bounded in~$d$ if $\sum_{j \ge 1} \gamma_j^{1/(2(\alpha - \delta))} < \infty$ and either $\sup_{\ell \ge 1} \Gamma_\ell^{1/\ell} < \infty$ or $\Gamma_\ell^{1/(2(\alpha - \delta))} \le B^\ell \, \ell!$ for some $B > 0$.
\end{itemize}
\end{theorem}
\begin{proof}
For the unanchored space $\widetilde{H}^{\alpha,\bsgamma,\rmu}_{\mix,1}([0,1]^d)$ we have, see Table~\ref{tbl:periodic-equivalent-norms}, $r_\alpha(h) = |h|^\alpha$ for $h \ne 0$.
Using this in~\eqref{eq:def-mu-alpha-lambda} we get
\begin{align*}
  \mu_\alpha(\lambda)
  &=
  2\zeta(\alpha/\lambda)
  ,
\end{align*}
which is finite for $\alpha/\lambda > 1$, i.e., $\lambda < \alpha$.
We satisfy all the requirements of Theorem~\ref{thm:CBC-prime-or-composite} and hence obtain~\eqref{eq:CBC-periodic-general-totient-bound}.

Next we use
\begin{align*}
  1/\varphi(n)
  &\in
  o(n^{-1+\eta})
\end{align*}
for every $\eta > 0$, i.e., for every $c_\eta > 0$ there is $n_0 = n_0(\eta,c_\eta)$ such that for all $n \ge n_0$ we have
\begin{align*}
  \frac1{\varphi(n)}
  &\le
  c_\eta \, n^{-1+\eta}
  .
\end{align*}
Taking $\lambda = \alpha - \delta$ and $\eta = \epsilon/\lambda = \epsilon/(\alpha - \delta)$, we then arrive at~\eqref{eq:CBC-periodic-general-upper-bound} with $C_{\alpha,\delta,\epsilon} = C_{\lambda,\epsilon} = c_\eta^\lambda = c_{\epsilon/\lambda}^\lambda = c_{\epsilon/(\alpha - \delta)}^{\alpha - \delta}$.

Using Theorem~\ref{thm:periodic-upper-bound-transfer} we can use norm embeddings to transfer the result to any of the listed function spaces, including the anchored periodic space $\widetilde{H}^{\alpha,\bsgamma,\bsa}_\mix([0,1]^d)$ for product weights.

The conditions for dimension-independent bounds for product weights follow from Lemma~\ref{lem:product-weight-sum} and Theorem~\ref{thm:periodic-upper-bound-transfer}, and for POD weights from Lemma~\ref{lem:POD-product-weight-bound}.
\end{proof}

\subsection{Reproducing kernel Hilbert spaces}\label{sec:rkhs}

In the previous sections we derived lower and upper bounds using relatively direct techniques.
Since we (mostly) stuck to Hilbert versions of our function spaces with sufficient smoothness to have continuous embedding into $C([0,1]^d)$, these spaces are all reproducing kernel Hilbert spaces (RKHS).
A lot of the statements we have shown can also be shown using RKHS techniques.
In particular, for deterministic worst-case integration over the unit ball of an RKHS with bounded integration functional, we still need to prove one last result already used in the proof of Theorem~\ref{thm:curse-nonlinear-adaptive}: together with the symmetry of the unit ball and the zero-function argument used there, it shows that nonlinear and adaptive algorithms cannot improve upon an optimal non-adaptive linear rule.
We will first introduce some general notions of reproducing kernel Hilbert spaces.
We state these without proof.
These can be found, e.g., in~\cite{Aro50,DKS2013,NW2008}.
A particularly nice short overview is given in \cite[Chapter~9]{Moo85} which applies this to one-dimensional quadrature in an anchored Sobolev space of smoothness~$1$ and also gives the reproducing kernels for $\alpha \in \N$.

\subsubsection{Reproducing kernels and the worst-case error}\label{sec:rkhs-representers-worst-case-error}

Let $H(K)$ be a complex- or real-valued reproducing kernel Hilbert space on~$[0,1]^d$.
We assume its reproducing kernel $K : [0,1]^d \times [0,1]^d \to \R$ is real-valued and symmetric.
(This holds for all the RKHSs considered here, also if they are complex-valued.)
The kernel needs to be symmetric positive definite, that is
\begin{align*}
  \sum_{i, j = 1}^n \overline{w_i} \, w_j \, K(\bsx_i, \bsx_j)
  &\ge
  0
  ,
\end{align*}
for any $n \in \N$, $\bsx_1, \ldots, \bsx_n \in [0,1]^d$ and $w_1, \ldots, w_n \in \bbC$.
A positive definite kernel uniquely defines an RKHS and vice versa.
We use $\langle f, g \rangle_{H(K)}$ to denote its inner product and we use the convention that the inner product is linear in the first argument and conjugate-linear in the second argument.
The inner product induces the norm
\begin{align*}
  \|f\|_{H(K)}
  &=
  \sqrt{\langle f, f \rangle_{H(K)}}
  .
\end{align*}

The defining property of a reproducing kernel Hilbert space is that for $f \in H(K)$ and $\bsx \in [0,1]^d$ we have
\begin{align*}
  f(\bsx)
  &=
  \langle f , K_\bsx \rangle_{H(K)}
  ,
\end{align*}
where $K_\bsx \in H(K)$ is defined by $K_\bsx := K(\cdot, \bsx)$.
Then $K_\bsx$ is called the \emph{Riesz representer} of point evaluation at~$\bsx$.
Continuity of point evaluation at a fixed $\bsx$ means that there exists a constant $C_\bsx < \infty$ such that point evaluation is bounded as
\begin{align*}
  |f(\bsx)|
  &\le
  C_\bsx \, \|f\|_{H(K)}
  \qquad\text{for all } f \in H(K)
  .
\end{align*}
A continuous embedding $H(K) \hookrightarrow C([0,1]^d)$ is a sufficient condition for continuity of point evaluation, but it is not a necessary condition.
E.g., Walsh spaces of smoothness $\alpha > 1/2$ contain discontinuous Walsh functions even though point evaluation is continuous, see, e.g., \cite[Section~4.4]{Nuy2014} and \cite[Section~6.1]{DKS2013}.
Using Cauchy--Schwarz we obtain
\begin{align*}
  |f(\bsx)|
  &=
  |\langle f, K_\bsx \rangle_{H(K)}|
  \le
  \|f\|_{H(K)} \, \|K_\bsx\|_{H(K)}
  =
  \sqrt{K(\bsx, \bsx)} \, \|f\|_{H(K)}
  \numberthis{eq:RKHS-point-evaluation-bound}
  .
\end{align*}
Hence $C_\bsx$ above can be taken as $\sqrt{K(\bsx,\bsx)}$.

More generally the Riesz representation theorem states that every continuous linear functional $L : H(K) \to \bbC$ has a unique representer $h \in H(K)$ such that for all $f \in H(K)$ we have $L(f) = \langle f, h \rangle_{H(K)}$.
Again, continuity means there exists a constant $C_L < \infty$ such that $|L(f)| \le C_L \, \|f\|_{H(K)}$ for all $f \in H(K)$.

For the spaces considered in these lecture notes, the integration functional is continuous / bounded.
Let $f \in H(K)$.
We have
\begin{align*}
  I_d(f)
  &=
  \int_{[0,1]^d} f(\bsx) \rd\bsx
  =
  \int_{[0,1]^d} \langle f, K_\bsx \rangle_{H(K)} \rd\bsx
  =
  \langle f, \int_{[0,1]^d} K_\bsx \rd\bsx \rangle_{H(K)}
  .
\end{align*}
Hence the Riesz representer of integration $I_d$ on $H(K)$ is
\begin{align}\label{eq:RKHS-integration-representer}
  h_d(\bsy)
  &:=
  \int_{[0,1]^d} K_\bsx(\bsy) \rd\bsx
  ,
\end{align}
and we have $I_d(f) = \langle f, h_d \rangle_{H(K)}$ for all $f \in H(K)$.
For a cubature rule $Q^{X,\bsw}_{d,n}$ with given nodes $X = \{\bsx_1, \ldots, \bsx_n\} \subset [0,1]^d$ and given weights $\bsw = (w_1, \ldots, w_n) \in \R^n$, we have
\begin{align*}
  Q^{X,\bsw}_{d,n}(f)
  &=
  \sum_{i=1}^n w_i \, f(\bsx_i)
  =
  \sum_{i=1}^n w_i \, \langle f, K_{\bsx_i} \rangle_{H(K)}
  =
  \langle f, \sum_{i=1}^n w_i \, K_{\bsx_i} \rangle_{H(K)}
  ,
\end{align*}
which reveals the Riesz representer of the cubature rule $Q^{X,\bsw}_{d,n}$ on $H(K)$ as
\begin{align}\label{eq:RKHS-cubature-representer}
  q^{X,\bsw}_{d,n}(\bsy)
  &:=
  \sum_{i=1}^n w_i \, K_{\bsx_i}(\bsy)
  ,
\end{align}
and we have $Q^{X,\bsw}_{d,n}(f) = \langle f, q^{X,\bsw}_{d,n} \rangle_{H(K)}$ for all $f \in H(K)$.
In passing we remark that complex weights $w_i$ do not help in our case since the kernel is real-valued.

Using these representers we obtain an explicit formula for the (squared) worst-case error in the next lemma.
We have already met this formula in the proof of Theorem~\ref{thm:curse}.

\begin{lemma}[Worst-case error in an RKHS]\label{lem:RKHS-wce}
Let $Q^{X,\bsw}_{d,n}$ be a cubature rule with given nodes $X = \{\bsx_1, \ldots, \bsx_n\} \subset [0,1]^d$ and given weights $\bsw = (w_1, \ldots, w_n) \in \R^n$ of the form
\begin{align}\label{eq:Q-X-w}
  Q^{X,\bsw}_{d,n}(f)
  =
  Q_{d,n}(f; (w_i, \bsx_i)_{i=1}^n)
  &=
  \sum_{i=1}^n w_i \, f(\bsx_i)
  ,
\end{align}
and let $H(K)$ be a reproducing kernel Hilbert space on $[0,1]^d$ with a real-valued symmetric positive definite kernel $K$ and assume integration $I_d$ as in~\eqref{eq:Id} is bounded on~$H(K)$, then the squared worst-case error has the exact form
\begin{align*}
  &
  e(Q^{X,\bsw}_{d,n}, H(K))^2
  =
  \|h_d - q^{X,\bsw}_{d,n}\|_{H(K)}^2
  \journalonly{\numberthis{eq:RKHS-wce-formula}}
  \\*
  &\articleonly{\qquad}\journalonly{\;}=
  \journalonly{\iint_{[0,1]^{2d}}}\articleonly{\int_{[0,1]^d} \int_{[0,1]^d}} K(\bsx,\bsy) \rd\bsx \rd\bsy
  -
  2 \sum_{i=1}^n w_i \int_{[0,1]^d} K(\bsx_i,\bsy) \rd\bsy
  +
  \sum_{i,j=1}^n w_i \, w_j \, K(\bsx_i,\bsx_j)
  \articleonly{\numberthis{eq:RKHS-wce-formula}}
  .
\end{align*}
This formula is easy to remember by its form $IIK - 2QIK + QQK$ which is the expansion of the quadratic form $\|h_d - q^{X,\bsw}_{d,n}\|_{H(K)}^2$.
\end{lemma}
\begin{proof}
Combining~\eqref{eq:RKHS-integration-representer} and~\eqref{eq:RKHS-cubature-representer} we obtain, for $f \in H(K)$,
\begin{align*}
  I_d(f) - Q^{X,\bsw}_{d,n}(f)
  &=
  \langle f, h_d - q^{X,\bsw}_{d,n} \rangle_{H(K)}
  .
\end{align*}
Thus $h_d - q^{X,\bsw}_{d,n}$ is the Riesz representer of the error functional~$I_d - Q^{X,\bsw}_{d,n}$.
Applying Cauchy--Schwarz gives
\begin{align*}
  |I_d(f) - Q^{X,\bsw}_{d,n}(f)|
  =
  |\langle f, h_d - q^{X,\bsw}_{d,n} \rangle_{H(K)}|
  &\le
  \|f\|_{H(K)} \, \|h_d - q^{X,\bsw}_{d,n}\|_{H(K)}
  ,
\end{align*}
and hence
\begin{align*}
  e(Q^{X,\bsw}_{d,n}, H(K))
  &\le
  \|h_d - q^{X,\bsw}_{d,n}\|_{H(K)}
  .
\end{align*}
We will now show this is equality.
The case $h_d - q^{X,\bsw}_{d,n} = 0$ is trivial so let us consider $h_d - q^{X,\bsw}_{d,n} \ne 0$.
The function
\begin{align*}
  f^*(\bsy)
  &:=
  \frac{h_d(\bsy) - q^{X,\bsw}_{d,n}(\bsy)}{\|h_d - q^{X,\bsw}_{d,n}\|_{H(K)}}
  ,
\end{align*}
has norm one and satisfies
\begin{align*}
  |I_d(f^*) - Q^{X,\bsw}_{d,n}(f^*)|
  &=
  \frac{|\langle h_d - q^{X,\bsw}_{d,n}, h_d - q^{X,\bsw}_{d,n} \rangle_{H(K)}|}{\|h_d - q^{X,\bsw}_{d,n}\|_{H(K)}}
  =
  \|h_d - q^{X,\bsw}_{d,n}\|_{H(K)}
  .
\end{align*}
Hence we have $e(Q^{X,\bsw}_{d,n}, H(K)) = \|h_d - q^{X,\bsw}_{d,n}\|_{H(K)}$.

We now expand $\|h_d - q^{X,\bsw}_{d,n}\|_{H(K)}^2 = \langle h_d - q^{X,\bsw}_{d,n}, h_d - q^{X,\bsw}_{d,n} \rangle_{H(K)}$.
Since $K : [0,1]^d \times [0,1]^d \to \R$ and $\bsw \in \R^n$ we have
\begin{align*}
  e(Q^{X,\bsw}_{d,n}, H(K))^2
  &=
  \langle h_d, h_d \rangle_{H(K)}
  -
  2 \, \langle q^{X,\bsw}_{d,n}, h_d \rangle_{H(K)}
  +
  \langle q^{X,\bsw}_{d,n}, q^{X,\bsw}_{d,n} \rangle_{H(K)}
  .
\end{align*}
Filling in the expressions for $h_d$ and $q^{X,\bsw}_{d,n}$, see~\eqref{eq:RKHS-integration-representer} and~\eqref{eq:RKHS-cubature-representer}, gives~\eqref{eq:RKHS-wce-formula}.
\end{proof}

\Needspace{10\baselineskip}
\subsubsection{Some kernels for the dominating mixed smoothness Sobolev spaces}\label{sec:rkhs-sobolev-kernels}

\paragraph{Kernels for the non-periodic spaces.}
In the non-periodic spaces of dominating mixed smoothness from Section~\ref{sec:non-periodic-spaces} with integer smoothness $\alpha \ge 1$, point evaluation is continuous and hence they are reproducing kernel Hilbert spaces.
For the weighted unanchored space $H^{\alpha,\bsgamma,\rmu}_\mix([0,1]^d)$ with $\alpha \in \N$ the kernel is
\begin{align}\label{eq:unanchored-kernel}
  K(\bsx,\bsy)
  &=
  \sum_{\setu \subseteq \{1,\ldots,d\}} \gamma_\setu \, \prod_{j \in \setu}
  \Bigg(
    \sum_{\tau=1}^{\alpha} \frac{B_\tau(x_j) \, B_\tau(y_j)}{(\tau!)^2}
    +
    \frac{(-1)^{\alpha+1}}{(2\alpha)!} \, B_{2\alpha}(|x_j-y_j|)
  \Bigg)
  ,
\end{align}
see, e.g.,~\cite{KSWW2010b,DKS2013,DKP2022,NS2023}, and where $B_\tau$ denotes the Bernoulli polynomial of degree~$\tau$.
For the weighted anchored space $H^{\alpha,\bsgamma,\bsa}_\mix([0,1]^d)$ with anchor $\bsa \in [0,1]^d$ and $\alpha \in \N$ the kernel is
\begin{align*}
  &K(\bsx,\bsy)
  \numberthis{eq:anchored-kernel}
  =
  \sum_{\setu \subseteq \{1,\ldots,d\}} \gamma_\setu \, \prod_{j \in \setu}
  \Bigg(
    \sum_{\tau=1}^{\alpha-1} \frac{(x_j-a_j)^\tau (y_j-a_j)^\tau}{(\tau!)^2}
  \\*
  &\qquad\articleonly{\hphantom{K(\bsx,\bsy) =}\qquad}
    + \bbone\{(x_j-a_j)(y_j-a_j) > 0\} \,
  \int_0^1
    \frac{(|x_j-a_j|-t)_+^{\alpha-1} (|y_j-a_j|-t)_+^{\alpha-1}}{((\alpha-1)!)^2}
    \rd{t}
  \Bigg)
  ,
\end{align*}
see, e.g.,~\cite{DKS2013,NN2021uniform}, where $u_+ := \max\{u,0\}$ with $u_+^0 := \bbone\{u > 0\}$ when $\alpha = 1$.
See also~\cite[Appendix~A.2]{NN2021lognormal} for a derivation of the anchored univariate kernel.
As noted there, the formulas in~\cite[Example~4.2]{KSWW2010b} and~\cite[Section~5.2]{DG2014} contain a typo for $\alpha \ge 2$: their ``0 otherwise'' case needs to be fixed as above.
A more direct way of writing the integral is as follows
\begin{align*}
  \MoveEqLeft[1]
  \int_0^1
  \frac{(|x_j-a_j|-t)_+^{\alpha-1} (|y_j-a_j|-t)_+^{\alpha-1}}{((\alpha-1)!)^2}
  \rd{t}
  \\*
  &=
  \int_0^{\min\{|x_j-a_j|,|y_j-a_j|\}}
  \frac{(|x_j-a_j|-t)^{\alpha-1} (|y_j-a_j|-t)^{\alpha-1}}{((\alpha-1)!)^2}
  \rd{t}
  .
\end{align*}
(But page widths in books are limited\ldots)

\paragraph{Kernels for the periodic spaces.}
For the periodic spaces from Section~\ref{sec:periodic-spaces} with smoothness $\alpha > 1/2$, point evaluation is continuous and hence they are reproducing kernel Hilbert spaces.
For every space $F^{\alpha,\bsgamma}_r([0,1]^d)$ from Table~\ref{tbl:periodic-equivalent-norms}, with $\alpha$ in its corresponding range and $\alpha > 1/2$, the kernel takes the form
\begin{align}\label{eq:Fourier-kernel}
  K(\bsx,\bsy)
  &=
  \sum_{\bsh \in \Z^d}
  r_{\alpha,\bsgamma}^{-2}(\bsh) \, \exp(2\pi\rmi \, \bsh \cdot (\bsx-\bsy))
  .
\end{align}
Since these multipliers all take the absolute value on the components of $\bsh$ we have $r_{\alpha,\bsgamma}(\bsh) = r_{\alpha,\bsgamma}(-\bsh)$ and hence the kernel is real valued and we could write $\cos(2\pi \, \bsh \cdot (\bsx-\bsy))$ instead of $\exp(2\pi\rmi \, \bsh \cdot (\bsx-\bsy))$.
For the unanchored periodic spaces $\widetilde{H}^{\alpha,\bsgamma,\rmu}_{\mix,1}([0,1]^d)$ and $\widetilde{H}^{\alpha,\bsgamma,\rmu}_\mix([0,1]^d)$ we can write the kernel in closed form using~\eqref{eq:unanchored-omega-j} for $\alpha \in \N$ or~\eqref{eq:unanchored-polylog} for $\alpha > 1/2$.
E.g., for the unanchored space $\widetilde{H}^{\alpha,\bsgamma,\rmu}_{\mix,1}([0,1]^d)$ with $\alpha \in \N$ we have
\begin{align}\label{eq:periodic-unanchored-kernel}
  K(\bsx,\bsy)
  &=
  \sum_{\setu \subseteq \{1,\ldots,d\}}
  \gamma_\setu \,
  \left[ \frac{(-1)^{\alpha+1} \, (2\pi)^{2\alpha}}{(2\alpha)!} \right]^{|\setu|} \,
  \prod_{j \in \setu} B_{2\alpha}(|x_j - y_j|)
  ,
\end{align}
and for general $\alpha > 1/2$ we have
\begin{align}\label{eq:periodic-unanchored-kernel-polylog}
  K(\bsx,\bsy)
  &=
  \sum_{\setu \subseteq \{1,\ldots,d\}}
  \gamma_\setu \,
  2^{|\setu|} \,
  \prod_{j \in \setu} \Re\Big\{\Li_{2\alpha}\big(\exp(2\pi\rmi \, (x_j - y_j))\big)\Big\}
  .
\end{align}
Dividing each weight $\gamma_\setu$ by $(2\pi)^{2|\setu|\alpha}$ gives the kernel for the unanchored space $\widetilde{H}^{\alpha,\bsgamma,\rmu}_\mix([0,1]^d)$ in both cases.

\subsubsection{Optimal weights for fixed nodes}\label{sec:rkhs-optimal-fixed-nodes}

The following result was used in the proof of Theorem~\ref{thm:curse-nonlinear-adaptive} to pass from arbitrary deterministic nonlinear and adaptive algorithms to linear nonadaptive cubature rules.
We assume that we are given cubature nodes $X = \{\bsx_1, \ldots, \bsx_n\} \subset [0,1]^d$ and we are looking for weights $\bsw^* \in \R^n$ that minimise the worst-case error for the cubature rule $Q^{X,\bsw^*}_{d,n}$ in $H(K)$.
How to find these optimal weights is quite well known, e.g., see \cite[Chapter~9]{Moo85} which demonstrates this for the one-dimensional anchored Sobolev space~$H^{1,a}([0,1])$ with~$a = 0$.

\begin{theorem}[Optimal linear cubature on fixed nodes]\label{thm:RKHS-optimal-fixed-nodes}
Let $H(K)$ be a reproducing kernel Hilbert space on $[0,1]^d$ with a real-valued symmetric positive definite kernel $K$ and assume integration $I_d$ as in~\eqref{eq:Id} is bounded on~$H(K)$.
Let the $n$ nodes $X = \{\bsx_1, \ldots, \bsx_n\} \subset [0,1]^d$ be given and consider (non-adaptive) linear rules $Q^{X,\bsw}_{d,n}$ of the form~\eqref{eq:Q-X-w}, then there exist weights $\bsw^* = (w_1^*, \ldots, w_n^*) \in \R^n$ such that the worst-case error equals
\begin{align}\label{eq:optimal-wce}
  e(Q^{X,\bsw^*}_{d,n}, H(K))
  &=
  \sup_{\substack{f \in H(K) \\ \|f\|_{H(K)} \le 1 \\ f(\bsx_i) = 0,\ i = 1, \ldots, n}}
  |I_d(f)|
  .
\end{align}
Moreover, every deterministic cubature algorithm, that may be adaptive and/or nonlinear, that uses at most $n$ function values has worst-case error at least that of an optimal linear nonadaptive rule on the nodes it samples on the zero function.

Define the matrix
\begin{align*}
  \bsK
  =
  \bsK_{K,X,n}
  &:=
  \begin{bmatrix}
    K(\bsx_i, \bsx_j)
  \end{bmatrix}_{i,j=1,\ldots,n}
  \quad\in\quad
  \R^{n \times n}
  ,
\end{align*}
and the moments vector
\begin{align*}
  \bsb
  =
  \bsb_{K,X,n}
  &:=
  \begin{bmatrix}
    \displaystyle
    \int_{[0,1]^d} K_{\bsx_i}(\bsy) \rd\bsy
  \end{bmatrix}_{i=1,\ldots,n}
  \quad\in\quad
  \R^n
  ,
\end{align*}
then the optimal weights $\bsw^* \in \R^n$ are any solution of
\begin{align*}
  \bsK \, \bsw^* = \bsb
  ,
\end{align*}
and this system always has at least one solution.

We can write
\begin{align}\label{eq:optimal-wce-explicit}
  e^2(Q^{X,\bsw^*}_{d,n}, H(K))
  &=
  \int_{[0,1]^d}\int_{[0,1]^d} K(\bsx, \bsy) \rd\bsx \rd\bsy
  -
  (\bsw^*)^\mathsf{T} \, \bsK \, \bsw^*
  .
\end{align}
\end{theorem}
\begin{proof}
For arbitrary real weights $\bsw \in \R^n$ consider the cubature rule $Q^{X,\bsw}_{d,n}$ in~\eqref{eq:Q-X-w}.
By~\eqref{eq:RKHS-wce-formula} from the previous lemma, its squared worst-case error can be written in terms of $\bsK$ and $\bsb$ as
\begin{align*}
  e(Q^{X,\bsw}_{d,n}, H(K))^2
  &=
  \int_{[0,1]^d} \int_{[0,1]^d} K(\bsx,\bsy) \rd\bsx \rd\bsy
  -
  2 \, \bsw^\mathsf{T} \, \bsb
  +
  \bsw^\mathsf{T} \, \bsK \, \bsw
  .
\end{align*}
Assume for the moment that $\bsK \, \bsw^* = \bsb$ has a solution.
For arbitrary $\bsw \in \R^n$, expanding the difference of the squared worst-case errors for $Q^{X,\bsw}_{d,n}$ and $Q^{X,\bsw^*}_{d,n}$, and immediately substituting $\bsb = \bsK \, \bsw^*$, gives
\begin{align*}
  e(Q^{X,\bsw}_{d,n}, H(K))^2 - e(Q^{X,\bsw^*}_{d,n},H(K))^2
  &=
  \bsw^\mathsf{T} \, \bsK \, \bsw
  -
  2 \, \bsw^\mathsf{T} \, \bsK \, \bsw^*
  +
  (\bsw^*)^\mathsf{T} \, \bsK \, \bsw^*
  \\
  &=
  (\bsw - \bsw^*)^\mathsf{T} \, \bsK \, (\bsw - \bsw^*)
  \\
  &\ge
  0
  ,
\end{align*}
where we used that $\bsK$ is a symmetric positive-semidefinite matrix.
Thus $Q^{X,\bsw^*}_{d,n}$ has minimal worst-case error among all linear cubature rules on the fixed nodes.

By Lemma~\ref{lem:RKHS-wce}, the error functional $I_d - Q^{X,\bsw^*}_{d,n}$ has Riesz representer in $H(K)$ as follows
\begin{align*}
  g^{X,\bsw^*}_{d,n}(\bsy)
  &:=
  h_d(\bsy) - q^{X,\bsw^*}_{d,n}(\bsy)
  =
  h_d(\bsy) - \sum_{i=1}^n w_i^* \, K_{\bsx_i}(\bsy)
  .
\end{align*}
If $\bsw^*$ satisfies $\bsK \, \bsw^* = \bsb$ then, for every $j = 1, \ldots, n$,
\begin{align*}
  g^{X,\bsw^*}_{d,n}(\bsx_j)
  &=
  b_j - \sum_{i=1}^n \bsK_{j,i} \, w_i^*
  =
  0
  .
\end{align*}
For every $f \in H(K)$ with $\|f\|_{H(K)} \le 1$ and that vanishes at the nodes in~$X$ we have $Q^{X,\bsw^*}_{d,n}(f) = 0$ and hence
\begin{align*}
  |I_d(f)|
  &=
  |I_d(f) - Q^{X,\bsw^*}_{d,n}(f)|
  =
  |\langle f, g^{X,\bsw^*}_{d,n} \rangle_{H(K)}|
  \journalbreak\le
  \|g^{X,\bsw^*}_{d,n}\|_{H(K)}
  =
  e(Q^{X,\bsw^*}_{d,n}, H(K))
  .
\end{align*}
We now need to prove this actually is equality and will follow similar steps as in the proof of Lemma~\ref{lem:RKHS-wce}.
The case $g^{X,\bsw^*}_{d,n} = 0$ is trivial.
If $g^{X,\bsw^*}_{d,n} \ne 0$ then
\begin{align*}
  f^*(\bsy)
  &:=
  \frac{g^{X,\bsw^*}_{d,n}(\bsy)}{\|g^{X,\bsw^*}_{d,n}\|_{H(K)}}
\end{align*}
belongs to the unit ball, vanishes at all the nodes in~$X$ and satisfies
\begin{align*}
  |I_d(f^*)|
  &=
  |I_d(f^*) - Q^{X,\bsw^*}_{d,n}(f^*)|
  =
  |\langle f^*, g^{X,\bsw^*}_{d,n} \rangle_{H(K)}|
  \journalbreak=
  \|g^{X,\bsw^*}_{d,n}\|_{H(K)}
  =
  e(Q^{X,\bsw^*}_{d,n}, H(K))
  .
\end{align*}
This proves the exact worst-case error formula~\eqref{eq:optimal-wce}.

We still need to show that even if the matrix $\bsK$ is singular the system $\bsK \, \bsw^* = \bsb$ is consistent.
We will show that $\bsb \in \operatorname{col}(\bsK)$.
Since $\bsK$ is symmetric, we have $\operatorname{null}(\bsK)^\perp = \operatorname{col}(\bsK^\mathsf{T}) = \operatorname{col}(\bsK)$.
Suppose that $\bsc \in \operatorname{null}(\bsK) \subseteq \R^n$ then the reproducing property and symmetry of the kernel give
\begin{align*}
  \Big\| \sum_{i=1}^n c_i \, K_{\bsx_i} \Big\|_{H(K)}^2
  &=
  \sum_{i,j=1}^n c_i \, c_j \, \langle K_{\bsx_i},K_{\bsx_j}\rangle_{H(K)}
  \\
  &=
  \sum_{i,j=1}^n c_i \, c_j \, K(\bsx_i,\bsx_j)
  \\
  &=
  \bsc^\mathsf{T} \, \bsK \, \bsc
  =
  0
  .
\end{align*}
Hence the function $\sum_{i=1}^n c_i \, K_{\bsx_i}(\bsy) = 0$ in~$H(K)$.
Now, for any $\bsc \in \operatorname{null}(\bsK)$, it follows that
\begin{align*}
  \bsc^\mathsf{T} \, \bsb
  &=
  \sum_{i=1}^n c_i \, \int_{[0,1]^d} K_{\bsx_i}(\bsy) \rd\bsy
  =
  \int_{[0,1]^d} \sum_{i=1}^n c_i \, K_{\bsx_i}(\bsy) \rd\bsy
  =
  0
  .
\end{align*}
Thus $\bsb$ is orthogonal to any vector in the null space $\operatorname{null}(\bsK)$ which means $\bsb \in \operatorname{null}(\bsK)^\perp = \operatorname{col}(\bsK^\mathsf{T}) = \operatorname{col}(\bsK)$, since $\bsK$ is symmetric, and thus a solution always exists.

We still need to compare an arbitrary deterministic algorithm, that may be nonlinear and/or adaptive, to the optimal linear algorithm.
Let us denote the arbitrary algorithm by $Q_{d,n}$.
We run the algorithm on the zero function and denote the resulting nodes where the zero function was evaluated by $X^0 = \{\bsx_1^0, \ldots, \bsx_m^0\}$, for some $m \le n$, and let's denote its output by $q := Q_{d,n}(0)$.
Since the algorithm is deterministic, every function $f \in H(K)$ that is zero at these nodes will also result in $Q_{d,n}(f) = Q_{d,n}(0) = q$.
Similarly for $-f$ we have $Q_{d,n}(-f) = Q_{d,n}(0) = q$.
If we have $\|f\|_{H(K)} \le 1$ then obviously also $\|-f\|_{H(K)} \le 1$, and the diameter-of-information argument from Section~\ref{sec:lower-bounds}, cf.~\eqref{eq:triangle-fp-fpp}, gives
\begin{align*}
  \max\{|I_d(f) - q|, |I_d(-f) - q|\}
  &\ge
  |I_d(f)|
  .
\end{align*}
We have
\begin{align*}
  e(Q_{d,n}, H(K))
  &\ge
  \sup_{\substack{f \in H(K) \\ \|f\|_{H(K)} \le 1 \\ f(\bsx_i^0) = 0,\ i = 1, \ldots, m}}
  |I_d(f)|
  .
\end{align*}
The right-hand side coincides with the error of the optimal rule $Q_{d,m}^{X^0,\bsw^*}$.
This algorithm is admissible for the same maximum budget of $n$ since $m \le n$.
Hence, for each $n$ we can construct such a nonadaptive linear rule by running $Q_{d,n}$ on the zero function and computing the optimal weights.
This finishes the proof.
\end{proof}

The previous theorem shows that optimal weights for any prescribed nodes can be computed by solving $\bsK \, \bsw^* = \bsb$, possibly numerically.
E.g., this could be done for purely randomly sampled points, and for our periodic Sobolev spaces this achieves the optimal algebraic convergence rate with high probability, see~\cite[Corollary~8.6]{KUV2021}.
However, compared to the simple Monte Carlo estimator we now will have generally unequal weights, which could be positive or negative.
On the other hand, it is interesting to know that for structured QMC point sets, in some compatible function space, see the next lemma, it is possible to compute these optimal weights explicitly, and even more, for lattice points they are all equal.
(The same is true for digital nets but we will not show that here.)
We now show this in the next lemma which holds for all our periodic spaces $F^{\alpha,\bsgamma}_r([0,1]^d)$ from Table~\ref{tbl:periodic-equivalent-norms} with $\alpha > 1/2$.
(The result could also be shown for general rank lattice points.)

\begin{lemma}[Optimal weights for lattice points]\label{lem:optimal-weights-lattice-points}
Let $\alpha > 1/2$ and let $\bsgamma = \{\gamma_\setu\}_{\setu \in \frakU}$ with $\gamma_\emptyset = 1$ and $\gamma_\setu > 0$ for all~$\setu$.
Let $r_{\alpha,\bsgamma} : \Z^d \to \R_{>0}$ be a Fourier multiplier such that $r_{\alpha,\bsgamma}^{-1} \in \ell_2(\Z^d)$ and $r_{\alpha,\bsgamma}^{-2}(\bszero) = \gamma_\emptyset = 1$, and let $F^{\alpha,\bsgamma}_r([0,1]^d)$ be the space with the norm defined in~\eqref{eq:def-F-d-r-norm},
\begin{align*}
    \| f \|_{F^{\alpha,\bsgamma}_r([0,1]^d)}^2
    &:=
    \sum_{\bsh \in \Z^d} |\hat{f}_\bsh|^2 \, r_{\alpha,\bsgamma}^2(\bsh)
    .
  \end{align*}
Assume $F^{\alpha,\bsgamma}_r([0,1]^d) = H(K)$ where $K$ is a real-valued symmetric positive-definite kernel, given, as in~\eqref{eq:Fourier-kernel}, by
\begin{align*}
  K(\bsx,\bsy)
  &:=
  \sum_{\bsh \in \Z^d}
  r_{\alpha,\bsgamma}^{-2}(\bsh) \,
  \exp(2\pi\rmi \, \bsh \cdot (\bsx-\bsy))
  ,
\end{align*}
and assume that integration $I_d$, see~\eqref{eq:Id}, is bounded on~$H(K)$.
Let $X = \{\bsx_0, \ldots, \bsx_{n-1}\}$ be a lattice point set as in~\eqref{eq:R1L-points}, where we index from zero to make the connection to the group structure clear.
Then one choice of optimal weights on~$X$ is
\begin{align}\label{eq:optimal-lattice-point-weights}
  w^*
  :=
  w_k^*
  &:=
  \frac1{1 + e^2(Q_{d,n,\bsz}, F^{\alpha,\bsgamma}_r([0,1]^d))} \, \frac1n
  ,
  \qquad
  k = 0, \ldots, n-1
  ,
\end{align}
and the resulting squared worst-case error is
\begin{align}\label{eq:optimal-lattice-point-error}
  e^2(Q^{X,\bsw^*}_{d,n}, F^{\alpha,\bsgamma}_r([0,1]^d))
  &=
  \frac{e^2(Q_{d,n,\bsz}, F^{\alpha,\bsgamma}_r([0,1]^d))}{1 + e^2(Q_{d,n,\bsz}, F^{\alpha,\bsgamma}_r([0,1]^d))}
  .
\end{align}
The weights $w^*_k = w^* < 1/n$ and, for good lattice rules, are asymptotically equivalent to~$1/n$.
\end{lemma}
\begin{proof}
For a lattice point set $X = \{\bsx_0, \ldots, \bsx_{n-1}\}$, where we number from $0$ to $n-1$ to stress the group structure and to be consistent with~\eqref{eq:R1L-points}, also for indexing $\bsb$ and $\bsK$, we have, for $i = 0,\ldots, n-1$,
\begin{align*}
  b_i
  =
  \int_{[0,1]^d} K_{\bsx_i}(\bsy) \rd\bsy
  &=
  \sum_{\bsh \in \Z^d} r_{\alpha,\bsgamma}^{-2}(\bsh) \, \int_{[0,1]^d} \exp(2\pi\rmi \, \bsh \cdot (\bsx_i - \bsy)) \rd\bsy
  \\
  &=
  \sum_{\bsh \in \Z^d} r_{\alpha,\bsgamma}^{-2}(\bsh) \, \exp(2\pi\rmi \, \bsh \cdot \bsx_i) \, \prod_{j=1}^d \bbone\{h_j = 0\}
  \\
  &=
  r_{\alpha,\bsgamma}^{-2}(\bszero)
  =
  \gamma_\emptyset
  =
  1
  ,
\end{align*}
and, for $i = 0,\ldots, n-1$, the $i$th row sum of the matrix $\bsK$ can be written as, using $\bsone_n = (1, \ldots, 1) \in \R^n$,
\begin{align*}
  (\bsK \, \bsone_n)_i
  =
  \sum_{j=0}^{n-1} K(\bsx_i, \bsx_j)
  &=
  \sum_{\bsh \in \Z^d} r_{\alpha,\bsgamma}^{-2}(\bsh) \, \sum_{j=0}^{n-1} \exp(2\pi\rmi \, \bsh \cdot (\bsx_i - \bsx_j))
  \\
  &=
  \sum_{\bsh \in \Z^d} r_{\alpha,\bsgamma}^{-2}(\bsh) \, \sum_{k=0}^{n-1} \exp(2\pi\rmi \, \bsh \cdot \bsx_k)
  \\
  &=
  n \, \sum_{\substack{\bsh \in \Z^d \\ \bsh \cdot \bsz \equiv 0 \tpmod{n}}} r_{\alpha,\bsgamma}^{-2}(\bsh)
  \\
  &=
  n \, \Big( 1 + \sum_{\substack{\bszero \ne \bsh \in \Z^d \\ \bsh \cdot \bsz \equiv 0 \tpmod{n}}} r_{\alpha,\bsgamma}^{-2}(\bsh) \Big)
  \\
  &=
  n \, \Big(1 + e^2(Q_{d,n,\bsz}, F^{\alpha,\bsgamma}_r([0,1]^d))\Big)
  ,
\end{align*}
where we used the lattice group structure $(\bsx_i - \bsx_j) \bmod{1} = \bsx_k$ with $k = (i - j) \bmod{n}$, see Section~\ref{sec:lattice-points}, and the fact that $j \mapsto (i-j) \bmod{n}$ is a bijection of~$\Z_n$, the lattice character property~\eqref{eq:lattice-character-prop} and the expression for the worst-case error of the standard lattice rule $Q_{d,n,\bsz}$ in the space $F^{\alpha,\bsgamma}_r([0,1]^d)$ as given in~\eqref{eq:F-r-wce}.
Note that the rule $Q_{d,n,\bsz} = Q^{X,\bsone_n/n}_{d,n}$ uses the same point set $X$ but is not the rule $Q^{X,\bsw^*}_{d,n}$ that we are trying to construct.
Since the vector $\bsw^*$ satisfies $\bsK \, \bsw^* = \bsb$ and we have $\bsb = \bsone_n$ and $\bsK \, \bsone_n = n \, (1 + e^2(Q_{d,n,\bsz}, F^{\alpha,\bsgamma}_r([0,1]^d))) \, \bsone_n$ we conclude that we can take all $w^*_j$ the same, for $j = 0, \ldots, n-1$,
\begin{align*}
  w^*_j
  =
  w^*
  &=
  \frac1{1 + e^2(Q_{d,n,\bsz}, F^{\alpha,\bsgamma}_r([0,1]^d))} \, \frac1n
  .
\end{align*}
Hence the optimal weights are slightly smaller than $1/n$ and converge to $1/n$ for $n \to \infty$, in the sense that $n \, w_j^* \to 1$, whenever the equal-weight lattice rules converge.

From~\eqref{eq:optimal-wce-explicit} we obtain
\begin{align*}
  e^2(Q^{X,\bsw^*}_{d,n}, F^{\alpha,\bsgamma}_r([0,1]^d))
  &=
  1
  -
  (\bsw^*)^\mathsf{T} \, \bsK \, \bsw^*
  =
  1 - (\bsw^*)^\mathsf{T} \, \bsone_n
  \\
  &=
  1 - (1 + e^2(Q_{d,n,\bsz}, F^{\alpha,\bsgamma}_r([0,1]^d)))^{-1} \, n^{-1} \, \bsone_n^\mathsf{T} \, \bsone_n
  \\
  &=
  \frac{e^2(Q_{d,n,\bsz}, F^{\alpha,\bsgamma}_r([0,1]^d))}{1 + e^2(Q_{d,n,\bsz}, F^{\alpha,\bsgamma}_r([0,1]^d))}
  .
  \qedhere
\end{align*}
\end{proof}

From the previous result we see that
\begin{align*}
  Q^{X,\bsw^*}_{d,n}
  &=
  \frac1{(1 + e^2(Q_{d,n,\bsz}, F^{\alpha,\bsgamma}_r([0,1]^d)))} \, Q_{d,n,\bsz}
  ,
\end{align*}
i.e., the optimal rule is just a rescaling.
This does mean that constant functions are not integrated exactly anymore.
To get this nice form we made use of the group structure of the point set and its compatibility with the reproducing kernel.
A completely equivalent statement could be made with the point set being a digital net and a Walsh series based function space.
(Optimal weights can also be computed for any other point set, even a point set fixed from random samples, but there need not be any structure in the system $\bsK \, \bsw^* = \bsb$ nor in~$\bsw^*$.)

\section*{Acknowledgements}

This set of lecture notes accompanies three lectures given at the CIME workshop on ``High-Dimensional Approximation: From Theoretical Foundations to Machine Learning and PDEs'' organised by Albert Cohen, Domenico Marinucci and Hendrik Speleers, September 23--27, 2024 at Cetraro, Italy.
My sincere thanks go to CIME, the organisers and also to the participants whose enthusiasm and knowledge have impressed me very much.

I would like to thank Jakob Dilen, Stefano Cartabia, Javier Bevia Ripoll, Laurence Wilkes and Weiwen Mo for discussions which form the basis of the contents of these notes and in particular Lisa Vandeweyer for carefully reading part of these notes.

I want to acknowledge support from the Research Foundation Flanders (FWO) for two grants.
The material grew under project grant FWO G091920N when discussing concepts with PhD students on this project.
In that respect I also want to acknowledge the KU~Leuven research council C1 grant 3E230463.
I further acknowledge support from the sabbatical grant FWO K801424N under which I had discussions with my colleagues at UNSW Sydney (Australia) and TU Chemnitz (Germany) concerning function spaces and proof techniques for function approximation.
The CIME summer school fell right after my visit to TU Chemnitz and it became clear that a proper understanding of where these function spaces come from was of huge interest.
The actual lectures contained some basic function spaces and motivation, but more was needed.
Therefore, I promised to expand on this topic in the lecture notes.
I hope that this chapter can serve as a citable reference for other researchers by bringing together results scattered throughout the literature in a consistent overview.
As AI tools become more widely used in research, I hope that such a source will also be helpful for checking the assumptions and context of results found through these tools.

Part of the Julia code originates from my lectures at the Workshop and Summer School on Applied Analysis organised by Franziska Nestler, Martin Stoll, Daniel Potts and Tino Ullrich at TU Chemnitz, Germany, September 18--22, 2023.

I used AI tools to assist with proofreading, correcting typos, helping with graphs and checking consistency throughout the text.
All responsibility remains with the author.
I proudly acknowledge my mistakes.

\bibliographystyle{abbrv}
\bibliography{mrabbrev,bib}

\end{document}